\documentclass[11pt, reqno]{amsart}
\usepackage{txfonts}
\usepackage{amsmath,amssymb,amsthm,mathrsfs,enumerate,bm,xcolor,multirow,pbox}
\usepackage{graphicx,color,framed,tikz,caption,subcaption}
\usepackage{enumitem}
\setlist{leftmargin=9mm}
\usepackage[colorlinks,linkcolor=black,citecolor=black,urlcolor=black]{hyperref}
\allowdisplaybreaks[4]
\numberwithin{equation}{section}
\newcommand{\N}{\mathbb{N}}
\newcommand{\R}{\mathbb{R}}

\newcommand{\pnorm}[2]{\lVert #1\rVert_{#2}}
\newcommand{\bigpnorm}[2]{\big\lVert#1\big\rVert_{#2}}
\newcommand{\biggpnorm}[2]{\bigg\lVert#1\bigg\rVert_{#2}}
\newcommand{\abs}[1]{\lvert#1\rvert}
\newcommand{\bigabs}[1]{\big\lvert#1\big\rvert}
\newcommand{\biggabs}[1]{\bigg\lvert#1\bigg\rvert}
\newcommand{\iprod}[2]{\langle#1,#2\rangle}
\newcommand{\bigiprod}[2]{\big\langle#1,#2\big\rangle}

\renewcommand{\epsilon}{\varepsilon}

\renewcommand{\d}[1]{\mathrm{d}#1}

\newcommand{\smallop}{\mathfrak{o}_{\mathbf{P}}}

\newcommand{\bigop}{\mathcal{O}_{\mathbf{P}}}

\newcommand{\smallo}{\mathfrak{o}}
\newcommand{\bigo}{\mathcal{O}}

\renewcommand{\hat}{\widehat}
\renewcommand{\tilde}{\widetilde}

\DeclareMathOperator{\E}{\mathbb{E}}
\DeclareMathOperator{\Prob}{\mathbb{P}}

\DeclareMathOperator{\tr}{tr}

\DeclareMathOperator{\cov}{Cov}

\DeclareMathOperator{\op}{op}

\DeclareMathOperator{\err}{\mathsf{err}}

\DeclareMathOperator{\prox}{\mathsf{prox}}

\let\limsup\relax
\DeclareMathOperator*\limsup{\overline{lim}}

\DeclareMathOperator*{\argmin}{arg\,min\,}

\newcommand{\beq}{\begin{equation}}
\newcommand{\eeq}{\end{equation}}
\newcommand{\beqa}{\begin{equation} \begin{aligned}}
\newcommand{\eeqa}{\end{aligned} \end{equation}}
\newcommand{\beqas}{\begin{equation*} \begin{aligned}}
\newcommand{\eeqas}{\end{aligned} \end{equation*}}

\newcommand{\bit}{\begin{itemize}}
	\newcommand{\eit}{\end{itemize}}
\newcommand{\bmat}{\begin{bmatrix}}
	\newcommand{\emat}{\end{bmatrix}}

\theoremstyle{definition}\newtheorem{problem}{Problem}[section]
\theoremstyle{definition}\newtheorem{definition}[problem]{Definition}
\theoremstyle{remark}

\theoremstyle{remark}\newtheorem{remark}{Remark}
\theoremstyle{definition}
\theoremstyle{plain}\newtheorem{theorem}[problem]{Theorem}
\theoremstyle{plain}
\theoremstyle{plain}\newtheorem{lemma}[problem]{Lemma}
\theoremstyle{plain}\newtheorem{proposition}[problem]{Proposition}
\theoremstyle{plain}
\theoremstyle{plain}

\AtBeginDocument{
	\def\MR#1{}
}

\begin{document}

\title[Entrywise dynamics and universality of GFOMs]{Entrywise dynamics and universality of general first order methods}
\thanks{The research of Q. Han is partially supported by NSF grant DMS-2143468.}

\author[Q. Han]{Qiyang Han}

\address[Q. Han]{
Department of Statistics, Rutgers University, Piscataway, NJ 08854, USA.
}
\email{qh85@stat.rutgers.edu}

\date{\today}


\begin{abstract}
General first order methods (GFOMs), including many variants of gradient descent and approximate message passing algorithms, constitute a broad class of iterative algorithms widely applied in modern statistical learning problems. Some GFOMs also serve as constructive proof devices, iteratively characterizing the empirical distributions of statistical estimators in the asymptotic regime of large system limits for any fixed number of iterations.

This paper develops a non-asymptotic, entrywise characterization of the dynamics for a general class of GFOMs. Our characterizations capture the precise stochastic behavior of each coordinate of the GFOM iterates, and more importantly, hold universally across a broad class of heterogeneous random matrix models. As a corollary, we provide the first non-asymptotic description of the empirical distributions of the GFOM iterates beyond Gaussian ensembles.

We demonstrate the utility of our general GFOM theory through two sets of applications. In the first application, we develop a new algorithmic approach to prove universality for general empirical risk minimizers. Specifically, we establish new entrywise universality for a broad class of regularized least squares estimators in the linear model, by controlling the entrywise error relative to a suitably constructed GFOM iterate. This algorithmic proof method also systematically improves averaged universality results for general regularized regression estimators in the linear model, and resolves the universality conjecture for (regularized) maximum likelihood estimators in the logistic regression model. In the second application, we obtain entrywise Gaussian approximations for a general class of gradient descent algorithms. Our approach provides non-asymptotic state evolution for the bias and variance of the algorithm along the iteration path, applicable even to non-convex loss functions.

The proof relies on a new recursive leave-k-out method that provides `almost' delocalization for the GFOM iterates, and their higher-order derivatives with respect to the underlying random matrix ensembles. Crucially, our method ensures the validity of entrywise universality for up to poly-logarithmic many iterations, which facilitates effective $\ell_2/\ell_\infty$ control between certain GFOM iterates and statistical estimators in applications.
\end{abstract}

\maketitle


\setcounter{tocdepth}{1}
\tableofcontents

\sloppy


\section{Introduction}

\subsection{Overview and main results}
General first order methods (GFOMs), a concept first introduced in \cite{celentano2020estimation}, encompass a broad class of iterative algorithms that include many modern first order optimization methods. In its symmetric version, starting with an initialization $z^{(0)}\in \R^n$, the GFOM generates a sequence of iterates $z^{(1)},z^{(2)},\ldots,$ according to the update rule
\begin{align}\label{eqn:intro_GFOM}
z^{(t)} = A \mathsf{F}_t(z^{(0)},\cdots,z^{(t-1)})+\mathsf{G}_t (z^{(0)},\cdots, z^{(t-1)})\in \R^n,\quad t=1,2\ldots.
\end{align}
Here $A \in \R^{n\times n}$ is a symmetric matrix that remains fixed throughout the iterations, and $\mathsf{F}_t, \mathsf{G}_t:\R^{n\times [0:t-1]}\to \R^n$ are sufficiently smooth functions that act row-wise on the $n$-dimensional vectors $z^{(0)},\ldots,z^{(t-1)}$. 

Important special examples of (\ref{eqn:intro_GFOM}) (and its asymmetric version) include projected, proximal, and stochastic gradient descent algorithms, along with its accelerated (Nesterov) or noisy (Langevin) versions, and a general class of approximate message passing (AMP) algorithms, cf. \cite{celentano2020estimation,celentano2021high,gerbelot2024rigorous}. These algorithms have been widely used in statistical learning in high dimensions.

Beyond their intrinsic interest as a general class of iterative algorithms, (\ref{eqn:intro_GFOM}) can also be used as a proof device to understand the behavior of statistical estimators. Specifically, it is now well understood that various statistical estimators can be iteratively approximated by a suitable AMP algorithm \cite{bayati2011dynamics,bayati2012lasso,donoho2016high,sur2019likelihood,sur2019modern,bu2021algorithmic,li2021minimum,jiang2025new}, and therefore understanding towards the properties of the underlying AMP algorithm, or more generally the GFOM, provides direct insights into the behavior of the statistical estimators at hand. 

For instance, suppose $A$ and $B$ are symmetric $n\times n$ random matrices, whose upper triangular entries are independent, mean $0$, variance $1/n$ and suitably light-tailed variables. A standard formulation asserts that for a sufficiently good test function $\psi:\R \to \R$ and a fixed iteration $t \in \N$, there exists a deterministic constant $\kappa\big(\{\mathsf{F}_s,\mathsf{G}_s\}_{s\leq t},\psi\big)\in \R$ such that, almost surely or in probability,
\begin{align}\label{eqn:intro_dynamics_avg}
\lim_{n\to \infty} \frac{1}{n}\sum_{j \in [n]} \psi\big(z_j^{(t)}(A)\big) =\lim_{n\to \infty} \frac{1}{n}\sum_{j \in [n]} \psi\big(z_j^{(t)}(B)\big) =\kappa\big(\{\mathsf{F}_s,\mathsf{G}_s\}_{s\leq t},\psi\big).
\end{align}
In particular, (\ref{eqn:intro_dynamics_avg}) has been verified for the AMP algorithms in \cite{bayati2011dynamics,bayati2015universality,berthier2020state,chen2021universality}, and for some of the GFOM (\ref{eqn:intro_GFOM}) in the context of the so-called dynamical mean-field theory \cite{celentano2021high,montanari2022statistically,gerbelot2024rigorous}.

The primary utility of (\ref{eqn:intro_dynamics_avg}) lies in providing precise characterizations of global quantities associated with the GFOM iterate $z^{(t)}$---such as the standard $\ell_2$ estimation error---in the asymptotic regime of large system limits with a fixed number of iterations. Consequently, when the GFOM iterate $z^{(t)}$ is used as a proof device to approximate the statistical estimator---denoted $\hat{z}$---of interest, one may expect similar asymptotic characterizations to hold for global functionals of $\hat{z}$ in the large system limit. An important feature of this program is that the behavior of the GFOM iterate $z^{(t)}$---in the averaged sense of (\ref{eqn:intro_dynamics_avg})---is universal within the class of the aforementioned general Wigner ensembles, and therefore such averaged universality property is expected to carry over to the statistical estimator of interest.

In many high-dimensional statistical learning problems, however, it is more important to understand both the finite-sample, non-asymptotic behavior of the statistical estimator $\hat{z}$, as well as the behavior of its low-dimensional components beyond mere global functionals of the entire vector $\hat{z}$. Unfortunately, the existing technical machinery and the proof method of (\ref{eqn:intro_dynamics_avg}) present major obstacles in understanding both these aspects of the behavior of the GFOM iterate $z^{(t)}$, and ultimately, of the statistical estimator $\hat{z}$ of interest.

The main goal of this paper is to provide a non-asymptotic, entrywise distributional theory for the GFOM in its most general form (\ref{eqn:intro_GFOM}), which addresses both challenges mentioned above for the existing theory (\ref{eqn:intro_dynamics_avg}). Our main abstract results are two-fold:
\begin{enumerate}
	\item[(R1)]  (\emph{Universality}). We show in Theorem \ref{thm:universality} that under regularity conditions on $\{\mathsf{F}_t,\mathsf{G}_t\}$ and `typical scenarios', for any sufficiently good $\psi:\R\to \R$,
	\begin{align}\label{eqn:intro_entrywise_universality}
	\max_{j \in [n]} \bigabs{\E\psi \big(z_j^{(t)}(A)\big)- \E \psi\big(z_j^{(t)}(B)\big)}\leq (C_0\log n)^{c_0 t^3}\cdot n^{-1/2}.
	\end{align}
    Here $A,B$ are two symmetric $n\times n$ matrices with independent, mean $0$, light-tailed variables with matching second moments on their upper triangles.
    \item[(R2)] (\emph{State evolution}). We show in Theorem \ref{thm:GFOM_se_sym} that again under regularity conditions on $\{\mathsf{F}_t,\mathsf{G}_t\}$ and `typical scenarios', for any sufficiently good $\psi:\R\to \R$, 
    \begin{align}\label{eqn:intro_entrywise_se}
    \max_{j \in [n]} \bigabs{\E\psi \big(z_j^{(t)}(\texttt{Gaussian})\big)- \E \psi\big(\Theta_t(\mathfrak{Z}_j^{[0:t]})\big)}\leq (C_0\log n)^{c_0 t^5}\cdot n^{-1/c_0^t}.
    \end{align}
    Here $\Theta_t: \R^{t+1}\to \R$ is an implicit mapping, and $\mathfrak{Z}_j^{[0:t]}$ is the $j$-th row of a centered Gaussian matrix $\mathfrak{Z}^{[0:t]}\in \R^{n\times [0:t]}$, both determined recursively via the state evolution in Definition \ref{def:GFOM_se_sym} ahead.
\end{enumerate}

An important distinction of (\ref{eqn:intro_entrywise_universality}) and (\ref{eqn:intro_entrywise_se}) lies in the permissible range of $t$. In particular, the poly-logarithmic range $t\ll (\log n/\log \log n)^{1/3}$ in (\ref{eqn:intro_entrywise_universality}) will play a crucial role for the application of this result into proving several universality results of empirical risk minimizers in Section \ref{section:ERM}. 

As a straightforward by-product of (\ref{eqn:intro_entrywise_universality})-(\ref{eqn:intro_entrywise_se}), we obtain a non-asymptotic version of (\ref{eqn:intro_dynamics_avg}) that provides explicit error bounds for arbitrary moments of the difference between the two averages therein. To the best of our knowledge, non-asymptotic characterizations, even in the averaged formulation (\ref{eqn:intro_dynamics_avg}), have been confined to the Gaussian setting in the literature, cf. \cite{rush2018finite,li2022non,cademartori2024non}.

\subsection{Applications}

To demonstrate the technical scope of our main results (\ref{eqn:intro_entrywise_universality})-(\ref{eqn:intro_entrywise_se}) in concrete applications, we consider a canonical statistical learning setting, where we observe i.i.d. $(A_i,Y_i)\in \R^n\times \R$'s from the standard linear model
\begin{align}\label{def:linear_model}
Y_i= A_i^\top \mu_0+\xi_i,\quad i\in [m],
\end{align}
where $\xi_i\in \R$'s play the role of measurement errors. We focus on the following class of empirical risk minimizers:
\begin{align}\label{def:ERM}
\hat{\mu} = \hat{\mu}(A)\in \argmin_{\mu \in \R^n } \bigg\{\sum_{i \in [m]} \mathsf{L}(Y_i-A_i^\top \mu)+ \sum_{j \in [n]} \mathsf{f}(\mu_j) \bigg\}.
\end{align}
Here $\mathsf{L}:\R\to \R$ is a loss function (not necessarily convex at this point), and  $\mathsf{f}:\R\to \R_{\geq 0}$ is a convex regularizer.

\subsubsection{Universality of empirical risk minimizers}

A significant recent line of statistical theory for $\hat{\mu}$ in (\ref{def:ERM}) has focused on characterizing its precise stochastic behavior under the Gaussian design $A$, cf. \cite{bayati2012lasso,donoho2016high,thrampoulidis2018precise,sur2019modern,miolane2021distribution,celentano2023lasso,han2023noisy}. While the exact Gaussianity condition on $A$ appears rather restrictive, numerical experiments strongly suggest `universality' of these results for general non-Gaussian designs, which has therefore attracted significant recent attention in the literature\footnote{A complete literature review will be deferred to Section \ref{subsection:universality_review}.}.

In the first application, we develop a general algorithmic approach to proving the universality of both $\hat{\mu}$ in (\ref{def:ERM}) and beyond. The underlying conceptual idea is straightforward: we construct a GFOM iterate $\{\mu^{(t)}\}$ that approximates $\hat{\mu}$, and leverage our universality theory (\ref{eqn:intro_entrywise_universality}) to propagate the universality of  $\{\mu^{(t)}\}$ to its `limit' $\hat{\mu}$. A version of this high-level strategy was implemented in \cite{dudeja2024spectral} to establish asymptotic, averaged universality—specifically in the sense of  (\ref{eqn:intro_dynamics_avg})---for a class regularized estimators in the linear model under the squared loss. Here, we significantly enhance the scope of this algorithmic proof method by utilizing our non-asymptotic, entrywise universality theory (\ref{eqn:intro_entrywise_universality}). 

As a first example to illustrate the power of our new approach, we prove entrywise universality of $\hat{\mu}$ for the squared loss $\mathsf{L}(x)=x^2/2$.  The crux of our new approach, is to design a suitable sequence of (proximal) gradient descent estimates $\{\mu^{(t)}\}$---a special case of GFOM---that iteratively approximate $\hat{\mu}$ in an entrywise sense. Then by using an asymmetric version of (\ref{eqn:intro_entrywise_universality}), we show in Theorem \ref{thm:universality_ls} that for sufficiently regular $\mathsf{f}$, the entrywise universality of the iterate $\mu^{(t)}$ in (\ref{eqn:intro_entrywise_universality}) can be upgraded to $\hat{\mu}$ itself: for any sufficiently good $\psi:\R \to \R$, 
\begin{align}\label{eqn:intro_ls_universality}
\max_{j \in [n]} \bigabs{\E\psi \big(\hat{\mu}_j(A)\big)- \E \psi\big(\hat{\mu}_j(B)\big)}\leq C_0 e^{-(\log n)^{1/4}/C_0}.
\end{align}
To the best of our knowledge, (\ref{eqn:intro_ls_universality}) presents the first entrywise universality result for $\hat{\mu}$ that goes beyond the exclusive focus on the universality for global functionals associated with $\hat{\mu}$ in the literature\footnotemark[1].

Interestingly, the proof of the entrywise universality for $\hat{\mu}$ in (\ref{eqn:intro_ls_universality}) intrinsically requires the validity of our non-asymptotic theory (\ref{eqn:intro_entrywise_universality}) that allows the iteration $t$ to grow with $n$ at a poly-logarithmic rate. Fundamentally, this need arises because converting the entrywise universality result from $\mu^{(t)}$ to $\hat{\mu}$ requires precise control of $\pnorm{\mu^{(t)}-\hat{\mu}}{\infty}$ on the order $\smallop(1)$. We prove such a sharp, high probability control $\pnorm{\mu^{(t)}-\hat{\mu}}{\infty}\lesssim \log n\cdot (1-\epsilon_0)^t$ via a `second-order' leave-one-out argument in Section \ref{section:proof_ERM_ls},  which specifically requires $t\gg \log \log n$ to achieve effective $\ell_\infty$ control between $\mu^{(t)}$ and $\hat{\mu}$.

The technical flexibility in our non-asymptotic universality theory for general GFOMs in (\ref{eqn:intro_entrywise_universality}) also proves useful within the formulation of the standard, averaged universality (\ref{eqn:intro_dynamics_avg}). In particular:
\begin{itemize}
	\item In Theorem \ref{thm:avg_universality_ERM}, we prove the averaged universality for $\hat{\mu}$ for a general convex loss-penalty pair $(\mathsf{L},\mathsf{f})$---including Lasso and regularized robust regression estimators---under a large class of heterogeneous random matrix models. Surprisingly, the simple universality proof method via controlling the global $\ell_2$ error between a proximal gradient descent iterate $\{\mu^{(t)}\}$ and $\hat{\mu}$, provides systematic improvements over existing universality results for $\hat{\mu}$ that are confined either to Wigner-type ensembles \cite{han2023universality}, or to strongly convex problems under the squared loss \cite{dudeja2024spectral}.
	\item Beyond the linear model (\ref{def:linear_model}), we prove in Theorem \ref{thm:universality_logistic} the averaged universality for a general class of regularized maximum likelihood estimators in logistic regression. This is achieved by compensating for the discontinuity in the loss function with a progressively smoothed proximal gradient descent iterate, obtained after a suitable number of iterations that grow with $n$. To the best of our knowledge, universality for (regularized) logistic regression estimators remains open since the seminal works of \cite{salehi2019impact,sur2019likelihood,sur2019modern,candes2020phase} that provide exact characterizations under Gaussian designs. 
\end{itemize} 

An important feature of the above universality results is that they are entirely detached from understanding the global existence, uniqueness and the stability of the solution to the mean-field system of equations for $\hat{\mu}$. These properties, known only for specific instances of the pair $(\mathsf{L},\mathsf{f})$ and the random matrix ensemble $A$,  are now well recognized as the primary theoretical challenges in understanding the high dimensional behavior of $\hat{\mu}$, cf. \cite{sur2019modern,miolane2021distribution,celentano2023lasso,han2023noisy,montanari2023generalization}. Consequently, we expect that the algorithmic proof method developed here---which entirely circumvents the need to analyze these properties, much like the approach in \cite{dudeja2024spectral}, while also offering significant technical flexibility by accommodating non-trivial growth in the iterations---will have broad applicability in establishing universality for other empirical risk minimizers.

\subsubsection{Entrywise dynamics of gradient descent iterates}
In the second application, we precisely characterize the entrywise dynamics of a general class of gradient descent algorithms aimed at finding a solution $\hat{\mu}$ of (\ref{def:ERM}), under possibly non-convex loss functions $\mathsf{L}$. To keep the presentation simple, we shall focus on the Ridge regularizer $\mathsf{f}(x)=\lambda x^2/2$, where $\lambda\geq 0$. Specifically, we examine a stochastic gradient descent version of the following iterative algorithm: for a step size $\eta>0$, let for $t=1,2,\ldots$, 
\begin{align}\label{def:ERM_grad_descent_general_loss}
\mu^{(t)} \equiv \mu^{(t-1)}-\eta\cdot  \big(-A^\top \mathsf{L}'(Y-A \mu^{(t-1)})+\lambda \mu^{(t-1)}\big),
\end{align}
with the initialization $\mu^{(0)}=0$ (for simplicity). Here $\mathsf{L}':\R\to \R$ is applied component-wise. Using an asymmetric version of the state evolution (\ref{eqn:intro_entrywise_se}), we show in Theorem \ref{thm:grad_descent} that (uniformly) for $\ell \in [n]$,
\begin{align}\label{eqn:intro_gd}
\mu^{(t)}_\ell-\mu_{0,\ell}\stackrel{d}{\approx} \mathcal{N}\Big(b^{(t)}_{\ell;\mathsf{GD}} \mu_{0,\ell}, \sigma^{2,(t)}_{\ell;\mathsf{GD}} \Big).
\end{align}
Here $\big(b^{(t)}_{\ell;\mathsf{GD}}\big)_{\ell \in [n]}\in \R^n, \big(\sigma^{2,(t)}_{\ell;\mathsf{GD}}\big)_{\ell \in [n]} \in \R_{\geq 0}^n$, defined in (\ref{def:grad_descent_key_para}) via a deterministic state evolution (cf. Definition \ref{def:grad_descent_se}), precisely quantify the entrywise bias and variance of the (stochastic) gradient descent algorithm $\{\mu^{(t)}\}$ in approximating the underlying, unknown signal $\mu_0$ in the linear model (\ref{def:linear_model}). 

A closely related line of research \cite{celentano2021high,gerbelot2024rigorous} examines the high-dimensional, asymptotic behavior of the gradient descent algorithm (\ref{def:ERM_grad_descent_general_loss}) and its continuum gradient flow version. Here our non-asymptotic, entrywise distributional characterizations (\ref{eqn:intro_gd}) pave the way for a deeper understanding of the actual, algorithmic behavior of gradient descent methods beyond the scope of these existing technical tools. For instance, we expect results of the type (\ref{eqn:intro_gd}) to be useful for, but not limited to, (i) tracking the precise impact of the step size $\eta>0$ in the convergence/divergence of $\{\mu^{(t)}\}$ and its accelerated/noisy versions, even for possibly non-convex loss function $\mathsf{L}$'s; (ii) systematically quantifying the effect of implicit regularization in gradient descent methods due to early stopping, beyond the existing approaches to the special squared loss $\mathsf{L}(x)=x^2/2$ via direct random matrix techniques \cite{ali2019continuous,ali2020implicit}, and (iii) providing a distributional foundation for algorithmic debiased statistical inference methods, as recently introduced in \cite{bellec2024uncertainty}. Detailed applications along these directions fall beyond the scope of this paper and will be pursued elsewhere; see, e.g. the follow-up work \cite{han2024gradient} in the direction (iii).

\subsection{Proof techniques}

The method of proof for the entrywise universality (\ref{eqn:intro_entrywise_universality}) differs substantially from existing methods aimed at proving the averaged universality for the special case of the AMP algorithm. For instance, \cite{bayati2015universality,dudeja2023universality,wang2024universality} used the method of moments that involve intricate combinatorial calculations; \cite{chen2021universality} used Gaussian interpolation techniques, coupled with combinatorial estimates that control arbitrary moments of the derivatives in the sense that $\max_{k \in [n]}\E^{1/p}\abs{\partial^m_\cdot z^{(t)}_k)}^p\lesssim 1$ for a general order $m\in \N$. 

Here our approach is based on a variant of Chatterjee's version of the Lindeberg principle \cite{chatterjee2006generalization}, coupled with `almost' delocalization of the GFOM iterate $z^{(t)}\in \R^n$, its derivatives $\partial_{ij} z^{(t)},\partial_{ij}^2 z^{(t)}\in \R^n$ ($\partial_{ij}\equiv \partial/\partial A_{ij}$) and certain other higher-order interactions. In particular, we prove in Propositions \ref{prop:loo_l2_bound}-\ref{prop:der_cross_cubic} the following estimate under `typical scenarios': with probability at least $1-C_0 n^{-D}$, 
\begin{align}\label{eqn:intro_delocalization}
&\pnorm{z^{(t)}}{\infty}+ \max_{i,j\in [n]}\Big\{ \Big(\max_{k\in [n]\setminus \{i,j\}} n^{1/2} + \max_{k\in \{i,j\} }\Big) \bigabs{\partial_{ij} z^{(t)}_k}\vee \bigabs{\partial_{ij}^2 z^{(t)}_k} \Big\}\nonumber\\
&\qquad\qquad\qquad\qquad + n^{1/2}\max_{k \in [n]}\biggabs{\sum_{i,j\in [n]} A_{ij}^3 \partial_{ij}^3 z^{(t)}_k} \leq (C_0\log n)^{c_0 t^3}.
\end{align}
A concrete example of GFOM that shows the optimality of the estimate in (\ref{eqn:intro_delocalization}) (up to logarithmic factors) is provided in Remark \ref{rmk:deloc_der} ahead. Interestingly, already in the AMP setting, the second term in our delocalization estimate in (\ref{eqn:intro_delocalization}) provides a crucial, second-order non-asymptotic refinement to the derivative estimates obtained in \cite{chen2021universality} that holds beyond Gaussian ensembles.

The proof of the above delocalization estimate (\ref{eqn:intro_delocalization}) is based on a recursive leave-k-out method that iteratively reduces the size of the complicated summation in the formulae of the derivatives $\partial_{ij}^{(\cdot)} z^{(t)}$ (cf. Lemma \ref{lem:z_der_form}). In extreme synthesis, this method in the simplest form provides an inductive scheme to bound the complicated summation term
\begin{align}\label{eqn:intro_A_hadamard}
&\sum_{\ell_1,\ldots,\ell_r \in [n]} A_{k,\ell_1}A_{\ell_1,\ell_2}\cdots A_{\ell_{r-1},\ell_r}A_{\ell_r, i} \cdot \mathsf{H}_{t_1}(z_{\ell_1}^{(t_1)})\cdots \mathsf{H}_{t_r}(z_{\ell_r}^{(t_r)})
\end{align}
of size $r$, by those similar summation terms of size at most $r-1$. The major difficulty in controlling (\ref{eqn:intro_A_hadamard}) lies in the complicated dependence between the GFOM iterate $\{z^{(t)}\}$ and $A$. The leave-k-out method is used to formally establish that (\ref{eqn:intro_A_hadamard}) does not change significantly in order, when $\{A_{k,\ell_1}, A_{\ell_r, i}: \ell_1,\ell_r \in [n]\}$ are replaced by their independent copies. Once this is verified, a reduction in the size of (\ref{eqn:intro_A_hadamard}) can then be achieved through the concentration of linear and quadratic forms. 

The leave-k-out method discussed above is conceptually related to the leave-one-out approach in \cite{han2023universality}, which is used to establish delocalization estimates for statistical estimators. The key distinction is that the leave-one-out method in \cite{han2023universality} is applied only once to the statistical estimator of interest, whereas the leave-k-out method in this paper must be implemented recursively to reduce the complexity of (\ref{eqn:intro_A_hadamard}).

On the other hand, the key to the proof of the entrywise state evolution (\ref{eqn:intro_entrywise_se}) is to relate the GFOM iterate to an underlying AMP iterate via the implicit functions $\{\Theta_t\}$ in Definition \ref{def:GFOM_se_sym} ahead. This idea appeared in \cite{celentano2020estimation,celentano2021high,dudeja2024spectral} in an asymptotic form for (Gaussian) i.i.d. design matrices; here we provide a non-asymptotic correspondence in the case of random matrix ensembles with a general variance profile. Once the correspondence between the GFOM and an underlying AMP is established, we may then leverage the recent entrywise distribution theory for the AMP (in the Gaussian setting) developed in \cite{bao2025leave} to prove (\ref{eqn:intro_entrywise_se}).

\subsection{Further related literature}\label{subsection:universality_review}

Universality for various global functionals of statistical estimators has been intensively investigated in the literature. We refer the readers to \cite{korada2011applications,elkaroui2013asymptotic,montanari2017universality,panahi2017universal,elkaroui2018impact,oymak2018universality,abbasi2019universality,han2023universality,dudeja2024spectral} for a number of settings arising from the linear model, and \cite{montanari2022universality,gerace2022gaussian,montanari2023universality,hu2023universality} for recent universality results for training/test errors arising from generalized/non-linear models. 

From a technical perspective, the relationship between delocalization and universality of statistical estimators is explored in \cite{han2023universality,han2023distribution} within a related theoretical framework under the so-called Convex Gaussian Min-Max Theorem (CGMT), cf. \cite{gordon1985some,stojnic2013framework,thrampoulidis2018precise,miolane2021distribution,celentano2023lasso,han2023noisy,montanari2023generalization}. In essence, this approach asserts averaged universality of statistical estimators resulted from convex optimization problems, provided simultaneous delocalization is verified for both the primal and dual variables.

The technical approach of this paper draws inspiration from this line of works, but proves a much stronger phenomenon: delocalization of the GFOM iterate, as described in (\ref{eqn:intro_delocalization}), ultimately leads to the entrywise universality (\ref{eqn:intro_entrywise_universality}). When specialized to AMP iterates associated with convex problems, these delocalization results can be roughly understood as approximate verification of the simultaneous primal-dual delocalization condition required in the CGMT universality framework in \cite{han2023universality}. Interestingly, while our constructive approach iteratively captures the entrywise universality of the statistical estimator, this goal falls beyond the scope of the CGMT machinery, as the fluctuation of the cost optimum is usually much larger than the effect of individual coordinates.

In a very different direction, the stochastic behavior of various statistical estimators under rotational invariant $A$'s has been studied in \cite{gerbelot2020asymptotic,gerbelot2020asymptotic2,fan2022approximate}, with further asymptotic, averaged universality results obtained in \cite{dudeja2023universality,dudeja2024spectral,wang2024universality} via the method of moments. While some form of the leave-one-out method is known to be effective for rotational invariant ensembles \cite{bao2017local}, it remains open to extend these ideas to provide non-asymptotic entrywise dynamics and universality results under the prescribed random matrix models.

Finally, as mentioned above, our entrywise state evolution theory (\ref{eqn:intro_entrywise_se}) builds upon the recent non-asymptotic entrywise theory for AMP developed in \cite{bao2025leave}. Additionally, several related asymptotic entrywise results for AMP-type algorithms, varying in generality, can be found in \cite{chen2021convergence,hachem2024approximate,gufler2025stein}.

\subsection{Organization}
The rest of the paper is organized as follows. Section \ref{section:main_results_sym} provides formal statements for the entrywise universality (\ref{eqn:intro_entrywise_universality}) and state evolution (\ref{eqn:intro_entrywise_se}) for the GFOM iterate, and analogous results for an asymmetric version are stated in Section \ref{section:main_results_asym}. Section \ref{section:ERM} is devoted to the application to universality of regularized estimators in both linear and logistic regression models. Section \ref{section:grad_des} details the entrywise dynamics characterization (\ref{eqn:intro_gd}) for a stochastic gradient descent version of  (\ref{def:ERM_grad_descent_general_loss}). All proofs are then presented in Sections 
\ref{section:proof_universality}-\ref{section:proof_grad_descent}.

\subsection{Notation}
For any two integers $m,n$, let $[m:n]\equiv \{m,m+1,\ldots,n\}$, $(m,n]\equiv \{m+1,\ldots,n\}$ and $[m,n)\equiv \{m,m+1,\ldots, n-1\}$. We sometimes write for notational convenience $[n]\equiv [1:n]$. When $m>n$, it is understood that $[m:n]=\emptyset$.  

For $a,b \in \R$, $a\vee b\equiv \max\{a,b\}$ and $a\wedge b\equiv\min\{a,b\}$. For $a \in \R$, let $a_\pm \equiv (\pm a)\vee 0$. For a multi-index $a \in \mathbb{Z}_{\geq 0}^n$, let $\abs{a}\equiv \sum_{i \in [n]}a_i$. For $x \in \R^n$, let $\pnorm{x}{p}$ denote its $p$-norm $(0\leq p\leq \infty)$, and $B_{n;p}(R)\equiv \{x \in \R^n: \pnorm{x}{p}\leq R\}$. We simply write $\pnorm{x}{}\equiv\pnorm{x}{2}$ and $B_n(R)\equiv B_{n;2}(R)$. For $x \in \R^n$, let $\mathrm{diag}(x)\equiv (x_i\bm{1}_{i=j})_{i,j \in [n]} \in \R^{n\times n}$.

For a matrix $M \in \R^{m\times n}$, let $\pnorm{M}{\op},\pnorm{M}{F}$ denote the spectral and Frobenius norm of $M$, respectively. $I_n$ is reserved for an $n\times n$ identity matrix, written simply as $I$ (in the proofs) if no confusion arises. For two matrices $M,N \in \R^{m\times n}$ of the same size, let $M\circ N =(M_{ij}N_{ij})\in \R^{m\times n}$ be their Hadamard product.

We use $C_{x}$ to denote a generic constant that depends only on $x$, whose numeric value may change from line to line unless otherwise specified. $a\lesssim_{x} b$ and $a\gtrsim_x b$ mean $a\leq C_x b$ and $a\geq C_x b$, abbreviated as $a=\bigo_x(b), a=\Omega_x(b)$ respectively;  $a\asymp_x b$ means $a\lesssim_{x} b$ and $a\gtrsim_x b$. $\bigo$ and $\smallo$ (resp. $\mathcal{O}_{\mathbf{P}}$ and $\mathfrak{o}_{\mathbf{P}}$) denote the usual big and small O notation (resp. in probability). By convention, sum and product over an empty set are understood as $\Sigma_{\emptyset}(\cdots)=0$ and $\Pi_{\emptyset}(\cdots)=1$. 

For a random variable $X$, we use $\Prob_X,\E_X$ (resp. $\Prob^X,\E^X$) to indicate that the probability and expectation are taken with respect to $X$ (resp. conditional on $X$).

For $\Lambda>0$ and $\mathfrak{p}\in \N$, a measurable map $f:\R^n \to \R$ is called \emph{$\Lambda$-pseudo-Lipschitz of order $\mathfrak{p}$} iff 
\begin{align}\label{cond:pseudo_lip}
\abs{f(x)-f(y)}\leq \Lambda\cdot  (1+\pnorm{x}{}+\pnorm{y}{})^{\mathfrak{p}-1}\cdot\pnorm{x-y}{},\quad \forall x,y \in \R^{n}.
\end{align}
Moreover, $f$ is called \emph{$\Lambda$-Lipschitz} iff $f$ is $\Lambda$-pseudo-Lipschitz of order $1$, and in this case we often write $\pnorm{f}{\mathrm{Lip}}\leq L$, where $\pnorm{f}{\mathrm{Lip}}\equiv \sup_{x\neq y} \abs{f(x)-f(y)}/\pnorm{x-y}{}$. For a proper, closed convex function $f$ defined on $\R^n$, its \emph{proximal operator} $\prox_f(\cdot)$  is defined by $\prox_f(x)\equiv \argmin_{z \in \R^n} \big\{\frac{1}{2}\pnorm{x-z}{}^2+f(z)\big\}$.

\section{General first order methods: symmetric case}\label{section:main_results_sym}

\subsection{Formal setup and examples}

Recall the GFOM in (\ref{eqn:intro_GFOM}), where $A \in \R^{n\times n}$ is a symmetric random matrix, $\mathsf{F}_t, \mathsf{G}_t:\R^{n\times [0:t-1]} \to \R^n$, and $z^{(0)}\in \R^n$ is a deterministic or random initialization independent of $A$. For notational simplicity we write
\begin{align*}
z^{([0:t-1])}\equiv (z^{(0)},z^{(1)},\cdots,z^{(t-1)})\in \R^{n\times [0:t-1]},
\end{align*}
so the GFOM (\ref{eqn:intro_GFOM}) can be written compactly as 
\begin{align}\label{def:GFOM_sym}
z^{(t)}&=A \mathsf{F}_t(z^{([0:t-1])})+\mathsf{G}_t(z^{([0:t-1])}).
\end{align}
For $z^{([0:t-1])} \in \R^{n\times [0:t-1]}$, we write $z^{([0:t-1])}_\ell$ as the $\ell$-th row of $z^{([0:t-1])}$. We assume that $\{\mathsf{F}_t,\mathsf{G}_t\}$ are \emph{row-separate}, in the sense that for some measurable functions $\{\mathsf{F}_{t,\ell},\mathsf{G}_{t,\ell}: \R^{[0:t-1]} \to \R\}_{\ell \in [n]}$,
\begin{align}\label{def:row_sep}
\mathsf{F}_t(z^{([0:t-1])})=\big(\mathsf{F}_{t,\ell}(z_\ell^{([0:t-1])})\big)_{\ell \in [n]},\quad \mathsf{G}_t(z^{([0:t-1])})=\big(\mathsf{G}_{t,\ell}(z_\ell^{([0:t-1])})\big)_{\ell \in [n]}.
\end{align}
We shall work with the normalization that the variance for the entries of $A$ is of order $1/n$, so we typically expect $\pnorm{z^{(t)}}{}/\sqrt{n}=\bigop(1)$. 

Some canonical examples for the symmetric GFOM in (\ref{def:GFOM_sym}) include:
\begin{enumerate}
	\item (\emph{Power iteration}). The power iteration method (without normalization) can be identified by setting 
	\begin{align*}
	\mathsf{F}_t(z^{([0:t-1])}) \equiv z^{(t-1)},\, \mathsf{G}_t(z^{([0:t-1])})\equiv 0.
	\end{align*}
	\item (\emph{AMP algorithms}). The standard AMP algorithm in the symmetric form (cf. \cite{bayati2011dynamics,bayati2015universality,bao2025leave}) can be recovered by setting
	\begin{align*}
	&\mathsf{F}_t(z^{([0:t-1])})\equiv \mathfrak{F}_t(z^{(t-1)}),\\ &\mathsf{G}_t(z^{([0:t-1])})\equiv - \big[\mathscr{V}_A \E \mathfrak{F}_t'(z^{(t-1)})\big]\circ \mathfrak{F}_{t-1}(z^{(t-2)}). 
	\end{align*}
	Here $\{\mathfrak{F}_{t}:\R \to \R\}_{t\geq 1}$ is class of sufficiently smooth functions applied coordinate-wise with $\mathfrak{F}_0(\cdot)\equiv 0$, and $A$ is a generalized Wigner ensemble with a variance profile matrix $\mathscr{V}_A\equiv \E A\circ A \in \R_{\geq 0}^{n\times n}$. 
\end{enumerate}
More examples of relevance in applications in this paper will be detailed after the introduction of the asymmetric GFOM in Section \ref{section:main_results_asym} ahead.

\subsection{Universality} 

The following theorem establishes a precise, non-asymptotic version of the entrywise universality (\ref{eqn:intro_entrywise_universality}) for the symmetric GFOM in (\ref{def:GFOM_sym}).

\begin{theorem}\label{thm:universality}
	Fix $t \in \N$ and $n \in \N$. Suppose the following hold:
	\begin{enumerate}
		\item[(U1)] $A=A_0/\sqrt{n}$, $B=B_0/\sqrt{n}$, where (i) $A_0, B_0$ are symmetric $n\times n$ random matrices, (ii) the entries of its upper triangle are independent mean $0$ variables, and (iii) for all $i,j \in [n]$, $\E A_{0,ij}^2=\E B_{0,ij}^2$ and $\pnorm{A_{0,ij}}{\psi_2}\vee \pnorm{B_{0,ij}}{\psi_2}\leq K$ for some $K\geq 2$.
		\item[(U2)] For all $s \in [t],\ell \in [n]$, $\mathsf{F}_{s,\ell},\mathsf{G}_{s,\ell} \in C^3(\R^{[0:s-1]})$. Moreover, there exists some $\Lambda\geq 2$ and $\mathfrak{p}\in \N$ such that 
		\begin{align*}
		\max_{s \in [t]}\max_{\mathsf{E}_s \in \{\mathsf{F}_s,\mathsf{G}_s\}}\max_{\ell \in [n]}\Big\{\pnorm{\mathsf{E}_{s,\ell}}{\mathrm{Lip}}+\max_{a\in \mathbb{Z}_{\geq 0}^{[0:s-1]}, \abs{a}\leq 3} \bigpnorm{(1+\pnorm{\cdot}{})^{-\mathfrak{p}}\abs{\partial^a \mathsf{E}_{s,\ell}(\cdot)}}{\infty}\Big\}\leq \Lambda.
		\end{align*}
	\end{enumerate}
	Then for any $\emptyset \neq S\subset [n]$, and any $\Psi \in C^3(\R^{\abs{S}\times t})$ satisfying
	\begin{align}\label{cond:Psi}
	\max_{a \in \mathbb{Z}_{\geq 0}^{\abs{S}\times t}, \abs{a}\leq 3} \sup_{x \in \R^{\abs{S}\times t }}\bigg(\sum_{(s,\tau) \in [\abs{S}]\times [t]}(1+\abs{x_{s,\tau}})^{\mathfrak{p}}\bigg)^{-1} \abs{\partial^a \Psi(x)} \leq \Lambda_{\Psi}
	\end{align}
    for some $\Lambda_{\Psi}\geq 2$, it holds for some constant $c_1=c_1(\mathfrak{p})>0$ that with $ \E^{(0)} \equiv \E\big[\cdot | z^{(0)}\big]$,
	\begin{align*}
	& \bigabs{\E^{(0)} \Psi\big(z_S^{([t])}(A)\big)-\E^{(0)} \Psi\big(z_S^{([t])}(B)\big) } \\
	&\qquad \leq \abs{S}^3 \Lambda_\Psi \cdot \big(K\Lambda \log n\cdot (1+\pnorm{z^{(0)}}{\infty})\big)^{c_1 t^3} \cdot n^{-1/2}.
	\end{align*}
\end{theorem}

We note that the dependence on the dimension $n$ in the above theorem is optimal up to a multiplicative logarithmic factor. On the other hand, the dependence on $t$ in the above theorem typically allows the number of iteration to grow as large as $t\ll (\log n/\log \log n)^{1/3}$. While this growth rate may be further improved with additional structural assumptions on $\{\mathsf{F}_\cdot,\mathsf{G}_\cdot\}$, it appears sufficiently powerful for our applications to empirical risk minimization problems in Section \ref{section:ERM} ahead. 

We also note that although the above theorem allows $\abs{S}$ to grow polynomially with $n$, it remains open to examine the optimal dependence with respect to $\abs{S}$. Moreover, the constant $2$ in the requirement $K,\Lambda,\Lambda_{\Psi}\geq 2$ can be replaced by any constant larger than $1$; the same convention applies to all other theorems.

As entrywise universality is stronger than averaged universality, using Theorem \ref{thm:universality} above, we may also provide a non-asymptotic estimate for the universality claim in (\ref{eqn:intro_dynamics_avg}). 

\begin{theorem}\label{thm:universality_avg}
Suppose (U1) in Theorem \ref{thm:universality} holds, and (U2) is replaced by
\begin{enumerate}
	\item[(U2)'] 
	$
	\max\limits_{s \in [t]}\max\limits_{\mathsf{E}_s \in \{\mathsf{F}_s,\mathsf{G}_s\}}\max\limits_{\ell \in [n]}\big\{\pnorm{\mathsf{E}_{s,\ell}}{\mathrm{Lip}}+\abs{\mathsf{E}_{s,\ell}(0)}\big\}\leq \Lambda$ for some $\Lambda\geq 2$.
\end{enumerate}
Fix a sequence of $\Lambda_\psi$-pseudo-Lipschitz functions $\{\psi_k:\R^t \to \R\}_{k \in [n]}$ of order $\mathfrak{p}$, where $\Lambda_\psi\geq 2$. Then for any $q \in \N$, there exists some $C_0=C_0(\mathfrak{p},q)>0$ such that with $ \E^{(0)} \equiv \E\big[\cdot | z^{(0)}\big]$,
\begin{align*}
& \E^{(0)} \biggabs{\frac{1}{n}\sum_{k \in [n]} \psi_k\big(z_k^{([t])}(A)\big) - \frac{1}{n}\sum_{k \in [n]}  \psi_k\big(z_k^{([t])}(B)\big)  }^q\\
&\leq \big(K\Lambda\Lambda_\psi\log n\cdot (1+\pnorm{z^{(0)}}{\infty})\big)^{C_0 t^3}\cdot n^{-1/(C_0t^3)}. 
\end{align*}
\end{theorem}

Asymptotic versions of averaged universality have been previously obtained through moment calculations in \cite{chen2021universality,dudeja2023universality,wang2024universality}. However, these techniques do not provide non-asymptotic estimates as presented in our Theorem \ref{thm:universality_avg} above. To the best of our knowledge, this theorem provides the first non-asymptotic estimates for averaged universality for the general class of GFOMs in (\ref{def:GFOM_sym}).

We note that the worsened dependence on $n$ in the above theorem is primarily due to a smoothing argument to weaken the regularity condition (U2) in Theorem \ref{thm:universality} to the much weaker (U2)'. This leads to a slightly worsened condition $t\ll (\log n/\log \log n)^{1/6}$ to guarantee a vanishing error for the averaged universality.

\subsection{State evolution}

The goal of this subsection is to provide a deterministic description of the quantities involved in Theorems \ref{thm:universality} and \ref{thm:universality_avg}, at a possibly worsened error estimate. Such deterministic description is known as \emph{state evolution} for the special class of AMP algorithms in an averaged sense \cite{bayati2011dynamics,berthier2020state}, and more recently in an entrywise sense \cite{bao2025leave}. 

For the symmetric GFOM in (\ref{def:GFOM_sym}), such state evolution is iteratively described in the following definition by two objects, namely, (i) a row-separate map $\Theta_t: \R^{n\times [0:t]} \to \R^{n\times [0:t]}$, and (ii) a centered $n\times \infty$ Gaussian matrix $\mathfrak{Z}^{([1:\infty))}\in \R^{n\times [1:\infty)}$.

\begin{definition}\label{def:GFOM_se_sym}
	Initialize with $\Theta_0\equiv \mathrm{id}(\R^n)$ and $\mathfrak{Z}^{(0)}\equiv z^{(0)}$. For $t=1,2,\ldots$, with $\E^{(0)}\equiv \E\big[\cdot | \mathfrak{Z}^{(0)}\big]$, execute the following steps:
	\begin{enumerate}
		\item  Let $\Theta_t: \R^{n\times [0:t]} \to \R^{n\times [0:t]}$ be defined as follows: for $w \in [0:t-1]$, let $\big[\Theta_t(\mathfrak{z}^{([0:t])})\big]_{\cdot,w}\equiv [\Theta_w(\mathfrak{z}^{([0:w])})]_{\cdot,w}$, and for $w=t$,
		\begin{align*}
		\big[\Theta_t(\mathfrak{z}^{([0:t])})\big]_{\cdot,t} \equiv  \mathfrak{z}^{(t)}+ \sum_{s \in [1:t-1]} \mathfrak{b}_s^{(t)}\circ \mathsf{F}_{s}\big(\Theta_{s-1}(\mathfrak{z}^{([0:s-1])})\big)+ \mathsf{G}_t\big(\Theta_{t-1}(\mathfrak{z}^{([0:t-1])})\big),
		\end{align*} 
		where for $s \in [1:t-1]$, the coefficient vector $\mathfrak{b}_s^{(t)}\in \R^n$ is defined by 
		\begin{align*}
		\mathfrak{b}_{s,k}^{(t)} \equiv \sum_{\ell \in [n]} \E A_{k\ell}^2\cdot  \E^{(0)}\partial_{\mathfrak{Z}_\ell^{(s)} } \big(\mathsf{F}_{t,\ell}\circ \Theta_{t-1,\ell} \big) \big(\mathfrak{Z}^{([0:t-1])}_\ell\big),\quad k \in [n].
		\end{align*}
		\item Let the Gaussian law of $\mathfrak{Z}^{(t)}$ be determined via the following correlation specification: for $s \in [1:t]$ and $k \in [n]$,
		\begin{align*}
		\cov(\mathfrak{Z}^{(t)}_k, \mathfrak{Z}^{(s)}_k)\equiv \sum_{\ell \in [n]} \E A_{k\ell}^2 \cdot \E^{(0)}\prod_{\tau \in \{t,s\}} \mathsf{F}_{\tau,\ell} \big(\Theta_{\tau-1,\ell}(\mathfrak{Z}^{([0:\tau-1])}_\ell)\big).
		\end{align*}
		Here $\mathfrak{Z}^{([0:\tau])}_\ell\equiv (\mathfrak{Z}^{(0)}_\ell,\mathfrak{Z}^{(1)}_\ell,\ldots,\mathfrak{Z}^{(\tau)}_\ell)$, and $\Theta_{\tau,\ell}:\R^{[0:\tau]} \to \R^{[0:\tau]}$ is the $\ell$-th row of the map $\Theta_{\tau}: \R^{n\times [0:\tau]}\to \R^{n\times [0:\tau]}$. 
	\end{enumerate}
\end{definition}

Using the state evolution parameters $(\{\Theta_t\}, \mathfrak{Z}^{([0:\infty))})$ in the above Definition \ref{def:GFOM_se_sym}, we may establish an entrywise distribution theory for the symmetric GFOM in (\ref{def:GFOM_sym}); its proof can be found in Section \ref{section:proof_GFOM_se_sym}.

\begin{theorem}\label{thm:GFOM_se_sym}
	Fix $t \in \N$ and $n \in \N$. Suppose the following hold:
	\begin{enumerate}
		\item[(D1)] $A\equiv A_0/\sqrt{n}$, where $A_0$ is a symmetric matrix whose upper triangle entries are independent mean $0$ variables such that $\max_{i,j \in [n]}\pnorm{A_{0,ij}}{\psi_2}\leq K$ holds for some $K\geq 2$.
		\item[(D2)] For all $s \in [t],\ell \in [n]$, $\mathsf{F}_{s,\ell},\mathsf{G}_{s,\ell} \in C^3(\R^{[0:s-1]})$. Moreover, there exists some $\Lambda\geq 2$ such that 
		\begin{align*}
		\max_{s \in [t]}\max_{\mathsf{E}_s \in \{\mathsf{F}_s,\mathsf{G}_s\}}\max_{\ell \in [n]}\Big\{\abs{\mathsf{E}_{s,\ell}(0)}+\max_{0\neq a \in \mathbb{Z}_{\geq 0}^{[0:s-1]}, \abs{a}\leq 3} \pnorm{ \partial^a \mathsf{E}_{s,\ell} }{\infty}\Big\}\leq \Lambda.
		\end{align*}
	\end{enumerate}
	Then for any $\emptyset \neq S\subset [n]$, and any $\Psi \in C^3(\R^{\abs{S}\times t})$ satisfying (\ref{cond:Psi}) for some $\Lambda_{\Psi}\geq 2$, it holds for some universal constant $c_0>0$ and another constant $c_1\equiv c_1(\abs{S},\mathfrak{p})>0$, such that with $\E^{(0)}\equiv \E\big[\cdot | z^{(0)}\big]$
	\begin{align*}
	& \bigabs{\E^{(0)}\Psi\big[z_S^{([t])}(A)\big]-\E^{(0)}\Psi\big[\big(\Theta_t(\mathfrak{Z}^{([0:t])})\big)_{S,[1:t]}\big] } \\
	&\qquad \leq \Lambda_\Psi \cdot \big(K\Lambda \log n\cdot (1+\pnorm{z^{(0)}}{\infty})\big)^{c_1 t^{5}} \cdot n^{-1/c_0^t}.
	\end{align*}
\end{theorem}

We may also derive an averaged version of the above Theorem \ref{thm:GFOM_se_sym} with significantly weaker regularity assumptions on $\{\mathsf{F}_\cdot,\mathsf{G}_\cdot\}$, in a similar spirit to Theorem \ref{thm:universality_avg}. For ease of reference for future applications, we present this result below.

\begin{theorem}\label{thm:GFOM_se_sym_avg}
	Fix $t \in \N$ and $n \in \N$. Suppose (D1) in Theorem \ref{thm:GFOM_se_sym} and (U2)' in Theorem \ref{thm:universality_avg} hold. Fix a sequence of $\Lambda_\psi$-pseudo-Lipschitz functions $\{\psi_k:\R^t \to \R\}_{k \in [n]}$ of order $\mathfrak{p}$, where $\Lambda_\psi\geq 2$. Then for any $q \in \N$, there exists some $C_0=C_0(\mathfrak{p},q)>0$ such that with $\E^{(0)}\equiv \E\big[\cdot | z^{(0)}\big]$,
	\begin{align*}
	&\E^{(0)} \biggabs{\frac{1}{n}\sum_{k \in [n]} \psi_k\big(z_k^{([t])}(A)\big) - \frac{1}{n}\sum_{k \in [n]}  \E^{(0)} \psi_k\big[\big(\Theta_t(\mathfrak{Z}^{([0:t])})\big)_{k,[1:t]}\big]  }^q \\
	&\leq \big(K\Lambda\Lambda_\psi\log n\cdot (1+\pnorm{z^{(0)}}{\infty})\big)^{C_0 t^{5} }\cdot n^{-1/C_0^t}. 
	\end{align*}
\end{theorem}

The proof of the above theorem is largely similar to that of Theorem \ref{thm:universality_avg}, so will be omitted to avoid repetitive details. 

We note that the error bounds in both Theorems \ref{thm:GFOM_se_sym} and \ref{thm:GFOM_se_sym_avg} above typically allow for $t\ll \log \log n$, which is significantly smaller than the typical range of $t\ll (\log n/\log \log n)^{c_0}$ in the universality Theorems \ref{thm:universality} and \ref{thm:universality_avg}. This worsened error bound is closely tied to the increasingly singular covariance for the AMP iterate $\{\mathfrak{z}^{(t)}\}$ (cf. \cite{bao2025leave}). 

In a related direction, for the Gaussian AMP iterate $\mathfrak{z}^{(t)}(G)$ where $G\equiv \mathrm{GOE}(n)$, it is known that the averaged state evolution characterization $n^{-1}\sum_{k \in [n]} \psi\big(\mathfrak{z}_k^{(t)}(G)\big)\approx \E \psi (\sigma_t Z)$ ($Z\sim \mathcal{N}(0,1))$ holds up to $t\ll \log n/\log \log n$ in general, cf. \cite{rush2018finite}, and up to $t\leq  n/\mathrm{polylog}(n)$ in the spiked model and in robust regression, cf. \cite{li2022non,li2024non}. It remains open to examine the optimal dependence of $t$ in the context of the above results (for specific choices of $\{\mathsf{F}_\cdot,\mathsf{G}_\cdot\}$).

\section{General first order methods: asymmetric case}\label{section:main_results_asym}

\subsection{Formal setup and examples}
Next we consider an asymmetric version of the iterative algorithm (\ref{def:GFOM_sym}), which is initialized with $(u^{(0)},v^{(0)})\in \R^m\times \R^n$, and subsequently updated according to 
\begin{align}\label{def:GFOM_asym}
\begin{cases}
u^{(t)} = A \mathsf{F}_t^{\langle 1\rangle}(v^{([0:t-1])})+ \mathsf{G}_{t}^{\langle 1\rangle}(u^{([0:t-1])}) \in \R^m,\\
v^{(t)}= A^\top \mathsf{G}_t^{\langle 2\rangle}(u^{([0:t])})+\mathsf{F}_t^{\langle 2\rangle}(v^{([0:t-1])})\in \R^n.
\end{cases}
\end{align}
Here with a slight abuse of notation, we denote $A$ now as an $m\times n$ random matrix, and the row-separate functions $\mathsf{F}_t^{\langle 1 \rangle}, \mathsf{F}_t^{\langle 2 \rangle}:\R^{n\times [0:t-1]} \to \R^n$, $\mathsf{G}_t^{\langle 1 \rangle}:\R^{m\times [0:t-1]} \to \R^m$ and $\mathsf{G}_t^{\langle 2 \rangle}:\R^{m\times [0:t]} \to \R^m$ are understood as applied row-wise in the sense of (\ref{def:row_sep}). 

Let us now give some examples for the asymmetric GFOM in (\ref{def:GFOM_asym}):
\begin{enumerate}
	\item (\emph{AMP algorithms}). The standard asymmetric AMP algorithm can be recovered by setting
	\begin{align*}
	&\mathsf{F}_t^{\langle 1\rangle}(v^{([0:t-1])})\equiv \mathfrak{F}_t(v^{(t-1)}),\\
	&\mathsf{G}_{t}^{\langle 1\rangle}(u^{([0:t-1])})\equiv - \big[\mathscr{V}_A \E \mathfrak{F}_t'(v^{(t-1)}) \big]\circ \mathfrak{G}_{t-1}(u^{(t-1)}),\\
	&\mathsf{G}_{t}^{\langle 2\rangle}(u^{([0:t])})\equiv \mathfrak{G}_{t}(u^{(t)}),\\
	 &\mathsf{F}_t^{\langle 2\rangle}(v^{([0:t-1])})\equiv -[\mathscr{V}_A^\top \E \mathfrak{G}_{t}'(u^{(t)})]\circ \mathfrak{F}_t(v^{(t-1)}).
	\end{align*}
	Here $\{\mathfrak{F}_{t},\mathfrak{G}_t:\R \to \R\}_{t\geq 1}$ is a class of sufficiently smooth functions applied coordinate-wise with $\mathfrak{G}_0(\cdot)\equiv 0$, and $A \in \R^{m\times n}$ has independent entries with a general variance profile $\mathscr{V}_A\equiv \E A\circ A \in \R_{\geq 0}^{m\times n} $. 
	\item (\emph{Variants of gradient descent algorithms}). Recall the linear model (\ref{def:linear_model}) and the empirical risk minimizer $\hat{\mu}$ in (\ref{def:ERM}). A natural gradient descent algorithm for solving $\hat{\mu}$ can be described as follows: With initialization $\mu^{(0)}=0$ for simplicity, and for a chosen step size $\eta>0$, let
	\begin{align}\label{def:ERM_grad_descent_general_loss_regularizer}
	\mu^{(t)}&\equiv \mu^{(t-1)}-\eta  \big(-A^\top \mathsf{L}'(Y-A \mu^{(t-1)})+\lambda \mathsf{f}_n'(\mu^{(t-1)})\big),
	\end{align}
	where $\mathsf{f}_n(\mu)\equiv \sum_{j \in [n]} \mathsf{f}(\mu_j)$. Then (\ref{def:ERM_grad_descent_general_loss_regularizer}) can be reduced to the asymmetric GFOM (\ref{def:GFOM_asym}) by setting $\mu^{(t)}\equiv v^{(t)}$ and
	\begin{align*}
	&\mathsf{F}_t^{\langle 1\rangle}(v^{([0:t-1])})\equiv -v^{(t-1)}+\mu_0,\, \mathsf{G}_{t}^{\langle 1\rangle}(u^{([0:t-1])})\equiv 0,\\
	&\mathsf{G}_{t}^{\langle 2\rangle}(u^{([0:t])})\equiv \eta \mathsf{L}'(u^{(t)}+\xi),\, \mathsf{F}_t^{\langle 2\rangle}(v^{([0:t-1])})\equiv (\mathrm{id}-\eta \lambda \mathsf{f}_n')(v^{(t-1)}).
	\end{align*}	
	Some variants of the gradient descent algorithm (\ref{def:ERM_grad_descent_general_loss_regularizer}) can be easily incorporated by modifying the above identifications:
	\begin{itemize}
		\item (\emph{Stochastic gradient descent}). This algorithm uses a random subsample $\{(A_i,Y_i): i \in S\}$ (i.e., $S$ is a random subset of $[m]$) in (\ref{def:ERM_grad_descent_general_loss_regularizer}) to reduce the computational cost. It can be reformulated into a GFOM by replacing $\mathsf{G}_{t}^{\langle 2\rangle}$ above with  $\mathsf{G}_{t}^{\langle 2\rangle}(u^{([0:t])})\equiv \eta\big(s^{(t)}\circ \mathsf{L}'(u^{(t)}+\xi)\big)$, where the entries of $s^{(t)}\in \{0,1\}^m$ are i.i.d $\mathrm{Bern}(p)$, independent of all other variables.
		\item (\emph{Polyak's momentum speed up}). This method \cite{poljak1964some} amounts to adding a `momentum' term $\beta(\mu^{(t-1)}-\mu^{(t-2)})$ with $\beta>0$ in (\ref{def:ERM_grad_descent_general_loss_regularizer}). It can be reformulated into a GFOM by replacing $\mathsf{F}_t^{\langle 2\rangle}(v^{([0:t-1])})$ above with $\mathsf{F}_t^{\langle 2\rangle}(v^{([0:t-1])})\equiv (\mathrm{id}-\eta \lambda \mathsf{f}_n')(v^{(t-1)})+\beta(v^{(t-1)}-v^{(t-2)}\bm{1}_{t\geq 2})$.
	\end{itemize}
    For more variants of (\ref{def:ERM_grad_descent_general_loss_regularizer}), including the discrete Langevin algorithm and Nesterov's accelerated gradient descent method, the readers are referred to \cite[Section 3.1]{gerbelot2024rigorous}, where these methods are reformulated into GFOMs.
    
    We note that when the regularizer $\mathsf{f}_n$ is convex but not necessarily differentiable, the proximal version of the gradient descent algorithm (\ref{def:ERM_grad_descent_general_loss_regularizer}) can also be reformulated into an asymmetric GFOM. The readers are referred to Section \ref{subsection:proof_strategy_ERM} for several examples along this line.
\end{enumerate}

\subsection{Universality}
The following theorem provides an analogue of the universality Theorem \ref{thm:universality}, now for the asymmetric GFOM in (\ref{def:GFOM_asym}).

\begin{theorem}\label{thm:universality_asym}
Fix $t \in \N$ and $m,n \in \N$. Suppose the following hold:
\begin{enumerate}
	\item[($U^\ast $1)] $A=A_0/\sqrt{m+n}$, $B=B_0/\sqrt{m+n}$, where (i) $A_0, B_0$ are $m\times n$ random matrices whose entries are independent mean $0$ variables, and (ii) for all $i \in [m], j \in [n]$, $\E A_{0,ij}^2=\E B_{0,ij}^2$ and $\pnorm{A_{0,ij}}{\psi_2}\vee \pnorm{B_{0,ij}}{\psi_2}\leq K$ for some $K\geq 2$.
	\item[($U^\ast $2)] For all $s \in [t], \#\in \{1,2\}, k \in [m], \ell \in [n]$, $\mathsf{F}_{s,\ell}^{\langle \# \rangle},\mathsf{G}_{s,k}^{\langle 1 \rangle} \in C^3(\R^{[0:s-1]})$ and $\mathsf{G}_{s,k}^{\langle 2 \rangle} \in C^3(\R^{[0:s]})$. Moreover, there exists some $\Lambda\geq 2$ and $\mathfrak{p}\in \N$ such that 
	\begin{align*}
	&\max_{s \in [t]}\max_{\# =1,2}\max_{k\in[m], \ell \in [n]}\Big\{ \pnorm{ \mathsf{F}_{s,\ell}^{\langle \# \rangle} }{\mathrm{Lip}}+ \pnorm{\mathsf{G}_{s,k}^{\langle \# \rangle}}{\mathrm{Lip}}+\max_{a \in \mathbb{Z}_{\geq 0}^{[0:s-1]},b \in \mathbb{Z}_{\geq 0}^{[0:s]},\abs{a}\vee \abs{b}\leq 3 }\\ &\qquad \bigpnorm{(1+\abs{\cdot})^{-\mathfrak{p}}\big\{\abs{\partial^a \mathsf{F}_{s,\ell}^{\langle \# \rangle}(\cdot)}+ \abs{\partial^a \mathsf{G}_{s,k}^{\langle 1 \rangle}(\cdot)}+ \abs{\partial^b \mathsf{G}_{s,k}^{\langle 2 \rangle}(\cdot)}  \big\}  }{\infty}   \Big\}\leq \Lambda.
	\end{align*}
\end{enumerate}
Then for any $\emptyset \neq S_1\subset [m]$, $\emptyset\neq S_2 \subset [n]$, and any $\Psi_{1} \in C^3(\R^{\abs{S_{1}}\times t}), \Psi_{2} \in C^3(\R^{\abs{S_{2}}\times t })$ satisfying 
\begin{align}\label{cond:Psi_asym}
\max_{\#=1,2}\max_{a \in \mathbb{Z}_{\geq 0}^{\abs{S_{\#}}\times t}, \abs{a}\leq 3} \sup_{x \in \R^{\abs{S_{\#}}\times t } }\bigg(\sum_{ (s,\tau) \in [\abs{S_{\#}}]\times [t] }(1+\abs{x_{s,\tau}})^{\mathfrak{p}}\bigg)^{-1} \abs{\partial^a \Psi_{\#}(x)} \leq \Lambda_{\Psi}
\end{align}
for some $\Lambda_{\Psi}\geq 2$, it holds for some constant $c_1=c_1(\mathfrak{p})>0$ that with $\E^{(0)}\equiv \E\big[\cdot | (u^{(0)},v^{(0)})\big]$,
\begin{align*}
& \max \Big\{ \abs{S_1}^{-3}\bigabs{\E^{(0)} \Psi_1\big(u_{S_1}^{([t])}(A)\big)-\E^{(0)} \Psi_1\big(u_{S_1}^{([t])}(B)\big)}, \nonumber\\
&\qquad\qquad  \abs{S_2}^{-3}\bigabs{\E^{(0)} \Psi_2\big(v_{S_2}^{([t])}(A)\big)-\E^{(0)} \Psi_2\big(v_{S_2}^{([t])}(B)\big)}   \Big\}\nonumber\\ 
& \leq  \Lambda_\Psi \cdot\big(K\Lambda\log n\cdot (1+\pnorm{u^{(0)}}{\infty}+\pnorm{v^{(0)}}{\infty})\big)^{c_1 t^3} \cdot (m+n)^{-1/2}.
\end{align*}
\end{theorem}

Similar to Theorem \ref{thm:universality_avg}, we may use Theorem \ref{thm:universality_asym} to provide a precise non-asymptotic version of the averaged universality for the asymmetric GFOM (\ref{def:GFOM_asym}).

\begin{theorem}\label{thm:universality_asym_avg}
	Suppose ($U^\ast $1) in Theorem \ref{thm:universality_asym} holds, and ($U^\ast $2) is replaced by
	\begin{enumerate}
		\item[($U^\ast $2)'] 
		$
		\max\limits_{s \in [t]}\max\limits_{  \# = 1,2   }\max\limits_{k \in [m], \ell \in [n]}\big\{\pnorm{ \mathsf{F}_{s,\ell}^{\langle \# \rangle} }{\mathrm{Lip}}+ \pnorm{\mathsf{G}_{s,k}^{\langle \# \rangle}}{\mathrm{Lip}}+\abs{\mathsf{F}_{s,\ell}^{\langle \# \rangle} (0)}+ \abs{\mathsf{G}_{s,k}^{\langle \# \rangle} (0)}\big\}\leq \Lambda$ for some $\Lambda\geq 2$.
	\end{enumerate}
	Fix a sequence of $\Lambda_\psi$-pseudo-Lipschitz functions $\{\psi_k:\R^t \to \R\}_{k \in [m\vee n]}$ of order $\mathfrak{p}$, where $\Lambda_\psi\geq 2$. Then for any $q \in \N$, there exists some $C_0=C_0(\mathfrak{p},q)>0$ such that with $\E^{(0)}\equiv \E\big[\cdot | (u^{(0)},v^{(0)})\big]$,
	\begin{align*}
	&\E^{(0)} \biggabs{\frac{1}{m+n}\sum_{k \in [m]} \psi_k\big(u_k^{([t])}(A)\big) - \frac{1}{m+n}\sum_{k \in [m]}  \psi_k\big(u_k^{([t])}(B)\big)  }^q\\
	&\qquad \vee \E^{(0)} \biggabs{\frac{1}{m+n}\sum_{\ell \in [n]} \psi_\ell\big(v_\ell^{([t])}(A)\big) - \frac{1}{m+n}\sum_{\ell \in [n]}  \psi_\ell\big(v_\ell^{([t])}(B)\big)  }^q\\
	&\leq \big(K\Lambda\Lambda_\psi\log n\cdot (1+\pnorm{u^{(0)}}{\infty}+\pnorm{v^{(0)}}{\infty})\big)^{C_0 t^3}\cdot (m+n)^{-1/(C_0t^3)}. 
	\end{align*}
\end{theorem}

The proofs for the above theorems proceed with a reduction of the asymmetric GFOM (\ref{def:GFOM_asym}) to a symmetric GFOM. Such a reduction scheme has been well-known for the special case of the AMP iterate; see, e.g., \cite[Section 6]{berthier2020state}, \cite[Section 6.3]{bao2025leave}. The details of this reduction in the context of GFOM can be found in Section \ref{subsection:proof_GFOM_universality_asym}.

We note that in both Theorems \ref{thm:universality_asym} and \ref{thm:universality_asym_avg}, explicit conditions on $m$ and $n$ are not placed due to the normalization $(m+n)^{-1/2}$ in (U$^*$1). In typical applications, the normalization will be either $m^{-1/2}$ or $n^{-1/2}$, so the conclusions of the above results remain valid in the proportional regime $1/K\leq m/n\leq K$.

\subsection{State evolution}

In the asymmetric case, the state evolution for the GFOM (\ref{def:GFOM_asym}) is more complicated than that in Definition \ref{def:GFOM_se_sym}. In fact, it is iteratively described---in the following definition---by (i) two row-separate maps $\Phi_t: \R^{m\times [0:t]} \to \R^{m\times [0:t]}$ and $\Xi_t: \R^{n\times [0:t]} \to \R^{n\times [0:t]}$, and (ii) two centered Gaussian matrices $\mathfrak{U}^{([1:\infty))}\in \R^{m\times [1:\infty)}$ and $\mathfrak{V}^{([1:\infty))}\in \R^{n\times [1:\infty)}$.

\begin{definition}\label{def:GFOM_se_asym}
Initialize with $\Phi_0=\mathrm{id}(\R^m)$, $\Xi_0\equiv \mathrm{id}(\R^n)$, and $\mathfrak{U}^{(0)}\equiv u^{(0)}$, $\mathfrak{V}^{(0)}\equiv v^{(0)}$. For $t=1,2,\ldots$, with $\E^{(0)}\equiv \E\big[\cdot | (\mathfrak{U}^{(0)}, \mathfrak{V}^{(0)})\big]$, execute the following steps:
\begin{enumerate}
	\item Let $\Phi_{t}:\R^{m\times [0:t]}\to \R^{m\times [0:t]}$ be defined as follows: for $w \in [0:t-1]$, $\big[\Phi_{t}(\mathfrak{u}^{([0:t])})\big]_{\cdot,w}\equiv \big[\Phi_{w}(\mathfrak{u}^{([0:w])})\big]_{\cdot,w}$, and for $w=t$,
	\begin{align*}
	\big[\Phi_{t}(\mathfrak{u}^{([0:t])})\big]_{\cdot,t} \equiv \mathfrak{u}^{(t)}+\sum_{s \in [1:t-1]} \mathfrak{f}_{s}^{(t-1) }\circ \mathsf{G}_{s}^{\langle 2\rangle}\big(\Phi_{s}(\mathfrak{u}^{([0:s])}) \big)+\mathsf{G}_{t}^{\langle 1\rangle}\big(\Phi_{t-1}(\mathfrak{u}^{([0:t-1])}) \big),
	\end{align*}
	where the coefficient vectors $\{\mathfrak{f}_{s}^{(t-1) } \}_{s \in [1:t-1]}\subset \R^m$ are determined by
	\begin{align*}
	\mathfrak{f}_{s,k}^{(t-1) } \equiv \sum_{\ell \in [n]} \E A_{k\ell}^2\cdot \E^{(0)}\partial_{\mathfrak{V}_\ell^{(s)}} \big\{ \mathsf{F}_{t,\ell}^{\langle 1\rangle}\circ \Xi_{t-1,\ell}\big\} (\mathfrak{V}_\ell^{([0:t-1])}), \quad  k \in [m].
	\end{align*}
	\item Let the Gaussian law of $\mathfrak{U}^{(t)}$ be determined via the following correlation specification: for $s \in [1:t]$ and $k \in [m]$,
	\begin{align*}
	\mathrm{Cov}\big(\mathfrak{U}_k^{(t)}, \mathfrak{U}_k^{(s)} \big)\equiv \sum_{\ell \in [n]} \E A_{k\ell}^2\cdot \E^{(0)}  \prod_{\tau \in \{t,s\}}\big\{ \mathsf{F}_{\tau,\ell}^{\langle 1\rangle}\circ \Xi_{\tau-1,\ell}\big\} (\mathfrak{V}_\ell^{([0:\tau-1])}).
	\end{align*}
	\item Let $\Xi_{t}:\R^{n\times [0:t]}\to \R^{n\times [0:t]}$ be defined as follows: for $w \in [0:t-1]$, $\big[\Xi_{t}(\mathfrak{v}^{([0:t])})\big]_{\cdot,w}\equiv \big[\Xi_{w}(\mathfrak{v}^{([0:w])})\big]_{\cdot,w}$, and for $w=t$,
	\begin{align*}
	\big[\Xi_{t}(\mathfrak{v}^{([0:t])})\big]_{\cdot,t} \equiv \mathfrak{v}^{(t)}+\sum_{s \in [1:t]} \mathfrak{g}_{s}^{(t)}\circ \mathsf{F}_{s}^{\langle 1\rangle}\big(\Xi_{s-1}(\mathfrak{v}^{([0:s-1])}) \big)+\mathsf{F}_{t}^{\langle 2\rangle}\big(\Xi_{t-1}(\mathfrak{v}^{([0:t-1])}) \big),
	\end{align*}
	where  the coefficient vectors $\{\mathfrak{g}_{s}^{(t)}\}_{s \in [1:t]}\subset \R^n$ are determined via
	\begin{align*}
	\mathfrak{g}_{s,\ell}^{(t)}\equiv \sum_{k \in [m]} \E A_{k\ell}^2\cdot  \E^{(0)} \partial_{\mathfrak{U}_k^{(s)}} \big\{ \mathsf{G}_{t,k}^{\langle 2\rangle}\circ\Phi_{t,k}\big\} (\mathfrak{U}_k^{([0:t])}),\quad \ell \in [n].
	\end{align*}
	\item Let the Gaussian law of $\mathfrak{V}^{(t)}$ be determined via the following correlation specification: for $s \in [1:t]$ and $\ell \in [n]$,
	\begin{align*}
	\mathrm{Cov}(\mathfrak{V}_\ell^{(t)},\mathfrak{V}_\ell^{(s)})\equiv \sum_{k \in [m]} \E A_{k\ell}^2\cdot  \E^{(0)} \prod_{\tau \in \{t,s\}} \big\{ \mathsf{G}_{\tau,k}^{\langle 2\rangle}\circ\Phi_{\tau,k}\big\} (\mathfrak{U}_k^{([0:\tau])}).
	\end{align*}
\end{enumerate}
\end{definition}

The following theorems establish analogues of Theorems \ref{thm:GFOM_se_sym} and \ref{thm:GFOM_se_sym_avg} for the asymmetric GFOM (\ref{def:GFOM_asym}) by using the state evolution parameters $(\{\Phi_t\},\{\Xi_t\}, \mathfrak{U}^{([1:\infty))},\mathfrak{V}^{([1:\infty))})$ in Definition \ref{def:GFOM_se_asym} above.

\begin{theorem}\label{thm:GFOM_se_asym}
	Fix $t \in \N$ and $n \in \N$. Suppose the following hold:
	\begin{enumerate}
		\item[($D^\ast$1)] $A\equiv A_0/\sqrt{m}$, where the entries of $A_0\in \R^{m\times n}$ are independent mean $0$ variables such that $\max_{i,j \in [n]}\pnorm{A_{0,ij}}{\psi_2}\leq K$ holds for some $K\geq 2$.
		\item[($D^\ast$2)] For all $s \in [t], \#\in \{1,2\}, k \in [m], \ell \in [n]$, $\mathsf{F}_{s,\ell}^{\langle \# \rangle},\mathsf{G}_{s,k}^{\langle 1 \rangle} \in C^3(\R^{[0:s-1]})$ and $\mathsf{G}_{s,k}^{\langle 2 \rangle} \in C^3(\R^{[0:s]})$. Moreover, there exists some $\Lambda\geq 2$ and $\mathfrak{p}\in \N$ such that 
		\begin{align*}
		&\max_{s \in [t]}\max_{\# =1,2}\max_{k\in[m], \ell \in [n]}\Big\{ \abs{ \mathsf{F}_{s,\ell}^{\langle \# \rangle}(0) }+ \abs{\mathsf{G}_{s,k}^{\langle \# \rangle}(0)}+\max_{0\neq a \in \mathbb{Z}_{\geq 0}^{[0:s-1]},0\neq b \in \mathbb{Z}_{\geq 0}^{[0:s]},\abs{a}\vee \abs{b}\leq 3 }\\ &\qquad \Big(\bigpnorm{ \partial^a \mathsf{F}_{s,\ell}^{\langle \# \rangle} }{\infty}+\bigpnorm{ \partial^a \mathsf{G}_{s,k}^{\langle 1 \rangle}  }{\infty}+ \bigpnorm{  \partial^b \mathsf{G}_{s,k}^{\langle 2 \rangle}   }{\infty} \Big)   \Big\}\leq \Lambda.
		\end{align*}
	\end{enumerate}
	Further suppose $1/K\leq m/n\leq K$. Then for any $\emptyset \neq S_1\subset [m]$, $\emptyset\neq S_2 \subset [n]$, and any $\Psi_{1} \in C^3(\R^{\abs{S_{1}}\times t}), \Psi_{2} \in C^3(\R^{\abs{S_{2}}\times t })$ satisfying (\ref{cond:Psi_asym}) for some $\Lambda_{\Psi}\geq 2$, it holds for some universal constant $c_0>0$ and another constant $c_1\equiv c_1(\abs{S_1},\abs{S_2},\mathfrak{p})>0$, such that with $\E^{(0)}\equiv \E\big[\cdot | (u^{(0)},v^{(0)})\big]$,
	\begin{align*}
	& \bigabs{\E^{(0)} \Psi_1\big[u_{S_1}^{([t])}(A)\big]-\E^{(0)}\Psi_1\big[\big(\Phi_t(\mathfrak{U}^{([0:t])})\big)_{S_1,[1:t]}\big] } \\
	& \qquad \vee \bigabs{ \E^{(0)} \Psi_2\big[v_{S_2}^{([t])}(A)\big]-\E^{(0)}\Psi_2\big[\big(\Xi_t(\mathfrak{V}^{([0:t])})\big)_{S_2,[1:t]}\big] } \\
	& \leq \Lambda_\Psi \cdot \big(K\Lambda \log n\cdot (1+\pnorm{u^{(0)}}{\infty}+\pnorm{v^{(0)}}{\infty})\big)^{c_1 t^5} \cdot n^{-1/c_0^t}.
	\end{align*}
\end{theorem}

\begin{theorem}\label{thm:GFOM_se_asym_avg}
	Fix $t \in \N$ and $n \in \N$, and suppose $1/K\leq m/n\leq K$ for some $K\geq 2$. Suppose ($D^\ast$1) in Theorem \ref{thm:GFOM_se_asym} and ($U^\ast$2)' in Theorem \ref{thm:universality_asym_avg} hold. Fix a sequence of $\Lambda_\psi$-pseudo-Lipschitz functions $\{\psi_k:\R^t \to \R\}_{k \in [m\vee n]}$ of order $\mathfrak{p}$, where $\Lambda_\psi\geq 2$. Then for any $q \in \N$, there exists some $C_0=C_0(\mathfrak{p},q)>0$ such that with $\E^{(0)}\equiv \E\big[\cdot | (u^{(0)},v^{(0)})\big]$,
	\begin{align*}
	&\E^{(0)} \bigg[\biggabs{\frac{1}{m}\sum_{k \in [m]} \psi_k\big(u_k^{([t])}(A)\big) - \frac{1}{m}\sum_{k \in [m]}  \E^{(0)} \psi_k\big[\big(\Phi_t(\mathfrak{U}^{([0:t])})\big)_{k,[1:t]}\big]  }^q\bigg]\\
	&\quad \vee \E^{(0)} \bigg[\biggabs{\frac{1}{n}\sum_{\ell \in [n]} \psi_\ell\big(v_\ell^{([t])}(A)\big) - \frac{1}{n}\sum_{\ell \in [n]} \E^{(0)}  \psi_\ell\big[\big(\Xi_t(\mathfrak{V}^{([0:t])})\big)_{\ell,[1:t]}\big]  }^q\bigg]  \\
	&\leq \big(K\Lambda\Lambda_\psi\log n\cdot (1+\pnorm{u^{(0)}}{\infty}+\pnorm{v^{(0)}}{\infty})\big)^{C_0 t^5}\cdot n^{-1/C_0^t}. 
	\end{align*}
\end{theorem}

We note that the variance normalization $1/m$ in ($D^\ast$1) is introduced to align with the typical normalization used in the AMP literature, cf. \cite{bayati2011dynamics,berthier2020state,bao2025leave}. Clearly, this normalization can be changed to, say, $1/n$ or $1/(m+n)$, due to the assumption $m\asymp n$ and the scale-free formulation of the state evolution in Definition \ref{def:GFOM_se_asym}. 

We remark that instead of proving Theorem \ref{thm:GFOM_se_asym} by using the same reduction scheme as in the proofs of Theorems \ref{thm:universality_asym} and \ref{thm:universality_asym_avg}, here we take a direct approach that connects the asymmetric GFOM iterate to an asymmetric AMP iterate, as in the proof of Theorem \ref{thm:GFOM_se_sym}. This approach has the advantage of both avoiding unnecessary notational complications, and explicitly connecting to an iteratively constructed AMP. The details of the arguments can be found in Section \ref{subsection:proof_GFOM_se_asym}.

\section{Application I: Universality of empirical risk minimizers}\label{section:ERM}

In this section, we will apply the universality results in the previous section in two canonical empirical risk minimization problems, namely, (i) the regularized regression estimators in the linear model (cf. Section \ref{section:ERM_ls}), and (ii) the regularized maximum likelihood estimators in the logistic regression model (cf. Section \ref{section:ERM_logit}).

\subsection{Regularized estimators in linear regression}\label{section:ERM_ls}
Recall the linear regression model (\ref{def:linear_model}) and the regularized empirical risk minimizer $\hat{\mu}$ in (\ref{def:ERM}).

\begin{theorem}\label{thm:avg_universality_ERM}
Suppose the following hold for some $K\geq 2$. 
\begin{enumerate}
	\item $m/n\in [1/K,K]$, $\pnorm{\mu_0}{\infty}\leq K$.
	\item $\mathsf{L} \in C^1(\R\to \R_{\geq 0})$ is convex, and $\pnorm{\mathsf{L}'}{\mathrm{Lip}}\vee \abs{\mathsf{L}'(0)}\vee \pnorm{\mathsf{L}'(\xi_1)}{\psi_2}\leq K$. 
	\item $A=A_0/\sqrt{m}$, $B=B_0/\sqrt{m}$, where (i) $A_0, B_0$ are $m\times n$ random matrices whose entries are independent mean $0$ variables, and (ii) for all $i \in [m], j \in [n]$, $\E A_{0,ij}^2=\E B_{0,ij}^2$ and $\pnorm{A_{0,ij}}{\psi_2}\vee \pnorm{B_{0,ij}}{\psi_2}\leq K$.
	\item The regularizer $\mathsf{f}\geq \mathsf{f}(0)=0$, and is $1/K$-strongly convex.
\end{enumerate}
Fix a sequence of $K$-pseudo-Lipschitz functions $\{\psi_j:\R \to \R\}_{j \in [n]}$ of order 2. Then for $q \in \N$, there exists some $C_0=C_0(q,K)>0$ such that for $\hat{\mu}$ defined in (\ref{def:ERM}), with $L_{\xi}\equiv 1+m^{-1/2}\pnorm{\xi}{}$,
\begin{align*}
&\E \biggabs{\frac{1}{n}\sum_{j \in [n]} \psi_j\big(\hat{\mu}_j(A)\big)-\frac{1}{n}\sum_{j \in [n]} \psi_j\big(\hat{\mu}_j(B)\big) }^q \leq C_0\cdot  \big(e^{-(\log n)^{1/8}/C_0}+e^{-n/C_0}\cdot \E L_{\xi}^q\big).
\end{align*}
Moreover, for the squared loss $\mathsf{L}(x)=x^2/2$:
\begin{itemize}
	\item If furthermore $m/n\geq 1+1/K$ and $1/K\leq \E A_{0,ij}^2=\E B_{0,ij}^2\leq K$ for all $i \in [m], j \in [n]$, then the above display holds without assuming the strong convexity of $\mathsf{f}$ in the condition (4).  
	\item For Lasso with $\mathsf{f}(x)=\lambda\abs{x}$, the condition $m/n\geq 1+1/K$ above can be further removed if $\lambda\geq C_1$ for some large $C_1=C_1(K)>0$.
\end{itemize}
\end{theorem}

To put Theorem \ref{thm:avg_universality_ERM} in the literature, it covers (almost) all universality examples in high dimensional regression in \cite{han2023universality}\footnote{Except for the absolute loss $\mathsf{L}(x)=\abs{x}$, which can however be handled using a similar smoothing argument as in logistic regression in Section \ref{section:ERM_logit}.}, and improves its proof-theoretic machinery in at least the following aspects:
\begin{itemize}
	\item Conceptually, the above theorem completely separates the universality problem of $\hat{\mu}$ from understanding its precise high-dimensional distributions as required in \cite{han2023universality}, using the so-called Convex Gaussian Min-Max Theorem (CGMT). The CGMT method, as it current stands, does not accommodate the large class of random matrix models as permitted by the above theorem. More importantly, the CGMT method must operate under the pair $(\mathsf{L},\mathsf{f})$ for which the solutions to the associated fixed point equations can be characterized. This requirement is now well-known as the major technical bottleneck for the CGMT method to be applicable in concrete problems, cf. \cite{celentano2023lasso,han2023noisy,montanari2023generalization}.	
	\item Technically, the proof of the above theorem is fairly straightforward by applying the universality results in Sections \ref{section:main_results_sym} and \ref{section:main_results_asym}, upon constructing a natural proximal gradient descent iterate $\{\mu^{(t)}\}$ that approaches $\hat{\mu}$ as $t \to \infty$. A similar idea is employed in \cite{dudeja2024spectral} under the squared loss. A major technical advantage of using gradient descent type iterates over the AMP iterate, is that they are easier to analyze due to their dependence only on the last iteration, whereas the AMP depends on past two iterations that lead to more involved convergence analysis via the state evolution \cite{bayati2012lasso,donoho2016high,sur2019modern,bao2025leave}; see Section \ref{subsection:proof_strategy_ERM} for more discussions.
\end{itemize}

Under the squared loss $\mathsf{L}(x)=x^2/2$ and stronger smooth conditions on the regularizer $\mathsf{f}$, we may establish a stronger, entrywise universality result of $\hat{\mu}$ with respect to the design matrix $A$.
\begin{theorem}\label{thm:universality_ls}
	Suppose $\mathsf{L}(x)=x^2/2$ and the following hold for some $K\geq 2$. 
	\begin{enumerate}
		\item $m/n\in [1/K,K]$, $\pnorm{\mu_0}{\infty}\leq K$ and $\pnorm{\xi_1}{\psi_2}\leq K$. 
		\item $A=A_0/\sqrt{m}$, $B=B_0/\sqrt{m}$, where (i) $A_0, B_0$ are $m\times n$ random matrices whose entries are independent mean $0$ variables, and (ii) for all $i \in [m], j \in [n]$, $\E A_{0,ij}^2=\E B_{0,ij}^2$ and $\pnorm{A_{0,ij}}{\psi_2}\vee \pnorm{B_{0,ij}}{\psi_2}\leq K$.
		\item The regularizer $\mathsf{f}\geq \mathsf{f}(0)=0$, $\mathsf{f}\in C^4(\R)$ and satisfies $\max_{q\in [2:4]}\pnorm{\mathsf{f}^{(q)}}{\infty}\vee \pnorm{(\mathsf{f}^{(2)})^{-1}}{\infty}\leq K$.
	\end{enumerate}
	Fix $\Psi\in C^3(\R)$ with $\max_{q \in [0:3]} \pnorm{\Psi^{(q)}(\cdot)}{\infty}\leq K$. Then there exists some $C_0=C_0(K)>1$ such that for $\hat{\mu}$ defined via (\ref{def:ERM}),
	\begin{align*}
	\max_{j \in [n]}\bigabs{\E \big[\Psi\big(\hat{\mu}_j(A)\big)\big]-\E \big[\Psi\big(\hat{\mu}_j(B)\big)\big]} \leq C_0 e^{-(\log n)^{1/4}/C_0}.
	\end{align*}
	If $m/n\geq 1+1/K$, then the above estimate holds without the strong convexity condition $\pnorm{(\mathsf{f}^{(2)})^{-1}}{\infty}\leq K$.
\end{theorem}

Our result above appears to be new already in the Ridge regression setting (i.e., $\mathsf{f}(x)=\lambda x^2/2$ for some tuning parameter $\lambda>0$), where $\hat{\mu}(A)=(A^\top A+\lambda I)^{-1} A^\top Y$ admits a closed form. In this setting, when $\E A_{0,ij}^2=1$ for all $i \in [m], j \in [n]$, we may directly leverage powerful tools from random matrix theory, cf., \cite{knowles2013isotropic,knowles2017anisotropic}, to derive entrywise Gaussian approximations of $\hat{\mu}(A)$. For general variance profiles $\E A_0\circ A_0\neq c \bm{1}_{m\times n}$, the recent work \cite{bao2025leave} provides characterizations for the average $n^{-1}\sum_{j \in [n]}\psi_j\big(\hat{\mu}_j(A_0)\big)$ when the entries of $A_0$ are Gaussian. Here our result above provides a much stronger, entrywise universality result for Ridge regression under such design matrices with heteroscedastic variance profiles.

An interesting open question is to understand the extent to which the above Theorem \ref{thm:universality_ls} holds for non-smooth, and non-strongly-convex regularizers $\mathsf{f}$. A canonical setting is the Lasso where $\mathsf{f}(x)=\lambda\abs{x}$. In this Lasso setting, entrywise distributional controls may not hold in general due to the possible variance spike phenomenon; see e.g., \cite{bellec2023debias}. Moreover, it is of interest to remove smoothness regularity conditions to include singular test function such as $\Psi(x)=\bm{1}_{x=0}$, which can be used to understand universality properties for the Lasso sparsity under heteroscedastic design matrices.

\subsection{Logistic regression}\label{section:ERM_logit}
Consider the binary regression model: Let $(A_i,Y_i) \in \R^n\times \{\pm 1\}$ be i.i.d. samples from
\begin{align}\label{def:logit_reg}
\Prob\big(Y_i=1|A_i\big)\equiv \rho'(A_i^\top \mu_0),\quad i \in [m].
\end{align}
Here $\rho(x)\equiv \log (1+e^x)$ and therefore $\rho'(x)=1/(1+e^{-x})$. We will be interested in the regularized maximum likelihood estimator (MLE), defined via
\begin{align}\label{def:ERM_logit}
\hat{\mu}\equiv \hat{\mu}(A)\equiv \argmin_{\mu \in \R^n }\bigg\{\sum_{i \in [m]}\rho\big(-Y_i\cdot A_i^\top \mu\big)+\sum_{j \in [n]} \mathsf{f}(\mu_j)\bigg\},
\end{align}
where $\mathsf{f}:\R\to \R_{\geq 0}$ is a (convex) regularizer. Here we have slightly abused the notation $\hat{\mu}$; this notation will be local in this subsection. 

The following theorem establishes the averaged universality of the regularized MLE $\hat{\mu}$ in logistic regression with respect to the design matrix $A$.

\begin{theorem}\label{thm:universality_logistic}
	Suppose the following hold for some $K\geq 2$.
	\begin{enumerate}
		\item  $m/n \in [1/K,K]$ and $\pnorm{\mu_0}{\infty}\leq K$.
		\item $A=A_0/\sqrt{m}$, $B=B_0/\sqrt{m}$, where (i) $A_0, B_0$ are $m\times n$ random matrices whose entries are independent mean $0$ variables, and (ii) for all $i \in [m], j \in [n]$, $\E A_{0,ij}^2=\E B_{0,ij}^2$ and $\pnorm{A_{0,ij}}{\psi_2}\vee \pnorm{B_{0,ij}}{\psi_2}\leq K$.
		\item The regularizer $\mathsf{f}\geq \mathsf{f}(0)=0$, and is $1/K$-strongly convex.
	\end{enumerate}
	Fix a sequence of $K$-pseudo-Lipschitz functions $\{\psi_j:\R \to \R\}_{j \in [n]}$ of order 2.
	Then there exists some $C_0=C_0(K)>1$ such that for $\hat{\mu}$ defined in (\ref{def:ERM_logit}),
	\begin{align*}
	\E\biggabs{\frac{1}{n}\sum_{j \in [n]} \psi_j\big(\hat{\mu}_j(A)\big)-\frac{1}{n}\sum_{j \in [n]} \psi_j\big(\hat{\mu}_j(B)\big) }\leq C_0 e^{-(\log n)^{1/8}/C_0}.
	\end{align*}
\end{theorem}
It is easy to extend the above theorem to other link functions $\rho$ in the binary regression model (\ref{def:logit_reg}); we omit these non-essential ramifications.

When $\E A_{0,ij}^2=\E B_{0,ij}^2=1$, we may take $B=G$ with i.i.d. $\mathcal{N}(0,1/m)$ entries. In this case, the value of $n^{-1}\sum_{j \in [n]} \psi(\hat{\mu}_j(G),\mu_{0,j})$ in the proportional high dimensional limiting regime (i.e., $m/n\to \gamma$ and $\mu_0$ has i.i.d. entries) can be determined via the solution of a set of six equations with six unknowns, cf. \cite[Theorem 1]{salehi2019impact}. Moreover, for the Ridge regularizer $\mathsf{f}(x)\equiv \lambda x^2/2$ with a fixed tuning parameter $\lambda>0$, using rotational invariance of Gaussian distributions, it is easy to prove (see, e.g.,  \cite[Lemma 2.1]{zhao2022asymptotic}) the approximate normality of $\hat{\mu}_j(G)-\alpha\mu_{0,j}\stackrel{d}{\approx} \sigma\cdot \mathcal{N}(0,1)$ with `global quantities' $\alpha=\iprod{\hat{\mu}(G)}{\mu_0}/\pnorm{\mu_0}{}^2$, $\sigma^2=\pnorm{\mathsf{P}_{\mu_0}^\perp\hat{\mu}(G)}{}^2$ that can be determined with a reduced system of three equations with three unknowns, cf. \cite[Theorem 2]{salehi2019impact}. 

After this paper was posted online, \cite{verchand2024high} obtained universality results for Ridge-regularized logistic regression estimators using the CGMT method, following a similar approach to \cite{han2023universality}. As noted earlier, the methods in \cite{han2023universality}---and consequently in \cite{verchand2024high}---do not accommodate heterogeneous design matrices as in Theorem \ref{thm:universality_logistic} and rely on understanding the mean-field equations for Ridge-regularized logistic regression estimators. In contrast, the universality results in Theorem \ref{thm:universality_logistic} apply even in heterogeneous settings where the corresponding mean-field equations remain unknown. 

An interesting open question is to prove universality for the plain MLE $\hat{\mu}$ in the unregularized case (i.e., $\mathsf{f}\equiv 0$). In this case, it is well known that the existence and $\ell_2$ boundedness of MLE $\hat{\mu}(G)$ undergo a sharp phase transition in terms of the aspect ratio $m/n$ under the Gaussian design, cf. \cite{candes2020phase,sur2019likelihood,sur2019modern}. This phase transitional threshold is recently proved to exhibit universality in terms of linear separability of the data, cf. \cite[Theorem 1]{montanari2023universality}. Interestingly, a careful examination of our proof shows that the averaged universality of the plain MLE  $\hat{\mu}$ would follow, provided that a sufficiently strong $\ell_2$ control can be proven for $\hat{\mu}$ (and its smoothed version) below the phase transition curve.

\subsection{Proof strategies  via universality results in Sections \ref{section:main_results_sym} and \ref{section:main_results_asym}}\label{subsection:proof_strategy_ERM}

\subsubsection{Proof strategy of Theorem \ref{thm:avg_universality_ERM}}

The empirical risk minimizer $\hat{\mu}$ satisfies the following first-order condition: for $\eta>0$,
\begin{align}\label{ineq:avg_universality_ERM_1}
\hat{\mu} = \prox_{\eta \mathsf{f}_n}\Big(\hat{\mu}+\eta\cdot  A^\top \mathsf{L}'(Y-A \hat{\mu})\Big).
\end{align}
This suggests the following proximal gradient algorithm: for $\eta>0$, let for $t=1,2,\ldots$,
\begin{align}\label{ineq:avg_universality_ERM_2}
\mu^{(t)} = \prox_{\eta \mathsf{f}_n}\Big(\mu^{(t-1)}+\eta\cdot  A^\top \mathsf{L}'(Y-A \mu^{(t-1)})\Big),
\end{align}
with the initialization $\mu^{(0)}=0$. Here $\mathsf{f}_n(\mu)=\sum_{j \in [n]} \mathsf{f}(\mu_j)$. As the above iterative algorithm (\ref{ineq:avg_universality_ERM_2}) is a special case of the generic asymmetric GFOM iterate in (\ref{def:GFOM_asym}), Theorem \ref{thm:universality_asym_avg} indicates that the averaged universality of $\mu^{(t)}$ carries over to $\hat{\mu}$, if an $\ell_2$ control between $\mu^{(t)}$ and $\hat{\mu}$ can be proven in the following sense: for some small $\delta_0 \in (0,1/2)$ and small $\eta>0$,
\begin{align}\label{ineq:universality_ERM_l2}
\frac{\pnorm{\mu^{(t)}-\hat{\mu}}{} }{\sqrt{n}}\lesssim (1-\delta_0\eta)^{t}.
\end{align}
An estimate of the type (\ref{ineq:universality_ERM_l2}) can be proven in a straightforward manner, with either the presence of the strong convexity of $\mathsf{f}$, or using the non-singularity of the sample covariance in the squared loss case. For the Lasso, an estimate of the above type (\ref{ineq:universality_ERM_l2})  is more delicate due to the apparent lack of strong convexity in the regime $m/n<1-\epsilon$. We prove such an estimate by controlling the sparsity of the GFOM iterate $\{\mu^{(t)}\}$ in (\ref{ineq:avg_universality_ERM_2}) on the linear order uniformly in $t$, provided that $\lambda$ is not too small. The readers are referred to Section \ref{subsection:proof_thm_avg_universality_ERM} for more details.

We remark that for (a transformation of) the AMP iterate $\{\mu^{(t)}_{\mathrm{AMP}}\}$, a typical estimate that formalizes $\mu^{(t)}_{\mathrm{AMP}}\approx \hat{\mu}$ reads
\begin{align}\label{ineq:universality_ERM_l2_AMP}
\lim_{t\to \infty} \limsup_{m/n \to \gamma }  \frac{\pnorm{\mu^{(t)}_{\mathrm{AMP}}-\hat{\mu}}{} }{\sqrt{n}}=0\quad \hbox{a.s. or in probability.}
\end{align} 
Estimates of the type (\ref{ineq:universality_ERM_l2_AMP}) are obtained under a standard Gaussian design for, e.g., the Lasso \cite{bayati2012lasso}, robust regression estimators \cite{donoho2016high}, the maximum likelihood estimator in logistic regression \cite{sur2019modern}, the SLOPE \cite{bu2021algorithmic} and the minimum $\ell_1$-norm interpolator \cite{li2021minimum}. The proof of (\ref{ineq:universality_ERM_l2_AMP}) usually involves sharp controls of two consecutive AMP iterates, or equivalently, the convergence of the AMP state evolution to its equilibrium. Such a task is technically highly non-trivial, as it is again strongly tied to the global existence, uniqueness and the stability of the mean-field system of equations for $\hat{\mu}$ that must be studied case-by-case under different pairs of $(\mathsf{L},\mathsf{f})$ and random matrix ensembles $A$. A direct convergence estimate of the type (\ref{ineq:universality_ERM_l2}) that holds for the gradient descent iterate (\ref{ineq:avg_universality_ERM_2}) appears unavailable for the AMP iterate.

\subsubsection{Proof strategy of Theorem \ref{thm:universality_ls}}
For the squared loss, (\ref{ineq:avg_universality_ERM_1})-(\ref{ineq:avg_universality_ERM_2}) can be further simplified as follows. Let $\hat{\Sigma}\equiv \sum_{i \in [m]} A_iA_i^\top=A^\top A$ be the sample covariance. Then (\ref{ineq:avg_universality_ERM_1}) reduces to
\begin{align}\label{eqn:ERM_first_order}
\hat{\mu} = \prox_{\eta \mathsf{f}_n}\Big(\big(I-\eta\hat{\Sigma}\big)\hat{\mu}+\eta\cdot  A^\top Y\Big),
\end{align}
and (\ref{ineq:avg_universality_ERM_2}) becomes the following: 
\begin{align}\label{def:ERM_grad_descent}
\mu^{(t)} \equiv \mu^{(t)}(A) \equiv \prox_{\eta \mathsf{f}_n}\Big(\big(I-\eta\hat{\Sigma}\big)\mu^{(t-1)}+\eta\cdot  A^\top Y\Big).
\end{align}
So an application of Theorem \ref{thm:universality_asym} yields that uniformly in $j \in [n]$,
\begin{align*}
{\mu}_j^{(t)}(A)\stackrel{d}{\approx} {\mu}_j^{(t)}(B),\quad \hbox{for }t\ll (\log n/\log \log n)^{1/3}.
\end{align*}
The key challenge from here is to prove an \emph{entrywise control} for $\mu^{(t)}-\hat{\mu}$ beyond the $\ell_2$ control as in (\ref{ineq:universality_ERM_l2}). We provide such a control in Proposition \ref{prop:ERM_Delta_r_l2}: for some small $\delta_0 \in (0,1/2)$, if $t\lesssim \log n$ and $\eta>0$ is chosen sufficiently small,
\begin{align}\label{ineq:universality_ls_loo}
\pnorm{\mu^{(t)}-\hat{\mu}}{\infty}\lesssim \log n\cdot  (1-\delta_0\eta)^t\, \hbox{ with high probability}.
\end{align} 
It is easy to see that both terms $\log n$ and $(1-\delta_0\eta)^t$ are necessary in the above estimate. Using (\ref{ineq:universality_ls_loo}) above, we may then run  (\ref{def:ERM_grad_descent}) for $\log \log n \ll t \ll (\log n/\log\log n)^{1/3}$ many iterations to conclude Theorem \ref{thm:universality_ls}. 

The proof of (\ref{ineq:universality_ls_loo}) is fairly non-standard. In fact, our method of proof relies on a `second-order' leave-one-out argument that provides entrywise control for the difference of $\hat{\mu},\mu^{(t)}$ and their leave-one(-predictor)-out versions $\hat{\mu}_{[-\ell]}, \mu^{(t)}_{[-\ell]}$ (formally defined in Eqns.  (\ref{def:ERM_loo_predictor}) and (\ref{eqn:ERM_grad_descent_loo_predictor})). In essence, the standard leave-one-out method provides control for $\pnorm{\hat{\mu}-\hat{\mu}_{[-\ell]}}{}$ and $\pnorm{\mu^{(t)}-\mu^{(t)}_{[-\ell]}}{}$---and therefore also for $\pnorm{\hat{\mu}}{\infty}$ and $\pnorm{\mu^{(t)}}{\infty}$---approximately on the order $1$. Here the key step in proving (\ref{ineq:universality_ls_loo}) is a second-order, entrywise control on $\hat{\mu}-\hat{\mu}_{[-\ell]}$ and $\mu^{(t)}-\mu^{(t)}_{[-\ell]}$: In Lemma \ref{lem:ERM_loo_predictor_error} we prove that, for $t\lesssim \log n$, with high probability,
\begin{align*}
\Big(\max_{k\neq \ell} n^{1/2} \abs{(\hat{\mu}-\hat{\mu}_{[-\ell]})_k}+ \abs{\hat{\mu}_\ell}\Big)\vee \Big(\max_{k\neq \ell} n^{1/2} \abs{({\mu}^{(t)}-{\mu}^{(t)}_{[-\ell]})_k}+ \abs{{\mu}_\ell^{(t)}}\Big)\lesssim \log n.
\end{align*}
The readers are referred to Section \ref{section:proof_ERM_ls} for more details.

\subsubsection{Proof strategy of Theorem \ref{thm:universality_logistic}}
It is convenient to rewrite (\ref{def:ERM_logit}) to suit the purpose of our analysis. First we let
\begin{align}\label{def:L_logit}
\mathsf{L}(x,y;\xi)\equiv \rho\big(-(2\cdot\bm{1}_{y+\xi\geq 0}-1)x\big).
\end{align}
Now with $\xi_1,\ldots,\xi_m$ being i.i.d. logistic random variables with c.d.f. $\Prob(\xi_1\leq t)=1/(1+e^{-t})$, some simple algebra shows that the regularized MLE $\hat{\mu}$ defined in (\ref{def:ERM_logit}) is equivalent to
\begin{align}\label{def:ERM_logit_equiv}
\hat{\mu} = \hat{\mu}(A)=\argmin_{\mu \in \R^n } \bigg\{\sum_{i \in [m]} \mathsf{L}(A_i^\top \mu, A_i^\top \mu_0;\xi_i)+  \sum_{j \in [n]} \mathsf{f}(\mu_j) \bigg\}.
\end{align}
Using the above representation, similar to (\ref{eqn:ERM_first_order}), $\hat{\mu}$ satisfies the following first-order condition: for any $\eta>0$, 
\begin{align}\label{eqn:ERM_first_order_logit}
\hat{\mu} = \prox_{\eta \mathsf{f}_n}\bigg(\hat{\mu}-\eta \sum_{i \in [m]} A_i \partial_1 \mathsf{L}(A_i^\top \hat{\mu}, A_i^\top \mu_0;\xi_i) \bigg).
\end{align}
Here we have written $\partial_1 \mathsf{L}(x,y;\xi)\equiv (\partial \mathsf{L}/\partial x) (x,y;\xi)$. We now consider, similar to (\ref{def:ERM_grad_descent}), the following proximal gradient descent algorithm: for $\eta>0$, let for $t=1,2,\ldots,$
\begin{align}\label{def:ERM_grad_descent_logit}
\mu^{(t)} \equiv \prox_{\eta \mathsf{f}_n}\bigg(\mu^{(t-1)}-\eta \sum_{i \in [m]} A_i \partial_1 \mathsf{L}(A_i^\top \mu^{(t-1)}, A_i^\top \mu_0;\xi_i) \bigg),
\end{align}
with the initialization $\mu^{(0)}=0$. A major technical difficulty to apply Theorem \ref{thm:universality_asym_avg} to the iterative scheme (\ref{def:ERM_grad_descent_logit}) lies in the non-smoothness of $y\mapsto \partial_1 \mathsf{L}(x,y;\xi)$ due to the indicator structure. To overcome this, we consider smoothed versions $\hat{\mu}_\sigma,\mu^{(t)}_\sigma$ of $\hat{\mu},\mu^{(t)}$, formally defined  in the beginning of Section \ref{section:proof_ERM_logistic} via a smoothed version $\mathsf{L}_\sigma$ of $\mathsf{L}$ in (\ref{def:L_logit}). We then prove in Lemmas \ref{lem:logistic_error} and \ref{lem:logistic_smooth_error} that there exists some small $\delta_0 \in (0,1/2)$ such that for any $t\leq n$, sufficiently small $\eta>0$ and any smoothing parameter $\sigma\geq \log n/n$, with high probability,
\begin{align}\label{ineq:universality_logit_l2}
\frac{ \pnorm{\mu_\sigma^{(t)}-\hat{\mu}}{}  }{\sqrt{n}}\lesssim \sqrt{\log n}\cdot (1-\delta_0\eta)^t+ (\sigma \log n)^{1/4}. 
\end{align}
A remaining technical subtlety to apply Theorem \ref{thm:universality_asym_avg} for the smoothed loss function $\mathsf{L}_\sigma$ lies in that the Lipschitz constant of $y\mapsto \partial_1 \mathsf{L}_\sigma(x,y;\xi)$ depends on $\abs{x}$. We prove that the vector $\big( A_i^\top \mu_\sigma^{(t-1)}\big)_{i \in [m]}$, playing the role of $x$ here, is delocalized via a leave-one-sample-out method (cf. Lemma \ref{lem:logistic_loo}), so with high probability the Lipschitz constants of the maps being applied in Theorem \ref{thm:universality_asym_avg} are under control.

From here we may run the smoothed version of the algorithm (\ref{def:ERM_grad_descent_logit}) with the smoothing parameter $\log n/n\leq \sigma\ll 1/\log n$ for $\log \log n\ll t\leq n$ many iterations to conclude Theorem \ref{thm:universality_logistic}. The smoothing parameter $\sigma$ need to be tuned carefully, as the error bound for the averaged universality of $\mu_\sigma^{(t)}$ blows up as $\sigma \downarrow 0$. The details of these arguments may be found in Section \ref{section:proof_ERM_logistic}.

\section{Application II: Entrywise dynamics of gradient descent iterates}\label{section:grad_des}

Recall the general empirical risk minimization problem in (\ref{def:ERM}) with $\mathsf{f}(x)=\lambda x^2/2$. Here we consider a stochastic gradient descent version of (\ref{def:ERM_grad_descent_general_loss}): for $t=1,2,\ldots$, with $S_{t-1}\subset [m]$ a sub-sample chosen independent of $A$,
\begin{align}\label{def:ERM_stoc_grad_descent_general_loss}
\mu^{(t)} \equiv \mu^{(t-1)}-\eta\cdot  \bigg(-\sum_{i \in S_{t-1}}A_i \mathsf{L}'\big(Y_i-\iprod{A_i}{ \mu^{(t-1)}}\big)+\lambda \mu^{(t-1)}\bigg).
\end{align}
Also recall that the loss function $\mathsf{L}$ need not be convex. For notational convenience, we let $\mathsf{L}_{t-1}:\R^m\to \R^m$ be defined by 
\begin{align*}
\mathsf{L}_{t-1}(x)\equiv \big(x_i \bm{1}_{i \in S_{t-1} }\big)_{i \in [m]} \in \R^m.
\end{align*}
Clearly, $\mathsf{L}_{t-1}$ is row-separate.

\subsection{State evolution for $\{\mu^{(t)}\}$}
The state evolution for the gradient descent iterate $\{\mu^{(t)}\}$ in (\ref{def:ERM_stoc_grad_descent_general_loss}) can be described by three objects:
\begin{enumerate}
	\item[(i)] two sequences of symmetric covariance matrices $\{\Sigma^{\mathfrak{U}}_k\}_{k \in [m]}\subset \R^{[1:\infty)\times [1:\infty)}, \{\Sigma^{\mathfrak{V}}_\ell\}_{\ell \in [n]} \subset \R^{[0:\infty)\times [0:\infty)}$;
	\item[(ii)] another sequence of matrices $\{\mathsf{M}_\ell^{\mathfrak{V}}\}_{\ell \in [n]}\subset \R^{[0:\infty)\times [0:\infty)}$ whose lower triangle elements are $0$;
	\item[(iii)] a sequence of maps $\{\Phi_{t}:\R^{m\times [0:t]}\to \R^{m}\}_{t \in \N}$.
\end{enumerate}
We also associate $\{\Sigma^{\mathfrak{U}}_k\}_{k \in [m]}$ with a row-independent, centered Gaussian matrix $\mathfrak{U}^{[1:\infty)}\in \R^{m\times [1:\infty)}$. All these objects are defined recursively as follows.

\begin{definition}\label{def:grad_descent_se}
	Initialize with $(\Sigma_\ell^{\mathfrak{V}})_{0,0}\equiv \mu_{0,\ell}^2$, $(\mathsf{M}_\ell^{\mathfrak{V}})_{0,0}\equiv 1$ for $\ell \in [n]$, $\Phi_0\equiv \mathrm{id}(\R^m)$  and $\mathfrak{U}^{(0)}=0_m$.  For $t=1,2,\ldots,$ execute the following steps:
	\begin{enumerate}
		\item For $k \in [m]$ and $s \in [1:t]$, let
		\begin{align*}
		(\Sigma^{\mathfrak{U}}_k)_{t,s}&\equiv \mathrm{Cov}\big(\mathfrak{U}^{(t)}_k, \mathfrak{U}^{(s)}_k \big) = \sum_{\ell \in [n]} \E A_{k\ell}^2\cdot \bigiprod{(\mathsf{M}_{\ell}^{\mathfrak{V}})_{\cdot,t-1} }{\Sigma_\ell^{\mathfrak{V}} (\mathsf{M}_{\ell}^{\mathfrak{V}})_{\cdot,s-1}}.
		\end{align*}
		\item
		Let $\Phi_{t}:\R^{m\times [0:t]}\to \R^{m}$ be defined by
		\begin{align*}
		\Phi_{t}(\mathfrak{u}^{([0:t])}) \equiv \mathfrak{u}^{(t)}+\eta\sum_{s \in [1:t-1]} \mathfrak{f}_{s}^{(t-1) }\circ  \mathsf{L}_{s-1}'\big(\xi-\Phi_{s}(\mathfrak{u}^{([0:s])})\big),
		\end{align*}
		where for $k \in [m]$,
		\begin{align*}
		\mathfrak{f}_{s,k}^{(t-1)} \equiv \sum_{\ell \in [n]} \E A_{k\ell}^2\cdot (\mathsf{M}_{\ell}^{\mathfrak{V}})_{s,t-1},\quad s \in [1:t-1].
		\end{align*}
		
		\item For $\ell \in [n]$ and $r \in [0:\infty)$, let 
		\begin{align*}
		(\mathsf{M}_{\ell}^{\mathfrak{V}})_{r,t}&\equiv \eta\lambda\cdot \bm{1}_{r=0}+\bigg(\sum_{s \in [r+1:t]} \mathfrak{g}_{s,\ell}^{(t)}  (\mathsf{M}_{\ell}^{\mathfrak{V}})_{r,s-1}  +(1-\eta \lambda) (\mathsf{M}_\ell^{\mathfrak{V}})_{r,t-1}\bigg)\cdot \bm{1}_{r \in [0:t-1]}+\bm{1}_{r=t},
		\end{align*}
		where with $W^{(t)}(\mathfrak{U}^{([0:t])})\equiv \mathsf{L}_{t-1}'' \big(\xi-\Phi_{t}(\mathfrak{U}^{([0:t])})\big)\in \R^m$,
		\begin{align*}
		\mathfrak{g}_{s,\ell}^{(t)}&\equiv -\eta \sum_{k \in [m]} \E A_{k\ell}^2\cdot  \E^\xi   \bigg[W_k^{(t)}(\mathfrak{U}^{([0:t])}) \\
		&\quad \times \bigg\{ \delta_{s,t}+\sum_{\tau \in [0:t-s-1]}(-\eta)^{\tau+1}  \sum_{r_{[1:\tau]}\subset [s+1: t-1]} \prod_{\iota \in [1:\tau+1]}  \mathfrak{f}_{r_\iota,k}^{(r_{\iota-1}-1) } W_k^{(r_\iota)}(\mathfrak{U}^{([0:r_\iota])})\bigg\}\bigg].
		\end{align*}
		Here we write $r_0\equiv t$ and $r_{\tau+1}\equiv s$ when the summation is run over $r_{[1:\tau]}\subset [s+1: t-1]$ (including $\tau=0$).
		\item For $\ell \in [n]$, let $(\Sigma^{\mathfrak{V}}_\ell)_{t,0}\equiv 0$, and let for $s \in [1:t]$
		\begin{align*}
		(\Sigma^{\mathfrak{V}}_\ell)_{t,s}&\equiv \eta^2 \sum_{k \in [m]} \E A_{k\ell}^2\cdot  \E^\xi \prod_{\tau \in \{t,s\}} \bigiprod{e_k}{ \mathsf{L}_{\tau-1}'\big(\xi-\Phi_{\tau} (\mathfrak{U}^{([0:\tau])})\big)}.
		\end{align*}
	\end{enumerate} 
\end{definition}

\subsection{Entrywise dynamics of $\{{\mu}^{(t)}\}$}
Using the state evolution in the above Definition \ref{def:grad_descent_se}, let us define the key parameters 
\begin{align}\label{def:grad_descent_key_para}
b^{(t)}_{\ell;\mathsf{GD}}\equiv -(\mathsf{M}_{\ell}^{\mathfrak{V}})_{0,t},\quad \sigma^{2,(t)}_{\ell;\mathsf{GD}}\equiv \bigiprod{(\mathsf{M}^{\mathfrak{V}}_\ell)_{{[1:\infty),t}}}{(\Sigma_\ell^{\mathfrak{V}})_{[1:\infty)^2} (\mathsf{M}^{\mathfrak{V}}_\ell)_{{[1:\infty),t}} }.
\end{align}
These quantities will play a crucial role in understanding the dynamics of $\{\mu^{(t)}\}$ in (\ref{def:ERM_stoc_grad_descent_general_loss}) in the theorem below; its proof can be found in  Section \ref{section:proof_grad_descent}.

\begin{theorem}\label{thm:grad_descent}
	Suppose $\lambda\geq 0$, $K^{-1}\leq m/n\leq K$ and (D*1) in Theorem \ref{thm:GFOM_se_asym} hold for some $K\geq 2$. Further assume that $\mathsf{L} \in C^4(\R)$. Then for any $\Lambda$-pseudo-Lipschitz $\psi:\R \to \R$, there exists some universal constant $c_0>0$ such that
	\begin{align*}
	&\max_{\ell \in [n]}\bigabs{\E^\xi \psi \big(\mu^{(t)}_\ell-\mu_{0,\ell}\big)- \E^\xi \psi \big[\mathcal{N}\big(b^{(t)}_{\ell;\mathsf{GD}} \mu_{0,\ell}, \sigma^{2,(t)}_{\ell;\mathsf{GD}} \big)\big]}\\
	&\leq  \Big(K\Lambda \log n\cdot  (1+ \eta + \lambda + \pnorm{\mu_0}{\infty})\cdot \big(\pnorm{\mathsf{L}'(\xi)}{\infty}+\max_{q \in [2:4]} \pnorm{\mathsf{L}^{(q)}}{\infty}\big)   \Big)^{c_0 t^5}\cdot n^{-1/c_0^t}.
	\end{align*}
	Here in the expectation $\E^\xi$ we treat $\mu_0,\xi,\{S_\cdot\}$ all as fixed. 
\end{theorem}

The distributional description in Theorem \ref{thm:grad_descent} admits a natural interpretation. Indeed, as will be clear from the proof, the decomposition
\begin{align}\label{ineq:grad_descent_decomposition}
\mu^{(t)}_\ell-\mu_{0,\ell}=  -(\mathsf{M}_{\ell}^{\mathfrak{V}})_{0,t}\cdot\mu_{0,\ell}+\sum_{s \in [1:t]} (\mathsf{M}_{\ell}^{\mathfrak{V}})_{s,t} \mathfrak{v}^{(s)}_\ell,\quad \ell \in [n]
\end{align}
holds, where $\{\mathfrak{v}^{(t)}\}$ is an underlying AMP iterate, and $(\mathfrak{v}^{(s)}_{\ell})_{s \in [1:t]}$ is approximately a centered Gaussian vector in $\R^t$, whose covariance $(\Sigma_\ell^{\mathfrak{V}})_{[1:t]^2}$ can be tracked easily by the state evolution in Definition \ref{def:grad_descent_se}. In typical situations where the gradient descent iterate $\{\mu^{(t)}\}$  converges as $t\to \infty$, this underlying AMP $\{\mathfrak{v}^{(t)}\}$ is also expected to converge as $t\to \infty$, and the coefficient is expected to behave as $(\mathsf{M}_{\ell}^{\mathfrak{V}})_{s,t}\approx \mathfrak{m}_\ell^{\mathfrak{V}}(t-s)$ for some function $\mathfrak{m}_\ell^{\mathfrak{V}}:\R\to \R$ and large $t,s$, so that the effect of initialization (i.e., behavior of $\mathfrak{v}^{(t)}$ for small values of $t$) eventually dies out. The readers are also referred to related discussions in \cite[Section 4.2]{celentano2021high} for the continuum limit gradient flow case.

\begin{remark}
Some technical remarks are in order.
\begin{enumerate}
	\item In (\ref{def:ERM_stoc_grad_descent_general_loss}) we consider gradient descent iterates associated with a general loss in the linear model (\ref{def:linear_model}). It is also possible to characterize the slightly more general form of the gradient descent iterates, as studied (in continuum limit) in \cite{celentano2021high}, that apply beyond linear models. In particular, we may use the same data augmentation reduction in Eqn. (27)-(28) in \cite{celentano2021high}; these non-essential generalizations are omitted to keep the presentation simple.
	\item The decomposition (\ref{ineq:grad_descent_decomposition}) and the resulting approximate Gaussianity of $\{\mu^{(t)}_\ell\}_{\ell \in [n]}$ are crucially tied to the $\ell_2$ regularization with $\mathsf{f}(x)=\lambda x^2/2$. While it is possible to provide abstract characterizations for $\{\mu^{(t)}\}$ under other convex regularizer $\mathsf{f}$'s, an easy-to-interpret distributional description as (\ref{ineq:grad_descent_decomposition}) will no longer be available.
	\item The regularity conditions on $\mathsf{L}$ can be further weakened for an averaged characterization of $\{\mu^{(t)}\}$. We omit these ramifications.
\end{enumerate}
\end{remark}

\subsubsection{An illustrative setting}
	As an illustration, consider the homogeneous case with $\E A_{k\ell}^2=1/n$ for $k \in [m],\ell \in [n]$, and with $\lambda =0$. We focus on the full sample case $S_{\cdot}\equiv [m]$ for simplicity. In this case, Definition \ref{def:grad_descent_se} admits major simplifications. Let $\phi\equiv m/n$ be the aspect ratio. Initialized with $\Sigma^{\mathfrak{V}}_{0,0}\equiv \pnorm{\mu_0}{}^2/n$, $\mathsf{M}^{\mathfrak{V}}_{0,0}\equiv 1$, $\Phi_0\equiv \mathrm{id}(\R)$  and $\mathfrak{U}^{(0)}=0$, for $t=1,2,\ldots,$ we update sequentially $\Sigma^{\mathfrak{U}}_{t,\cdot}$, $\Phi_t$, $\mathsf{M}^{\mathfrak{V}}_{\cdot,t}$ and $\Sigma^{\mathfrak{V}}_{t,\cdot}$ as follows: 
	\begin{enumerate}
		\item[(H1)] For $s \in [1:t]$, let $
		\Sigma^{\mathfrak{U}}_{t,s}\equiv \mathrm{Cov}\big(\mathfrak{U}^{(t)}, \mathfrak{U}^{(s)} \big) = \bigiprod{\mathsf{M}^{\mathfrak{V}}_{\cdot,t-1} }{\Sigma^{\mathfrak{V}} \mathsf{M}^{\mathfrak{V}}_{\cdot,s-1}}$. 
		\item[(H2)]
		Let $
		\Phi_{t}(\mathfrak{u}^{([0:t])}) \equiv \mathfrak{u}^{(t)}+\eta\sum_{s \in [1:t-1]}  \mathsf{M}^{\mathfrak{V}}_{s,t-1} \mathsf{L}'\big(\xi-\Phi_{s}(\mathfrak{u}^{([0:s])})\big)\in \R^m$. 	
		\item[(H3)] For $r \in [0:\infty)$, let $
		\mathsf{M}^{\mathfrak{V}}_{r,t}\equiv \big(\sum_{s \in [r+1:t]} \mathfrak{g}_{s}^{(t)}  \mathsf{M}^{\mathfrak{V}}_{r,s-1}  + \mathsf{M}^{\mathfrak{V}}_{r,t-1}\big)\cdot \bm{1}_{r \in [0:t-1]}+\bm{1}_{r=t}$, 
		where with $W^{(t)}_q\equiv  \mathsf{L}^{(q)} \big(\xi-\Phi_{t}\big[(\mathfrak{U}^{([0:t])}_k)_{k \in [m]}\big]\big)\in \R^m, q=1,2$,
		\begin{align*}
		\mathfrak{g}_{s}^{(t)}\equiv -\phi\cdot \eta \E  W^{(t)}_{2,\pi_m} \bigg\{ \delta_{s,t}+\sum_{\tau \in [0:t-s-1]}(-\eta)^{\tau+1}  \sum_{r_{[1:\tau]}\subset [s+1, t-1]} \prod_{\iota \in [1:\tau+1]}  \mathsf{M}^{\mathfrak{V}}_{r_\iota,r_{\iota-1}-1 } W^{(r_\iota)}_{2,\pi_m}\bigg\}.
		\end{align*}
		Here $\{\mathfrak{U}^{([0:t])}_k\}_{k\in [m]}$ are i.i.d. copies of $\mathfrak{U}^{([0:t])}$, and $\pi_m\sim \mathrm{Unif}([1:m])$ is independent of all other variables.
		\item[(H4)] Let $\Sigma^{\mathfrak{V}}_{t,0}\equiv 0$, and for $s \in [1:t]$, let $
		\Sigma^{\mathfrak{V}}_{t,s}\equiv \phi\cdot\eta^2   \E  W_{1,\pi_m}^{(t)}W_{1,\pi_m}^{(s)}$. 
	\end{enumerate} 
	Let us now examine the behavior of $\mu^{(t)}$ using Theorem \ref{thm:grad_descent}, under two extreme scenarios of the aspect ratio $\phi$:
	\begin{itemize}
		\item In the (substantially) overparametrized regime $\phi\ll 1$, under suitable growth condition on $\mathsf{L}$,  $\mathfrak{g}_{s}^{(t)}\approx 0$ and by (H3), $\mathsf{M}^{\mathfrak{V}}_{0,t}\approx \mathsf{M}^{\mathfrak{V}}_{0,t-1}\approx\cdots \approx \mathsf{M}^{\mathfrak{V}}_{0,0}=1$. Consequently,  $b^{(t)}_{\ell;\mathrm{GD}}\approx -1$ and $\sigma^{2,(t)}_{\ell;\mathrm{GD}}\approx 0$ for $\ell \in [n]$, and therefore Theorem \ref{thm:grad_descent} implies that $\mu^{(t)}_\ell \approx 0$ for large $t$ and $n$. For the squared loss, $\mu^{(t)}$ is known to converge to the Ridgeless interpolator $\hat{\mu}_0$ as $t \to \infty$ in the overparametrized regime \cite{hastie2022surprises}. Our results here align with the averaged distributional characterization in \cite{han2023distribution} that asserts  $\hat{\mu}_0\approx 0$. 
		\item In the (substantially) underparametrized regime $\phi\gg 1$, under the scaling $\eta = \epsilon_0/\phi$ with some small $\epsilon_0>0$, the gradient descent iterate $\{\mu^{(t)}\}$ admits a stable evolution. Then $\mathfrak{g}_{s}^{(t)}\approx-\epsilon_0 \E W^{(t)}_{2,\pi_m}\cdot\delta_{s,t}$ and by (H3), $\mathsf{M}^{\mathfrak{V}}_{0,t}\approx (1-\epsilon_0 \E W^{(t)}_{2,\pi_m})\mathsf{M}^{\mathfrak{V}}_{0,t-1}\approx\cdots \approx \prod_{s \in [t]}(1-\epsilon_0 \E W^{(s)}_{2,\pi_m})$. If $\inf_{s \in [t]}\E W^{(s)}_{2,\pi_m}\geq c_0$ for some small constant $c_0>0$, then $\abs{\mathsf{M}^{\mathfrak{V}}_{0,t}}\lesssim (1-\epsilon_0c_0)^t\approx 0$ for large $t$ and $n$. Consequently,  $b^{(t)}_{\ell;\mathrm{GD}}\approx 0$ and $\sigma^{2,(t)}_{\ell;\mathrm{GD}}\approx 0$ for $\ell \in [n]$, and therefore Theorem \ref{thm:grad_descent} implies that $\mu^{(t)}_\ell \approx \mu_{0,\ell}$ for large $t$ and $n$. This aligns with the classical consistency result $\hat{\mu}\approx \mu_0$ in this underparametrized regime \cite{van2000asymptotic}. Note that the low-dimensional CLT for $\hat{\mu}$ (or $\mu^{(t)}$ for large $t$) is beyond the scope of Theorem \ref{thm:grad_descent}, due of its error bound that must scale at least $\bigo(n^{-1/2})$.
	\end{itemize}

\section{Proof of Theorem \ref{thm:universality} via delocalization}\label{section:proof_universality}

\subsection{A variant of Chatterjee's Lindeberg principle}
The basic tool we need to establish universality in Theorem \ref{thm:universality} is the following version of Chatterjee's elegant Lindeberg principle \cite{chatterjee2006generalization}. 

\begin{proposition}\label{prop:lindeberg}
	Let $X=(X_1,\ldots,X_n)$ and $Y=(Y_1,\ldots,Y_n)$ be two random vectors in $\R^n$ with independent components. Then for any $f \in C^3(\R^n)$,
	\begin{align*}
	\bigabs{\E f(X) - \E f(Y)}&\leq \sum_{\ell=1,2}\biggabs{ \sum_{i=1}^n \frac{1}{\ell!}\big(\E X_i^\ell-\E Y_i^\ell\big) \E \partial_i^\ell f( X_{[1:(i-1)]},0, Y_{[(i+1):n]})} \\
	&\quad+ \max_{U_i \in \{X_i,Y_i\}}\biggabs{\sum_{i=1}^n\E U_i^3 \int_0^{1} \partial_i^3 f(X_{[1:(i-1)]},tU_i, Y_{[(i+1):n]} )(1-t)^2\,\d{t}}.
	\end{align*}
\end{proposition}
\begin{proof}
	The proof is essentially a repetition of \cite[Theorem 1.1]{chatterjee2006generalization} by using the integral remainder in the Taylor expansion. We spell out some details below. Let $Z_i\equiv (X_1,\ldots,X_{i-1},X_i,Y_{i+1},\ldots,Y_n)$ and $Z_i^0\equiv (X_1,\ldots,X_{i-1},0,Y_{i+1},\ldots,Y_n)$. Then $\E f(X)-\E f(Y)=\sum_{i=1}^n \big(\E f(Z_i)-\E f(Z_{i-1})\big)$. On the other hand, using Taylor expansion up to order $3$ with integral remainder,
	\begin{align*}
	f(Z_i) &= f(Z_i^0)+X_i \partial_i f(Z_i^0)+ \frac{X_i^2}{2} \partial_i^2 f(Z_i^0)+\frac{X_i^3}{2}\int_{0}^{1} \partial_i^3 f(Z_i^0+t X_i)(1-t)^2\,\d{t},\\
	f(Z_{i-1}) &= f(Z_i^0)+Y_i \partial_i f(Z_i^0)+ \frac{Y_i^2}{2} \partial_i^2 f(Z_i^0)+\frac{Y_i^3}{2}\int_{0}^{1} \partial_i^3 f(Z_i^0+t Y_i)(1-t)^2\,\d{t}.
	\end{align*}
	The claim follows by using $(X_i,Y_i)\perp Z_i^0$.
\end{proof}

\subsection{Delocalization estimates}

\subsubsection{Delocalization for GFOM iterates}
Fix $\mathcal{P}\subset [n]$. Let $A_{[-\mathcal{P}]}\equiv (A_{ij}\bm{1}_{i,j\notin \mathcal{P}})\in \R^{n\times n}$ and $A_{[\mathcal{P}]}\equiv A-A_{[-\mathcal{P}]}$. Let $z^{(t)}_{[-\mathcal{P}]}\in \R^n$ be defined iteratively via 
\begin{align}\label{def:GFOM_sym_loo}
z^{(t)}_{[-\mathcal{P}]}=A_{[-\mathcal{P}]} \mathsf{F}_t(z^{([t-1])}_{[-\mathcal{P}]})+\mathsf{G}_t(z^{([t-1])}_{[-\mathcal{P}]}),
\end{align}
with the same initialization $z^{(0)}_{[-\mathcal{P}]}\equiv z^{(0)}$ that is independent of $A$. 

\begin{proposition}\label{prop:loo_l2_bound}
	Suppose the following hold:
	\begin{enumerate}
		\item $A=A_0/\sqrt{n}$, where $A_0$ is symmetric and the entries of its upper triangle are independent mean $0$ random variables with $\max_{i,j \in [n]} \pnorm{A_{0,ij}}{\psi_2}\leq K$.
		\item $\max\limits_{s \in [t]}\max\limits_{\mathsf{H}_s \in \{\mathsf{F}_s,\mathsf{G}_s\}}\max\limits_{\ell \in [n]}\big(\pnorm{\mathsf{H}_{s,\ell}}{\mathrm{Lip}}\vee \abs{\mathsf{H}_{s,\ell}(0)}\big)\leq \Lambda$ for some $\Lambda\geq 2$. 
	\end{enumerate}
	Then there exists some universal constant $c_0>0$ such that for $x\geq 0$,
	\begin{align*}
	&\Prob\bigg(\bigpnorm{ z^{(t)}-z_{[-\mathcal{P}]}^{(t)}  }{}\geq \abs{\mathcal{P}}\big(c_0K\Lambda (1+\sqrt{x/n})\big)^{t+1}\sqrt{x}\\
	&\qquad\qquad \times   \Big(1+\max_{k \in \mathcal{P}}\abs{z_k^{(0)}}+n^{-1/2}\pnorm{z^{(0)}}{}\Big) \big| z^{(0)}\bigg)\leq c_0\cdot t \abs{\mathcal{P}}e^{-x/c_0}.
	\end{align*}
	Moreover,
	\begin{align*}
	\max_{k \in [n]}\Prob\Big(\abs{z^{(t)}_k}\geq \big(c_0K\Lambda (1+\sqrt{x/n})\big)^{t+1}(1+\sqrt{x})\cdot \big(1+\pnorm{z^{(0)}}{\infty}\big) \big| z^{(0)}\Big)\leq c_0\cdot te^{-x/c_0}.
	\end{align*}
	Consequently, for any $p\geq 1$, there exists some $C_p>1$ such that 
	\begin{align*}
	\max_{k \in [n]} \E^{1/p} \big[ \abs{z_k^{(t)}}^p\big| z^{(0)}\big] \vee \frac{\E^{1/p}\big[ \pnorm{z^{(t)}}{\infty}^p\big| z^{(0)}\big]}{(\log n)^{2t}} \leq (C_p t K\Lambda)^{t+1}\big(1+\pnorm{z^{(0)}}{\infty}\big).
	\end{align*} 
\end{proposition}

The proof of the above proposition can be found in Section \ref{subsection:proof_delocalization_iterate}.

\subsubsection{Delocalization for derivatives}

Here and below, we write $\partial_{ij}\equiv \partial/\partial A_{ij}$ for notational simplicity.

\begin{proposition}\label{prop:z_deloc}
	Suppose the following hold:
	\begin{enumerate}
		\item $A=A_0/\sqrt{n}$, where $A_0$ is symmetric and the entries of its upper triangle are independent mean $0$ random variables with $\max_{i,j \in [n]} \abs{A_{0,ij}}\leq K$.
		\item For all $s \in [t],\ell \in [n]$, $\mathsf{F}_{s,\ell},\mathsf{G}_{s,\ell} \in C^3(\R)$. Moreover, there exists some $\Lambda\geq 2$ and $\mathfrak{p}\in \N$ such that 
		\begin{align*}
		\max_{s \in [t]}\max_{\mathsf{E}_s \in \{\mathsf{F}_s,\mathsf{G}_s\}}\max_{\ell \in [n]}\Big\{\pnorm{\mathsf{E}_{s,\ell}}{\mathrm{Lip}}+\max_{q \in [0:3]} \bigpnorm{(1+\abs{\cdot})^{-\mathfrak{p}}\abs{\mathsf{E}_{s,\ell}^{(q)}(\cdot)}}{\infty}\Big\}\leq \Lambda.
		\end{align*}
	\end{enumerate}
	Then for any $D>0$, there exist $C_0=C_0(D)>0$, a universal $c_0>0$ and another $c_1=c_1(\mathfrak{p})>0$ such that if $t\leq c_0^{-1}\log n$,
	\begin{align*}
	&\Prob\bigg[  \Big(\max_{k\neq {i,j}} \sqrt{n} + \max_{k}\Big) \bigabs{\partial_{ij} z^{(t)}_k}\vee \bigabs{\partial_{ij}^2 z^{(t)}_k}\\ &\qquad \geq \big(C_0 K\Lambda \log n\cdot (1+\pnorm{z^{(0)}}{\infty})\big)^{c_1 t^3}\Big| z^{(0)}  \bigg]\leq c_0^{t}\cdot C_0 n^{-D}.
	\end{align*}
	Moreover, for any $p\geq 1$, there exists a constant $C_p>0$ such that 
	\begin{align*}
	\Big(\max_{k\neq {i,j}} \sqrt{n} + \max_{k}\Big) \E^{1/p}\bigabs{\partial_{ij} z^{(t)}_k}^p \leq \big(C_p K\Lambda   \log n\cdot (1+\pnorm{z^{(0)}}{\infty})\big)^{c_1 t^3}.
	\end{align*}
\end{proposition}

The proof of the above proposition can be found in Section \ref{subsection:proof_delocalization_derivative}.

\begin{remark}\label{rmk:deloc_der}
	It is useful to understand why the delocalization for the derivatives $\partial_{ij} z^{(t)}, \partial_{ij}^2 z^{(t)}$ exhibits inhomogeneity. To see this, consider the very simple case $\mathsf{F}_t(z^{[0:t-1]})\equiv z^{(t-1)}$, and $\mathsf{G}_t\equiv 0$. Then $z^{(t)}=A z^{(t-1)}=\cdots =A^t z^{(0)}$. For instance, for $t=2$, 
	\begin{align*}
	\partial_{ij} z^{(2)}= \partial_{ij} (A^2 z^{(0)})=(A z^{(0)})_j e_i+(A z^{(0)})_i e_j + z^{(0)}_j A_i+ z^{(0)}_i A_j,
	\end{align*}
	which implies $\partial_{ij} z^{(2)}_k = \bigop(n^{-1/2}+\bm{1}_{k \in \{i,j\}})$. A similar calculation applies to the second derivative vector $\partial_{ij}^2 z^{(3)}$. In particular, this shows that the inhomogeneous delocalization estimates in Proposition \ref{prop:z_deloc} are optimal for GFOM iterates. 
\end{remark}

\subsubsection{Delocalization for other interaction terms}

\begin{proposition}\label{prop:der_cross_cubic}
	Suppose the following hold:
	\begin{enumerate}
		\item $A=A_0/\sqrt{n}$, where $A_0$ is symmetric and the entries of its upper triangle are independent mean $0$ random variables with $\max_{i,j \in [n]} \abs{A_{0,ij}}\leq K$.
		\item For all $s \in [t],\ell \in [n]$, $\mathsf{F}_{s,\ell},\mathsf{G}_{s,\ell} \in C^3(\R)$, Moreover, there exists some $\Lambda\geq 2$ and $\mathfrak{p}\in \N$ such that 
		\begin{align*}
		\max_{s \in [t]}\max_{\mathsf{E}_s \in \{\mathsf{F}_s,\mathsf{G}_s\}}\max_{\ell \in [n]}\Big\{\pnorm{\mathsf{E}_{s,\ell}}{\mathrm{Lip}}+\max_{q \in [0:3]} \bigpnorm{(1+\abs{\cdot})^{-\mathfrak{p}}\abs{\mathsf{E}_{s,\ell}^{(q)}(\cdot)}}{\infty}\Big\}\leq \Lambda.
		\end{align*}
	\end{enumerate}
	Then for any $D>0$, there exist $C_0=C_0(D)>0$, a universal $c_0>0$ and another $c_1=c_1(\mathfrak{p})>0$ such that if $t\leq c_0^{-1}\log n$, 
	\begin{align*}
	&\Prob\bigg[ n^{1/2}\max_{k \in [n]}\biggabs{\sum_{i,j\in [n]} A_{ij}^3 \partial_{ij}^3 z^{(t)}_k} \geq \big(C_0 K\Lambda  \log n\cdot (1+\pnorm{z^{(0)}}{\infty})\big)^{c_1 t^3}\Big| z^{(0)}  \bigg]\leq c_0^{t}\cdot C_0 n^{-D}.
	\end{align*}
	The above display remains valid when $z^{(t)}$ is generated from the Lindeberg interpolating random matrix $D_{ij}(s)$, $s \in [0,1]$ between $A$ and $B$ with matching first and second moments (formally defined in the proof of Theorem \ref{thm:universality}).
\end{proposition}

The proof of the above proposition can be found in Section \ref{subsection:proof_delocalization_other}.

\subsection{Proof of Theorem \ref{thm:universality}}

	For notational simplicity in the proof, we shall write $\E[\cdot|z^{(0)}]=\E$ and similarly for $\Prob$. We will first consider $\Psi: \R\to \R$ and then indicate the (minor) differences for the general case $\Psi: \R^{\abs{S}}\to \R$.
	
	\noindent (\textbf{Step 1}). Let $\overline{A}_{ij}\equiv (A_{ij} \wedge C_0 K \sqrt{\log n/n})\vee (-C_0 K \sqrt{\log n/n})$ and similarly define $\overline{B}_{ij}$ for some large enough $C_0=C_0(D)>0$, such that on an event $\mathscr{E}_0$ with $\Prob(\mathscr{E}_0^c)\leq C_0 n^{-4D}$ where $D>1$,  $\overline{A}_{ij}=A_{ij}$ and $\overline{B}_{ij}=B_{ij}$ uniformly in $i,j \in [n]$, and for $M \in \{A,B\}$,
	\begin{align*}
	\max_{\ell \in [2]}\max_{i,j \in [n]} \bigabs{\E M_{ij}^\ell-\E \overline{M}_{ij}^\ell}=\bigabs{\E (M_{ij}^\ell-\overline{M}_{ij}^\ell)\bm{1}_{\mathscr{E}_0}^c}\leq C_0' K^2 \log n\cdot n^{-2D}.
	\end{align*}
	Using the first and second moment match $\E A_{ij}^\ell=\E B_{ij}^\ell$ for $\ell \in [2]$, by possibly enlarging $C_0$, we have
	\begin{align}\label{ineq:universality_1}
	\max_{\ell \in [2]}\max_{i,j \in [n]}\bigabs{\E \overline{A}_{ij}^\ell- \E \overline{B}_{ij}^\ell}\leq C_0 K^2 \log n\cdot n^{-2D}.
	\end{align}
	Moreover, on $\mathscr{E}_0$, $z^{(t)}(\overline{A})=z^{(t)}(A)$ and $z^{(t)}(\overline{B})=z^{(t)}(B)$. So by Proposition \ref{prop:loo_l2_bound},
	\begin{align}\label{ineq:universality_2}
	&\max_{M \in \{A,B\}}\bigabs{\E \Psi\big(z_k^{(t)}(M)\big)- \E \Psi\big(z_k^{(t)}(\overline{M})\big)} \nonumber
	\\
	&\leq \max_{M \in \{A,B\}} \bigabs{\E \Psi\big(z_k^{(t)}(M)\big)\bm{1}_{\mathscr{E}_0^c}-  \E \Psi\big(z_k^{(t)}(\overline{M})\big)\bm{1}_{\mathscr{E}_0^c} }\nonumber\\
	&\leq C_0 \cdot \Lambda_\Psi \big( t K\Lambda \log n\cdot (1+\pnorm{z^{(0)}}{\infty}) \big)^{c_0\mathfrak{p}  t}\cdot n^{-2D}. 
	\end{align}
	Here $c_0>0$ is universal. 
	
	\noindent (\textbf{Step 2}). We shall now establish an error bound for $
	\bigabs{\E \Psi\big(z_k^{(t)}(\overline{A})\big)-\E \Psi\big(z_k^{(t)}(\overline{B})\big)}$. 
	For $i\leq j$, consider the Lindeberg path $\{\overline{D}_{ij}(s):s \in [0,1]\}$ between two random symmetric matrices $\overline{A},\overline{B}$, defined symmetrically by setting all elements in the upper triangle part of $\overline{D}_{ij}(s)$ before (resp. after) the position $(i,j)$ as those of $\overline{A}$ (resp. $\overline{B}$) and $(\overline{D}_{ij}(s))_{ij}=s \overline{A}_{ij}$. By Proposition \ref{prop:lindeberg}, we only need to handle the following two terms: 
	\begin{align*}
	\mathscr{T}_1&\equiv \max_{\ell \in [2]}\biggabs{ \sum_{i,j \in [n], i\leq j} \big(\E \overline{A}_{ij}^\ell- \E \overline{B}_{ij}^\ell\big) \E \partial_{ij}^\ell \Psi\big(z_k^{(t)}(\overline{D}_{ij}(0))\big)}, \\
	\mathscr{T}_2(s)&\equiv \biggabs{\E \sum_{i,j \in [n], i\leq j}\overline{A}_{ij}^3 \partial_{ij}^3 \Psi \big(z_k^{(t)}(\overline{D}_{ij}(s))\big)},\quad s \in [0,1].
	\end{align*}
	To bound $\mathscr{T}_1,\mathscr{T}_2(s)$, first note the derivative formulae
	\begin{align}\label{ineq:universality_3}
	\partial_{ij} \Psi(z_k^{(t)})&= \Psi'(z_k^{(t)})\partial_{ij}z_k^{(t)},\quad \partial_{ij}^2 \Psi(z_k^{(t)})= \Psi^{(2)}(z_k^{(t)})\big(\partial_{ij}z_k^{(t)}\big)^2+ \Psi'(z_k^{(t)})\partial_{ij}^2z_k^{(t)},\nonumber\\
	\partial_{ij}^3 \Psi(z_k^{(t)})&= \Psi^{(3)}(z_k^{(t)}) (\partial_{ij} z_k^{(t)})^3+3 \Psi^{(2)}(z_k^{(t)}) \cdot \partial_{ij} z_k^{(t)}\partial_{ij}^2 z_k^{(t)}+ \Psi'(z_k^{(t)})\partial_{ij}^3 z_k^{(t)}.
	\end{align}
	Here note that the meaning of $\partial_{ij}$ is slightly different on the left and right hand sides of the above display. In particular, for the right hand side we have shorthanded $\partial_{ij}\equiv \partial/\partial \overline{A}_{ij}$. 
	
	We first handle $\mathscr{T}_1$. Using the moment residual estimate in (\ref{ineq:universality_1}), the derivative formula in (\ref{ineq:universality_3}) and the moment estimate in Proposition \ref{prop:z_deloc},
	\begin{align}\label{ineq:universality_4}
	\mathscr{T}_1&\leq \max_{\ell \in [2]} \biggabs{ \sum_{i,j \in [n], i\leq j} \big(\E \overline{A}_{ij}^\ell- \E \overline{B}_{ij}^\ell\big) \E \partial_{ij}^\ell \Psi\big(z_k^{(t)}(\overline{D}_{ij}(0))\big)}\nonumber\\
	&\lesssim \Lambda_{\Psi}\cdot\max_{\ell \in [2]} \max_{i,j \in [n]} \bigabs{\E \overline{A}_{ij}^\ell- \E \overline{B}_{ij}^\ell}\cdot \sum_{i,j \in [n]} \E \Big\{1+\big(\partial_{ij}z_k^{(t)}\big)^2+ \abs{\partial_{ij}^2z_k^{(t)}}\Big\}\big(1+\abs{z_k^{(t)}}\big)^{\mathfrak{p}} \nonumber\\
	&\leq \Lambda_\Psi \cdot \big(C_1 K\Lambda \log n\cdot (1+\pnorm{z^{(0)}}{\infty})\big)^{c_1 \mathfrak{p} t^3}\cdot n^{-D}. 
	\end{align}
	Next we handle $\mathscr{T}_2(s)$. First note that for any $s \in [0,1]$,
	\begin{align*}
	&\mathscr{T}_2(s)\leq \biggabs{\E \sum_{i,j \in [n]}\overline{A}_{ij}^3 \partial_{ij}^3 \Psi \big(z_k^{(t)}(\overline{D}_{ij}(s))\big)}+ \frac{1}{2}\biggabs{\E \sum_{i \in [n]}\overline{A}_{ii}^3 \partial_{ii}^3 \Psi \big(z_k^{(t)}(\overline{D}_{ii}(s))\big)}\equiv (I)+(II).
	\end{align*}
	The second term $(II)$ is easier to handle: using the third derivative formula in (\ref{ineq:universality_3}) and the moment estimates in Propositions \ref{prop:loo_l2_bound}, \ref{prop:z_deloc} and Lemma \ref{lem:z_apriori_moment} ahead,
	\begin{align}\label{ineq:universality_5}
	(II)&\lesssim \Lambda_\Psi K^3 n^{-3/2}\log^{3/2}n\cdot \sum_{i \in [n]} \E\Big\{\abs{\partial_{ii} z_k^{(t)}}^3+ \abs{\partial_{ii} z_k^{(t)}\partial_{ii}^2 z_k^{(t)}} + \abs{\partial_{ii}^3 z_k^{(t)}}\Big\}\big(1+\abs{z_k^{(t)}}\big)^{\mathfrak{p}}\nonumber\\
	&\leq \Lambda_\Psi \cdot \big(C_2 K\Lambda \log n\cdot (1+\pnorm{z^{(0)}}{\infty})\big)^{c_2 \mathfrak{p} t^3}\cdot n^{-1/2}. 
	\end{align}
	Now we handle $(I)$. By Propositions \ref{prop:loo_l2_bound}-\ref{prop:der_cross_cubic} and possibly enlarging $c_2,C_2>0$, on an event $\mathscr{E}_2$ with $\Prob(\mathscr{E}_2^t)\leq c_2^t\cdot C_2 n^{-2D}$, 
	\begin{align*}
	\biggabs{\sum_{i,j \in [n]}\overline{A}_{ij}^3 \partial_{ij}^3 \Psi \big(z_k^{(t)}(\overline{D}_{ij}(s))\big)}\leq \Lambda_\Psi\cdot\big(C_2 K\Lambda \log n\cdot (1+\pnorm{z^{(0)}}{\infty})\big)^{c_2 \mathfrak{p} t^3}\cdot n^{-1/2}.
	\end{align*}
	On the other hand, using again (\ref{ineq:universality_3}), Proposition \ref{prop:z_deloc} and Lemma \ref{lem:z_apriori_moment}, we have the following trivial apriori estimate:
	\begin{align*}
	&\E^{1/2} \bigg( \sum_{i,j \in [n]}\overline{A}_{ij}^3 \partial_{ij}^3 \Psi \big(z_k^{(t)}(\overline{D}_{ij}(s))\big) \bigg)^2 \lesssim \Lambda_\Psi \cdot \big(C_2 K\Lambda \log n\cdot (1+\pnorm{z^{(0)}}{\infty})\big)^{c_2 \mathfrak{p} t^3}\cdot n^{1/2}.
	\end{align*}
	Consequently,
	\begin{align}\label{ineq:universality_6}
	(I)&\leq  \big(C_2 K\Lambda \log n\cdot (1+\pnorm{z^{(0)}}{\infty})\big)^{c_2 t^3}\cdot n^{-1/2}+ \E \biggabs{ \sum_{i,j \in [n]}\overline{A}_{ij}^3 \partial_{ij}^3 \Psi \big(z_k^{(t)}(\overline{D}_{ij}(s))\big) }\bm{1}_{\mathscr{E}_2^c}\nonumber\\
	&\leq \big(C_2' K\Lambda \log n\cdot (1+\pnorm{z^{(0)}}{\infty})\big)^{c_2' \mathfrak{p} t^3}\cdot n^{-1/2}.
	\end{align}
	Combining (\ref{ineq:universality_5})-(\ref{ineq:universality_6}), again by possibly enlarging $c_2,C_2>0$, we arrive at
	\begin{align}\label{ineq:universality_7}
	\sup_{s \in [0,1]}\mathscr{T}_2(s)\leq \Lambda_\Psi \cdot \big(C_2 K\Lambda \log n\cdot (1+\pnorm{z^{(0)}}{\infty})\big)^{c_2 \mathfrak{p} t^3}\cdot n^{-1/2}. 
	\end{align}
	Now combining (\ref{ineq:universality_4}) and (\ref{ineq:universality_7}), we obtain 
	\begin{align*}
	\bigabs{\E \Psi\big(z_k^{(t)}(\overline{A})\big)-\E \Psi\big(z_k^{(t)}(\overline{B})\big)}\leq \Lambda_\Psi \cdot \big(C_3 K\Lambda \log n\cdot (1+\pnorm{z^{(0)}}{\infty})\big)^{c_3 \mathfrak{p} t^3}\cdot n^{-1/2}. 
	\end{align*}
	The claimed comparison inequality for $\Psi:\R \to \R$ now follows by combining the above display and (\ref{ineq:universality_2}), and noting that the constraint $t\leq c_0^{-1}\log n$ can be removed for free. 
	
	\noindent (\textbf{Step 3}). We consider now the general case $\Psi: \R^{\abs{S}\times t}\to \R$. The difference is mostly formal, which we sketch below. We write $S\equiv \{S_{\ell}: \ell \in [\abs{S}]\}$. First, the estimate (\ref{ineq:universality_2}) comes with an additional multiplicative factor of $\abs{S}$ (the additional factor on $t$ can be assimilated by adjusting constants). Second, with $\omega_\tau \equiv (\ell_\tau,t_\tau) \in [\abs{S}]\times [t]$ for $\tau \in \N$, the derivative formulae (\ref{ineq:universality_3}) now reads
	\begin{align*}
	\partial_{ij} \Psi(z_S^{([t])})&= \sum_{\omega_{1}} \partial_{\omega_{1}}\Psi(z_S^{([t])})\cdot \partial_{ij} z_{S_{\ell_1}}^{(t_1)},\\
	\partial_{ij}^2 \Psi(z_S^{([t])})&= \sum_{\omega_{[2]}} \partial_{\omega_{[1:2]}}\Psi(z_S^{([t])})\cdot \partial_{ij} z_{S_{\ell_1}}^{(t_1)}  \partial_{ij} z_{S_{\ell_2}}^{(t_2)} + \sum_{\omega_1} \partial_{\omega_1} \Psi(z_S^{([t])})\cdot \partial_{ij}^2 z_{S_{\ell_1}}^{(t_1)},\\
	\partial_{ij}^3 \Psi(z_S^{(t)})&=\sum_{\omega_{[3]}} \partial_{\omega_{[1:3]}}\Psi(z_S^{([t])})\cdot \partial_{ij} z_{S_{\ell_1}}^{(t_1)}  \partial_{ij} z_{S_{\ell_2}}^{(t_2)} \partial_{ij} z_{S_{\ell_3}}^{(t_3)} \\
	&\qquad + 3 \sum_{\omega_{[2]}} \partial_{\omega_{[1:2]}}\Psi(z_S^{([t])})\cdot \partial_{ij}^2 z_{S_{\ell_1}}^{(t_1)}  \partial_{ij} z_{S_{\ell_2}}^{(t_2)}+ \sum_{\omega_1} \partial_{\omega_1} \Psi(z_S^{([t])})\cdot \partial_{ij}^3 z_{S_{\ell_1}}^{(t_1)}.
	\end{align*}
	We may then parallel the proofs in Step 2 to obtain the same estimate with an additional multiplicative factor of $\abs{S}^3$, and adjusting constants to assimilate  $t^3$.\qed

\section{Higher order interactions}\label{section:proof_high_order_inter}

The goal of this section is to provide estimates for a generalized version of the summation term (\ref{eqn:intro_A_hadamard}). These estimates will play a key role in the proof of the delocalization estimates in Section \ref{section:proof_delocalization} ahead.

\subsection{Some further notation}
Fix $\mathcal{P}\subset [n]$. 
\begin{itemize}
	\item For any $t \in \N$, let $z^{(t)}(A) \in \R^n$ be a vector-valued measurable map of $A$, and $z^{(t)}_{[-\mathcal{P}]}(A) \in \R^n$ depend on $A$ only through $A_{[-\mathcal{P}]}$. 
	\item Let $\mathsf{H}_t=((\mathsf{H}_t)_\ell): \R^n\to \R^n$ be a measurable, separable map in the sense that  $\big(\mathsf{H}_t(z)\big)_\ell= (\mathsf{H}_t)_\ell(z_\ell)$. 	For notational simplicity, let $\mathsf{H}_{t_s;\ell_s}\equiv \big(\mathsf{H}_{t_s}\big)_{\ell_s}(z^{(t_s)}_{\ell_s})$ and similarly $\mathsf{H}_{t_s;\ell_s}^{[-\mathcal{P}]}\equiv \big(\mathsf{H}_{t_s}\big)_{\ell_s}(z^{(t_s)}_{[-\mathcal{P}],\ell_s})$. 
	\item For any set of consecutive integers $\mathcal{I}\subset \mathbb{N}$, let $\mathcal{I}^-\equiv  \min \mathcal{I}$ and $\mathcal{I}^+\equiv \max \mathcal{I}$. 
	\item For $k,\ell \in [n]$, let 
	\begin{align*}
	A_{k,[\ell_\mathcal{I}],\ell}&\equiv A_{k,[\ell_{\mathcal{I}^-}:\ell_{\mathcal{I}^+}],\ell}\equiv A_{k,\ell_{\mathcal{I}^-} } \prod_{s \in [\mathcal{I}^-: \mathcal{I}^+) } A_{\ell_s,\ell_{s+1}}\cdot  A_{ \ell_{\mathcal{I}^+},\ell   },\\
	A_{k,[\ell_{\mathcal{I}}]}&\equiv A_{k,\ell_{\mathcal{I}^-} } \prod_{s \in [\mathcal{I}^-: \mathcal{I}^+) } A_{\ell_s,\ell_{s+1}}.
	\end{align*}
	Furthermore, for a chosen $t_{\mathcal{I}}\subset \N$, let
	\begin{align*}
	\mathscr{A}_{k,\ell}(\mathcal{I})&\equiv \sum_{\ell_{[\mathcal{I}^-:\mathcal{I}^+]} \in [n] }  A_{k,[\ell_{\mathcal{I}^-}:\ell_{\mathcal{I}^+}],\ell} \prod_{s \in \mathcal{I}}\mathsf{H}_{t_s;\ell_s},\\
	\mathscr{A}_{k,\ell}^{[-\mathcal{P}]}(\mathcal{I})&\equiv \sum_{\ell_{[\mathcal{I}^-:\mathcal{I}^+]} \in [n] }  A_{k,[\ell_{\mathcal{I}^-}:\ell_{\mathcal{I}^+}],\ell} \prod_{s \in \mathcal{I}}\mathsf{H}_{t_s;\ell_s}^{[-\mathcal{P}]},\\
	\mathscr{A}_{k,\ell;[-\mathcal{P}]}^{[-\mathcal{P}]}(\mathcal{I})&\equiv \sum_{\ell_{[\mathcal{I}^-:\mathcal{I}^+]} \in [n]\setminus \mathcal{P} }  A_{k,[\ell_{\mathcal{I}^-}:\ell_{\mathcal{I}^+}],\ell} \prod_{s \in \mathcal{I}}\mathsf{H}_{t_s;\ell_s}^{[-\mathcal{P}]}.
	\end{align*}
    Notational dependence of the above quantities on $\mathsf{H}$ and $t_{\mathcal{I}}$ will be omitted for simplicity.
    \item  For any three non-overlapping sets $\mathcal{I}_0,\mathcal{I}_1,\mathcal{I}_2$ of consecutive integers, let
    \begin{align*}
    \mathscr{A}_{k,(i,j)}\big(\mathcal{I}_0,(\mathcal{I}_1,\mathcal{I}_2)\big)&\equiv \sum_{\ell_{\mathcal{I}_0^+} \in [n]} \mathscr{A}_{k,\ell_{\mathcal{I}_0^+}}(\mathcal{I}_0^{[)}) \mathsf{H}_{t_{\mathcal{I}_0^+};\ell_{\mathcal{I}_0^+}}\cdot  \mathscr{A}_{\ell_{\mathcal{I}_0^+}, i}(\mathcal{I}_1) \mathscr{A}_{\ell_{\mathcal{I}_0^+}, j}(\mathcal{I}_2),
    \end{align*}
    and similarly for $\mathscr{A}_{k,(i,j)}^{[-\mathcal{P}]}\big(\mathcal{I}_0,(\mathcal{I}_1,\mathcal{I}_2)\big)$ and $\mathscr{A}_{k,(i,j);[-\mathcal{P}]}^{[-\mathcal{P}]}\big(\mathcal{I}_0,(\mathcal{I}_1,\mathcal{I}_2)\big)$.
    
    In the above definition the index sets $t_{\mathcal{I}_0}, t_{\mathcal{I}_I}$ and $t_{\mathcal{I}_2}$ are allowed to overlap with each other. Moreover, we may take $\mathcal{I}_1,\mathcal{I}_2 = \emptyset$. When $\mathcal{I}_0=\emptyset$, the corresponding term is formally understood as $\mathscr{A}_{k,\ell_{\mathcal{I}_0^+}}(\mathcal{I}_0^{[)})=\mathscr{A}_{k,\ell_{\mathcal{I}_0^+}}(\emptyset^{[)})=\delta_{k,\ell_{\mathcal{I}_0^+}}$. The same formal rule is applied when $\mathcal{I}_1,\mathcal{I}_2=\emptyset^{[)}$. 
    \item 
    For any finite set $T\subset \N$ (including $T=\emptyset$), let $\abs{ \{T,\emptyset^{[)}\} }\equiv \abs{T}-1$, $\abs{ \{T,\emptyset^{[)}, \emptyset^{[)}\} }\equiv \abs{T}-2$.
\end{itemize}

\subsection{Controls of higher order interactions}

For $L>0$, an index set $\mathcal{T}\subset \mathbb{N}$ associated with maps $\{\mathsf{H}_{t_s}:\R\to \R\}_{s \in \mathcal{T}}$, we define the events
\begin{align*}
\mathscr{E}_{\mathcal{T}}(L)&\equiv \Big\{\max_{ \substack{\mathcal{P}\subset \mathcal{P}'\subset [n],\\ \abs{\mathcal{P}}\leq 3(\abs{\mathcal{T}}+2), \abs{\mathcal{P}'}-\abs{\mathcal{P}}\leq 3} }\max_{s \in \mathcal{T}}\bigpnorm{z^{(t_s)}_{[-\mathcal{P}]}-z^{(t_s)}_{[-\mathcal{P}']}}{}\leq L\Big\},\\
\mathscr{E}_{\mathsf{H}}(L)&\equiv \mathscr{E}_{\mathsf{H};\mathcal{T}}(L)\equiv \Big\{\max_{s\in \mathcal{T}, q \in \{0,1\}}\pnorm{\mathsf{H}_{t_s}^{(q)}(z^{(t_s)}) }{\infty}\leq L \Big\}.
\end{align*}
The main goal of this section is to prove the following control for  $\mathscr{A}_{k,(i,j)}\big(\mathcal{T}_0,(\mathcal{T}_{1},\mathcal{T}_2)\big)$ on the intersection of the above events.

\begin{proposition}\label{prop:A_hprod_quad}
Suppose $A=A_0/\sqrt{n}$, where $A_0$ is symmetric and the entries of its upper triangle are independent, centered random variables. Fix three non-overlapping sets $\mathcal{T}_0,\mathcal{T}_1,\mathcal{T}_2\subset \N$ of consecutive integers, where $\mathcal{T}_1,\mathcal{T}_2$ are allowed to take $\emptyset^{[)}$. Suppose $\max_{i,j\in [n]}\abs{A_{0,ij}}\leq K$ for some $K>1$ and $\mathcal{T}\equiv \cup_{\ell \in [0:2]} \mathcal{T}_\ell$. Then there exists a universal constant $c_0>0$ such that for $\abs{\mathcal{T}}_\ast\leq \log^2 n$,
\begin{align*}
&\Prob\Big(\max_{i,j,k \in [n]} n^{(\abs{\{i,j,k\}}-1)/2} \bigabs{\mathscr{A}_{k,(i,j)}\big(\mathcal{T}_0,(\mathcal{T}_{1},\mathcal{T}_2)\big) }\\
&\qquad\qquad \geq \big(K L \abs{\mathcal{T}}_\ast\log n\big)^{c_0\abs{\mathcal{T}}_\ast^2},\mathscr{E}_{\mathcal{T}}(L)\cap \mathscr{E}_{\mathsf{H}}(L)\Big)\leq   c_0^{\abs{\mathcal{T}}_\ast}e^{-(\log n)^{100}/c_0}. 
\end{align*}
Here $\abs{\mathcal{T}}_\ast\equiv 1\vee \abs{\mathcal{T}}$.
\end{proposition}
The above proposition also includes the case for $\mathscr{A}_{i,j}(\mathcal{T})$ by identifying
\begin{align*}
\max_{i,j\in [n]} n^{(\abs{\{i,j\}}-1 )/2}\bigabs{\mathscr{A}_{i,j}(\mathcal{T})}=\max_{i,j\in [n]} n^{(\abs{\{i,j\}}-1 )/2}\bigabs{\mathscr{A}_{i,(i,j)}\big(\emptyset,(\emptyset^{[)},\mathcal{T})\big) }.
\end{align*}

For technical reasons, we also need the following estimate.

\begin{proposition}\label{prop:A_cross_cubic}
	Suppose $A=A_0/\sqrt{n}$, where $A_0$ is symmetric and the entries of its upper triangle are independent, centered random variables. Fix $t_0 \in \N$, $q_0\geq 2$ and non-overlapping sets $\mathcal{T}_0,\mathcal{T}_1,\cdots,\mathcal{T}_{q_0}\subset \N$ of consecutive integers, where $\mathcal{T}_1,\cdots,\mathcal{T}_{q_0}$ are allowed to take $\emptyset^{[)}$. Suppose $\max_{i,j\in [n]}\abs{A_{0,ij}}\leq K$ for some $K>1$ and $\mathcal{T}\equiv \cup_{\ell \in [0:q_0]} \mathcal{T}_\ell \cup \{t_0\}$. Then there exists a constant $c_0=c_0(q_0)>0$ such that for $\abs{\mathcal{T}}_\ast\leq \log^2 n$,
	\begin{align*}
	&\Prob\bigg(n^{3/2}\max_{j\neq k \in [n]}\biggabs{\sum_{\ell_{0} \in [n]}\mathscr{A}_{k,(\ell_{0},\ell_{0})}(\mathcal{T}_{[0:2]}) \prod_{q=[3:q_0]} \mathscr{A}_{\ell_0,\ell_0}(\mathcal{T}_q) \cdot A_{\ell_{0},j}^3 \mathsf{H}_{t_{0};\ell_{0}}} \\
	&\qquad\qquad \geq \big(K L \abs{\mathcal{T}}_\ast\log n\big)^{c_0\abs{\mathcal{T}}_\ast^2},\mathscr{E}_{\mathcal{T}}(L)\cap \mathscr{E}_{\mathsf{H}}(L)\bigg)\leq  c_0^{\abs{\mathcal{T}}_\ast}e^{-(\log n)^{100}/c_0}. 
	\end{align*}
	Here $\abs{\mathcal{T}}_\ast\equiv 1\vee \abs{\mathcal{T}}$.
\end{proposition}

\begin{remark}
	Some technical remarks are in order:
	\begin{enumerate}
		\item Per notation convention, we interpret $\prod_{q=[3:q_0]}(\cdots)=1$ when $q_0=2$.
		\item In both propositions above, we may replace $\mathsf{H}_{t_s,\ell_s}=\big(\mathsf{H}_{t_s}\big)_{\ell_s}(z^{(t_s)}_{\ell_s})$ in the definition of $\mathscr{A}_{k,\ell}(\mathcal{I})$ by, e.g., $\big(\mathsf{H}_{t_s}\big)_{\ell_s}(z^{(t_s)}_{\ell_s})\big(\mathsf{H}_{t_s'}\big)_{\ell_s}(z^{(t_s')}_{\ell_s})$ for two chosen sets $t_{\mathcal{I}},t_{\mathcal{I}}'\subset \N$ (or more), at the cost of (i) a possibly enlarged exponent for $K$ in the above bound, and (ii) $\pnorm{z^{(t_s)}_{[-\mathcal{P}]}-z^{(t_s)}_{[-\mathcal{P}']}}{}\vee \pnorm{z^{(t_s')}_{[-\mathcal{P}]}-z^{(t_s')}_{[-\mathcal{P}']}}{}$ in the definition of $\mathscr{E}_{\mathcal{T}}(L)$.
		\item The constant $100$ in the probability estimates in both propositions above can be replaced by any fixed large real number.
	\end{enumerate}
\end{remark}

\noindent \textbf{Convention}. In the proofs of the above propositions below, $c>0$ denotes a universal constant whose numerical value may change from line to line. 

\subsection{The simplest possible case}

Let us consider the simplest possible case for Proposition \ref{prop:A_hprod_quad}, where $\mathcal{T}_0$ is a singleton and $\mathcal{T}_1=\mathcal{T}_2=\emptyset$. In this case, we may write, for some $t_0 \in \N$, $
\mathscr{A}_{k,(i,j)}\big(\mathcal{T}_0,(\mathcal{T}_{1},\mathcal{T}_2)\big) = \sum_{\ell \in [n]} A_{k\ell} A_{i\ell} A_{j\ell} \mathsf{H}(z_\ell^{(t_0)})$. Our goal is to prove
\begin{align}\label{ineq:A_hprod_simple_0}
\bigabs{n^{(\abs{ \{i,j,k\} }-1)/2}\mathscr{A}_{k,(i,j)} \big(\mathcal{T}_0,(\mathcal{T}_{1},\mathcal{T}_2)\big)}=\tilde{\bigop}(1).
\end{align}
Here $\tilde{\bigop}$ hides multiplicative logarithmic factors. Recall that for any $\mathcal{P}\subset [n]$:
\begin{itemize}
	\item $\mathscr{A}_{k,(i,j);[-\mathcal{P}]}^{ [-\mathcal{P}] }\big(\mathcal{T}_0,(\mathcal{T}_{1},\mathcal{T}_2)\big) = \sum_{\ell \in [n]\setminus \mathcal{P}} A_{k\ell} A_{i\ell} A_{j\ell} \mathsf{H}(z_{[-\mathcal{P}],\ell}^{(t_0)})$,
	\item $\mathscr{A}_{k,(i,j)}^{ [-\mathcal{P}] }\big(\mathcal{T}_0,(\mathcal{T}_{1},\mathcal{T}_2)\big) = \sum_{\ell \in [n]} A_{k\ell} A_{i\ell} A_{j\ell} \mathsf{H}(z_{[-\mathcal{P}],\ell}^{(t_0)})$.
\end{itemize}
Below we shall sketch a proof of (\ref{ineq:A_hprod_simple_0}) in three steps, and we shall omit $\big(\mathcal{T}_0,(\mathcal{T}_{1},\mathcal{T}_2)\big)$ for notational simplicity.

\noindent (\textbf{Step 1}). In the first step, we show that if one of $i,j,k$ belongs to $\mathcal{P}$, then
\begin{align}\label{ineq:A_hprod_simple_1}
\bigabs{n^{(\abs{ \{i,j,k\} }-1)/2}\mathscr{A}_{k,(i,j);[-\mathcal{P}]}^{ [-\mathcal{P}] }}=\tilde{\bigop}(1).
\end{align}
For instance, if $i,j \in \mathcal{P}$ and $k \in [n]\setminus \mathcal{P}$, then with $\tilde{\Lambda}_k\equiv \mathrm{diag}\big(A_{k\ell}\mathsf{H}(z_{[-\mathcal{P}],\ell}^{(t_0)}) \big)_{\ell \in [n]\setminus \mathcal{P}}$, using Hanson-Wright inequality and the simple estimate $\tr(\tilde{\Lambda}_k)=\tilde{\bigop}(n^{1/2})$,
\begin{align*}
\mathscr{A}_{k,(i,j);[-\mathcal{P}]}^{ [-\mathcal{P}] } &= A_{i\cdot}^\top \tilde{\Lambda}_k A_{j\cdot} = n^{-1}\cdot  \tilde{\bigop}\big(\delta_{ij} \tr(\tilde{\Lambda}_k)+ \pnorm{ \tilde{\Lambda}_k }{F}\big)=  \tilde{\bigop} \big(\delta_{ij} n^{-1/2}+n^{-1}\big).
\end{align*}
Other cases can also be computed explicitly.

For a general index $(\mathcal{T}_0,\mathcal{T}_1,\mathcal{T}_2)$, concentration of quadratic forms (via Hanson-Wright) or linear forms (via sub-gaussian inequality) leads to a size reduction of the index  $(\mathcal{T}_0,\mathcal{T}_1,\mathcal{T}_2)$. This program is detailed in Lemma \ref{lem:A_cross_ind_quad}.

\noindent (\textbf{Step 2}). In the second step, we show that for any $\mathcal{P}\subset [n]$,
\begin{align}\label{ineq:A_hprod_simple_2}
n^{(\abs{\{i,j,k\}}-1)/2} \bigabs{\mathscr{A}_{k,(i,j)}^{[-\mathcal{P}]} - \mathscr{A}_{k,(i,j);[-\mathcal{P}]}^{[-\mathcal{P}]}}=\tilde{\bigop}(1).
\end{align}
The claimed estimate in (\ref{ineq:A_hprod_simple_2}) is trivial in our case, as 
\begin{align*}
\bigabs{\mathscr{A}_{k,(i,j)}^{[-\mathcal{P}]} - \mathscr{A}_{k,(i,j);[-\mathcal{P}]}^{[-\mathcal{P}]}}=\biggabs{ \sum_{\ell \in \mathcal{P}} A_{k\ell} A_{i\ell} A_{j\ell} \mathsf{H}(z_{[-\mathcal{P}],\ell}^{(t_0)}) }=\tilde{\bigop}(n^{-3/2}).
\end{align*}
For a general index $(\mathcal{T}_0,\mathcal{T}_1,\mathcal{T}_2)$, the term $A_{k\ell} A_{i\ell} A_{j\ell}$ above need to be replaced by $\mathscr{A}_{k,\ell} \mathscr{A}_{i,\ell} \mathscr{A}_{j,\ell}$ which is generally bounded, up to multiplicative logarithmic factors, by $\big(\delta_{k\ell}+n^{-1/2}\big)\big(\delta_{i\ell}+n^{-1/2}\big)\big(\delta_{j\ell}+n^{-1/2}\big)$, and is therefore bounded by $n^{-(\abs{\{i,j,k\}}-1)/2}$ for the worst-case choice of $\ell$. Details of this  program are carried out in Lemma \ref{lem:A_P_cross_quad}. 

\noindent (\textbf{Step 3}). In the third step, we show that for any $\mathcal{P}\subset [n]$,
\begin{align}\label{ineq:A_hprod_simple_3}
n^{(\abs{\{i,j,k\}}-1)/2}\bigabs{\mathscr{A}_{k,(i,j)}-\mathscr{A}_{k,(i,j)}^{[-\mathcal{P}]} }=\tilde{\bigop}(1).
\end{align}
To see (\ref{ineq:A_hprod_simple_3}), it suffices to note that
\begin{align*}
\bigabs{\mathscr{A}_{k,(i,j)}-\mathscr{A}_{k,(i,j)}^{[-\mathcal{P}]} }&=\biggabs{ \sum_{\ell \in [n]} A_{k\ell} A_{i\ell} A_{j\ell} \big(\mathsf{H}(z_{\ell}^{(t_0)})-\mathsf{H}(z_{[-\mathcal{P}],\ell}^{(t_0)})  \big)  }\\
&\leq \bigg\{\sum_{\ell \in [n]} A_{k\ell}^2 A_{i\ell}^2 A_{j\ell}^2 \sum_{\ell \in [n]} \big(\mathsf{H}(z_{\ell}^{(t_0)})-\mathsf{H}(z_{[-\mathcal{P}],\ell}^{(t_0)})  \big)^2  \bigg\}^{1/2} \\
& = \tilde{\bigop}\big(n^{-1}\cdot \pnorm{z^{(t_0)}-z^{(t_0)}_{[-\mathcal{P}]}  }{}\big) = \tilde{\bigop}(n^{-1}).
\end{align*}
For a general index $(\mathcal{T}_0,\mathcal{T}_1,\mathcal{T}_2)$, using a similar argument as in Step 2, the factor $n^{-1}$ need to be replaced by  $n^{-(\abs{\{i,j,k\}}-1)/2}$ in the worst case. The details of the proof are implemented in Step 1 of the proof of Proposition \ref{prop:A_hprod_quad}.

The claimed estimate (\ref{ineq:A_hprod_simple_0}) now follows by combining Steps 1-3 with the choice $\mathcal{P}\equiv \{i,j,k\}$. For a general index $(\mathcal{T}_0,\mathcal{T}_1,\mathcal{T}_2)$, running Steps 1-3 above once only reduces its total size by 1, so we need to recursively run these steps until the index size trivializes. 

\subsection{Proof of Proposition \ref{prop:A_hprod_quad}}

For three non-overlapping sets $\{\mathcal{I}_q:q\in [0:2]\}$ of consecutive integers, and any $s_{[0:2]}\in \mathbb{Z}_{\geq 0}\times  \mathbb{Z}_{\geq -1}\times  \mathbb{Z}_{\geq -1}$, let
{\small\begin{align*}
\mathscr{U}_{ \{\mathcal{I}_{[0:2]}\};[-\mathcal{P}]}^{[-\mathcal{P}]}(s_{[0:2]})&\equiv  \max_{ \substack{ \{(\mathcal{J}_{w_q}^{(q)})\}\in \mathscr{C}_{\empty}(\{\mathcal{I}_{[0:2]}\};s_{[0:2]}) } } \prod_{\substack{w_0 \in [N_{\mathcal{I}_0}-1], \\w_{[2]} \in [2: N_{\mathcal{I}_{[2]}}] } } \bigg(\sum_{\gamma \in [2]}n^{(\gamma-1)/2} \max_{\substack{i,j \in [n]\setminus \mathcal{P},\\ \abs{\{i,j\}}=\gamma }  }\bigg) \bigabs{\mathscr{A}_{i,j;[-\mathcal{P}]}^{[-\mathcal{P}]}\big(\mathcal{J}_{w_q}^{(q)}\big)  } \\
&\qquad \times  \bigg(\sum_{\gamma \in [3]}n^{(\gamma-1)/2}\max_{ \substack{i,j,k \in [n]\setminus \mathcal{P},\\\abs{\{i,j,k\}}=\gamma} }\bigg) \bigabs{\mathscr{A}_{k,(i,j);[-\mathcal{P}]}^{[-\mathcal{P}]}\big(\mathcal{J}_{N_{\mathcal{I}_0}}^{(0)}, \big(\mathcal{J}_1^{(1)}, \mathcal{J}_1^{(2)}\big) \big) },
\end{align*}}
where $\mathscr{C}_{\empty}(\{\mathcal{I}_{[0:2]}\};s_{[0:2]}) $ collects all non-overlapping sets (including $\emptyset$) of consecutive integers $\{\mathcal{J}_{w_q}^{(q)}:w_q \in [N_{\mathcal{I}_q}]\}\subset \mathcal{I}_q$, $q= [0:2]$, with  $\sum_{w_q \in [N_{\mathcal{I}_q}]}\abs{\mathcal{J}_{w_q}^{(q)}}\in [-\bm{1}_{q\in \{1,2\}}: (s_q\vee -\bm{1}_{q\in \{1,2\}})]$ (if any of $s_q=-1$ for $q\in \{1,2\}$, then $\{\mathcal{J}_{w_q}^{(q)}\}= \emptyset^{[)}$). We shall use the convention that for more general $s_{[0:2]}\in \mathbb{Z}^3$, 
\begin{align*}
\mathscr{U}_{ \{\mathcal{I}_{[0:2]}\};[-\mathcal{P}]}^{[-\mathcal{P}]}(s_{[0:2]})\equiv \mathscr{U}_{ \{\mathcal{I}_{[0:2]}\};[-\mathcal{P}]}^{[-\mathcal{P}]}\big(s_0\vee 0, s_1\vee (-1), s_2\vee (-1)\big).
\end{align*}
Moreover, with the prescribed $\{\mathcal{I}_\cdot\}\equiv \{\mathcal{I}_{[0:2]}\}$, for $q\leq \abs{\mathcal{I}}\equiv \abs{\mathcal{I}_0\cup \mathcal{I}_1\cup \mathcal{I}_2}$,
\begin{align}
\mathfrak{U}_{\{\mathcal{I}_\cdot\};[-\mathcal{P}]}^{[-\mathcal{P}]}(q)\equiv \max_{ \substack{s_\ell\leq \abs{\mathcal{I}_\ell}, \abs{s_{[0:2]}}= q\vee (-2) } }\mathscr{U}_{\{\mathcal{I}_\cdot\};[-\mathcal{P}]}^{[-\mathcal{P}]}(s_{[0:2]}).
\end{align}

\subsubsection{Weighted estimates for $\mathscr{A}_{\ast;(\ast,\ast);[-\mathcal{P}]}^{[-\mathcal{P}]}(\ast)$ }
We first prove that for any $\mathcal{P}\subset [n]$, the size of $\mathscr{A}_{k,(i,j);[-\mathcal{P}]}^{[-\mathcal{P}]}\big(\mathcal{T}_0,(\mathcal{T}_1,\mathcal{T}_2)\big)$ can be reduced as long as one of the indices $i,j,k \in \mathcal{P}$.

\begin{lemma}\label{lem:A_cross_ind_quad}
	Suppose $A=A_0/\sqrt{n}$, where $A_0$ is symmetric and the entries of its upper triangle are independent, centered random variables. Fix $\mathcal{P}\subset [n]$ and three non-overlapping sets $\mathcal{T}_0,\mathcal{T}_1,\mathcal{T}_2\subset \N$ of consecutive integers, where $\mathcal{T}_1,\mathcal{T}_2$ are allowed to take $\emptyset^{[)}$. Suppose $\max_{i,j\in [n]}\abs{A_{0,ij}}\vee \max_{s\in \mathcal{T}, \ell \in [n]}\pnorm{\mathsf{H}_{t_s,\ell}}{\infty}\leq K$ for some $K>1$ and $\mathcal{T}\equiv \cup_{\ell \in [0:2]} \mathcal{T}_\ell$. Then there exists some universal $c_0>0$ such that the following hold with $\Prob(\cdot|A_{[-\mathcal{P}]})$-probability at least $1-c_0  e^{-(\log n)^{100}/c_0}$: uniformly for any tuple $(i,j,k) \in [n]^3$ such that $(i,j,k)\notin ([n]\setminus \mathcal{P})^3$, 
	\begin{align*}
	&n^{(\abs{ \{i,j,k\} }-1)/2}\bigabs{\mathscr{A}_{k,(i,j);[-\mathcal{P}]}^{[-\mathcal{P}]}\big(\mathcal{T}_0,(\mathcal{T}_1,\mathcal{T}_2)\big)} \leq (K \log n)^{c_0}\cdot \mathfrak{U}_{\{\mathcal{T}_\cdot\};[-\mathcal{P}]}^{[-\mathcal{P}]}\big(\abs{\mathcal{T}}-1\big).
	\end{align*}
\end{lemma}
\begin{proof}[Proof of Lemma \ref{lem:A_cross_ind_quad}: $\mathcal{T}_0\neq \emptyset$ and $\mathcal{T}_1,\mathcal{T}_2 \neq \emptyset^{[)}$]
    As the estimates below only involve concentration of linear and quadratic forms, we assume without loss of generality that $\E A_0\circ A_0 =1$. Let us now assume that $\mathcal{T}_0\equiv [r_0]$, $\mathcal{T}_1\equiv [r_0+1:r_1]$ and $\mathcal{T}_2\equiv [r_1+1:r_2]$ for notational simplicity. Then for $\mathcal{P}\subset [n]$,
	\begin{align*}
	&\mathscr{A}_{k,(i,j);[-\mathcal{P}]}^{[-\mathcal{P}]}\big(\mathcal{T}_0,(\mathcal{T}_{1},\mathcal{T}_2)\big)
	= \sum_{\ell_1,\ell_{r_1},\ell_{r_2}\in [n]\setminus \mathcal{P}} \lambda_{\ell_1,(\ell_{r_1},\ell_{r_2})} A_{k,\ell_1} A_{\ell_{r_1},i} A_{\ell_{r_2},j} \\
	& = \sum_{\ell_1 \in [n]\setminus \mathcal{P}} A_{k,\ell_1} \big(A_{i\cdot}^\top \Lambda_{\ell_1} A_{j\cdot}\big) = A_{i\cdot}^\top \bigg(\sum_{\ell_1 \in [n]\setminus \mathcal{P}} A_{k,\ell_1} \Lambda_{\ell_1}\bigg) A_{j\cdot}\equiv A_{i\cdot}^\top \tilde{\Lambda}_k A_{j\cdot},
	\end{align*}
	where, with slight abuse of notation, we write $A_{i\cdot}\equiv (A_{i\ell})_{\ell \in [n]\setminus \mathcal{P}} \in \R^{n-\abs{\mathcal{P}}}$, and $\Lambda_{\ell_1}\equiv \big(\lambda_{\ell_1,(\ell_{r_1},\ell_{r_2})}\big)_{\ell_{r_1},\ell_{r_2} \in [n]\setminus \mathcal{P}}$ with
	\begin{align}\label{ineq:A_cross_ind_quad_lambda}
	\lambda_{\ell_1,(\ell_{r_1},\ell_{r_2})}&\equiv \sum_{\ell_{\mathcal{T}_0^{(]}},\ell_{\mathcal{T}_1^{[)}}, \ell_{\mathcal{T}_2^{[)}} \in [n]\setminus \mathcal{P}}A_{\ell_1,[\ell_{ \mathcal{T}_0^{()}}],\ell_{\mathcal{T}_0^+}} A_{\ell_{\mathcal{T}_0^+},[\ell_{\mathcal{T}_1^{[)}}],\ell_{r_1}} A_{\ell_{\mathcal{T}_0^+},[\ell_{\mathcal{T}_2^{[)}}],\ell_{r_2}} \prod_{s \in[r_2]}\mathsf{H}_{t_s;\ell_{s}}^{[-\mathcal{P}]}\nonumber\\
	& = \mathscr{A}_{\ell_1,(\ell_{r_1},\ell_{r_2});[-\mathcal{P}]}^{[-\mathcal{P}]}\big(\mathcal{T}_0^{(]},(\mathcal{T}_1^{[)},\mathcal{T}_2^{[)})\big)\cdot \prod_{s \in \{1,{r_1},{r_2}\}} \mathsf{H}_{t_s;\ell_{s}}^{[-\mathcal{P}]}.
	\end{align}
	It is also easy to verify that
	\begin{align}\label{ineq:A_cross_ind_quad_lambda_tilde}
	(\tilde{\Lambda}_k)_{\ell_{r_1},\ell_{r_2}}= \mathscr{A}_{k,(\ell_{r_1},\ell_{r_2});[-\mathcal{P}]}^{[-\mathcal{P}]}\big(\mathcal{T}_0,(\mathcal{T}_1^{[)},\mathcal{T}_2^{[)})\big)\cdot \prod_{s \in \{{r_1},{r_2}\}} \mathsf{H}_{t_s;\ell_{s}}^{[-\mathcal{P}]}.
	\end{align}
	Consequently,
	\begin{align*}
	\max_{k \in [n]\setminus \mathcal{P}}\Big\{ n^{-1/2}\abs{\tr(\Lambda_k)}\vee \pnorm{\Lambda_k}{F}\Big\}&\leq K^c\cdot \mathscr{U}_{\{\mathcal{T}_\cdot\};[-\mathcal{P}]}^{[-\mathcal{P}]}\Big(\abs{\mathcal{T}_0}-1,\big(\abs{\mathcal{T}_1}-1,\abs{\mathcal{T}_2}-1\big)\Big),\\
	\max_{k \in [n]\setminus \mathcal{P}}\Big\{ n^{-1/2}\abs{\tr(\tilde{\Lambda}_k)}\vee \pnorm{\tilde{\Lambda}_k}{F}\Big\}&\leq K^c\cdot \mathscr{U}_{\{\mathcal{T}_\cdot\};[-\mathcal{P}]}^{[-\mathcal{P}]}\Big(\abs{\mathcal{T}_0},\big(\abs{\mathcal{T}_1}-1,\abs{\mathcal{T}_2}-1\big)\Big).
	\end{align*}

	\noindent (\textbf{Case 1}). Consider the case $i,j \in \mathcal{P}$. Using Hanson-Wright inequality, on an event $\mathscr{E}_{1}$ with $\Prob(\mathscr{E}_{1}^c|A_{[-\mathcal{P}]})\leq c  e^{-(\log n)^{100}/c}$, uniformly in $\ell_1 \in [n]\setminus \mathcal{P}$, we have $A_{i\cdot}^\top \Lambda_{\ell_1} A_{j\cdot}=n^{-1}\big(\delta_{ij} \tr(\Lambda_{\ell_1})+Z_{(i,j);\ell_1}\big)$ for some centered random variable $Z_{(i,j);\ell_1}$ that depends only on $A_{i\cdot},A_{j\cdot}$ and satisfies $\abs{Z_{(i,j);\ell_1}}\lesssim (K \log n)^c\cdot \pnorm{\Lambda_{\ell_1}}{F} $. 
	
	\noindent (\textbf{Subcase 1-(a)}). Suppose $k \in \mathcal{P}\setminus \{i,j\}$. Then $Z_{(i,j);\cdot}$ is independent of $A_{k\cdot}$. So using subgaussian concentration, on an event $\mathscr{E}_{1,a}$ with $\Prob(\mathscr{E}_{1,a}^c|A_{[-\mathcal{P}]})\leq c e^{-(\log n)^{100}/c}$,
	\begin{align*}
	&\bigabs{\mathscr{A}_{k,(i,j);[-\mathcal{P}]}^{[-\mathcal{P}]}\big(\mathcal{T}_0,(\mathcal{T}_1,\mathcal{T}_2)\big)}\\
	&\leq \biggabs{\frac{\delta_{ij}}{n} \sum_{\ell_1 \in [n]\setminus \mathcal{P}} A_{k,\ell_1} \tr(\Lambda_{\ell_1})}+ \frac{1}{n} \biggabs{  \sum_{\ell_1 \in [n]\setminus \mathcal{P}} A_{k,\ell_1} Z_{(i,j);\ell_1} }\\
	&\leq \frac{(K \log n)^c}{n^{3/2}}\bigg\{ \sum_{k \in [n]\setminus \mathcal{P}} \Big(\delta_{ij}\tr^2(\Lambda_k)+ \pnorm{\Lambda_k}{F}^2\Big)\bigg\}^{1/2}\\
	&\leq \frac{(K\log n)^{c}}{n} \big(n^{1/2}\delta_{ij}+1\big) \mathscr{U}_{\{\mathcal{T}_\cdot\};[-\mathcal{P}]}^{[-\mathcal{P}]}\Big(\abs{\mathcal{T}_0}-1,\big(\abs{\mathcal{T}_1}-1,\abs{\mathcal{T}_2}-1\big)\Big).
	\end{align*}
	
	\noindent (\textbf{Subcase 1-(b)}). Suppose $k \in \{i,j\}$. Then using the Cauchy-Schwarz inequality rather than subgaussian concentration, on $\mathscr{E}_{1}$, 
	\begin{align*}
	&\bigabs{\mathscr{A}_{k,(i,j);[-\mathcal{P}]}^{[-\mathcal{P}]}\big(\mathcal{T}_0,(\mathcal{T}_1,\mathcal{T}_2)\big)}\leq \frac{(K\log n)^c}{n}\bigg\{ \sum_{k \in [n]\setminus \mathcal{P}} \Big(\delta_{ij}\tr^2(\Lambda_k)+ \pnorm{\Lambda_k}{F}^2\Big)\bigg\}^{1/2}\\
	&\leq \frac{(K\log n)^c}{n^{1/2}} \big(n^{1/2}\delta_{ij}+1\big) \mathscr{U}_{\{\mathcal{T}_\cdot\};[-\mathcal{P}]}^{[-\mathcal{P}]}\Big(\abs{\mathcal{T}_0}-1,\big(\abs{\mathcal{T}_1}-1,\abs{\mathcal{T}_2}-1\big)\Big).
	\end{align*}
	
	\noindent (\textbf{Subcase 1-(c)}). Suppose $k \in [n]\setminus \mathcal{P}$. Then using Hanson-Wright inequality conditionally on $A_{[-\mathcal{P}]}$, on an event $\mathscr{E}_{1,c}$ with $\Prob(\mathscr{E}_{1,c}^c|A_{[-\mathcal{P}]})\leq c e^{-(\log n)^{100}/c}$,
	\begin{align*}
	&\bigabs{\mathscr{A}_{k,(i,j);[-\mathcal{P}]}^{[-\mathcal{P}]}\big(\mathcal{T}_0,(\mathcal{T}_1,\mathcal{T}_2)\big)} = \abs{A_{i\cdot}^\top \tilde{\Lambda}_k A_{j\cdot}}\leq n^{-1}\big(\delta_{ij} \abs{\tr(\tilde{\Lambda}_k)}+ (K \log n)^c \pnorm{ \tilde{\Lambda}_k }{F}\big)\\
	&\leq \frac{(K \log n)^c}{n} \big(n^{1/2}\delta_{ij}+1\big) \mathscr{U}_{\{\mathcal{T}_\cdot\};[-\mathcal{P}]}^{[-\mathcal{P}]}\Big(\abs{\mathcal{T}_0},\big(\abs{\mathcal{T}_1}-1,\abs{\mathcal{T}_2}-1\big)\Big).
	\end{align*}
	
	\noindent (\textbf{Case 2}). Consider the case $i \in \mathcal{P}, j \in [n]\setminus \mathcal{P}$. By subgaussian concentration, on an event $\mathscr{E}_2$ with $\Prob(\mathscr{E}_2^c|A_{[-\mathcal{P}]})\leq c  e^{-(\log n)^{100}/c}$, uniformly in $\ell_1 \in [n]\setminus \mathcal{P}$, 
	\begin{align*}
	\bigabs{A_{i\cdot}^\top \Lambda_{\ell_1} A_{j\cdot}}&\leq  n^{-1/2}(K\log n)^c \cdot  \pnorm{\Lambda_{\ell_1} A_{j\cdot}}{}\\
	&\leq n^{-1}(K\log n)^c \cdot  \big(n^{1/2}\delta_{\ell_1,j}+1\big) \mathscr{U}_{\{\mathcal{T}_\cdot\};[-\mathcal{P}]}^{[-\mathcal{P}]}\Big(\abs{\mathcal{T}_0}-1,\big(\abs{\mathcal{T}_1}-1,\abs{\mathcal{T}_2}\big)\Big).
	\end{align*}
	Here the last inequality follows as, using (\ref{ineq:A_cross_ind_quad_lambda}),
	\begin{align}\label{ineq:A_cross_ind_quad_1}
	\pnorm{\Lambda_{\ell_1} A_{j\cdot}}{}  &=\bigg\{\sum_{\ell_{r_1}\in [n]\setminus \mathcal{P}} \bigg(\sum_{\ell_{r_2}\in [n]\setminus \mathcal{P}}\lambda_{\ell_1,(\ell_{r_1},\ell_{r_2})}A_{\ell_{r_2},j}\bigg)^2\bigg\}^{1/2} \nonumber\\
	&= \bigg\{\sum_{\ell_{r_1} \in [n]\setminus \mathcal{P}} \bigg(\mathscr{A}_{\ell_1,(\ell_{r_1},j);[-\mathcal{P}]}^{[-\mathcal{P}]}\big(\mathcal{T}_0^{(]},(\mathcal{T}_1^{[)},\mathcal{T}_2)\big)\cdot \prod_{s \in \{1,{r_1}\}} \mathsf{H}_{t_s;\ell_{s}}^{[-\mathcal{P}]}\bigg)^2\bigg\}^{1/2}\nonumber\\
	&\lesssim K^c \big(\delta_{\ell_1,j}+n^{-1/2}\big) \mathscr{U}_{\{\mathcal{T}_\cdot\};[-\mathcal{P}]}^{[-\mathcal{P}]}\Big(\abs{\mathcal{T}_0}-1,\big(\abs{\mathcal{T}_1}-1,\abs{\mathcal{T}_2}\big)\Big).
	\end{align}

	\noindent (\textbf{Subcase 2-(a)}). Suppose $k \in \mathcal{P}\setminus \{i\}$. As $\abs{A_{i\cdot}^\top \Lambda_{\ell_1} A_{j\cdot}}$ is independent of $A_{k\cdot}$ conditional on $A_{[-\mathcal{P}]}$, using subgaussian concentration again, on an event $\mathscr{E}_{2,a}$ with $\Prob(\mathscr{E}_{2,a}^c|A_{[-\mathcal{P}]})\leq c e^{-(\log n)^{100}/c}$,
	\begin{align*}
	&\bigabs{\mathscr{A}_{k,(i,j);[-\mathcal{P}]}^{[-\mathcal{P}]}\big(\mathcal{T}_0,(\mathcal{T}_1,\mathcal{T}_2)\big)}\leq  n^{-1/2}(K\log n)^c\cdot \bigg\{\sum_{\ell_1 \in [n]\setminus \mathcal{P}} \abs{A_{i\cdot}^\top \Lambda_{\ell_1} A_{j\cdot}}^2\bigg\}^{1/2}\\
	&\leq \frac{(K\log n)^c}{n}\cdot \mathscr{U}_{\{\mathcal{T}_\cdot\};[-\mathcal{P}]}^{[-\mathcal{P}]}\Big(\abs{\mathcal{T}_0}-1,\big(\abs{\mathcal{T}_1}-1,\abs{\mathcal{T}_2}\big)\Big).
	\end{align*}
	\noindent (\textbf{Subcase 2-(b)}). Suppose $k\in \{i\}$. Using Cauchy-Schwarz, on an event $\mathscr{E}_{2,b}$ with $\Prob(\mathscr{E}_{2,b}^c|A_{[-\mathcal{P}]})\leq c e^{-(\log n)^{100}/c}$, 
	\begin{align*}
	&\bigabs{\mathscr{A}_{k,(i,j);[-\mathcal{P}]}^{[-\mathcal{P}]}\big(\mathcal{T}_0,(\mathcal{T}_1,\mathcal{T}_2)\big)}\leq K \bigg\{\sum_{\ell_1 \in [n]\setminus \mathcal{P}} \big(A_{i\cdot}^\top \Lambda_{\ell_1} A_{j\cdot}\big)^2\bigg\}^{1/2}\\
	&\leq \frac{(K\log n)^c }{n^{1/2}}\cdot \mathscr{U}_{\{\mathcal{T}_\cdot\};[-\mathcal{P}]}^{[-\mathcal{P}]}\Big(\abs{\mathcal{T}_0}-1,\big(\abs{\mathcal{T}_1}-1,\abs{\mathcal{T}_2}\big)\Big).
	\end{align*}
	
	\noindent (\textbf{Subcase 2-(c)}). Suppose $k\in [n]\setminus \mathcal{P}$. Then using subgaussian concentration conditionally on $A_{[-\mathcal{P}]}$ and then taking unconditional probability, on an event $\mathscr{E}_{2,c}$ with $\Prob(\mathscr{E}_{2,c}^c|A_{[-\mathcal{P}]})\leq c e^{-(\log n)^{100}/c}$, 
	\begin{align*}
	&\bigabs{\mathscr{A}_{k,(i,j);[-\mathcal{P}]}^{[-\mathcal{P}]}\big(\mathcal{T}_0,(\mathcal{T}_1,\mathcal{T}_2)\big)} = \abs{A_{i\cdot}^\top \tilde{\Lambda}_k A_{j\cdot}}\leq  n^{-1/2}(K\log n)^c\cdot \pnorm{ \tilde{\Lambda}_k A_{j\cdot} }{}\\
	&\leq \frac{(K\log n)^c}{n} \big(n^{1/2}\delta_{kj}+1\big) \mathscr{U}_{\{\mathcal{T}_\cdot\};[-\mathcal{P}]}^{[-\mathcal{P}]}\Big(\abs{\mathcal{T}_0},\big(\abs{\mathcal{T}_1}-1,\abs{\mathcal{T}_2}\big)\Big).
	\end{align*}
	Here the last inequality follows by similar arguments to (\ref{ineq:A_cross_ind_quad_1}), now using (\ref{ineq:A_cross_ind_quad_lambda_tilde}) instead of (\ref{ineq:A_cross_ind_quad_lambda}),
	\begin{align*}
	&\pnorm{ \tilde{\Lambda}_k A_{j\cdot} }{}
	 = \bigg\{\sum_{\ell_{r_1} \in [n]\setminus \mathcal{P}} \bigg( \mathscr{A}_{k,(\ell_{r_1},j);[-\mathcal{P}]}^{[-\mathcal{P}]}\big(\mathcal{T}_0,(\mathcal{T}_1^{[)},\mathcal{T}_2)\big) \mathsf{H}_{t_{r_1};\ell_{r_1}}^{[-\mathcal{P}]}  \bigg)^2\bigg\}^{1/2}\\
	&\leq K^c (\delta_{kj}+n^{-1/2})\mathscr{U}_{\{\mathcal{T}_\cdot\};[-\mathcal{P}]}^{[-\mathcal{P}]}\Big(\abs{\mathcal{T}_0},\big(\abs{\mathcal{T}_1}-1,\abs{\mathcal{T}_2}\big)\Big).
	\end{align*}

	\noindent (\textbf{Case 3}). Consider the case $i, j \in [n]\setminus \mathcal{P}$ and $k \in \mathcal{P}$. Then by subgaussian concentration, on an event $\mathscr{E}_3$ with $\Prob(\mathscr{E}_3^c|A_{[-\mathcal{P}]})\leq c e^{-(\log n)^{100}/c}$, 
	\begin{align*}
	&\bigabs{\mathscr{A}_{k,(i,j);[-\mathcal{P}]}^{[-\mathcal{P}]}\big(\mathcal{T}_0,(\mathcal{T}_1,\mathcal{T}_2)\big)}\leq  n^{-1/2}(K\log n)^c\cdot \bigg\{\sum_{\ell_1 \in [n]\setminus \mathcal{P}} \abs{A_{i\cdot}^\top \Lambda_{\ell_1} A_{j\cdot}}^2\bigg\}^{1/2}\\
	&\leq \frac{(K\log n)^c}{n} \big(n^{1/2}\delta_{ij}+1\big) \mathscr{U}_{\{\mathcal{T}_\cdot\};[-\mathcal{P}]}^{[-\mathcal{P}]}\Big(\abs{\mathcal{T}_0}-1,\big(\abs{\mathcal{T}_1},\abs{\mathcal{T}_2}\big)\Big).
	\end{align*}
	Here the last inequality follows as $\sum_{\ell_1 \in [n]\setminus \mathcal{P}}\abs{A_{i\cdot}^\top \Lambda_{\ell_1} A_{j\cdot}}^2$ is equal to 
	\begin{align*}
	& \sum_{\ell_1 \in [n]\setminus \mathcal{P}}\bigg(\sum_{\ell_{r_1},\ell_{r_2} \in [n] \setminus \mathcal{P}} A_{i,\ell_{r_1}}A_{j,\ell_{r_2}} \mathscr{A}_{\ell_1,(\ell_{r_1},\ell_{r_2});[-\mathcal{P}]}^{[-\mathcal{P}]}\big(\mathcal{T}_0^{(]},(\mathcal{T}_1^{[)},\mathcal{T}_2^{[)})\big)\cdot \prod_{s \in \{1,{r_1},{r_2}\}} \mathsf{H}_{t_s;\ell_{s}}^{[-\mathcal{P}]}\bigg)^2\\
	& = \sum_{\ell_1 \in [n]\setminus \mathcal{P}} \bigg(\mathscr{A}_{\ell_1,(i,j);[-\mathcal{P}]}^{[-\mathcal{P}]}\big(\mathcal{T}_0^{(]},(\mathcal{T}_1,\mathcal{T}_2)\big)\cdot \mathsf{H}_{t_1;\ell_{1}}^{[-\mathcal{P}]}\bigg)^2 \\
	&\leq K^c  \big(\delta_{ij}+n^{-1}\big)\cdot \bigg(\mathscr{U}_{\{\mathcal{T}_\cdot\};[-\mathcal{P}]}^{[-\mathcal{P}]}\Big(\abs{\mathcal{T}_0}-1,\big(\abs{\mathcal{T}_1},\abs{\mathcal{T}_2}\big)\Big)\bigg)^2.
	\end{align*}
	Summarizing all the cases to conclude the scenario $\mathcal{T}_0\neq \emptyset$ and $\mathcal{T}_1,\mathcal{T}_2 \neq \emptyset^{[)}$.
	\end{proof}
	
	\begin{proof}[Proof of Lemma \ref{lem:A_cross_ind_quad}: Any of $\mathcal{T}_0\neq \emptyset, \mathcal{T}_1\neq \emptyset^{[)}, \mathcal{T}_2 \neq \emptyset^{[)}$ fails, but $(\mathcal{T}_0,\mathcal{T}_1,\mathcal{T}_2)\neq \big(\emptyset,\emptyset^{[)},\emptyset^{[)}\big)$]
	We have two cases:
	\begin{itemize}
		\item For $\mathcal{T}_0= \emptyset$, $\mathscr{A}_{k,(i,j);[-\mathcal{P}]}^{[-\mathcal{P}]}\big(\mathcal{T}_0,(\mathcal{T}_{1},\mathcal{T}_2)\big)= \mathsf{H}_{\ast;\ast}\cdot \mathscr{A}_{k,i;[-\mathcal{P}]}^{[-\mathcal{P}]}(\mathcal{T}_1)\mathscr{A}_{k,j;[-\mathcal{P}]}^{[-\mathcal{P}]}(\mathcal{T}_2)$, so by Lemma \ref{lem:set_merge}, we have
		\begin{align*}
		&n^{(\abs{\{i,j,k\} }-1)/2}\bigabs{\mathscr{A}_{k,(i,j);[-\mathcal{P}]}^{[-\mathcal{P}]}\big(\mathcal{T}_0,(\mathcal{T}_{1},\mathcal{T}_2)\big)}\\
		&\leq K\cdot n^{(\abs{\{i,k\} }-1)/2}\bigabs{ \mathscr{A}_{k,i;[-\mathcal{P}]}^{[-\mathcal{P}]}(\mathcal{T}_1) }\cdot n^{(\abs{\{j,k\} }-1)/2}\bigabs{ \mathscr{A}_{k,j;[-\mathcal{P}]}^{[-\mathcal{P}]}(\mathcal{T}_2) }.
		\end{align*}
		In this case, $\abs{\mathcal{T}}=\abs{\{\mathcal{T}_1, \mathcal{T}_2\}}$.
		\item For $\mathcal{T}_1 = \emptyset^{[)}$,  $\mathscr{A}_{k,(i,j);[-\mathcal{P}]}^{[-\mathcal{P}]}\big(\mathcal{T}_0,(\mathcal{T}_{1},\mathcal{T}_2)\big)= \mathsf{H}_{\ast;\ast}\cdot \mathscr{A}_{k,i;[-\mathcal{P}]}^{[-\mathcal{P}]}(\mathcal{T}_0^{[)})\mathscr{A}_{i,j;[-\mathcal{P}]}^{[-\mathcal{P}]}(\mathcal{T}_2)$, so again by Lemma \ref{lem:set_merge}, we have
		\begin{align*}
		&n^{(\abs{\{i,j,k\} }-1)/2}\bigabs{\mathscr{A}_{k,(i,j);[-\mathcal{P}]}^{[-\mathcal{P}]}\big(\mathcal{T}_0,(\mathcal{T}_{1},\mathcal{T}_2)\big)}\\
		&\leq K\cdot n^{(\abs{\{i,k\} }-1)/2}\bigabs{ \mathscr{A}_{k,i;[-\mathcal{P}]}^{[-\mathcal{P}]}(\mathcal{T}_0^{[)}) }\cdot n^{(\abs{\{j,k\} }-1)/2}\bigabs{ \mathscr{A}_{i,j;[-\mathcal{P}]}^{[-\mathcal{P}]}(\mathcal{T}_2) }.
		\end{align*}
		In this case, $\abs{\mathcal{T}}=\abs{\{\mathcal{T}_0^{[)}, \mathcal{T}_2\}}$. The same argument applies to the case $\mathcal{T}_2 = \emptyset^{[)}$.
	\end{itemize}
	 For both cases, we are left to consider $\mathscr{A}_{i,j;[-\mathcal{P}]}^{[-\mathcal{P}]}(\mathcal{T})$ for $\abs{\mathcal{T}}\geq 0$ and $(i,j)\notin ([n]\setminus \mathcal{P})^2$, which is identified as $\mathscr{A}_{i,j;[-\mathcal{P}]}^{[-\mathcal{P}]}(\mathcal{T})=\mathscr{A}_{i,(i,j);[-\mathcal{P}]}^{[-\mathcal{P}]}\big(\emptyset,(\emptyset^{[)},\mathcal{T})\big)$. If $\mathcal{T}=\emptyset$, then by definition $\abs{\mathscr{A}_{i,j;[-\mathcal{P}]}^{[-\mathcal{P}]}(\mathcal{T})}\leq Kn^{-1/2}$, so let us consider $\abs{\mathcal{T}}\geq 1$. We may then write $\abs{\mathscr{A}_{i,j;[-\mathcal{P}]}^{[-\mathcal{P}]}(\mathcal{T})}=A_i^\top \Gamma_{\mathcal{T}} A_j$, where $\Gamma_{\mathcal{T}}=\big(\Gamma_{\mathcal{T};(\ell_{r_1},\ell_{r_2})}\big)_{ \ell_{r_1},\ell_{r_2}\in [n]\setminus \mathcal{P}}$ with $
	\Gamma_{\mathcal{T};(\ell_{r_1},\ell_{r_2})} = \mathscr{A}_{\ell_{r_1},\ell_{r_2};[-\mathcal{P}]}^{[-\mathcal{P}]}(\mathcal{T}^{()}) \cdot \prod_{s \in \{{r_1},{r_2}\}} \mathsf{H}_{t_s;\ell_{s}}^{[-\mathcal{P}]}$. 
	From here we may use similar concentration arguments as above using that either $i$ or $j$ belong to $\mathcal{P}$. We omit these repetitive details.
\end{proof}

\subsubsection{Replacing $\mathscr{A}_{\ast;(\ast,\ast);[-\mathcal{P}]}^{[-\mathcal{P}]}(\ast)$  by $\mathscr{A}_{\ast;(\ast,\ast)}^{[-\mathcal{P}]}(\ast)$}
Next we prove that $\mathscr{A}_{k,(i,j)}^{[-\mathcal{P}]}\big(\mathcal{T}_0,(\mathcal{T}_{1},\mathcal{T}_2)\big)$ is typically of the same order as  $\mathscr{A}_{k,(i,j);[-\mathcal{P}]}^{[-\mathcal{P}]}\big(\mathcal{T}_0,(\mathcal{T}_{1},\mathcal{T}_2)\big)$, adjusting the multiplicity of $(i,j,k)$.

\begin{lemma}\label{lem:A_P_cross_quad}
Suppose $A=A_0/\sqrt{n}$, where $A_0$ is symmetric and the entries of its upper triangle are independent, centered random variables. Fix $\mathcal{P}\subset [n]$ and three non-overlapping sets $ \mathcal{T}_0,\mathcal{T}_1,\mathcal{T}_2\subset \N$ of consecutive integers, where $\mathcal{T}_1,\mathcal{T}_2$ are allowed to take $\emptyset^{[)}$. Suppose $\max_{i,j\in [n]}\abs{A_{0,ij}}\vee \max_{s\in \mathcal{T}, \ell \in [n]}\pnorm{\mathsf{H}_{t_s,\ell}}{\infty}\leq K$ for some $K>1$ and $\mathcal{T}\equiv \cup_{\ell \in [0:2]} \mathcal{T}_\ell$. Then there exists a universal constant $c_0>0$ such that the following hold with $\Prob(\cdot|A_{[-\mathcal{P}]})$-probability at least $1- (c_0\abs{\mathcal{P}})^{\abs{\mathcal{T}}_\ast} e^{-(\log n)^{100}/c_0}$: uniformly for all tuples $(i,j,k) \in [n]^3$,
\begin{align*}
&n^{(\abs{\{i,j,k\}}-1)/2} \bigabs{\mathscr{A}_{k,(i,j)}^{[-\mathcal{P}]}\big(\mathcal{T}_0,(\mathcal{T}_{1},\mathcal{T}_2)\big) - \mathscr{A}_{k,(i,j);[-\mathcal{P}]}^{[-\mathcal{P}]}\big(\mathcal{T}_0,(\mathcal{T}_{1},\mathcal{T}_2)\big)}\\
&\leq (\abs{\mathcal{P}} K \log n)^{c_0 \abs{\mathcal{T}}_\ast}\cdot \mathfrak{U}_{\{\mathcal{T}_\cdot\};[-\mathcal{P}]}^{[-\mathcal{P}]}\big(\abs{\mathcal{T}}-1\big).
\end{align*}
\end{lemma}
\begin{proof}[Proof of Lemma \ref{lem:A_P_cross_quad}: $\mathcal{T}_0\neq \emptyset$ and $\mathcal{T}_1,\mathcal{T}_2 \neq \emptyset^{[)}$]
Without loss of generality, we assume that $\mathcal{T}_0\equiv [r_0]$, $\mathcal{T}_1\equiv [r_0+1:r_1]$ and $\mathcal{T}_2\equiv [r_1+1:r_2]$. Then 
\begin{align}\label{ineq:A_P_cross_quad_1}
&\mathscr{A}_{k,(i,j)}^{[-\mathcal{P}]}\big(\mathcal{T}_0,(\mathcal{T}_{1},\mathcal{T}_2)\big) - \mathscr{A}_{k,(i,j);[-\mathcal{P}]}^{[-\mathcal{P}]}\big(\mathcal{T}_0,(\mathcal{T}_{1},\mathcal{T}_2)\big)\\
 &\quad = \sum_{\substack{\mathscr{B},\mathscr{C},\mathscr{D} \in \{\mathscr{A},\mathscr{R}\},\\ (\mathscr{B},\mathscr{C},\mathscr{D})\neq (\mathscr{A},\mathscr{A},\mathscr{A})} } \sum_{\ell_{\mathcal{T}_0^+} \in [n]} \mathscr{B}_{k,\ell_{\mathcal{T}_0^+};[-\mathcal{P}]}^{[-\mathcal{P}]}(\mathcal{T}_0^{[)}) \mathsf{H}_{t_{\mathcal{T}_0^+};\ell_{\mathcal{T}_0^+}}\cdot  \mathscr{C}_{\ell_{\mathcal{T}_0^+}, i;[-\mathcal{P}]}^{[-\mathcal{P}]}(\mathcal{T}_1) \mathscr{D}_{\ell_{\mathcal{T}_0^+}, j;[-\mathcal{P}]}^{[-\mathcal{P}]}(\mathcal{T}_2),\nonumber
\end{align}
and we denote the terms on the right hand side of the above display as $\Delta_{(i,j,k)}$. Here for $k,\ell \in [n]$ and any set ${\mathcal{I}}$ of consecutive integers,
\begin{align*}
\mathscr{R}_{k,\ell;[-\mathcal{P}]}^{[-\mathcal{P}]}({\mathcal{I}})& \equiv \mathscr{A}_{k,\ell}^{[-\mathcal{P}]}({\mathcal{I}})-\mathscr{A}_{k,\ell;[-\mathcal{P}]}^{[-\mathcal{P}]}({\mathcal{I}}) = \sum_{\mathcal{Q}\subsetneq {\mathcal{I}} } \sum_{\ell_{\mathcal{Q}}\in [n]\setminus \mathcal{P}}\sum_{\ell_{\mathcal{I}\setminus \mathcal{Q}} \in \mathcal{P} } A_{k,[\ell_{\mathcal{I}}],\ell} \prod_{s \in \mathcal{I} } \mathsf{H}_{t_s;\ell_s}^{[-\mathcal{P}]}.
\end{align*}
Consequently, each term in  $\sum_{(\mathscr{B},\mathscr{C},\mathscr{D})}$ in (\ref{ineq:A_P_cross_quad_1}) takes the following form: for some  $\mathcal{Q}_0\subset \mathcal{T}_0^{[)},\mathcal{Q}_1\subset \mathcal{T}_1,\mathcal{Q}_2 \subset \mathcal{T}_2$, one of which being a proper subset, 
\begin{align}\label{ineq:A_P_cross_quad_2}
&\sum_{\ell_{\mathcal{T}_0^+} \in [n] }\sum_{ \ell_{\mathcal{T}_0^{[)}\setminus \mathcal{Q}_0}, \ell_{\mathcal{T}_{[2]}\setminus \mathcal{Q}_{[2]}} \in \mathcal{P} } \sum_{\ell_{\mathcal{Q}_{[0:2]}}\in [n]\setminus \mathcal{P}} A_{k,[\ell_{\mathcal{T}_0^{[)}}],\ell_{\mathcal{T}_0^+}} A_{\ell_{\mathcal{T}_0^+,[\ell_{\mathcal{T}_1}],i}} A_{\ell_{\mathcal{T}_0^+,[\ell_{\mathcal{T}_2}],j}}   \prod_{s \in \mathcal{T} } \mathsf{H}_{t_s;\ell_s}^{[-\mathcal{P}]}\nonumber\\
&= \bigg(  \sum_{ \ell_{\mathcal{T}_0^{[)}\setminus \mathcal{Q}_0}, \ell_{\mathcal{T}_{[2]}\setminus \mathcal{Q}_{[2]}} \in \mathcal{P} }\sum_{\ell_{\mathcal{T}_0^+}, \ell_{\mathcal{Q}_{[0:2]}}\in [n]\setminus \mathcal{P}}+   \sum_{\ell_{\mathcal{T}_0^+}, \ell_{\mathcal{T}_0^{[)}\setminus \mathcal{Q}_0}, \ell_{\mathcal{T}_{[2]}\setminus \mathcal{Q}_{[2]}} \in \mathcal{P} }\sum_{ \ell_{\mathcal{Q}_{[0:2]}}\in [n]\setminus \mathcal{P}}\bigg)(\cdots)\nonumber\\
&\equiv (I)_{i,j,k;\mathcal{Q}_{[0:2]}}+ (II)_{i,j,k;\mathcal{Q}_{[0:2]}}.
\end{align}
For $q\in [0:2]$, let
$\mathcal{Q}_q=\cup_{w \in [N_{\mathcal{Q}_q}]} \mathcal{U}^{(q)}_w$ be a consecutive-integer-set representation  of $\mathcal{Q}_q$. We then have the following consecutive-integer-set representation of $\mathcal{T}$:
\begin{align}\label{ineq:A_P_cross_quad_cir}
&\mathcal{V}_1^{(0)},\mathcal{U}_1^{(0)},\ldots, \mathcal{V}_{N_{\mathcal{Q}_0}}^{(0)},\mathcal{U}_{N_{\mathcal{Q}_0}}^{(0)}, \mathcal{V}_{N_{\mathcal{Q}_0}+1}^{(0)}, \mathcal{T}_0^+, \nonumber \\
&\mathcal{V}_1^{(1)},\mathcal{U}_1^{(1)},\ldots, \mathcal{V}_{N_{\mathcal{Q}_1}}^{(1)},\mathcal{U}_{N_{\mathcal{Q}_1}}^{(1)}, \mathcal{V}_{N_{\mathcal{Q}_1}+1}^{(1)},\nonumber\\
&\mathcal{V}_1^{(2)},\mathcal{U}_1^{(2)},\ldots, \mathcal{V}_{N_{\mathcal{Q}_2}}^{(2)},\mathcal{U}_{N_{\mathcal{Q}_2}}^{(2)}, \mathcal{V}_{N_{\mathcal{Q}_2}+1}^{(2)}.
\end{align}
Then the sums are over $[n]\setminus \mathcal{P}$ for indices in $\mathcal{U}_\cdot^{(\cdot)}$, and over $\mathcal{P}$ for indices in $\mathcal{V}_\cdot^{(\cdot)}$.

Let us now consider $(I)_{i,j,k;\mathcal{Q}_{[0:2]}}$, where the sum is over $[n]\setminus \mathcal{P}$ for $\mathcal{T}_0^+$. For each term in the summation over $\{\mathcal{V}_{\cdot}^{(\cdot)}\}$, we search factors in the following order: 
\begin{enumerate}
	\item $ A_{[\ell_{ \mathcal{J} }]}$ where for $\mathcal{J}=\emptyset$, it is understood $ A_{[\ell_{ \mathcal{J} }]}=1$;
	\item  $\mathscr{A}_{-,+;[-\mathcal{P}]}^{[-\mathcal{P}]}(\mathcal{U}_\cdot^{(\cdot)})$ with indices $-,+\notin ([n]\setminus\mathcal{P})^2$;
	\item  $\mathscr{A}_{*,(-,+);[-\mathcal{P}]}^{[-\mathcal{P}]}(\mathcal{I},(\cdot,\cdot))$ for some $\mathcal{I}$ with $\mathcal{I}^+=\mathcal{T}_0^+$, and $(\ast,-,+)\notin ([n]\setminus \mathcal{P})^3$. 
\end{enumerate}
We set the multiplicity of the factor in (1) as $\abs{\mathcal{J}}$, the factor in (2) as $\abs{\{-,+\}}-1$, and the factor in (3) as $\abs{\{\ast,-,+\}}-1$. The search for each term in the summation over $\{\mathcal{V}_{\cdot}^{(\cdot)}\}$ terminates with multiplicity $\geq\abs{\{i,j,k\}}-1$. Now an application of Lemma \ref{lem:A_cross_ind_quad} above shows that on an event $\mathscr{E}_{1,\mathcal{Q}_{[0:2]}}$ with $\Prob(\mathscr{E}_{1,\mathcal{Q}_{[0:2]}}|A_{[-\mathcal{P}]})\leq (c \abs{\mathcal{P}})^{r_2}  e^{-(\log n)^{100}/c}$, uniformly in $(i,j,k) \in [n]^3$,
\begin{align}\label{ineq:A_P_cross_quad_3}
n^{(\abs{\{i,j,k\}}-1)/2}\abs{(I)_{i,j,k;\mathcal{Q}_{[0:2]}}} \leq (\abs{\mathcal{P}} K \log n)^{c r_2}\cdot \mathfrak{U}_{\{\mathcal{T}_\cdot\};[-\mathcal{P}]}^{[-\mathcal{P}]}\big(\abs{\mathcal{T}}-1\big).
\end{align}
Let us now consider $(II)_{i,j,k;\mathcal{Q}_{[0:2]}}$, where the sum is over $\mathcal{P}$ for $\mathcal{T}_0^+$. Using a completely similar argument as above (but without the case (3)), on an event $\mathscr{E}_{2,\mathcal{Q}_{[0:2]}}$ with $\Prob(\mathscr{E}_{2,\mathcal{Q}_{[0:2]} }|A_{[-\mathcal{P}]})\leq  (c \abs{\mathcal{P}})^{r_2}  e^{-(\log n)^{100}/c}$, uniformly in $(i,j,k) \in [n]^3$,
\begin{align}\label{ineq:A_P_cross_quad_4}
n^{(\abs{\{i,j,k\}}-1)/2}\abs{(II)_{i,j,k;\mathcal{Q}_{[0:2]}}} \leq (\abs{\mathcal{P}} K \log n)^{c r_2}\cdot \mathfrak{U}_{\{\mathcal{T}_\cdot\};[-\mathcal{P}]}^{[-\mathcal{P}]}\big(\abs{\mathcal{T}}-1\big).
\end{align}
Consequently, combining (\ref{ineq:A_P_cross_quad_2})-(\ref{ineq:A_P_cross_quad_4}), on the event $E_0\equiv \cap_{\mathcal{Q}_{[0:2]}} \big(\mathscr{E}_{1,\mathcal{Q}_{[0:2]}}\cap \mathscr{E}_{2,\mathcal{Q}_{[0:2]}}\big)$, where the outside intersection is taken over $\mathcal{Q}_0\subset \mathcal{T}_0^{[)},\mathcal{Q}_1\subset \mathcal{T}_1,\mathcal{Q}_2 \subset \mathcal{T}_2$ with at least one of them being proper, we have uniformly in $(i,j,k) \in [n]^3$
\begin{align}\label{ineq:A_P_cross_quad_5}
 n^{(\abs{\{i,j,k\}}-1)/2}\abs{\Delta_{2,(i,j,k)}}\leq (\abs{\mathcal{P}} K\log n)^{c_0 r_2}\cdot \mathfrak{U}_{\{\mathcal{T}_\cdot\};[-\mathcal{P}]}^{[-\mathcal{P}]}\big(\abs{\mathcal{T}}-1\big).
\end{align}
The claim for $\mathcal{T}_0\neq \emptyset$ and $\mathcal{T}_1,\mathcal{T}_2 \neq \emptyset^{[)}$ follows by combining (\ref{ineq:A_P_cross_quad_1}) and (\ref{ineq:A_P_cross_quad_5}), upon noting that $\Prob(E_0^c|A_{[-\mathcal{P}]})\leq   (c \abs{\mathcal{P}})^{r_2}  e^{-(\log n)^{100}}$. 
\end{proof}
\begin{proof}[Proof of Lemma \ref{lem:A_P_cross_quad}: Any of $\mathcal{T}_0\neq \emptyset, \mathcal{T}_1\neq \emptyset^{[)}, \mathcal{T}_2 \neq \emptyset^{[)}$ fails, but $(\mathcal{T}_0,\mathcal{T}_1,\mathcal{T}_2)\neq \big(\emptyset,\emptyset^{[)},\emptyset^{[)}\big)$]
 It suffices to consider 
\begin{align*}
n^{(\abs{\{i,j\}}-1)/2} \bigabs{\mathscr{A}_{i,j}^{[-\mathcal{P}]}\big(\mathcal{T}\big) - \mathscr{A}_{i,j;[-\mathcal{P}]}^{[-\mathcal{P}]}\big(\mathcal{T}\big)}=n^{(\abs{\{i,j\}}-1)/2} \bigabs{\mathscr{R}_{k,\ell;[-\mathcal{P}]}^{[-\mathcal{P}]}({\mathcal{T}})},
\end{align*}
which can be handled similar to $(II)_{i,j,k;\mathcal{Q}_{[0:2]}}$ in the above proof. Details are omitted.
\end{proof}

\subsubsection{Proof of Proposition \ref{prop:A_hprod_quad}: replacing $\mathscr{A}_{\ast;(\ast,\ast)}^{[-\mathcal{P}]}(\ast)$ by $\mathscr{A}_{\ast;(\ast,\ast)}(\ast)$ }

	Throughout we shall work on the event $\mathscr{E}_{\mathcal{T}}(L)\cap \mathscr{E}_{\mathsf{H}}(L)$.
	
	\noindent (\textbf{Step 1}). We claim that
	on an event $E\cap \mathscr{E}_{\mathcal{T}}(L)\cap \mathscr{E}_{\mathsf{H}}(L)$ with $\Prob(E^c)\leq c_0^{r}  e^{-(\log n)^{100}}$, for any sets of consecutive integers $\mathcal{T}_0,\mathcal{T}_1,\mathcal{T}_2\subset \mathbb{N}$ where $\mathcal{T}_1,\mathcal{T}_2$ may take $\emptyset^{[)}$ and $\abs{\mathcal{T}_0\cup \mathcal{T}_1\cup \mathcal{T}_2}\leq r$,
	\begin{align}\label{ineq:A_hprod_quad_claim}
	&\max_{i,j,k \in [n]} n^{(\abs{\{i,j,k\}}-1)/2} \bigabs{\mathscr{A}_{k,(i,j)}\big(\mathcal{T}_0,(\mathcal{T}_{1},\mathcal{T}_2)\big) }\nonumber\\
	&\leq \big(K L \log n\big)^{c_0 r} \cdot \max_{\substack{\mathcal{P}\subset [n], \abs{\mathcal{P}}\leq 3 }}\mathfrak{U}_{\{\mathcal{T}_\cdot\};[-\mathcal{P}]}^{[-\mathcal{P}]}\big(\abs{\mathcal{T}}-1\big).
	\end{align}
	
	\noindent (\textbf{Step 1-(a)}). In this step we consider the scenario that $\mathcal{T}_0\neq \emptyset$ and $\mathcal{T}_1,\mathcal{T}_2 \neq \emptyset^{[)}$. 
	Without loss of generality, we assume that $\mathcal{T}_0\equiv [r_0]$, $\mathcal{T}_1\equiv [r_0+1:r_1]$ and $\mathcal{T}_2\equiv [r_1+1:r_2]$. Fix a generic set $\mathcal{P}\subset [n]$ to be chosen later. Let $\Delta \mathsf{H}_{t_s;\ell_s}\equiv \mathsf{H}_{t_s}\big(z^{(t_s)}_{\ell_s}\big)-\mathsf{H}_{t_s}\big(z^{(t_s)}_{[-\mathcal{P}],\ell_s}\big)$. Note that
	\begin{align}\label{ineq:A_hprod_quad_1}
	&\mathscr{A}_{k,(i,j)}\big(\mathcal{T}_0,(\mathcal{T}_1,\mathcal{T}_2)\big)-\mathscr{A}_{k,(i,j)}^{[-\mathcal{P}]}\big(\mathcal{T}_0,(\mathcal{T}_1,\mathcal{T}_2)\big)\nonumber\\
	& =  \sum_{\ell_{\mathcal{T}} \in [n]} A_{k,[\ell_{\mathcal{T}_0^{[)}}], \ell_{\mathcal{T}_0^+}} A_{\ell_{\mathcal{T}_0^+}, [\ell_{\mathcal{T}_1}], i}A_{\ell_{\mathcal{T}_0^+}, [\ell_{\mathcal{T}_2}], j}  \cdot\bigg(\sum_{ \mathcal{Q}_{[0:2]}\subset \mathcal{T}\setminus \{\mathcal{T}_0^+\} }\prod_{s \in \mathcal{Q}_{[0:2]} } \mathsf{H}_{t_s;\ell_s}^{[-\mathcal{P}]} \prod_{s \in \mathcal{T}\setminus \mathcal{Q}_{[0:2]} } \Delta \mathsf{H}_{t_s;\ell_s}\nonumber\\
	&\qquad + \sum_{ \mathcal{Q}_{[0:2]}\subsetneq \mathcal{T}\setminus \{\mathcal{T}_0^+\}  }\prod_{s \in \overline{\mathcal{Q}}_{[0:2]} } \mathsf{H}_{t_s;\ell_s}^{[-\mathcal{P}]} \prod_{s \in \mathcal{T}\setminus \overline{\mathcal{Q}}_{[0:2]} } \Delta \mathsf{H}_{t_s;\ell_s} \bigg),
	\end{align}
	where the first summation $\sum_{ \mathcal{Q}_{[0:2]}\subset \mathcal{T}\setminus \{\mathcal{T}_0^+\}  }$ runs over $\mathcal{Q}_0\subset \mathcal{T}_0^{[)},\mathcal{Q}_1\subset \mathcal{T}_1,\mathcal{Q}_2 \subset \mathcal{T}_2$, the second $\sum_{ \mathcal{Q}_{[0:2]}\subsetneq \mathcal{T}\setminus \{\mathcal{T}_0^+\}  }$ with at least one of them being proper, and $\overline{\mathcal{Q}}_{[0:2]}\equiv \mathcal{Q}_{[0:2]}\cup \{\mathcal{T}_0^+\}$. We write the right hand side of the above display as $\sum_{ \mathcal{Q}_{[0:2]}\subset \mathcal{T}\setminus \{\mathcal{T}_0^+\}   }\Delta_{(i,j,k); \mathcal{Q}_{[0:2]}}+\sum_{ \mathcal{Q}_{[0:2]}\subsetneq \mathcal{T}\setminus \{\mathcal{T}_0^+\}   } {\Delta}_{(i,j,k); \overline{\mathcal{Q}}_{[0:2]}} $. 
	
	\noindent (\textbf{Term ${\Delta}_{(i,j,k); \overline{\mathcal{Q}}_{[0:2]}}$}). For each such configuration $\mathcal{Q}_{[0:2]}\subsetneq \mathcal{T}\setminus \{\mathcal{T}_0^+\} $, 
	\begin{align}\label{ineq:A_hprod_quad_2}
	&\abs{{\Delta}_{(i,j,k);\overline{\mathcal{Q}}_{[0:2]}}}\\
	&=\biggabs{ \sum_{\ell_{\mathcal{T}\setminus \overline{\mathcal{Q}}_{[0:2]}} \in [n]} \prod_{s \in \mathcal{T}\setminus \overline{\mathcal{Q}}_{[0:2]} } \Delta \mathsf{H}_{t_s;\ell_s} \bigg(\sum_{\ell_{\overline{\mathcal{Q}}_{[0:2]}} \in [n]  } A_{k,[\ell_{\mathcal{T}_0^{[)}}], \ell_{\mathcal{T}_0^+}} A_{\ell_{\mathcal{T}_0^+}, [\ell_{\mathcal{T}_1}], i}A_{\ell_{\mathcal{T}_0^+}, [\ell_{\mathcal{T}_2}], j} \prod_{s \in \overline{\mathcal{Q}}_{[0:2]} } \mathsf{H}_{t_s;\ell_s}^{[-\mathcal{P}]} \bigg)}\nonumber\\
	&\leq (KL)^{r_2}\bigg\{\sum_{\ell_{\mathcal{T}\setminus \overline{\mathcal{Q}}_{[0:2]}} \in [n]} \bigg(\sum_{\ell_{\overline{\mathcal{Q}}_{[0:2]} }\in [n]  } A_{k,[\ell_{\mathcal{T}_0^{[)}}], \ell_{\mathcal{T}_0^+}} A_{\ell_{\mathcal{T}_0^+}, [\ell_{\mathcal{T}_1}], i}A_{\ell_{\mathcal{T}_0^+}, [\ell_{\mathcal{T}_2}], j} \prod_{s \in \overline{\mathcal{Q}}_{[0:2]} } \mathsf{H}_{t_s;\ell_s}^{[-\mathcal{P}]} \bigg)^2\bigg\}^{1/2},\nonumber
	\end{align}
	where the last line follows from an application of Cauchy-Schwarz inequality over $\sum_{\ell_{\mathcal{T}\setminus \overline{\mathcal{Q}}_{[0:2]}} \in [n]}$, and the estimate $\sum_{\ell_{\mathcal{T}\setminus \overline{\mathcal{Q}}_{[0:2]}}\in [n]} \prod_{s \in \mathcal{T}\setminus \overline{\mathcal{Q}}_{[0:2]}} \big(\Delta \mathsf{H}_{t_s;\ell_s}\big)^2=\prod_{s \in \mathcal{T}\setminus \overline{\mathcal{Q}}_{[0:2]} }\big(\sum_{\ell_s \in [n]} (\Delta \mathsf{H}_{t_s;\ell_s})^2\big)\leq \prod_{ s \in \mathcal{T}\setminus \overline{\mathcal{Q}}_{[0:2]}} K^2 \pnorm{z^{(t_s)}-z^{(t_s)}_{[-\mathcal{P}]}}{}^2\leq (K L)^{2r_2}$.

	Consider the same consecutive-integer representation for $\mathcal{Q}_{[0:2]}$ and $\mathcal{T}$ as in (\ref{ineq:A_P_cross_quad_cir}). Let $\mathcal{V}^{(q)}\equiv \cup_{w_q \in [N_{\mathcal{Q}_q}+1]} \mathcal{V}_{w_q}^{(q)}$ collect all words in $\mathcal{V}_\cdot^{(q)}$ and let $\mathscr{P}_2(\mathcal{V}^{(q)})$ collect all pairs $(v_\alpha,v_\beta)$ of adjacent words in $\mathcal{V}^{(q)}$. We may then write the summation term in (\ref{ineq:A_hprod_quad_2}) in the following form: 
	{\small
	\begin{align}\label{ineq:A_hprod_quad_3}
	&\bigg\{\sum_{\ell_{\mathcal{T}\setminus \overline{\mathcal{Q}}_{[0:2]}} \in [n]} \mathfrak{a}^{\ast,2}_{k,\ell_{ (\mathcal{V}^{(0)})^-}  }       \mathscr{A}^{2,[-\mathcal{P}]}_{\ell_{ (\mathcal{V}^{(0)})^+}, (\ell_{ (\mathcal{V}^{(1)})^-}, \ell_{ (\mathcal{V}^{(2)})^-} ) }(*)  \cdot \mathfrak{a}^{\ast,2}_{\ell_{ (\mathcal{V}^{(1)})^+},i } \mathfrak{a}^{\ast,2}_{\ell_{ (\mathcal{V}^{(2)})^+},j }    \prod_{\substack{q \in [0:2],\\(\alpha,\beta) \in \mathscr{P}_2(\mathcal{V}^{(q)})}} \mathfrak{a}_{\alpha,\beta}^{\ast,2}  \bigg\}^{1/2},
	\end{align}
}
	where each $\mathfrak{a}_{\alpha,\beta}^\ast \equiv \mathscr{A}_{\alpha,\beta}^{[-\mathcal{P}]}(\mathcal{I}_*)$ has non-overlapping index set $\mathcal{I}_*$. When $\mathcal{V}^{(\cdot)}=\emptyset$, we set $\ell_{ (\mathcal{V}^{(\cdot)})^-}=\ell_{ (\mathcal{V}^{(\cdot)})^+}$ and $\mathfrak{a}^{\ast}_{\#,\ell_{ (\mathcal{V}^{(\cdot)})^-}  }\equiv \delta_{ \#,\ell_{ (\mathcal{V}^{(\cdot)})^-} }$, $\mathfrak{a}^{\ast}_{\ell_{ (\mathcal{V}^{(\cdot)})^+}, \#}\equiv \delta_{\ell_{ (\mathcal{V}^{(\cdot)})^+}, \#}$.
	
	Note that (i) the total size for the products in each term of $\sum_{\ell_{\mathcal{T}\setminus \overline{\mathcal{Q}}_{[0:2]}} \in [n]}$ in (\ref{ineq:A_hprod_quad_3}) does not exceed $\abs{\mathcal{T}}-1$ because $\mathcal{Q}_{[0:2]}\subsetneq \mathcal{T}\setminus \{\mathcal{T}_0^+\} $ and therefore $\overline{\mathcal{Q}}_{[0:2]}\subsetneq \mathcal{T}$; (ii) the number of products in each term in the summation $\sum_{\ell_{\mathcal{T}\setminus \overline{\mathcal{Q}}_{[0:2]}} \in [n]}$ in the above display (\ref{ineq:A_hprod_quad_3}) is $\leq \abs{\mathcal{T}\setminus \overline{\mathcal{Q}}_{[0:2]}}+1$. Now using first Lemma \ref{lem:A_P_cross_quad} for each of these terms to replace $\mathscr{A}^{[-\mathcal{P}]}_{\ast,\ast}$ and $\mathscr{A}^{[-\mathcal{P}]}_{\ast,(\ast,\ast)}$ by $\mathscr{A}^{[-\mathcal{P}]}_{\ast,\ast;[-\mathcal{P}]}$ and $\mathscr{A}^{[-\mathcal{P}]}_{\ast,(\ast,\ast);[-\mathcal{P}]}$, and then Lemma \ref{lem:A_cross_ind_quad} for terms with subscript indices in $\mathcal{P}$ (so with further reduction in size), it follows on an event $\mathscr{E}_{(i,j,k);\overline{\mathcal{Q}}_{[0:2]}}$ with $\Prob(\mathscr{E}_{(i,j,k);\overline{\mathcal{Q}}_{[0:2]}}^c|A_{[-\mathcal{P}]})\leq (c\abs{\mathcal{P}})^{r_2}  e^{-(\log n)^{100}/c}$,
	\begin{align}\label{ineq:A_hprod_quad_4}
	(\ref{ineq:A_hprod_quad_3})\leq (\abs{\mathcal{P}} K \log n)^{c r_2}\cdot \mathfrak{U}_{\{\mathcal{T}_\cdot\};[-\mathcal{P}]}^{[-\mathcal{P}]}\big(\abs{\mathcal{T}}-1\big)\cdot\bigg(\sum_{\ell_{\mathcal{T}\setminus \overline{\mathcal{Q}}_{[0:2]}} \in [n]} n^{-\mathcal{G}(\ell_{\mathcal{T}\setminus \overline{\mathcal{Q}}_{[0:2]}})}\bigg)^{1/2},
	\end{align}
	where for $\mathcal{V}^{(\cdot)}$ all being non-empty,
	\begin{align}\label{ineq:A_hprod_quad_G_count}
	\mathcal{G}(\ell_{\mathcal{T}\setminus\overline{\mathcal{Q}}_{[0:2]}})&\equiv \abs{\{k,\ell_{ (\mathcal{V}^{(0)})^-}  \}}-1 +\abs{\{\ell_{ (\mathcal{V}^{(1)})^+},i \}}-1+\abs{\{\ell_{ (\mathcal{V}^{(2)})^+},j \}}-1\nonumber\\
	&\quad +  \sum_{\substack{q \in [0:2], (\alpha,\beta) \in \mathscr{P}_2(\mathcal{V}^{(q)})}} \big(\abs{\{ \ell_{\alpha}, \ell_{\beta} \}}-1\big)+ \abs{\{\ell_{ (\mathcal{V}^{(0)})^+}, \ell_{ (\mathcal{V}^{(1)})^-}, \ell_{ (\mathcal{V}^{(2)})^-} \}}-1\nonumber\\
	&\geq \bigabs{\big\{k,\ell_{\mathcal{V}^{(0)}},\ell_{ \mathcal{V}^{(1)}},\ell_{\mathcal{V}^{(2)}},i,j\big\}}-1.
	\end{align}
	Here the last inequality follows by repeatedly applying Lemma \ref{lem:set_merge}. When some of $\mathcal{V}^{(\cdot)}=\emptyset$, an easy modification leads to the same lower bound for $\mathcal{G}(\ell_{\mathcal{T}\setminus\overline{\mathcal{Q}}_{[0:2]}})$ as above. So with $V\equiv \sum_{q\in[0:2]}\sum_{w_q \in [N_{\mathcal{Q}_q}+1]} \abs{\mathcal{V}_{w_q}^{(q)} }$ (assumed to be $\geq 1$ without loss of generality) and using Lemma \ref{lem:count_total_num}, we have 
	\begin{align}\label{ineq:A_hprod_quad_4_0}
	\sum_{\ell_{\mathcal{T}\setminus \overline{\mathcal{Q}}_{[0:2]}} \in [n]} n^{-\mathcal{G}(\ell_{\mathcal{T}\setminus \overline{\mathcal{Q}}_{[0:2]}})} \leq \sum_{\ell_{[V]}\in [n]} n^{- \abs{ \{i,j,k, \ell_{[V]}\}  }+1 } \leq  (6V)^{V+1} n^{-\abs{\{i,j,k\} }+1}.
	\end{align}
	So using the trivial bound $(6V)^{V+1}\leq (6r_2)^{r_2+1}$, (\ref{ineq:A_hprod_quad_4}) yields that on $\mathscr{E}_{(i,j,k);\overline{\mathcal{Q}}_{[0:2]}}$,
	\begin{align*}
	n^{(\abs{\{i,j,k\}}-1)/2}\cdot (\ref{ineq:A_hprod_quad_3})\leq \big(\abs{\mathcal{P}} K r_2 \log n\big)^{c r_2}\cdot  \mathfrak{U}_{\{\mathcal{T}_\cdot\};[-\mathcal{P}]}^{[-\mathcal{P}]}\big(\abs{\mathcal{T}}-1\big).
	\end{align*}
	Now in view of (\ref{ineq:A_hprod_quad_2}), on the event $\overline{E}_0\cap \mathscr{E}_{\mathcal{T}}(L)\cap \mathscr{E}_{\mathsf{H}}(L)$, where $\overline{E}_0\equiv \cap_{i,j,k\in[n]}\cap_{ \mathcal{Q}_{[0:2]}\subsetneq \mathcal{T}\setminus \{\mathcal{T}_0^+\}  }\mathscr{E}_{(i,j,k);\overline{\mathcal{Q}}_{[0:2]}}$ with $\Prob(\overline{E}_0^c|A_{[-\mathcal{P}]})\leq  (c\abs{\mathcal{P}})^{r_2}  e^{-(\log n)^{100}/c}$, we have uniformly in $i,j,k \in [n]$,
	{\small
	\begin{align}\label{ineq:A_hprod_quad_delta_bar}
	&n^{(\abs{\{i,j,k\}}-1)/2}\biggabs{\sum_{ \mathcal{Q}_{[0:2] }\subsetneq \mathcal{T}\setminus \{\mathcal{T}_0^+\}  } \Delta_{(i,j,k); \overline{\mathcal{Q}}_{[0:2]}} }\leq \big( \abs{\mathcal{P}} K L  r_2\log n\big)^{c r_2}\cdot \mathfrak{U}_{\{\mathcal{T}_\cdot\};[-\mathcal{P}]}^{[-\mathcal{P}]}\big(\abs{\mathcal{T}}-1\big).
	\end{align}
}
	
	\noindent (\textbf{Term ${\Delta}_{(i,j,k); \mathcal{Q}_{[0:2]} }$}). Note (\ref{ineq:A_hprod_quad_2}) holds formally by replacing $\overline{\mathcal{Q}}_{[0:2]}$ via ${\mathcal{Q}}_{[0:2]}$. We may write the summation term in this revised form of (\ref{ineq:A_hprod_quad_2}) in the form
	{\small\begin{align}\label{ineq:A_hprod_quad_3_1}
	&\bigg\{\sum_{\ell_{\mathcal{T}\setminus {\mathcal{Q}}_{[0:2]}} \in [n]} \mathfrak{a}^{\ast,2}_{k,\ell_{ (\mathcal{V}^{(0)})^-}  }         \bigg(\mathfrak{a}^{\ast,2}_{\ell_{ (\mathcal{V}^{(0)})^+},\ell_{\mathcal{T}_0^+}  } \mathfrak{a}^{\ast,2}_{\ell_{\mathcal{T}_0^+},  \ell_{ (\mathcal{V}^{(1)})^-}}  \mathfrak{a}^{\ast,2}_{\ell_{\mathcal{T}_0^+},  \ell_{ (\mathcal{V}^{(2)})^-}}\bigg)   \mathfrak{a}^{\ast,2}_{\ell_{ (\mathcal{V}^{(1)})^+},i }  \mathfrak{a}^{\ast,2}_{\ell_{ (\mathcal{V}^{(2)})^+},j } \prod_{\substack{q \in [0:2],\\(\alpha,\beta) \in \mathscr{P}_2(\mathcal{V}^{(q)})}} \mathfrak{a}_{\alpha,\beta}^{\ast,2}  \bigg\}^{1/2}.
	\end{align}}	
	When some of $\mathcal{V}^{(\cdot)}=\emptyset$, we follow the same notational convention as explained after (\ref{ineq:A_hprod_quad_3}). Using Lemma \ref{lem:A_P_cross_quad} to replace $\mathscr{A}_{\ast,\ast}^{[-\mathcal{P}]}$ by $\mathscr{A}_{\ast,\ast; [-\mathcal{P}]}^{[-\mathcal{P}]}$ (with total size $\leq \abs{\mathcal{T}}-1$ due to $\mathcal{Q}_{[0,2]}\subset \mathcal{T}\setminus \{\mathcal{T}_0^+\}$), followed by Lemma \ref{lem:A_cross_ind_quad} for those terms with subscript indices in $\mathcal{P}$ for a further reduction in size, on an event $\mathscr{E}_{(i,j,k);{\mathcal{Q}}_{[0:2]}}$ with $\Prob(\mathscr{E}_{(i,j,k);{\mathcal{Q}}_{[0:2]}}^c|A_{[-\mathcal{P}]})\leq (c\abs{\mathcal{P}})^{r_2} e^{-(\log n)^{100}/c}$, 
	\begin{align}\label{ineq:A_hprod_quad_4_1}
	(\ref{ineq:A_hprod_quad_3_1})\leq (\abs{\mathcal{P}} K \log n)^{c r_2}\cdot \mathfrak{U}_{\{\mathcal{T}_\cdot\};[-\mathcal{P}]}^{[-\mathcal{P}]}\big(\abs{\mathcal{T}}-1\big)\cdot \bigg(\sum_{\ell_{\mathcal{T}\setminus {\mathcal{Q}}_{[0:2]}} \in [n]} n^{-\mathcal{G}(\ell_{\mathcal{T}\setminus {\mathcal{Q}}_{[0:2]}})}\bigg)^{1/2},
	\end{align}
	where for $\mathcal{V}^{(\cdot)}$ all being non-empty,
	\begin{align}\label{ineq:A_hprod_quad_G_count_0}
	\mathcal{G}(\ell_{\mathcal{T}\setminus{\mathcal{Q}}_{[0:2]}})&\equiv \abs{\{k,\ell_{ (\mathcal{V}^{(0)})^-}  \}}-1 +\abs{\{\ell_{ (\mathcal{V}^{(1)})^+},i \}}-1+\abs{\{\ell_{ (\mathcal{V}^{(2)})^+},j \}}-1\nonumber\\
	&\qquad +\sum_{\substack{q \in [0:2], (\alpha,\beta) \in \mathscr{P}_2(\mathcal{V}^{(q)})}} \big(\abs{\{ \ell_{\alpha}, \ell_{\beta} \}}-1\big)
	+ \abs{\{\ell_{ (\mathcal{V}^{(0)})^+}, \ell_{\mathcal{T}_0^+}\}}-1\nonumber\\
	&\qquad + \abs{\{\ell_{ (\mathcal{V}^{(1)})^-}, \ell_{\mathcal{T}_0^+} \}}-1+ \abs{\{\ell_{ (\mathcal{V}^{(2)})^-}, \ell_{\mathcal{T}_0^+} \}}-1\nonumber\\
	&\geq \bigabs{\big\{k,\ell_{ \mathcal{V}^{(0)}},\ell_{\mathcal{T}_0^+},\ell_{ \mathcal{V}^{(1)}},\ell_{\mathcal{V}^{(2)}},i,j\big\}}-1.
	\end{align}
	Here the last line follows from repeated applying Lemma \ref{lem:set_merge}. The same estimate holds when some of $\mathcal{V}^{(\cdot)}=\emptyset$. 
	 So by Lemma \ref{lem:count_total_num},
	\begin{align}\label{ineq:A_hprod_quad_4_0_0}
	&\sum_{\ell_{\mathcal{T}\setminus {\mathcal{Q}}_{[0:2]}} \in [n]} n^{-\mathcal{G}(\ell_{\mathcal{T}\setminus {\mathcal{Q}}_{[0:2]}})}
	\leq (6V)^{V+1} n^{-\abs{\{i,j,k\} }+1}.
	\end{align}
	Consequently, on the event ${E}_0\cap \mathscr{E}_{\mathcal{T}}(L)\cap \mathscr{E}_{\mathsf{H}}(L)$, where $E_0\equiv \cap_{i,j,k\in[n]}\cap_{\mathcal{Q}_{[0,2]}\subset \mathcal{T}\setminus \{\mathcal{T}_0^+\}}\mathscr{E}_{(i,j,k);\mathcal{Q}_{[0:2]}}$ with $\Prob({E}_0^c|A_{[-\mathcal{P}]})\leq  (c\abs{\mathcal{P}})^{r_2} e^{-(\log n)^{100}/c}$, we have uniformly in $i,j,k \in [n]$,
	{\small
	\begin{align}\label{ineq:A_hprod_quad_delta}
	&n^{(\abs{\{i,j,k\}}-1)/2}\biggabs{\sum_{ \mathcal{Q}_{[0,2]}\subset \mathcal{T}\setminus \{\mathcal{T}_0^+\}} \Delta_{(i,j,k); {\mathcal{Q}}_{[0:2]}} }\leq \big(\abs{\mathcal{P}} K L r_2 \log n\big)^{cr_2}\cdot \mathfrak{U}_{\{\mathcal{T}_\cdot\};[-\mathcal{P}]}^{[-\mathcal{P}]}\big(\abs{\mathcal{T}}-1\big).
	\end{align}
}
	Combining (\ref{ineq:A_hprod_quad_delta_bar}) and (\ref{ineq:A_hprod_quad_delta}), on the event ${E}_0\cap \overline{E}_0\cap \mathscr{E}_{\mathcal{T}}(L)\cap \mathscr{E}_{\mathsf{H}}(L)$,
	\begin{align*}
	n^{(\abs{\{i,j,k\}}-1)/2}\cdot\hbox{RHS of (\ref{ineq:A_hprod_quad_1})}\leq \big(\abs{\mathcal{P}} K L r_2\log n\big)^{c r_2}\cdot \mathfrak{U}_{\{\mathcal{T}_\cdot\};[-\mathcal{P}]}^{[-\mathcal{P}]}\big(\abs{\mathcal{T}}-1\big).
	\end{align*}
	Now we choose $\mathcal{P}\equiv \{i,j,k\}$ and apply Lemma \ref{lem:A_P_cross_quad} to the term $\mathscr{A}_{k,(i,j)}^{[-\mathcal{P}]}\big(\mathcal{T}_0,(\mathcal{T}_1,\mathcal{T}_2)\big)$ in (\ref{ineq:A_hprod_quad_1}), which then concludes that on an event $E_1\cap \mathscr{E}_{\mathcal{T}}(L)\cap \mathscr{E}_{\mathsf{H}}(L)$ with $\Prob(E_1^c)\leq c^{r_2} e^{-(\log n)^{100}/c}$, 
	\begin{align*}
	&n^{(\abs{\{i,j,k\}}-1)/2} \bigabs{\mathscr{A}_{k,(i,j)}\big(\mathcal{T}_0,(\mathcal{T}_{1},\mathcal{T}_2)\big) - \mathscr{A}_{k,(i,j);[-\mathcal{P}]}^{[-\mathcal{P}]}\big(\mathcal{T}_0,(\mathcal{T}_{1},\mathcal{T}_2)\big)}\\
	& \leq \big(K L r_2 \log n\big)^{c r_2}\cdot \mathfrak{U}_{\{\mathcal{T}_\cdot\};[-\mathcal{P}]}^{[-\mathcal{P}]}\big(\abs{\mathcal{T}}-1\big).
	\end{align*}
	Applying Lemma \ref{lem:A_cross_ind_quad} to the term $\mathscr{A}_{k,(i,j);[-\mathcal{P}]}^{[-\mathcal{P}]}\big(\mathcal{T}_0,(\mathcal{T}_{1},\mathcal{T}_2)\big)$ in the above display (as $(i,j,k) \notin ([n]\setminus \mathcal{P})^3$ by our choice of $\mathcal{P}$), on an event $E_2\cap \mathscr{E}_{\mathcal{T}}(L)\cap \mathscr{E}_{\mathsf{H}}(L)$ with $\Prob(E_2^c)\leq c^{r_2} e^{-(\log n)^{100}/c}$,
	\begin{align*}
	&\max_{i,j,k \in [n]} n^{(\abs{\{i,j,k\}}-1)/2} \bigabs{\mathscr{A}_{k,(i,j)}\big(\mathcal{T}_0,(\mathcal{T}_{1},\mathcal{T}_2)\big) }\nonumber\\
	&\leq \big(K L r_2\log n\big)^{c r_2} \cdot \max_{\substack{\mathcal{P}\subset [n], \abs{\mathcal{P}}\leq 3 }}\mathfrak{U}_{\{\mathcal{T}_\cdot\};[-\mathcal{P}]}^{[-\mathcal{P}]}\big(\abs{\mathcal{T}}-1\big).
	\end{align*}
	This concludes the proof of (\ref{ineq:A_hprod_quad_claim}) for the case $\mathcal{T}_0\neq \emptyset$ and $\mathcal{T}_1,\mathcal{T}_2 \neq \emptyset^{[)}$.
	 
	\noindent (\textbf{Step 1-(b)}). In this step, we consider the scenario that any of $\mathcal{T}_0\neq \emptyset, \mathcal{T}_1\neq \emptyset^{[)}, \mathcal{T}_2 \neq \emptyset^{[)}$ fails, but $(\mathcal{T}_0,\mathcal{T}_1,\mathcal{T}_2)\neq \big(\emptyset,\emptyset^{[)},\emptyset^{[)}\big)$. In this case, it suffices to provide a bound for $\mathscr{A}_{i,j}(\mathcal{T})$. Without loss of generality, we assume $\mathcal{T}=[r]$ and $\mathcal{P}= \{i,j\}$. Note that using the same arguments as in (\ref{ineq:A_hprod_quad_1})-(\ref{ineq:A_hprod_quad_2}), on the event $\mathscr{E}_{\mathcal{T}}(L)\cap \mathscr{E}_{\mathsf{H}}(L)$,
	\begin{align*}
	&\bigabs{\mathscr{A}_{i,j}(\mathcal{T})-\mathscr{A}_{i,j}^{[-\mathcal{P}]}(\mathcal{T}) }=\biggabs{ \sum_{ \mathcal{Q} \subsetneq \mathcal{T} } \bigg\{\sum_{\ell_{\mathcal{T}} \in [n]} A_{i, [\ell_{\mathcal{T}}], j}  \bigg(\prod_{s \in \mathcal{Q} } \mathsf{H}_{t_s;\ell_s}^{[-\mathcal{P}]} \prod_{s \in \mathcal{T}\setminus \mathcal{Q} } \Delta \mathsf{H}_{t_s;\ell_s}\bigg)\bigg\} } \\
	& = \biggabs{\sum_{ \mathcal{Q} \subsetneq \mathcal{T} }\bigg\{ \sum_{\ell_{\mathcal{T}\setminus \mathcal{Q}} \in [n]  } \prod_{s \in \mathcal{T}\setminus \mathcal{Q} } \Delta \mathsf{H}_{t_s;\ell_s}\bigg(  \sum_{\ell_{\mathcal{Q}} \in [n] } A_{i, [\ell_{\mathcal{T}}], j}  \prod_{s \in \mathcal{Q} } \mathsf{H}_{t_s;\ell_s}^{[-\mathcal{P}]}  \bigg)\bigg\} }\\
	&\leq (KL)^r \sum_{ \mathcal{Q} \subsetneq \mathcal{T} } \bigg\{ \sum_{\ell_{\mathcal{T}\setminus \mathcal{Q}} \in [n]  } \bigg(  \sum_{\ell_{\mathcal{Q}} \in [n] } A_{i, [\ell_{\mathcal{T}}], j}  \prod_{s \in \mathcal{Q} } \mathsf{H}_{t_s;\ell_s}^{[-\mathcal{P}]}  \bigg)^2\bigg\}^{1/2}\equiv (KL)^r \sum_{ \mathcal{Q} \subsetneq \mathcal{T} } \overline{\Delta}_{(ij);\mathcal{Q}}.
	\end{align*}
	Let 
	$\mathcal{Q}=\cup_{w \in [N_{\mathcal{Q}}]} \mathcal{U}_w$ be a consecutive-integer-set representation  of $\mathcal{Q}$, and $\mathcal{V}_1,\mathcal{U}_1,\ldots, \mathcal{V}_{N_{\mathcal{Q}}},\mathcal{U}_{N_{\mathcal{Q}}}, \mathcal{V}_{N_{\mathcal{Q}}+1}$ be that of $\mathcal{T}$. Let $\mathcal{V}\equiv \cup_{w \in [N_{\mathcal{Q}}+1]} \mathcal{V}_{w}$ collect all words in $\mathcal{V}_\cdot$ and recall $\mathscr{P}_2(\mathcal{V})$ collects all pairs of adjacent words in $\mathcal{V}$. We may then write $\overline{\Delta}_{(ij);\mathcal{Q}}$ in the form as
	\begin{align*}
	\bigg\{\sum_{\ell_{\mathcal{T}\setminus {\mathcal{Q}} } \in [n]} \mathfrak{a}^{\ast,2}_{i,\ell_{ \mathcal{V}^-}  }          \mathfrak{a}^{\ast,2}_{\ell_{ \mathcal{V}^+},j } \prod_{\substack{(\alpha,\beta) \in \mathscr{P}_2(\mathcal{V})}} \mathfrak{a}_{\alpha,\beta}^{\ast,2}  \bigg\}^{1/2}.
	\end{align*}
	From here, we may (i) apply Lemma \ref{lem:A_P_cross_quad} to replace $\mathscr{A}_{\ast,\ast}^{[-\mathcal{P}]}$ by $\mathscr{A}_{\ast,\ast; [-\mathcal{P}]}^{[-\mathcal{P}]}$ (with total size $\leq \abs{\mathcal{T}}-1$ due to $\mathcal{Q}\subsetneq \mathcal{T}$), followed by Lemma \ref{lem:A_cross_ind_quad} for those terms with subscript indices in $\mathcal{P}$, and (ii) use the same counting techniques as in Step 1-(a), on an event $E_3\cap \mathscr{E}_{\mathcal{T}}(L)\cap \mathscr{E}_{\mathsf{H}}(L)$ with $\Prob(E_3^c)\leq  c^{r} e^{-(\log n)^{100}/c}$,
	\begin{align*}
	&\max_{i,j \in [n]} n^{(\abs{\{i,j\}}-1)/2} \bigabs{\mathscr{A}_{i,j}(\mathcal{T}) }\leq \big(K L r_2\log n\big)^{cr} \cdot \max_{\substack{\mathcal{P}\subset [n], \abs{\mathcal{P}}\leq 3 }}\mathfrak{U}_{\{\mathcal{T}_\cdot\};[-\mathcal{P}]}^{[-\mathcal{P}]}\big(\abs{\mathcal{T}}-1\big).
	\end{align*}
	This concludes the proof of (\ref{ineq:A_hprod_quad_claim}) for the current case.
	
	\noindent (\textbf{Step 2}). By a union bound via (\ref{ineq:A_hprod_quad_claim}) in Step 1 over all non-overlapping sets  of consecutive integers $\{\mathcal{J}_{w_q}^{(q)}:w_q \in [N_{\mathcal{T}_q}]\}\subset \mathcal{T}_q$, $q= [0:2]$ with  $\sum_{w_q \in [N_{\mathcal{T}_q}]}\abs{\mathcal{J}_{w_q}^{(q)}}\in [-\bm{1}_{q\in \{1,2\}}: (\abs{\mathcal{T}}\vee -\bm{1}_{q=1,2})]$, we have the estimate
	\begin{align}\label{ineq:A_hprod_quad_6}
	&\Prob\Big(\mathscr{W}_{\{\mathcal{T}_{\cdot}\};-0}(\abs{\mathcal{T}})\geq  \big(K L \abs{\mathcal{T}}_\ast\log n\big)^{c \abs{\mathcal{T}}_\ast}\cdot  \mathscr{W}_{\{\mathcal{T}_{\cdot}\};-3}\big(\abs{\mathcal{T}}-1\big), \nonumber\\
	&\qquad \qquad  \mathscr{E}_{\mathcal{T}}(L)\cap \mathscr{E}_{\mathsf{H}}(L) \Big)  \leq  c^{\abs{\mathcal{T}}_\ast } e^{-(\log n)^{100}/c},
	\end{align}
	where
	{\small\begin{align*}
	\mathscr{W}_{\{\mathcal{I}_{\cdot}\};-\alpha}(s)& \equiv  \max_{\substack{s_\ell\leq \abs{\mathcal{I}_\ell},\\ \abs{s_{[0:2]} } =s}  }\max_{\substack{ \{(\mathcal{J}_{w_q}^{(q)})\}\\\in \mathscr{C}_{\empty}(\{\mathcal{I}_{[0:2]}\};s_{[0:2]})} } \prod_{\substack{w_0 \in [N_{\mathcal{I}_0}-1], \\w_{[2]} \in [2: N_{\mathcal{I}_{[2]}}] } } \max_{ \substack{\mathcal{P}\subset [n],\\ \abs{\mathcal{P}}\leq \alpha } } \bigg(\sum_{\gamma \in [2]}n^{(\gamma-1)/2} \max_{\substack{i,j \in [n]\setminus \mathcal{P},\\ \abs{\{i,j\}}=\gamma }  }\bigg) \bigabs{\mathscr{A}_{i,j;[-\mathcal{P}]}^{[-\mathcal{P}]}\big(\mathcal{J}_{w_q}^{(q)}\big)  } \\
	&\quad \times  \max_{ \substack{\mathcal{P}\subset [n],\\ \abs{\mathcal{P}}\leq \alpha } }\bigg(\sum_{\gamma \in [3]}n^{(\gamma-1)/2}\max_{ \substack{i,j,k \in [n]\setminus \mathcal{P},\\\abs{\{i,j,k\}}=\gamma} }\bigg) \bigabs{\mathscr{A}_{k,(i,j);[-\mathcal{P}]}^{[-\mathcal{P}]}\Big(\mathcal{J}_{N_{\mathcal{I}_0}}^{(0)}, \big(\mathcal{J}_1^{(1)}, \mathcal{J}_1^{(2)}\big) \Big) }. 
	\end{align*}}
	Now for $\abs{\mathcal{T}}\leq \log^2 n$, iterating the bound (\ref{ineq:A_hprod_quad_6}) for $\abs{\mathcal{T}}+1$ times and using the simple estimate $\abs{\mathscr{W}_{\{\mathcal{T}_{\cdot}\};-3(\abs{\mathcal{T}}+1)}(-1)}\leq K^{c_0}$, 
	\begin{align}\label{ineq:A_hprod_quad_7}
	&\Prob\Big(\mathscr{W}_{\{\mathcal{T}_{\cdot}\};-0}(\abs{\mathcal{T}})\geq \big(K L \abs{\mathcal{T}}_\ast \log n\big)^{c\abs{\mathcal{T}}_\ast^2},\mathscr{E}_{\mathcal{T}}(L)\cap \mathscr{E}_{\mathsf{H}}(L)\Big)\nonumber\\
	&\leq  c^{\abs{\mathcal{T}}_\ast}\cdot n^{c \abs{\mathcal{T}}_\ast}e^{-(\log n)^{100}/c} \leq c^{\abs{\mathcal{T}}_\ast}e^{-(\log n)^{100}/c}. 
	\end{align}
	The claim follows. \qed

\subsection{Proof of Proposition \ref{prop:A_cross_cubic}}

We only prove the case for $q_0=2$; the general case follows from minor notational modifications.

To this end, we first prove the following lemma that serves as a debiasing estimate for the more complicated summation term $\sum_{\ell_{0} \in [n]}\mathscr{A}_{k,(\ell_{0},\ell_{0})}(\mathcal{T}_{[0:2]}) A_{\ell_{0},j}^3 \mathsf{H}_{t_{0};\ell_{0}}$ in Proposition \ref{prop:A_cross_cubic}. In particular, this estimate allows us to replace $A_{\cdot\cdot}^3$ by its expectation (which could be in general non-zero).

\begin{lemma}\label{lem:A_cross_sum}
Suppose $A=A_0/\sqrt{n}$, where $A_0$ is symmetric and the entries of its upper triangle are independent, centered random variables. 
Fix $t_0 \in \N$ and three non-overlapping sets $\mathcal{T}_0,\mathcal{T}_1,\mathcal{T}_2\subset \N$ of consecutive integers, where $\mathcal{T}_1,\mathcal{T}_2$ are allowed to take $\emptyset^{[)}$. Suppose $\max_{i,j\in [n]}\abs{A_{0,ij}}\leq K$ for some $K>1$ and $\mathcal{T}\equiv \cup_{\ell \in [0:2]} \mathcal{T}_\ell \cup \{0\}$.
Then there exists a universal constant $c_0>0$ such that for $\abs{\mathcal{T}}_\ast\leq \log^2 n$,
\begin{align*}
&\Prob\bigg(\max_{k \in [n]}\biggabs{ \sum_{\ell_{0} \in [n]}\mathscr{A}_{k,(\ell_{0},\ell_{0})}(\mathcal{T}_{[0:2]}) \mathsf{H}_{t_{0};\ell_{0}} }\\
&\qquad \qquad \geq \big(K L \abs{\mathcal{T}}_\ast\log n\big)^{c_0\abs{\mathcal{T}}_\ast^2},\mathscr{E}_{\mathcal{T}}(L)\cap \mathscr{E}_{\mathsf{H}}(L)\bigg)\leq  c_0^{\abs{\mathcal{T}}_\ast}e^{-(\log n)^{100}/c_0}. 
\end{align*}
\end{lemma}

\begin{proof}[Proof of Lemma \ref{lem:A_cross_sum}: $\mathcal{T}_0\neq \emptyset$ and $\mathcal{T}_1,\mathcal{T}_2 \neq \emptyset^{[)}$]
Without loss of generality, we assume that $\mathcal{T}_0\equiv [r_0]$, $\mathcal{T}_1\equiv [r_0+1:r_1]$ and $\mathcal{T}_2\equiv [r_1+1:r_2]$. We also write $r\equiv r_2$, $\ell_0\equiv \ell_{r+1}$ and reset $t=0$ as $t_{r+1}$ for notational simplicity. Then we may identify the set $\mathcal{T}$ as $\mathcal{T}=[r+1]$. 

Fix $\mathcal{P}\equiv \{k\}$. By writing 
\begin{align}\label{ineq:A_cross_quad_sum_0}
A_{k,[\ell_{\mathcal{T}_0^{[)}}], \ell_{\mathcal{T}_0^+}, ([\ell_{\mathcal{T}_1}],[\ell_{\mathcal{T}_2}]), \ell_{r+1}}\equiv A_{k,[\ell_{\mathcal{T}_0^{[)}}], \ell_{\mathcal{T}_0^+}} A_{\ell_{\mathcal{T}_0^+}, [\ell_{\mathcal{T}_1}], \ell_{r+1}}A_{\ell_{\mathcal{T}_0^+}, [\ell_{\mathcal{T}_2}], \ell_{r+1}},
\end{align}
we have
\begin{align*}
&\sum_{\ell_{r+1} \in [n]}\mathscr{A}_{k,(\ell_{r+1},\ell_{r+1})}(\mathcal{T}_{[0:2]}) \mathsf{H}_{t_{r+1};\ell_{r+1}} \\
&= \sum_{\ell_{[r+1]} \in [n]} A_{k,[\ell_{\mathcal{T}_0^{[)}}], \ell_{\mathcal{T}_0^+}, ([\ell_{\mathcal{T}_1}],[\ell_{\mathcal{T}_2}]), \ell_{r+1}}     \prod_{s \in [r+1]} \big(\mathsf{H}_{t_s;\ell_s}^{[-\mathcal{P}]}+\Delta \mathsf{H}_{t_s;\ell_s}\big)\\
& = \sum_{\mathcal{Q},\mathcal{Q}'\subset [r+1]}\sum_{\substack{\ell_{[r+1]\setminus \mathcal{Q}} \in \mathcal{P},\\ \ell_{\mathcal{Q}}\in [n]\setminus \mathcal{P}} }  A_{k,[\ell_{\mathcal{T}_0^{[)}}], \ell_{\mathcal{T}_0^+}, ([\ell_{\mathcal{T}_1}],[\ell_{\mathcal{T}_2}]), \ell_{r+1}}  \prod_{s \in \mathcal{Q}'} \mathsf{H}_{t_s;\ell_s}^{[-\mathcal{P}]} \prod_{s \in [r+1]\setminus \mathcal{Q}'} \Delta \mathsf{H}_{t_s;\ell_s}.
\end{align*}
Therefore, on the event $\mathscr{E}_{\mathcal{T}}(L)\cap \mathscr{E}_{\mathsf{H}}(L)$,
\begin{align}\label{ineq:A_cross_quad_sum_1}
&\biggabs{\sum_{\ell_{r+1} \in [n]}\mathscr{A}_{k,(\ell_{r+1},\ell_{r+1})}(\mathcal{T}_{[0:2]}) \mathsf{H}_{t_{r+1};\ell_{r+1}} }\\
&\leq (K\abs{\mathcal{P}})^{c r}\max_{\substack{\mathcal{Q}'\subset \mathcal{Q}\subset [r+1], \\\ell_{[r+1]\setminus \mathcal{Q}} \in \mathcal{P} } } \biggabs{\sum_{\ell_{\mathcal{Q}}\in [n]\setminus \mathcal{P}} A_{k,[\ell_{\mathcal{T}_0^{[)}}], \ell_{\mathcal{T}_0^+}, ([\ell_{\mathcal{T}_1}],[\ell_{\mathcal{T}_2}]), \ell_{r+1}} \prod_{s \in \mathcal{Q}'} \mathsf{H}_{t_s;\ell_s}^{[-\mathcal{P}]} \prod_{s \in \mathcal{Q}\setminus \mathcal{Q}'} \Delta \mathsf{H}_{t_s;\ell_s}}\nonumber\\
&\leq  (KL\abs{\mathcal{P}})^{c r}\max_{\substack{\mathcal{Q}'\subset \mathcal{Q}\subset [r+1], \\\ell_{[r+1]\setminus \mathcal{Q}} \in \mathcal{P} } } \bigg\{\sum_{\ell_{\mathcal{Q}\setminus \mathcal{Q}'} \in [n]\setminus \mathcal{P}} \bigg(\sum_{\ell_{\mathcal{Q}'}\in [n]\setminus \mathcal{P}} A_{k,[\ell_{\mathcal{T}_0^{[)}}], \ell_{\mathcal{T}_0^+}, ([\ell_{\mathcal{T}_1}],[\ell_{\mathcal{T}_2}]), \ell_{r+1}} \prod_{s \in \mathcal{Q}'} \mathsf{H}_{t_s;\ell_s}^{[-\mathcal{P}]}\bigg)^2\bigg\}^{1/2}.\nonumber
\end{align}
For the case $\mathcal{T}_0^+ \in \mathcal{Q}'$ (resp. $\mathcal{T}_0^+ \notin \mathcal{Q}'$), we use a similar representation as in (\ref{ineq:A_hprod_quad_3}) (resp. (\ref{ineq:A_hprod_quad_3_1})) and the counting technique in (\ref{ineq:A_hprod_quad_G_count})-(\ref{ineq:A_hprod_quad_4_0}) (resp. (\ref{ineq:A_hprod_quad_G_count_0})-(\ref{ineq:A_hprod_quad_4_0_0})). In particular:

\noindent (\textbf{Case 1}: $r+1\notin \mathcal{Q}'$). We (i) replace $(i,j)$ in (\ref{ineq:A_hprod_quad_3}) and (\ref{ineq:A_hprod_quad_G_count}) (resp. (\ref{ineq:A_hprod_quad_3_1}) and (\ref{ineq:A_hprod_quad_G_count_0})) by $(\ell_{r+1},\ell_{r+1})$, and (ii) replace the summation $\sum_{\ell_{\mathcal{T}\setminus \overline{\mathcal{Q}}_{[0:2]}} \in [n]}$ in (\ref{ineq:A_hprod_quad_3}) (resp. (\ref{ineq:A_hprod_quad_3_1})) by $\sum_{\ell_{\mathcal{Q}\setminus \mathcal{Q}'} \in [n]\setminus \mathcal{P}}$. Here $\mathfrak{a}_{\alpha,\beta}^\ast = \mathscr{A}_{\alpha,\beta;[-\mathcal{P}]}^{[-\mathcal{P}]}(\mathcal{I}_\ast)$. In this case, upon using Lemma \ref{lem:A_cross_ind_quad} for terms with subscript indices in $\mathcal{P}=\{k\}$, for some $V \subset [r]$, with probability at least $1-c^{r} e^{-(\log n)^{100}/c}$,
\begin{align*}
(\ref{ineq:A_cross_quad_sum_1}) &\leq (KL\log n)^{cr}\cdot \mathfrak{U}_{\{\mathcal{T}_\cdot\};[-\mathcal{P}]}^{[-\mathcal{P}]}\big(\abs{\mathcal{T}}\big)\cdot  \bigg\{ \sum_{\ell_{r+1},\ell_{[V]}\in [n]\setminus \mathcal{P}} n^{- \abs{ \{k, \ell_{r+1}, \ell_{[V]}\}  }+1 }\bigg\}^{1/2}\\
& \leq (KLr\log n)^{cr} \cdot \mathfrak{U}_{\{\mathcal{T}_\cdot\};[-\mathcal{P}]}^{[-\mathcal{P}]}\big(\abs{\mathcal{T}}\big).
\end{align*}
\noindent (\textbf{Case 2}: $r+1\in \mathcal{Q}'$). We (i) replace the term $\mathfrak{a}^{\ast,2}_{\ell_{ (\mathcal{V}^{(1)})^+},i }  \mathfrak{a}^{\ast,2}_{\ell_{ (\mathcal{V}^{(2)})^+},j}$ and the summation $\sum_{\ell_{\mathcal{T}\setminus \overline{\mathcal{Q}}_{[0:2]}} \in [n]}$ in (\ref{ineq:A_hprod_quad_3}) (resp. (\ref{ineq:A_hprod_quad_3_1})) by $\mathfrak{a}^{\ast,2}_{\ell_{ (\mathcal{V}^{(1)})^+},  \ell_{ (\mathcal{V}^{(2)})^+} }$ and $\sum_{\ell_{\mathcal{Q}\setminus \mathcal{Q}'} \in [n]\setminus \mathcal{P}}$, and (ii) replace $(i,j)$ in (\ref{ineq:A_hprod_quad_G_count}) (resp. (\ref{ineq:A_hprod_quad_G_count_0})) by $\big(\ell_{ (\mathcal{V}^{(1)})^+},  \ell_{ (\mathcal{V}^{(2)})^+} \big)$. In this case, again upon using Lemma \ref{lem:A_cross_ind_quad} for terms with subscript indices in $\mathcal{P}=\{k\}$, for some $V \subset [r]$, with probability at least $1-c^{r} e^{-(\log n)^{100}/c}$,
\begin{align*}
(\ref{ineq:A_cross_quad_sum_1}) &\leq (KL\log n)^{cr}\cdot \mathfrak{U}_{\{\mathcal{T}_\cdot\};[-\mathcal{P}]}^{[-\mathcal{P}]}\big(\abs{\mathcal{T}}\big)\cdot  \bigg\{ \sum_{\ell_{[V]}\in [n]\setminus \mathcal{P}} n^{- \abs{ \{k, \ell_{[V]}\}  }+1 }\bigg\}^{1/2}\\
&\leq (KLr\log n)^{cr} \cdot \mathfrak{U}_{\{\mathcal{T}_\cdot\};[-\mathcal{P}]}^{[-\mathcal{P}]}\big(\abs{\mathcal{T}}\big).
\end{align*}
The claim follows by using (\ref{ineq:A_hprod_quad_7}) (= a stronger version of Proposition \ref{prop:A_hprod_quad}) to control $\mathfrak{U}_{\{\mathcal{T}_\cdot\};[-\mathcal{P}]}^{[-\mathcal{P}]}(\abs{\mathcal{T}})$.
\end{proof}
\begin{proof}[Proof of Lemma \ref{lem:A_cross_sum}: Any of $\mathcal{T}_0\neq \emptyset, \mathcal{T}_1\neq \emptyset^{[)}, \mathcal{T}_2 \neq \emptyset^{[)}$ fails, but $(\mathcal{T}_0,\mathcal{T}_1,\mathcal{T}_2)\neq \big(\emptyset,\emptyset^{[)},\emptyset^{[)}\big)$]
\noindent (\textbf{Case 1}). For $\mathcal{T}_0= \emptyset$, as  $\mathscr{A}_{k,(\ell_0,\ell_0)}\big(\mathcal{T}_0,(\mathcal{T}_{1},\mathcal{T}_2)\big)= \mathsf{H}_{\ast;\ast}\cdot \mathscr{A}_{k,\ell_0}(\mathcal{T}_1)\mathscr{A}_{k,\ell_0}(\mathcal{T}_2)$, we have $\sum_{\ell_{0} \in [n]}\mathscr{A}_{k,(\ell_{0},\ell_{0})}(\mathcal{T}_{[0:2]}) \mathsf{H}_{t_{0};\ell_{0}} = \mathscr{A}_{k,k}(\mathcal{T})$ for which the desired estimate follows directly from Proposition \ref{prop:A_hprod_quad}.

\noindent (\textbf{Case 2}). For $\mathcal{T}_1 = \emptyset^{[)}$, we have $\mathscr{A}_{k,(\ell_0,\ell_0)}\big(\mathcal{T}_0,(\mathcal{T}_{1},\mathcal{T}_2)\big)= \mathsf{H}_{\ast;\ast}\cdot \mathscr{A}_{k,\ell_0}(\mathcal{T}_0^{[)})\mathscr{A}_{\ell_0,\ell_0}(\mathcal{T}_2)$. If $\mathcal{T}_0=\emptyset$, then $\abs{\sum_{\ell_{0} \in [n]}\mathscr{A}_{k,(\ell_{0},\ell_{0})}(\mathcal{T}_{[0:2]}) \mathsf{H}_{t_{0};\ell_{0}}}\leq K^c \abs{\mathscr{A}_{k,k}(\mathcal{T}_2)}$ for which Proposition \ref{prop:A_hprod_quad} applies to conclude. Now we consider the case $\mathcal{T}_0\neq \emptyset$. Using the same notation as in the above proof,  (\ref{ineq:A_cross_quad_sum_0}) becomes $A_{k,[\ell_{\mathcal{T}_0^{[)}}], \ell_{\mathcal{T}_0^+}, ([\ell_{\mathcal{T}_1}],[\ell_{\mathcal{T}_2}]), \ell_{r+1}}\equiv A_{k,[\ell_{\mathcal{T}_0^{[)}}], \ell_{r+1}} A_{\ell_{r+1}, [\ell_{\mathcal{T}_2}], \ell_{r+1}}$. The arguments in (\ref{ineq:A_cross_quad_sum_1}) then lead to, recall $\mathcal{P}=\{k\}$,
\begin{align*}
&\biggabs{\sum_{\ell_{r+1} \in [n]}\mathscr{A}_{k,(\ell_{r+1},\ell_{r+1})}(\mathcal{T}_{[0:2]}) \mathsf{H}_{t_{r+1};\ell_{r+1}} }\nonumber\\
&\leq (KL)^{cr}\cdot \bigg\{ \sum_{\ell_{[r+1]\setminus \mathcal{Q}'} \in [n]\setminus \mathcal{P}} \bigg(\sum_{\ell_{\mathcal{Q}'}\in [n]\setminus \mathcal{P}} A_{k,[\ell_{\mathcal{T}_0^{[)}}], \ell_{r+1}} A_{\ell_{r+1}, [\ell_{\mathcal{T}_2}], \ell_{r+1}}\prod_{s \in \mathcal{Q}'} \mathsf{H}_{t_s;\ell_s}^{[-\mathcal{P}]}\bigg)^2\bigg\}^{1/2}.
\end{align*}
From here we may use the same (in fact, simpler) arguments below (\ref{ineq:A_cross_quad_sum_1}) to conclude. 
\end{proof}

Now we are in a position to prove Proposition \ref{prop:A_cross_cubic}.

\begin{proof}[Proof of Proposition \ref{prop:A_cross_cubic}: $\mathcal{T}_0\neq \emptyset$ and $\mathcal{T}_1,\mathcal{T}_2 \neq \emptyset^{[)}$]
	Without loss of generality, we assume that $\mathcal{T}_0\equiv [r_0]$, $\mathcal{T}_1\equiv [r_0+1:r_1]$ and $\mathcal{T}_2\equiv [r_1+1:r_2]$. We also write $r\equiv r_2$, $\ell_0\equiv \ell_{r+1}$ and reset $t=0$ as $t_{r+1}$ for notational simplicity. Let $\mathcal{T}\equiv [r+1]$. Then we have the following decomposition
	\begin{align}\label{ineq:A_cross_quad_cubic_0}
	&\biggabs{\sum_{\ell_{r+1} \in [n]}\mathscr{A}_{k,(\ell_{r+1},\ell_{r+1})}(\mathcal{T}_{[0:2]}) A_{\ell_{r+1},j}^3 \mathsf{H}_{t_{r+1};\ell_{r+1}}}\leq \biggabs{\sum_{\ell_{r+1} \in [n]}\mathscr{A}_{k,(\ell_{r+1},\ell_{r+1})}(\mathcal{T}_{[0:2]})\E A_{\ell_{r+1},j}^3 \mathsf{H}_{t_{r+1};\ell_{r+1}}}\nonumber\\
	&\qquad +  \biggabs{ \sum_{\ell_{r+1} \in [n]}\mathscr{A}_{k,(\ell_{r+1},\ell_{r+1})}(\mathcal{T}_{[0:2]}) (A_{\ell_{r+1},j}^3-\E A_{\ell_{r+1},j}^3) \mathsf{H}_{t_{r+1};\ell_{r+1}} }.
	\end{align}
	By Lemma \ref{lem:A_cross_sum}, on an event $\mathscr{E}_0\cap \mathscr{E}_{\mathcal{T}}(L)\cap \mathscr{E}_{\mathsf{H}}(L)$ with $\Prob(\mathscr{E}_0^c)\leq c^r e^{-(\log n)^{100}/c}$, the first term on the right hand side is bounded by $\big(K L r\log n\big)^{c r^2} n^{-3/2}$, so we only need to show that the second term in (\ref{ineq:A_cross_quad_cubic_0}) can also be controlled at the same order.

	To do so, we fix $\mathcal{P}\equiv \{j\}$, so $k \in [n]\setminus \mathcal{P}$. We also write $A_{\ell_{r+1},j}^3-\E A_{\ell_{r+1},j}^3\equiv \overline{A_{\ell_{r+1},j}^3}$ for notational simplicity. Using the same reasoning as in (\ref{ineq:A_cross_quad_sum_1}), we only need to consider, for fixed $\mathcal{Q}'\subset \mathcal{Q}\subset [r+1], \ell_{[r+1]\setminus \mathcal{Q}} \in \mathcal{P}$, the following term 
	\begin{align}\label{ineq:A_cross_quad_cubic_1}
	\bigg\{\sum_{\ell_{\mathcal{Q}\setminus \mathcal{Q}'} \in [n]\setminus \mathcal{P}} \bigg(\sum_{\ell_{\mathcal{Q}'}\in [n]\setminus \mathcal{P}} A_{k,[\ell_{\mathcal{T}_0^{[)}}], \ell_{\mathcal{T}_0^+}, ([\ell_{\mathcal{T}_1}],[\ell_{\mathcal{T}_2}]), \ell_{r+1}}\overline{A_{\ell_{r+1},j}^3}\prod_{s \in \mathcal{Q}'} \mathsf{H}_{t_s;\ell_s}^{[-\mathcal{P}]}\bigg)^2\bigg\}^{1/2}.
	\end{align}
	\noindent (\textbf{Case 1: $\mathcal{T}_0^+ \in \mathcal{Q}'$}). Let (\ref{ineq:A_P_cross_quad_cir}) be a consecutive-integer-set representation of $\mathcal{Q}'$, where $\mathcal{Q}'$ contains all $\mathcal{V}_\cdot^{(\cdot)}$ along with $\mathcal{T}_0^+$. 
	
	\noindent (\textbf{Subcase 1-(a): $r+1 \in \mathcal{Q}'$}). (\ref{ineq:A_cross_quad_cubic_1}) takes the following form
	\begin{align}\label{ineq:A_cross_quad_cubic_2}
	&\bigg\{\sum_{\ell_{\mathcal{Q} \setminus \mathcal{Q}'} \in [n]\setminus \mathcal{P}}  \mathfrak{a}^{\ast,2}_{k,\ell_{ (\mathcal{V}^{(0)})^-}  }       \cdot   \mathscr{A}^2_{\ell_{ (\mathcal{V}^{(0)})^+}, (\ell_{ (\mathcal{V}^{(1)})^-}, \ell_{ (\mathcal{V}^{(2)})^-} ) }(*) \nonumber\\
	&\qquad \times \bigg( \sum_{\ell_{r+1}\in [n]\setminus \mathcal{P}}\mathfrak{a}^{\ast}_{\ell_{ (\mathcal{V}^{(1)})^+},\ell_{r+1} }\cdot  \mathfrak{a}^{\ast}_{\ell_{ (\mathcal{V}^{(2)})^+},\ell_{r+1} } \overline{A_{\ell_{r+1},j}^3}\bigg)^2 \prod_{\substack{q \in [0:2],\\(\alpha,\beta) \in \mathscr{P}_2(\mathcal{V}^{(q)})}} \mathfrak{a}_{\alpha,\beta}^{\ast,2}\bigg\}^{1/2}, 
	\end{align}
	where $\mathfrak{a}_{\alpha,\beta}^\ast = \mathscr{A}_{\alpha,\beta;[-\mathcal{P}]}^{[-\mathcal{P}]}(\mathcal{I}_\ast)$  has non-overlapping index sets $\mathcal{I}_\ast$. We also follow the same notational convention as explained after (\ref{ineq:A_hprod_quad_3}), when some of $\mathcal{V}^{(\cdot)}=\emptyset$.   
	
	By subgaussian concentration conditionally on $A_{[-\mathcal{P}]}=A_{[-j]}$, with probability at least $1-c e^{-(\log n)^{100}/c}$, we have
	\begin{align*}
	&\biggabs{\sum_{\ell_{r+1}\in [n]\setminus \mathcal{P}}\mathfrak{a}^{\ast}_{\ell_{ (\mathcal{V}^{(1)})^+},\ell_{r+1} }\cdot  \mathfrak{a}^{\ast}_{\ell_{ (\mathcal{V}^{(2)})^+},\ell_{r+1} } \overline{A_{\ell_{r+1},j}^3} }\\
	&\leq n^{-3/2} (K\log n)^c \cdot \bigg\{\sum_{\ell_{r+1}\in [n]\setminus \mathcal{P}}\mathfrak{a}^{\ast,2}_{\ell_{ (\mathcal{V}^{(1)})^+},\ell_{r+1} }\cdot  \mathfrak{a}^{\ast,2}_{\ell_{ (\mathcal{V}^{(2)})^+},\ell_{r+1} }\bigg\}^{1/2}.
	\end{align*}
	With the same probability, (\ref{ineq:A_cross_quad_cubic_2}) can be bounded by a $ n^{-3/2} (K\log n)^c$ multiplicative factor of
	{\small
	\begin{align}\label{ineq:A_cross_quad_cubic_3}
	&\bigg\{\sum_{ \substack{\ell_{\mathcal{Q}\setminus \mathcal{Q}'},\\ \ell_{r+1} \in [n]\setminus \mathcal{P} }  }  \mathfrak{a}^{\ast,2}_{k,\ell_{ (\mathcal{V}^{(0)})^-}  }      \mathscr{A}^{2,[-\mathcal{P}]}_{\ell_{ (\mathcal{V}^{(0)})^+}, (\ell_{ (\mathcal{V}^{(1)})^-}, \ell_{ (\mathcal{V}^{(2)})^-} ); [-\mathcal{P}] }(*)  \mathfrak{a}^{\ast,2}_{\ell_{ (\mathcal{V}^{(1)})^+},\ell_{r+1} } \mathfrak{a}^{\ast,2}_{\ell_{ (\mathcal{V}^{(2)})^+},\ell_{r+1} } \prod_{\substack{q \in [0:2],\\(\alpha,\beta) \in \mathscr{P}_2(\mathcal{V}^{(q)})}} \mathfrak{a}_{\alpha,\beta}^{\ast,2}\bigg\}^{1/2}.
	\end{align}
}
	Using the counting technique similar to (\ref{ineq:A_hprod_quad_G_count}) (where we identify $(i,j)=(\ell_{r+1},\ell_{r+1})$ with a further summation over $\ell_{r+1}$), upon using Lemma \ref{lem:A_cross_ind_quad} for terms with subscript indices in $\mathcal{P}=\{j\}$, on an event $\mathscr{E}_{1,a}\cap \mathscr{E}_{\mathcal{T}}(L)\cap \mathscr{E}_{\mathsf{H}}(L)$, where $\mathscr{E}_{1,a}$ with $\Prob(\mathscr{E}_{1,a}^c)\leq  c^r e^{-(\log n)^{100}/c}$, for some $V\subset [r+1]$,
	\begin{align*}
	(\ref{ineq:A_cross_quad_cubic_3})&\leq \big(K L \log n\big)^{cr}\cdot \mathfrak{U}_{\{\mathcal{T}_\cdot\};[-\mathcal{P}]}^{[-\mathcal{P}]}\big(\abs{\mathcal{T}}\big)\cdot \bigg(\sum_{\ell_{[V]} \in [n]\setminus \mathcal{P} }n^{-\{k,\ell_{[V]}\}+1 }\bigg)\\
	& \leq \big(K L r\log n\big)^{c r}\cdot \mathfrak{U}_{\{\mathcal{T}_\cdot\};[-\mathcal{P}]}^{[-\mathcal{P}]}\big(\abs{\mathcal{T}}\big).
	\end{align*}
	Consequently, on the same event, we have 
	\begin{align}\label{ineq:A_cross_quad_cubic_main_estimate}
	(\ref{ineq:A_cross_quad_cubic_1})\leq \big(K L r\log n\big)^{cr}\cdot \mathfrak{U}_{\{\mathcal{T}_\cdot\};[-\mathcal{P}]}^{[-\mathcal{P}]}\big(\abs{\mathcal{T}}\big)\cdot n^{-3/2}.
	\end{align} 
	
	\noindent (\textbf{Subcase 1-(b): $r+1 \notin \mathcal{Q}'$}). (\ref{ineq:A_cross_quad_cubic_1}) is bounded by a $K^c n^{-3/2}$ multiple of
	{\small
	\begin{align}\label{ineq:A_cross_quad_cubic_4}
	 &\bigg\{\sum_{\ell_{\mathcal{Q}\setminus \mathcal{Q}'} \in [n]\setminus \mathcal{P}}  \mathfrak{a}^{\ast,2}_{k,\ell_{ (\mathcal{V}^{(0)})^-}  }         \mathscr{A}^{2,[-\mathcal{P}]}_{\ell_{ (\mathcal{V}^{(0)})^+}, (\ell_{ (\mathcal{V}_{1}^{(1)})^-}, \ell_{ (\mathcal{V}^{(2)})^-});[-\mathcal{P}]  }(*) \mathfrak{a}^{\ast,2}_{\ell_{ (\mathcal{V}^{(1)})^+},\ell_{r+1} } \mathfrak{a}^{\ast,2}_{\ell_{ (\mathcal{V}^{(2)})^+},\ell_{r+1} } \prod_{\substack{q \in [0:2],\\(\alpha,\beta) \in \mathscr{P}_2(\mathcal{V}^{(q)})}} \mathfrak{a}_{\alpha,\beta}^{\ast,2}\bigg\}^{1/2}.
	\end{align}}	
	The above display corresponds to (\ref{ineq:A_cross_quad_cubic_3}), so upon using Lemma \ref{lem:A_cross_ind_quad} for terms with subscript indices in $\mathcal{P}=\{j\}$, on an event $\mathscr{E}_{1,b}\cap \mathscr{E}_{\mathcal{T}}(L)\cap \mathscr{E}_{\mathsf{H}}(L)$, where $\mathscr{E}_{1,b}$ with $\Prob(\mathscr{E}_{1,b}^c)\leq  c^r e^{-(\log n)^{100}/c}$, the estimate (\ref{ineq:A_cross_quad_cubic_main_estimate}) holds. 

    \noindent (\textbf{Case 2: $\mathcal{T}_0^+ \notin \mathcal{Q}'$}). This case follows from some modifications to the proceeding arguments. In particular, the term $\mathscr{A}^2_{\ell_{ (\mathcal{V}^{(0)})^+}, (\ell_{ (\mathcal{V}^{(1)})^-}, \ell_{ (\mathcal{V}^{(2)})^-} )}(*)$ in (\ref{ineq:A_cross_quad_cubic_2}) and (\ref{ineq:A_cross_quad_cubic_4}) will be replaced by $\mathfrak{a}^{\ast,2}_{\ell_{ (\mathcal{V}^{(0)})^+},\ell_{\mathcal{T}_0^+}  } \mathfrak{a}^{\ast,2}_{\ell_{\mathcal{T}_0^+},  \ell_{ (\mathcal{V}^{(1)})^-}}  \mathfrak{a}^{\ast,2}_{\ell_{\mathcal{T}_0^+},  \ell_{ (\mathcal{V}^{(2)})^-}}$, so on an event with probability at least $1- c e^{-(\log n)^{100}/c}$, (\ref{ineq:A_cross_quad_cubic_1}) is bounded by a $n^{-3/2} (K\log n)^c$ multiple of 
    {\small\begin{align}\label{ineq:A_cross_quad_cubic_5}
    &\bigg\{\sum_{ \substack{\ell_{\mathcal{Q} \setminus \mathcal{Q}'},\\ \ell_{r+1} \in [n]\setminus \mathcal{P} }  }  \mathfrak{a}^{\ast,2}_{k,\ell_{ (\mathcal{V}^{(0)})^-}  }          \mathfrak{a}^{\ast,2}_{\ell_{ (\mathcal{V}^{(0)})^+},\ell_{\mathcal{T}_0^+}  } \mathfrak{a}^{\ast,2}_{\ell_{\mathcal{T}_0^+},  \ell_{ (\mathcal{V}^{(1)})^-}}  \mathfrak{a}^{\ast,2}_{\ell_{\mathcal{T}_0^+},  \ell_{ (\mathcal{V}^{(2)})^-}}\mathfrak{a}^{\ast,2}_{\ell_{ (\mathcal{V}^{(1)})^+},\ell_{r+1} } \mathfrak{a}^{\ast,2}_{\ell_{ (\mathcal{V}^{(2)})^+},\ell_{r+1} }\prod_{\substack{q \in [0:2],\\(\alpha,\beta) \in \mathscr{P}_2(\mathcal{V}^{(q)})}} \mathfrak{a}_{\alpha,\beta}^{\ast,2} \bigg\}^{1/2}.
    \end{align}}
    Using the counting technique similar to (\ref{ineq:A_hprod_quad_G_count}) (where we identify $(i,j)=(\ell_{r+1},\ell_{r+1})$ with a further summation over $\ell_{r+1}$), upon using Lemma \ref{lem:A_cross_ind_quad} for terms with subscript indices in $\mathcal{P}=\{j\}$, on an event $\mathscr{E}_2\cap \mathscr{E}_{\mathcal{T}}(L)\cap \mathscr{E}_{\mathsf{H}}(L)$, where $\mathscr{E}_2$ with $\Prob(\mathscr{E}_2^c)\leq c^r e^{-(\log n)^{100}/c}$, 
    \begin{align*}
    (\ref{ineq:A_cross_quad_cubic_5})&\leq \big(K L r\log n\big)^{cr}\cdot \mathfrak{U}_{\{\mathcal{T}_\cdot\};[-\mathcal{P}]}^{[-\mathcal{P}]}\big(\abs{\mathcal{T}}\big).
    \end{align*}
    Consequently, on the same event, the estimate (\ref{ineq:A_cross_quad_cubic_main_estimate}) holds. 
    
    Summarizing all cases above  and applying (\ref{ineq:A_hprod_quad_7}) (= a stronger version of Proposition \ref{prop:A_hprod_quad}) to $\mathfrak{U}_{\{\mathcal{T}_\cdot\};[-\mathcal{P}]}^{[-\mathcal{P}]}\big(\abs{\mathcal{T}}\big)$ in (\ref{ineq:A_cross_quad_cubic_main_estimate}) conclude the proof of the claim.
\end{proof}

\begin{proof}[Proof of Proposition \ref{prop:A_cross_cubic}: Any of $\mathcal{T}_0\neq \emptyset, \mathcal{T}_1\neq \emptyset^{[)}, \mathcal{T}_2 \neq \emptyset^{[)}$ fails, but $(\mathcal{T}_0,\mathcal{T}_1,\mathcal{T}_2)\neq \big(\emptyset,\emptyset^{[)},\emptyset^{[)}\big)$]
For $\mathcal{T}_0= \emptyset$, as  $\mathscr{A}_{k,(\ell_0,\ell_0)}\big(\mathcal{T}_0,(\mathcal{T}_{1},\mathcal{T}_2)\big)= \mathsf{H}_{\ast;\ast}\cdot \mathscr{A}_{k,\ell_0}(\mathcal{T}_1)\mathscr{A}_{k,\ell_0}(\mathcal{T}_2)$, the corresponding term to (\ref{ineq:A_cross_quad_cubic_1}) reads 
\begin{align*}
\bigg\{\sum_{\ell_{\mathcal{Q} \setminus \mathcal{Q}'} \in [n]\setminus \mathcal{P}} \bigg(\sum_{\ell_{\mathcal{Q}'}\in [n]\setminus \mathcal{P}} A_{k, [\ell_{\mathcal{T}_1}], \ell_{r+1}}A_{k, [\ell_{\mathcal{T}_2}], \ell_{r+1}}   \overline{A_{\ell_{r+1},j}^3}\prod_{s \in \mathcal{Q}'} \mathsf{H}_{t_s;\ell_s}^{[-\mathcal{P}]}\bigg)^2\bigg\}^{1/2}.
\end{align*}
For $\mathcal{T}_1 = \emptyset^{[)}$, as  $\mathscr{A}_{k,(\ell_0,\ell_0)}\big(\mathcal{T}_0,(\mathcal{T}_{1},\mathcal{T}_2)\big)= \mathsf{H}_{\ast;\ast}\cdot \mathscr{A}_{k,\ell_0}(\mathcal{T}_0^{[)})\mathscr{A}_{\ell_0,\ell_0}(\mathcal{T}_2)$, the corresponding term to (\ref{ineq:A_cross_quad_cubic_1}) reads 
{\small\begin{align*}
\begin{cases}
\Big\{\sum_{\ell_{\mathcal{Q}\setminus \mathcal{Q}'} \in [n]\setminus \mathcal{P}} \Big(\sum_{\ell_{\mathcal{Q}'}\in [n]\setminus \mathcal{P}}  A_{\ell_{r+1}, [\ell_{\mathcal{T}_2}], \ell_{r+1}} \overline{A_{\ell_{r+1},j}^3} \prod_{s \in \mathcal{Q}'} \mathsf{H}_{t_s;\ell_s}^{[-\mathcal{P}]}\Big)^2\Big\}^{1/2}, & \mathcal{T}_0= \emptyset;\\
\Big\{\sum_{\ell_{\mathcal{Q} \setminus \mathcal{Q}'} \in [n]\setminus \mathcal{P}} \Big(\sum_{\ell_{\mathcal{Q}'}\in [n]\setminus \mathcal{P}} A_{k,[\ell_{\mathcal{T}_0^{[)}}], \ell_{r+1}} A_{\ell_{r+1}, [\ell_{\mathcal{T}_2}], \ell_{r+1}} \overline{A_{\ell_{r+1},j}^3} \prod_{s \in \mathcal{Q}'} \mathsf{H}_{t_s;\ell_s}^{[-\mathcal{P}]}\Big)^2\Big\}^{1/2}, & \mathcal{T}_0\neq \emptyset.
\end{cases}
\end{align*}
}
All three terms above can be handled using exactly the same method as above. We omit repetitive details. 
\end{proof}

\section{Proofs of delocalization estimates in Section \ref{section:proof_universality}}\label{section:proof_delocalization}

In this section, we shall focus on the simplified iterate
\begin{align}\label{def:GFOM_sym_proof}
z^{(t)} = A \mathsf{F}_t(z^{(t-1)})+\mathsf{G}_t(z^{(t-1)}),
\end{align}
and similarly for its leave-k-out version (\ref{def:GFOM_sym_loo}). The simplification made here is mostly formal for the proof. We choose to work with the above version to simplify notation. 

\subsection{Proof of Proposition \ref{prop:loo_l2_bound}}\label{subsection:proof_delocalization_iterate}

	Let us work on the event $
	\mathscr{E}_0\equiv \big\{\pnorm{A}{\op}\vee \pnorm{A_{[-\mathcal{P}]}}{\op}\vee\max_{k \in [n]}\pnorm{A_{k\cdot}}{}\leq c_0 K e_n(x)\big\}$ for some large enough $c_0>2$, and $e_n(x)\equiv 1+\sqrt{x/n}$. Without loss of generality we work with $x\geq 1$. By choosing $c_0$ large, in view of \cite[Theorem 4.4.5]{vershynin2018high} and the subsequent remarks, we have $\Prob(\mathscr{E}_0^c)\leq c e^{-x/c}$. Now note that on the event $\mathscr{E}_0$, 
	\begin{align}\label{ineq:gfom_deloc_1}
	\bigpnorm{ z^{(t)}-z_{[-\mathcal{P}]}^{(t)}  }{}&\leq \bigpnorm{A \mathsf{F}_t(z^{(t-1)})- A_{[-\mathcal{P}]} \mathsf{F}_t\big(z^{(t-1)}_{[-\mathcal{P}]}\big)}{}+ \bigpnorm{\mathsf{G}_t(z^{(t-1)})- \mathsf{G}_t\big(z^{(t-1)}_{[-\mathcal{P}]}  }{}\nonumber\\
	&\leq \pnorm{A}{\op}\bigpnorm{\mathsf{F}_t(z^{(t-1)})- \mathsf{F}_t\big(z^{(t-1)}_{[-\mathcal{P}]}\big)}{}+ \bigpnorm{ \big(A- A_{[-\mathcal{P}]}\big) \mathsf{F}_t\big(z^{(t-1)}_{[-\mathcal{P}]}\big)}{}\nonumber\\
	&\qquad\qquad + \bigpnorm{\mathsf{G}_t(z^{(t-1)})- \mathsf{G}_t\big(z^{(t-1)}_{[-\mathcal{P}]}}{}\nonumber\\
	&\leq 2c_0K\Lambda e_n(x)\cdot  \bigpnorm{ z^{(t-1)}-z_{[-\mathcal{P}]}^{(t-1)}  }{} + \bigpnorm{ A_{[\mathcal{P}]} \mathsf{F}_t\big(z^{(t-1)}_{[-\mathcal{P}]}\big)}{}.
	\end{align}
	Here $A_{[\mathcal{P}]}\equiv A-A_{[-\mathcal{P}]}$. On the other hand, as for any $k \in \mathcal{P}$, 
	\begin{align}\label{ineq:gfom_deloc_2}
	\bigabs{z^{(t-1)}_{[-\mathcal{P}],k}} \leq\Lambda \bigabs{z^{(t-2)}_{[-\mathcal{P}],k}}+\Lambda\leq\cdots\leq \Lambda^{t-1} \bigabs{z^{(0)}_{k}}+\sum_{s \in [t-1]}\Lambda^s,
	\end{align}
	we have on  $\mathscr{E}_0$,
	\begin{align}\label{ineq:gfom_deloc_3}
	&\bigpnorm{ A_{[\mathcal{P}]} \mathsf{F}_t\big(z^{(t-1)}_{[-\mathcal{P}]}\big)}{}\nonumber\\
	&= \bigg\{\sum_{\ell \notin \mathcal{P}}\bigiprod{A_{\mathcal{P},\ell}}{ \mathsf{F}_t\big(z^{(t-1)}_{[-\mathcal{P}],\mathcal{P}} \big)}^2+\bigpnorm{ A_{\mathcal{P},\cdot} \mathsf{F}_t\big(z^{(t-1)}_{[-\mathcal{P}]}\big)  }{}^2\bigg\}^{1/2}\nonumber\\
	&\leq \pnorm{A_{\mathcal{P},\cdot}}{F}\bigpnorm{ \mathsf{F}_t\big(z^{(t-1)}_{[-\mathcal{P}],\mathcal{P}}\big)  }{}+ \bigpnorm{ A_{\mathcal{P},\cdot} \mathsf{F}_t\big(z^{(t-1)}_{[-\mathcal{P}]}\big)  }{}\nonumber\\
	&\leq c_0Ke_n(x)\abs{\mathcal{P}}^{1/2}\bigg\{\sum_{k \in \mathcal{P}} \Lambda^2\Big(1+\bigabs{z^{(t-1)}_{[-\mathcal{P}],k}}\Big)^2\bigg\}^{1/2}+\bigg\{\sum_{k \in \mathcal{P}}\bigabs{\bigiprod{A_{k\cdot}}{\mathsf{F}_t\big(z_{[-\mathcal{P}]}^{(t-1)}\big) }}^2\bigg\}^{1/2}\nonumber\\
	&\leq c Ke_n(x)\abs{\mathcal{P}}\cdot\bigg(\Lambda^t \max_{k \in \mathcal{P}} \abs{z_k^{(0)}}+ \sum_{s \in [t]} \Lambda^s\bigg)+ \bigg\{\sum_{k \in \mathcal{P}}\bigabs{\bigiprod{A_{k\cdot}}{\mathsf{F}_t\big(z_{[-\mathcal{P}]}^{(t-1)}\big) }}^2\bigg\}^{1/2}.
	\end{align}
	As $A_{k\cdot}$ is independent of $z_{[-\mathcal{P}]}^{(t-1)}$, on an event $\mathscr{E}_{1,t}\cap \mathscr{E}_0$, where $\Prob(\mathscr{E}_{1,t}^c)\leq c\abs{\mathcal{P}}e^{-x/c}$, uniformly in $k \in \mathcal{P}$, 
	\begin{align}\label{ineq:gfom_deloc_4}
	\bigabs{\bigiprod{A_{k\cdot}}{\mathsf{F}_t\big(z_{[-\mathcal{P}]}^{(t-1)}\big) }}&\leq c\sqrt{x}\cdot K n^{-1/2}\bigpnorm{ \mathsf{F}_t\big(z_{[-\mathcal{P}]}^{(t-1)}\big)  }{}\nonumber\\
	&\leq c \sqrt{x}\cdot K \Lambda\cdot  \Big(1+n^{-1/2}\bigpnorm{z_{[-\mathcal{P}]}^{(t-1)} }{}\Big)\nonumber\\
	&\leq c \sqrt{x}\cdot K\Lambda\big(c K\Lambda e_n(x)+ \Lambda\big)\cdot \big(1+n^{-1/2} \bigpnorm{z_{[-k]}^{(t-2)} }{}\big)\nonumber\\
	&\leq\cdots\leq \big(c K\Lambda e_n(x) + \Lambda\big)^{t}\sqrt{x}\cdot \big(1+n^{-1/2}\pnorm{z^{(0)}}{}\big).
	\end{align}
	Combining the above two displays (\ref{ineq:gfom_deloc_3})-(\ref{ineq:gfom_deloc_4}), on the event $\mathscr{E}_0\cap \mathscr{E}_{1,t}$, 
	\begin{align*}
	\bigpnorm{ A_{[\mathcal{P}]} \mathsf{F}_t\big(z^{(t-1)}_{[-\mathcal{P}]}\big)}{}&\leq \abs{\mathcal{P}}\big(c K\Lambda e_n(x)\big)^{t+1}\sqrt{x}\cdot \Big(1+\max_{k \in \mathcal{P}}\abs{z_k^{(0)}}+n^{-1/2}\pnorm{z^{(0)}}{}\Big).
	\end{align*}
	Consequently, combined with (\ref{ineq:gfom_deloc_1}), on the event $\cap_{s \in [t]}\big(\mathscr{E}_0\cap \mathscr{E}_{1,s}\big)$, 
	\begin{align*}
	\bigpnorm{ z^{(t)}-z_{[-\mathcal{P}]}^{(t)}  }{}
	&\leq c K\Lambda e_n(x)\cdot  \bigpnorm{ z^{(t-1)}-z_{[-\mathcal{P}]}^{(t-1)}  }{} \\
	&\quad\quad + \abs{\mathcal{P}}\big(c K\Lambda e_n(x)\big)^{t+1}\sqrt{x}\cdot \Big(1+\max_{k \in \mathcal{P}}\abs{z_k^{(0)}}+n^{-1/2}\pnorm{z^{(0)}}{}\Big)\\
	&\leq \cdots \leq t \abs{\mathcal{P}}\big(c K\Lambda e_n(x)\big)^{t+1}\sqrt{x}\cdot \Big(1+\max_{k \in \mathcal{P}}\abs{z_k^{(0)}}+n^{-1/2}\pnorm{z^{(0)}}{}\Big),
	\end{align*}
	proving the first claimed inequality.
	
	For the second inequality, let us take $\mathcal{P}=\{k\}$. By (\ref{ineq:gfom_deloc_2}), on the same event as above,
	\begin{align*}
	\abs{z^{(t)}_k}\leq \bigpnorm{ z^{(t)}-z_{[-k]}^{(t)}  }{}+\bigabs{z^{(t)}_{[-k],k}}\leq \big(c K\Lambda e_n(x)\big)^{t+1}(1+\sqrt{x})\cdot \big(1+\pnorm{z^{(0)}}{\infty}\big).
	\end{align*}

	For the moment bound, by writing $Z_k\equiv \abs{z_k^{(t)}}/\big((c K\Lambda)^{t+1}(1+\pnorm{z^{(0)}}{\infty})\big)$, we have $\Prob\big(Z_k\geq x^{t/2+1}\big)\leq c\cdot t e^{-x/c}$ for $x\geq 1$. Now for $p\geq 1$, with the change of variable $z=x^{p(t/2+1)}$, 
	\begin{align*}
	\E Z_k^p &= \int_0^\infty \Prob(Z_k>z^{1/p})\,\d{z}\leq 1+ \int_1^\infty \Prob\big(Z_k>x^{t/2+1}\big)\,\d{x^{p(t/2+1)}}\\
	&\leq 1+ c \cdot p(t/2+1)t\int_1^\infty e^{-x/c} x^{p(t/2+1)-1}\,\d{x}\leq (C_p t)^{t+1}. 
	\end{align*}
	Finally for $Z_\infty=\max_{k \in [n]} Z_k$, a simple union bound leads to $\Prob\big(Z_\infty\geq x^{t/2+1}\big)\leq c\cdot t ne^{-x/c}$ for $x\geq 1$. So using the above calculation, we have
	\begin{align*}
	\E Z_\infty^p \leq C_p^{t+1} \cdot \int_0^\infty \big(ne^{-x}\wedge 1\big) x^{p(t/2+1)-1}\,\d{x} \leq (C_p t)^{t+1}\cdot (\log n)^{2pt}.
	\end{align*}
	The claimed moment bound follows.\qed

\subsection{Proof of Proposition \ref{prop:z_deloc}}\label{subsection:proof_delocalization_derivative}

To prove Proposition \ref{prop:z_deloc}, first we need the following derivative formula. 

\begin{lemma}\label{lem:z_der_form}
	For $t \in \N$, let the (random) matrices $M^{(t-1)}\equiv M_1^{(t-1)} ,M_2^{(t-1)},M_3^{(t-1)}\in \R^{n\times n}$ be defined via 
	\begin{align}\label{ineq:z_der_1}
	M^{(t-1)}_{q,k\ell}\equiv A_{k\ell} \mathsf{F}_{t,\ell}^{(q)}(z^{(t-1)}_\ell)+\delta_{k\ell} \mathsf{G}_{t,\ell}^{(q)}(z^{(t-1)}_\ell),\quad q=1,2,3.
	\end{align}
	Then for $i,j,k \in [n]$,
	\begin{align*}
	\partial_{ij} z^{(t)}_k& =\sum_{(a,b)}  \sum_{s \in [0:t-1]} M_{ka}^{(t-1:s+1)} \mathsf{F}_{s+1,b}(z^{(s)}_b),\\
	\partial_{ij}^2 z^{(t)}_k&= 2 \sum_{(a,b)} \sum_{s \in [0:t-1]} M_{ka}^{(t-1:s+1)} \mathsf{F}_{s+1,b}'(z^{(s)}_b) \partial_{ij} z_b^{(s)}+ \sum_{s \in [0:t-1]}\sum_{\ell \in [n]} \Big( M^{(t-1:s+1)} M_2^{(s)} \Big)_{k\ell} \big(\partial_{ij} z^{(s)}_\ell\big)^2,\\
	\partial_{ij}^3 z_k^{(t)}& = 3 \sum_{(a,b)} \sum_{s \in [0:t-1]} M_{ka}^{(t-1:s+1)} \Big(\mathsf{F}_{s+1,b}^{(2)}(z^{(s)}_b) \big(\partial_{ij} z_b^{(s)}\big)^2+\mathsf{F}_{s+1,b}'(z^{(s)}_b)\partial_{ij}^2 z_b^{(s)} \Big)\\
	&\quad + \sum_{s \in [0:t-1]} \sum_{\ell \in [n]} \bigg(\big(M^{(t-1:s+1)}M_3^{(s)}\big)_{k\ell}\big(\partial_{ij} z^{(s)}_{\ell}\big)^3 + 3 \big(M^{(t-1:s+1)}M_2^{(s)}\big)_{k\ell}\big(\partial_{ij} z^{(s)}_{\ell}\partial_{ij}^2 z^{(s)}_{\ell}\big)\bigg).
	\end{align*}
	Here the summation over $(a,b)$ in the above display runs over $\{(i,j),(j,i)\}$, and we write $M^{(t:s)}\equiv M^{(t)}\cdots M^{(s)}\bm{1}_{s\leq t}+I_n \bm{1}_{s>t} \in \R^{n\times n}$.
\end{lemma}

\begin{remark}
In the formulae for the second and third derivatives $\partial_{ij}^2 z^{(t)}, \partial_{ij}^3 z^{(t)}$, the summation $\sum_{s \in [0:t-1]}$ can be replaced by $\sum_{s \in [t-1]}$, as $\partial_{\cdot\cdot}^q z^{(0)}=0$ for $q\geq 1$. 
\end{remark}

\begin{proof}[Proof of Lemma \ref{lem:z_der_form}]
	We write $\mathsf{F}^{(\cdot)}_{t;\ell}\equiv \mathsf{F}^{(\cdot)}_{t,\ell}(z^{(t-1)}_\ell)$, $\overline{\mathsf{F}}^{(\cdot)}_{t}\equiv(\mathsf{F}^{(\cdot)}_{t;\ell})_{\ell \in [n]}\in \R^n$, and similarly for $\mathsf{G}$ for notational simplicity. For $i, j \in [n]$, let $\Delta_{ij}\equiv (\bm{1}_{i\neq j}+2^{-1}\bm{1}_{i=j})(e_ie_j^\top+ e_je_i^\top)$. Then $\partial_{ij} A = \Delta_{ij}$. 
	
	\noindent (1). 
	Then the recursion for $\partial_{ij} z^{(t)}$ can be rewritten as
	\begin{align*}
	\partial_{ij} z^{(t)} &= \Delta_{ij} \overline{\mathsf{F}}_t+A \big(\overline{\mathsf{F}}_t'\circ \partial_{ij} z^{(t-1)}\big)+ \overline{\mathsf{G}}_t'\circ \partial_{ij} z^{(t-1)}= \Delta_{ij} \overline{\mathsf{F}}_{t} + M^{(t-1)}\partial_{ij} z^{(t-1)}.
	\end{align*}
	Iterating this representation and using $\partial_{ij} z^{(0)}=0$, we arrive at
	\begin{align*}
	\partial_{ij} z^{(t)} = \Delta_{ij} \overline{\mathsf{F}}_{t}+ M^{(t-1)} \Delta_{ij} \overline{\mathsf{F}}_{t-1}+\cdots+ M^{(t-1:1)}\Delta_{ij} \overline{\mathsf{F}}_{1}=\sum_{s \in [0:t-1]} M^{(t-1:s+1)}\Delta_{ij} \overline{\mathsf{F}}_{s+1}.
	\end{align*}
	Consequently, for any $k \in [n]$,
	\begin{align}\label{ineq:z_der_2}
	\partial_{ij} z^{(t)}_k
	& =\sum_{(a,b)\in \{(i,j),(j,i)\}}  \sum_{s \in [0:t-1]} M_{ka}^{(t-1:s+1)} \mathsf{F}_{s+1;b}.
	\end{align}
	\noindent (2). The recursion for $\partial_{ij}^2 z^{(t)}$ can be written as 
	\begin{align*}
	\partial_{ij}^2 z^{(t)} 
	& = 2\Delta_{ij} \Big(\overline{\mathsf{F}}_t'\circ \partial_{ij} z^{(t-1)}\Big)+ M_2^{(t-1)}\big(\partial_{ij} z^{(t-1)}\big)^2+ M^{(t-1)} \partial_{ij}^2 z^{(t-1)}.
	\end{align*}
	Iterating this representation we obtain 
	\begin{align*}
	\partial_{ij}^2 z^{(t)}=2\sum_{s \in [t-1]} M^{(t-1:s+1)}\Delta_{ij} \big(\overline{\mathsf{F}}_{s+1}'\circ \partial_{ij} z^{(s)}\big)+\sum_{s \in [t-1]} M^{(t-1:s+1)} M_2^{(s)}\big(\partial_{ij} z^{(s)}\big)^2.
	\end{align*}
	Taking the $k$-component on both sides of the above display to conclude.

	\noindent (3). The recursion for $\partial_{ij}^3 z^{(t)}$ can be written as 
	\begin{align*}
	\partial_{ij}^3 z^{(t)}&=3\Delta_{ij} \Big(\overline{\mathsf{F}}_t^{(2)}\circ (\partial_{ij} z^{(t-1)})^2+\overline{\mathsf{F}}_t'\circ \partial_{ij}^2 z^{(t-1)}\Big)+ M_3^{(t-1)} (\partial_{ij} z^{(t-1)})^3\\
	&\qquad + 3 M_2^{(t-1)} (\partial_{ij} z^{(t-1)}\circ \partial_{ij}^2 z^{(t-1)}) + M^{(t-1)} \partial_{ij}^3 z^{(t-1)}.
	\end{align*}
	Iterating this representation we obtain
	\begin{align*}
	\partial_{ij}^3 z^{(t)}&=3 \sum_{s \in [t-1]} M^{(t-1:s+1)}\Delta_{ij} \Big(\overline{\mathsf{F}}_{s+1}^{(2)}\circ (\partial_{ij} z^{(s)})^2+\overline{\mathsf{F}}_{s+1}'\circ \partial_{ij}^2 z^{(s)}\Big)\\
	&\quad + \sum_{s \in [t-1]} M^{(t-1:s+1)}M_3^{(s)}\big(\partial_{ij} z^{(s)}\big)^3+3\sum_{s \in [t-1]} M^{(t-1:s+1)}M_2^{(s)}\big(\partial_{ij} z^{(s)}\circ \partial_{ij}^2 z^{(s)}\big).
	\end{align*}
	Taking the $k$-component on both sides of the above display to conclude. 
\end{proof}

\begin{lemma}\label{lem:z_apriori_moment}
	Suppose the following hold:
	\begin{enumerate}
		\item $A=A_0/\sqrt{n}$, where $A_0$ is symmetric and the entries of its upper triangle are independent mean $0$ random variables with $\max_{i,j \in [n]} \pnorm{A_{0,ij}}{\psi_2}\leq K$.
		\item For all $s \in [t],\ell \in [n]$, $\mathsf{F}_{s,\ell},\mathsf{G}_{s,\ell} \in C^3(\R)$, Moreover, there exists some $\Lambda\geq 2$ and $\mathfrak{p}\in \N$ such that 
		\begin{align*}
		\max_{s \in [t]}\max_{\mathsf{E}_s \in \{\mathsf{F}_s,\mathsf{G}_s\}}\max_{\ell \in [n]}\Big\{\pnorm{\mathsf{E}_{s,\ell}}{\mathrm{Lip}}+\max_{q \in [0:3]} \bigpnorm{(1+\abs{\cdot})^{-\mathfrak{p}}\abs{\mathsf{E}_{s,\ell}^{(q)}(\cdot)}}{\infty}\Big\}\leq \Lambda.
		\end{align*}
	\end{enumerate}
	Then for any $p\geq 1$, there exists a constant $c_1=c_1(p,\mathfrak{p})>0$ such that 
	\begin{align*}
	\max_{q\in [3]}\E^{1/p} \big[ \pnorm{ \partial_{ij}^q z^{(t)} }{}^p| z^{(0)}\big]\leq \big(t K\Lambda \log n\cdot (1+\pnorm{z^{(0)}}{\infty})\big)^{c_1 t^2}.
	\end{align*}
\end{lemma}
\begin{proof}
	For the first derivative, the recursive bound for $\partial_{ij} z^{(t)}$ reads
	\begin{align*}
	\pnorm{ \partial_{ij} z^{(t)} }{}&\lesssim \Lambda\big(1+ \abs{z^{(t-1)}_i}+\abs{z^{(t-1)}_j}\big)+ \big(\pnorm{A}{\op} \pnorm{\mathsf{F}_t'(z^{(t-1)}) }{\infty}+ \pnorm{\mathsf{G}_t'(z^{(t-1)}) }{\infty} \big) \cdot \pnorm{ \partial_{ij} z^{(t-1)} }{}\\
	&\lesssim \Lambda \big(\pnorm{A}{\op}\vee 1\big) \big(1\vee  \pnorm{z^{(t-1)}}{\infty}\big)^{\mathfrak{p}}\big(1\vee  \pnorm{\partial_{ij} z^{(t-1)}}{}\big).
	\end{align*}
	This means
	\begin{align}\label{ineq:z_apriori_moment_1}
	\pnorm{ \partial_{ij} z^{(t)} }{}\leq \big(c_1 \Lambda (\pnorm{A}{\op}\vee 1)\big)^t \cdot \prod_{s \in [0: t-1]} \big(1\vee \pnorm{z^{(s)}}{\infty}\big)^{\mathfrak{p}}. 
	\end{align}
	For the second derivative, the recursive bound for $\partial_{ij}^2 z^{(t)}$ reads
	\begin{align*}
	\pnorm{\partial_{ij}^2 z^{(t)}}{}&\lesssim \Lambda\big\{1+\abs{z_i^{(t-1)}}\cdot \abs{\partial_{ij} z^{(t-1)}_i}\big\}\\
	&\qquad +\big(\pnorm{A}{\op} \pnorm{\mathsf{F}_t''(z^{(t-1)}) }{\infty}+ \pnorm{\mathsf{G}_t''(z^{(t-1)}) }{\infty} \big)\cdot \pnorm{\partial_{ij} z^{(t-1)}}{}^2\\
	&\qquad + \big(\pnorm{A}{\op} \pnorm{\mathsf{F}_t'(z^{(t-1)}) }{\infty}+ \pnorm{\mathsf{G}_t'(z^{(t-1)}) }{\infty} \big) \cdot \pnorm{ \partial_{ij}^2 z^{(t-1)} }{}\\
	&\lesssim \big(c \Lambda (\pnorm{A}{\op}\vee 1)\big)^{2t}\prod_{s \in [0: t-1]} \big(1\vee \pnorm{z^{(s)}}{\infty}\big)^{2\mathfrak{p}}\\
	&\qquad + \Lambda (\pnorm{A}{\op}\vee 1)\big(1\vee  \pnorm{z^{(t-1)}}{\infty}\big)^{\mathfrak{p}}\cdot \pnorm{ \partial_{ij}^2 z^{(t-1)}}{ }.
	\end{align*}
	Iterating the bound, we obtain 
	\begin{align}\label{ineq:z_apriori_moment_2}
	\pnorm{\partial_{ij}^2 z^{(t)}}{}\leq \big(c_2 \Lambda (\pnorm{A}{\op}\vee 1)\big)^{2t}\prod_{s \in [0: t-1]} \big(1\vee \pnorm{z^{(s)}}{\infty}\big)^{2\mathfrak{p}}.
	\end{align}
	For the third derivative,   the recursive bound for $\partial_{ij}^3 z^{(t)}$ reads
	\begin{align*}
	\pnorm{\partial_{ij}^3 z^{(t)}}{}&\lesssim \Lambda\big\{1+\abs{z_i^{(t-1)}}\cdot \abs{\partial_{ij} z^{(t-1)}_i}^2+ \abs{z_i^{(t-1)}}\cdot \abs{\partial_{ij}^2 z^{(t-1)}_i} \big\}\\
	&\qquad +\big(\pnorm{A}{\op} \pnorm{\mathsf{F}_t^{(3)}(z^{(t-1)}) }{\infty}+ \pnorm{\mathsf{G}_t^{(3)}(z^{(t-1)}) }{\infty} \big)\cdot \pnorm{\partial_{ij} z^{(t-1)}}{}^3\\
	&\qquad + \big(\pnorm{A}{\op} \pnorm{\mathsf{F}_t^{(2)}(z^{(t-1)}) }{\infty}+ \pnorm{\mathsf{G}_t^{(2)}(z^{(t-1)}) }{\infty} \big) \cdot \pnorm{ \partial_{ij} z^{(t-1)} }{}\pnorm{ \partial_{ij}^2 z^{(t-1)} }{}\\
	&\qquad + \big(\pnorm{A}{\op} \pnorm{\mathsf{F}_t'(z^{(t-1)}) }{\infty}+ \pnorm{\mathsf{G}_t'(z^{(t-1)}) }{\infty} \big) \cdot \pnorm{ \partial_{ij}^3 z^{(t-1)} }{}\\
	&\lesssim \big(c \Lambda (\pnorm{A}{\op}\vee 1)\big)^{3t}\prod_{s \in [0: t-1]} \big(1\vee \pnorm{z^{(s)}}{\infty}\big)^{3\mathfrak{p}}\\
	&\qquad + \Lambda (\pnorm{A}{\op}\vee 1)\big(1\vee  \pnorm{z^{(t-1)}}{\infty}\big)^{\mathfrak{p}}\cdot \pnorm{ \partial_{ij}^3 z^{(t-1)}}{ }.
	\end{align*}
	Iterating the bound, we obtain 
	\begin{align}\label{ineq:z_apriori_moment_3}
	\pnorm{\partial_{ij}^3 z^{(t)}}{}\leq \big(c_3 \Lambda (\pnorm{A}{\op}\vee 1)\big)^{3t}\prod_{s \in [0: t-1]} \big(1\vee \pnorm{z^{(s)}}{\infty}\big)^{3\mathfrak{p}}.
	\end{align}
	The desired moment estimates then follow from (\ref{ineq:z_apriori_moment_1})-(\ref{ineq:z_apriori_moment_3}), and the last claimed inequality in Proposition \ref{prop:loo_l2_bound}.
\end{proof}

\begin{proof}[Proof of Proposition \ref{prop:z_deloc}]
	We write $\mathsf{F}^{(\cdot)}_{t;\ell}\equiv \mathsf{F}^{(\cdot)}_{t,\ell}(z^{(t-1)}_\ell)\in \R$, $\overline{\mathsf{F}}^{(\cdot)}_{t}\equiv(\mathsf{F}^{(\cdot)}_{t;\ell})_{\ell \in [n]}\in \R^n$, and similarly for $\mathsf{G}$ for notational simplicity.

	\noindent (1). 
	To avoid unnecessarily complicated notation, we analyze the last term in (\ref{ineq:z_der_2}) with $(a,b)=(i,j)$. Then with $\ell_1\equiv i, \ell_t\equiv k$, 
	\begin{align}\label{ineq:z_der_M}
	&M^{(t-1:1)}_{ki}= \sum_{\ell_{t-1},\ldots,\ell_2} \prod_{s \in [t-1]} M_{\ell_s,\ell_{s+1}}^{((s+1)-1)}=\sum_{\ell_{[2:t-1]}} \prod_{s\in [t-1]}\big(A_{\ell_{s},\ell_{s+1}}\mathsf{F}_{s+1;\ell_{s+1}}'+\delta_{\ell_{s}\ell_{s+1}} \mathsf{G}_{s+1;\ell_{s+1}}'\big)\nonumber\\
	&= \sum_{\mathcal{Q}\subset [t-1]}\bigg(\sum_{\ell_{[2:t-1]}}\mathsf{S}_{\mathcal{Q}}(\ell_{[2:t-1]}) \prod_{s \in \mathcal{Q}} A_{\ell_s,\ell_{s+1}}\prod_{s \in [t-1]\setminus \mathcal{Q}}\delta_{\ell_s,\ell_{s+1}} \bigg)\equiv \sum_{\mathcal{Q}\subset [t-1]}\mathscr{S}_{ki;\mathcal{Q}},
	\end{align}
	where $
	\mathsf{S}_{\mathcal{Q}}(\ell_{[2:t-1]})\equiv \prod_{s \in \mathcal{Q}} \mathsf{F}_{s+1;\ell_{s+1}}' \prod_{s \in [t-1]\setminus \mathcal{Q}} \mathsf{G}_{s+1;\ell_{s+1}}'$. 
    To avoid unnecessary notational complications, we work with the term with $\mathcal{Q}=\{1,\ldots,r\}$ for some $0\leq r\leq t-1$; the general case can be handled by a simple relabelling. The corresponding summand for the prescribed $\mathcal{Q}$ then becomes
	\begin{align*}
	\mathscr{S}_{ki;\mathcal{Q}}\equiv \bigg(\sum_{\ell_{[2:r]}} A_{k,\ell_2}\cdots A_{\ell_r,i}\cdot \mathsf{F}_{2;\ell_2}'\cdots \mathsf{F}_{r;\ell_r}'\bigg)\cdot \mathsf{F}_{r+1;i}' \mathsf{G}_{r+2;i}'\cdots \mathsf{G}_{t;i}'.
	\end{align*}
	For the choice 
	\begin{align}\label{ineq:z_der_L}
	L\equiv \big(C_1 K \Lambda \log n\cdot (1+\pnorm{z^{(0)}}{\infty})\big)^{c_1  t}
	\end{align}
	with some sufficiently large $C_1>0$, using Propositions \ref{prop:loo_l2_bound} and \ref{prop:A_hprod_quad}, on the event $\mathscr{E}_{[t]}(L)\cap \mathscr{E}_{\mathcal{Q}}$ with $\Prob(\mathscr{E}_{\mathcal{Q}}^c)\leq C_1 c_1^{t} n^{-D}$, 
	\begin{align*}
	\Big(\max_{k\neq i} \sqrt{n} + \max_{k,i}\Big)\bigabs{ \mathscr{S}_{ki;\mathcal{Q}}  }\leq \big(C_1 t L \log n\big)^{c_1 t^2}.
	\end{align*}
	So combining the above displays, by possibly enlarging $c_1,C_1>0$, on the event $\mathscr{E}_{[t]}(L)\bigcap \mathscr{E}_1$ with $\Prob(\mathscr{E}_1^c)\leq C_1 c_1^t n^{-D}$,  
	\begin{align}\label{ineq:z_der_3}
	\max_{s \in [t-1]}\Big(\max_{k\neq i} \sqrt{n} + \max_{k}\Big)\bigabs{ M^{(t-1:s)}_{ki}  }\leq  \big(C_1 t L \log n\big)^{c_1 t^2}.
	\end{align}
	Consequently, by a union bound and adjusting $c_1,C_1>0$, and we assume without loss of generality that on the same event $\mathscr{E}_{[t]}(L)\bigcap \mathscr{E}_1$, 
	\begin{align}\label{ineq:z_der_4}
	\max_{s \in [t]}\Big(\max_{k\neq {i,j}} \sqrt{n} + \max_{k}\Big) \bigabs{\partial_{ij} z^{(s)}_k}\leq \big(C_1 t L \log n\big)^{c_1 t^2}.
	\end{align}
	For the prescribed choice of $L$, we have $\Prob(\mathscr{E}_{[t]}(L)^c)\leq \binom{n}{ct}^2 \cdot Ct e^{-\log^2 n/C}\leq C n^{-D}$ when $t\leq c_0^{-1}\log n$. This concludes the claim.
	
	\noindent (2). 
	We write the two terms in the derivative formula for $\partial^2_{ij} z^{(t)}_k$ as $(I)$ and $(II)$. Using the same proof around as in (1) with minor modifications changing some $\mathsf{F}',\mathsf{G}'$ to $\mathsf{F}'',\mathsf{G}''$, the estimate (\ref{ineq:z_der_3}) also holds for $M^{(t-1:s+1)} M_2^{(s)}$. Consequently, on the same event $\mathscr{E}_{[t]}(L)\bigcap \mathscr{E}_1$, using the estimates (\ref{ineq:z_der_3})-(\ref{ineq:z_der_4}), we then have
	\begin{align*}
	\abs{(I)}&\lesssim t\cdot \Big\{\max_{s \in [t-1]} \abs{M_{ki}^{(t-1:s)}}\vee \abs{M_{kj}^{(t-1:s)}} \Big\}\cdot \Lambda\prod_{q=0,1}\max_{s \in [t]}\big(1+\pnorm{\partial_{ij}^q z^{(s)}}{\infty}\big)^{\mathfrak{p}}\\
	&\leq  \big(\bm{1}_{k \in \{i,j\}}+n^{-1/2}\big)\cdot \big(C_2 t L \log n\big)^{c_2 t^2},
	\end{align*}
	and with $\tilde{M}^{(t-1:s)}\equiv M^{(t-1:s+1)}M_2^{(s)}$, 
	\begin{align*}
	\abs{(II)}&\lesssim t\cdot \max_{s \in [t-1]} \bigg\{\Big(\abs{\tilde{M}_{ki}^{(t-1:s)}}\vee \abs{\tilde{M}_{kj}^{(t-1:s)}}\Big)\cdot  \pnorm{\partial_{ij} z^{(s)}}{\infty}^2 +  \abs{\tilde{M}_{kk}^{(t-1:s)}} \big(\partial_{ij} z^{(s)}_k\big)^2\\
	&\qquad\qquad + \max_{\ell \neq k} \abs{\tilde{M}_{k\ell}^{(t-1:s)}}\cdot \max_{k\neq i,j} \bigabs{\sqrt{n} \cdot \partial_{ij} z_k^{(s)}}^2 \bigg\}\leq \big(C_2 t L \log n\big)^{c_2 t^2}.
	\end{align*}
	Combining the above displays, on the event $\mathscr{E}_{[t]}(L)\bigcap \mathscr{E}_1$, we have
	\begin{align}\label{ineq:z_der_7}
	\max_{s \in [t]}\Big(\max_{k\neq {i,j}} \sqrt{n} + \max_{k}\Big) \bigabs{\partial_{ij}^2 z^{(s)}_k}\leq \big(C_2 t L \log n\big)^{c_2 t^2}.
	\end{align}
	The claimed large deviation estimate follows. The moment estimates follow from the large deviation estimate and the apriori control in Lemma \ref{lem:z_apriori_moment}.
\end{proof}

\subsection{Proof of Proposition \ref{prop:der_cross_cubic}}\label{subsection:proof_delocalization_other}

	Similar to the proof of Proposition \ref{prop:z_deloc}, we write $\mathsf{F}^{(\cdot)}_{t;\ell}\equiv \mathsf{F}^{(\cdot)}_{t,\ell}(z^{(t-1)}_\ell)\in \R$, $\overline{\mathsf{F}}^{(\cdot)}_{t}\equiv(\mathsf{F}^{(\cdot)}_{t;\ell})_{\ell \in [n]}\in \R^n$, and similarly for $\mathsf{G}$ for notational simplicity. Using the derivative formula in Lemma \ref{lem:z_der_form}, we may write
	\begin{align}\label{ineq:der_cross_cubic_1}
	A_{ij}^3 \partial_{ij}^3 z^{(t)}_k = \sum_{\ell \in [4]}S_{(i,j,k);\ell},
	\end{align}
	where
	\begin{align*}
	S_{(i,j,k);1}&\equiv 3 \sum_{(a,b)} \sum_{s \in [t-1]} M_{ka}^{(t-1:s+1)} A_{ij}^3  \big(\partial_{ij} z_b^{(s)}\big)^2\cdot \mathsf{F}_{s+1;b}^{(2)},\\
	S_{(i,j,k);2}&\equiv 3 \sum_{(a,b)} \sum_{s \in [t-1]} M_{ka}^{(t-1:s+1)} A_{ij}^3  \partial_{ij}^2 z_b^{(s)}\cdot \mathsf{F}_{s+1;b}',\\
	S_{(i,j,k);3}&\equiv \sum_{s \in [t-1]} \sum_{\ell \in [n]} \big(M^{(t-1:s+1)}M_3^{(s)}\big)_{k\ell}\big(\partial_{ij} z^{(s)}_{\ell}\big)^3\cdot A_{ij}^3,\\
	S_{(i,j,k);4}&\equiv 3\sum_{s \in [t-1]} \sum_{\ell \in [n]} \big(M^{(t-1:s+1)}M_2^{(s)}\big)_{k\ell}\big(\partial_{ij} z^{(s)}_{\ell}\partial_{ij}^2 z^{(s)}_{\ell}\big)\cdot A_{ij}^3.
	\end{align*}
	Let us now work on the event $\mathscr{E}_0\cap \mathscr{E}_{[t]}(L)$ with $\Prob(\mathscr{E}_0^c)\leq C_0 c_0^t n^{-D}$ and $L$ chosen according to (\ref{ineq:z_der_L}), where
	\begin{align}\label{ineq:der_cross_cubic_2}
	&\max_{s \in [t], u \in [s-1]}\Big(\max_{k\neq i} \sqrt{n} + \max_{k}\Big)\Big\{ \max_{q \in [3]}\bigabs{ M^{(s-1:u)}_{q,ki}  } \vee \max_{q=2,3} \bigabs{ \big(M^{(s-1:u+1)}M_q^{(u)}\big)_{ki} }  \Big\}\nonumber\\
	&\qquad + \max_{s \in [t]}\Big(\max_{k\neq {i,j}} \sqrt{n} + \max_{k}\Big) \Big\{\bigabs{\partial_{ij} z^{(s)}_k}\vee\bigabs{\partial_{ij}^2 z^{(s)}_k}\Big\}  \leq  \big(C_0 t L \log n\big)^{c_0 t^2}.
	\end{align}
	
	\noindent (\textbf{Term $S_{(i,j,k);1}$}). Note that
	\begin{align*}
	S_{(i,j,k);1}&= 3 \sum_{(a,b)} \sum_{s \in [t-1]} M_{ka}^{(t-1:s+1)} A_{ij}^3 \bigg(\sum_{(c,d)} \sum_{u \in [0:s-1]} M_{bc}^{(s-1:u+1)} \mathsf{F}_{u+1;d}\bigg)^2\mathsf{F}_{s+1;b}^{(2)}\\
	& = 3 \sum_{\substack{(a,b)\\(c,d),(c',d')}} \sum_{\substack{s \in [t-1],\\ u,u' \in [0:s-1]}} M_{ka}^{(t-1:s+1)}  M_{bc}^{(s-1:u+1)}  M_{bc'}^{(s-1:u'+1)} A_{ij}^3\cdot \mathsf{F}_{u+1;d}\mathsf{F}_{u'+1;d'}\mathsf{F}_{s+1;b}^{(2)}.
	\end{align*}
	Here the summations over $(a,b),(c,d),(c',d')$ run through $\{(i,j),(j,i)\}$. Let us consider $(a,b)=(i,j)$. The corresponding summation becomes
	\begin{align*}
	\sum_{(c,d),(c',d')} \sum_{\substack{s \in [t-1],\\ u,u' \in [0:s-1]}} M_{ki}^{(t-1:s+1)}  M_{jc}^{(s-1:u+1)}  M_{jc'}^{(s-1:u'+1)} A_{ij}^3\cdot \mathsf{F}_{u+1;d}\mathsf{F}_{u'+1;d'}\mathsf{F}_{s+1;j}^{(2)}.
	\end{align*}
	Except from the case $(c,d)=(c',d')=(j,i)$, all other cases can be bounded by
	\begin{align}\label{ineq:der_cross_cubic_2_0}
	n^{-3/2} \cdot (C_1 t K L \log n)^{c_1 t^2}\cdot n^{ -(\abs{\{i,j,k\}}-1)/2}
	\end{align}
	that holds with probability at least $1- C_1 n^{-D}$. Consequently, on an event $\mathscr{E}_{1,1}$ with $\Prob(\mathscr{E}_{1,1}^c)\leq C_1 n^{-D}$, by letting
	\begin{align*}
	\mathfrak{S}_{(i,j,k);1}\equiv \sum_{\substack{s \in [t-1],\\ u,u' \in [0:s-1]}} M_{ki}^{(t-1:s+1)}  M_{jj}^{(s-1:u+1)}  M_{jj}^{(s-1:u'+1)} A_{ij}^3\cdot \mathsf{F}_{u+1;i}\mathsf{F}_{u'+1;i}\mathsf{F}_{s+1;j}^{(2)},
	\end{align*}
	we have uniformly in $i,j \in [n]$,
	\begin{align}\label{ineq:der_cross_cubic_3}
	&\big|S_{(i,j,k);1}-3\mathfrak{S}_{(i,j,k);1}-3\mathfrak{S}_{(j,i,k);1} \big| \leq (C_1 t K L \log n)^{c_1 t^2}\cdot n^{ -\abs{\{i,j,k\}}/2-1}.
	\end{align}
	On the other hand, note that
	\begin{align*}
	\sum_{i,j \in [n]}\mathfrak{S}_{(i,j,k);1} = \sum_{j\in [n]}\sum_{\substack{s \in [t-1],\\ u,u' \in [0:s-1]}} \bigg[ \sum_{i \in [n]} M_{ki}^{(t-1:s+1)} A_{ij}^3\cdot \mathsf{F}_{u+1;i}\mathsf{F}_{u'+1;i}\bigg] M_{jj}^{(s-1:u+1)}  M_{jj}^{(s-1:u'+1)} \mathsf{F}_{s+1;j}^{(2)}.
	\end{align*}
	Using the representation (\ref{ineq:z_der_M}), we may write $M_{ki}^{(\cdot:\cdot)}$ as at most $2^t$ many sums of terms of the form $\sum_{\ell_{\mathcal{I}} \in [n] } A_{k,[\ell_{\mathcal{I}}],i}\prod_{s \in \mathcal{I}} \mathsf{F}_{t_s;\ell_s}'\cdot \prod_{s \in \mathcal{J}} \mathsf{G}_{t_s;\ell_s}' $ for some $\mathcal{I},\mathcal{J}\subset [t]$ with $\mathcal{I}\cap \mathcal{J}=\emptyset$. Using Proposition \ref{prop:A_cross_cubic} with $q_0=2$, on the event $\mathscr{E}_{1,1}\cap \mathscr{E}_{[t]}(L)$ where $\Prob(\mathscr{E}_{1,1}^c)\leq C_1\cdot c_1^t n^{-D}$, uniformly in $j\in [n]\setminus \{k\}$,
	\begin{align*}
	&\biggabs{\sum_{\substack{s \in [t-1],\\ u,u' \in [0:s-1]}} \bigg( \sum_{i \in [n]} M_{ki}^{(t-1:s+1)} A_{ij}^3\cdot \mathsf{F}_{u+1;i}\mathsf{F}_{u'+1;i}\bigg) M_{jj}^{(s-1:u+1)}  M_{jj}^{(s-1:u'+1)} \mathsf{F}_{s+1;j}^{(2)}}\\
	&\leq \big(C_1 tK L \log n\big)^{c_1t^2}\cdot n^{-3/2}. 
	\end{align*}
	Note that in the above display, for the case $s=t-1$, we may use a trivial bound as $M_{ki}^{(t-1:t)}=\delta_{ki}$. For $j=k$, we may also use a trivial bound. Consequently, on an event $\mathscr{E}_{1,+}\cap \mathscr{E}_{[t]}(L)$ with $\Prob(\mathscr{E}_{1,+}^c)\leq C_1\cdot c_1^t n^{-D}$,
	\begin{align}\label{ineq:der_cross_cubic_4}
	\biggabs{\sum_{i,j \in [n]}\mathfrak{S}_{(i,j,k);1}}\vee \biggabs{\sum_{i,j \in [n]}\mathfrak{S}_{(j,i,k);1}}\leq \big(C_1 tK L \log n\big)^{c_1t^2}\cdot n^{-1/2}.
	\end{align}
	Combining (\ref{ineq:der_cross_cubic_3})-(\ref{ineq:der_cross_cubic_4}), on an event $\mathscr{E}_{1}\cap \mathscr{E}_{[t]}(L)$ with $\Prob(\mathscr{E}_{1}^c)\leq C_1\cdot c_1^t n^{-D}$, as $\sum_{i,j \in [n]} n^{ -\abs{\{i,j,k\}}/2-1}\lesssim n^{-1/2}$, we have
	\begin{align}\label{ineq:der_cross_cubic_5}
	\biggabs{\sum_{i,j \in [n]} S_{(i,j,k);1}}&\leq \big(C_1 tK L \log n\big)^{c_1t^2}\cdot n^{-1/2}.
	\end{align}

	\noindent (\textbf{Term $S_{(i,j,k);2}$}). Note that
	\begin{align}\label{ineq:der_cross_cubic_6}
	S_{(i,j,k);2}&= 3 \sum_{(a,b)} \sum_{s \in [t-1]} M_{ka}^{(t-1:s+1)} A_{ij}^3\cdot\bigg(2 \sum_{(c,d)}\sum_{u \in [s-1]} M_{bc}^{(s-1:u+1)} \mathsf{F}_{u+1;d}' \partial_{ij} z_d^{(u)}\\
	&\qquad +\sum_{u \in [s-1]}\sum_{\ell \in [n]} \tilde{M}^{(s-1:u)}_{b\ell} (\partial_{ij} z_\ell^{(u)})^2\bigg)\cdot \mathsf{F}_{s+1;b}'\nonumber\\
	& = 6 \sum_{\substack{(a,b)\\(c,d),(c',d')}} \sum_{\substack{s \in [t-1],\\ u \in [s-1],\\ u' \in [0:u-1]}} M_{ka}^{(t-1:s+1)}M_{bc}^{(s-1:u+1)} M_{d c'}^{(u-1:u'+1)} A_{ij}^3\cdot \mathsf{F}_{u'+1;d'}  \mathsf{F}_{u+1;d}' \mathsf{F}_{s+1;b}'\nonumber\\
	&\qquad + 3 \sum_{(a,b)}\sum_{s \in [t-1], u \in [s-1]} \bigg(\sum_{\ell \in [n]}M_{ka}^{(t-1:s+1)} \tilde{M}^{(s-1:u)}_{b\ell} (\partial_{ij} z_\ell^{(u)})^2\bigg)\cdot A_{ij}^3\cdot \mathsf{F}_{s+1;b}'.\nonumber
	\end{align}
	We write the two terms on the right hand side of the above display as $S_{(i,j,k);2,I}$ and $S_{(i,j,k);2, II}$. 
	
	The first term $S_{(i,j,k);2,I}$ can be handled  in a similar fashion to  $S_{(i,j,k);1}$. Indeed, consider the configuration $(a,b)=(i,j)$. The corresponding inner sum in $S_{(i,j,k);2,I}$ then reads
	\begin{align*}
	\sum_{\substack{s \in [t-1], u \in [s-1], u' \in [0:u-1]}} M_{ki}^{(t-1:s+1)}M_{jc}^{(s-1:u+1)} M_{d c'}^{(u-1:u'+1)} A_{ij}^3\cdot \mathsf{F}_{u'+1;d'}  \mathsf{F}_{u+1;d}'\mathsf{F}_{s+1;j}'. 
	\end{align*}
	For the above term, a non-trivial bound (along the lines in between (\ref{ineq:der_cross_cubic_3})-(\ref{ineq:der_cross_cubic_4})) is needed for the configuration $(c,d)=(j,i)$ and $(c',d')=(i,j)$, whereas a trivial bound (\ref{ineq:der_cross_cubic_2_0}) suffices for the remaining configurations of $(c,d),(c',d')$. Consequently, on an event $\mathscr{E}_{2,1}\cap \mathscr{E}_{[t]}(L)$ with $\Prob(\mathscr{E}_{2,1}^c)\leq C_2\cdot c_2^t n^{-D}$, 
	\begin{align}\label{ineq:der_cross_cubic_7}
	\biggabs{\sum_{i,j \in [n]} S_{(i,j,k);2,I}}&\leq \big(C_2 tK L \log n\big)^{c_2t^2}\cdot n^{-1/2}.
	\end{align}
	The second term $S_{(i,j,k);2,II}$ can be handled as follows. For the configuration $(a,b)=(i,j)$, on an event $\mathscr{E}_{2,2}\cap \mathscr{E}_{[t]}(L)$ with $\Prob(\mathscr{E}_{2,2}^c)\leq C_2\cdot c_2^t n^{-D}$,
	\begin{align*}
	\biggabs{\sum_{\ell \neq j}M_{ki}^{(t-1:s+1)} \tilde{M}^{(s-1:u)}_{j\ell} (\partial_{ij} z_\ell^{(u)})^2} \leq \big(C_2 tK L \log n\big)^{c_2t^2}\cdot n^{- ( \abs{\{i,j,k\}}-1)/2}.
	\end{align*}
	By letting
	\begin{align*}
	\mathfrak{S}_{(i,j,k);2}\equiv \sum_{s \in [t-1], u \in [s-1]} M_{ki}^{(t-1:s+1)} \tilde{M}^{(s-1:u)}_{jj} (\partial_{ij} z_j^{(u)})^2\cdot A_{ij}^3\cdot \mathsf{F}_{s+1;j}',
	\end{align*}
	on an event $\mathscr{E}_{2,3}\cap \mathscr{E}_{[t]}(L)$ with $\Prob(\mathscr{E}_{2,3}^c)\leq C_2\cdot c_2^t n^{-D}$, uniformly in $i,j \in [n]$,
	\begin{align}\label{ineq:der_cross_cubic_8}
	\bigabs{S_{(i,j,k);2,II}-3\mathfrak{S}_{(i,j,k);2}-3\mathfrak{S}_{(j,i,k);2}}&\leq \big(C_2 tK L \log n\big)^{c_2t^2}\cdot n^{- \abs{\{i,j,k\}}/2-1}.
	\end{align}
	On the other hand, note that 
	\begin{align*}
	&\sum_{i,j \in [n]} \mathfrak{S}_{(i,j,k);2}=\sum_{ \substack{s \in [t-1], u \in [s-1],\\ u',u''\in [0:u-1] }} \sum_{j \in [n]} \tilde{M}^{(s-1:u)}_{jj}    \\
	&\qquad \times \bigg[\sum_{i \in [n]} \sum_{(c,d),(c',d')}M_{ki}^{(t-1:s+1)}  M_{jc}^{(u-1:u'+1)} M_{jc'}^{(u-1:u''+1)} \cdot A_{ij}^3\bigg]\cdot \mathsf{F}_{u'+1;d} \mathsf{F}_{u''+1;d'}  \mathsf{F}_{s+1;j}'.
	\end{align*}
	The term in the bracket above can be handled by a trivial bound expect for the case $c=c'=j$, in which case we may use the same argument below (\ref{ineq:der_cross_cubic_3}) and Proposition \ref{prop:A_cross_cubic} with $q_0=3$ to produce a similar bound as (\ref{ineq:der_cross_cubic_4}) for $\abs{\sum_{i,j \in [n]} \mathfrak{S}_{(i,j,k);2}}$ and $\abs{\sum_{i,j \in [n]} \mathfrak{S}_{(j,i,k);2}}$. Consequently, on an event $\mathscr{E}_{2,+}\cap \mathscr{E}_{[t]}(L)$ with $\Prob(\mathscr{E}_{2,+}^c)\leq C_2\cdot c_2^t n^{-D}$,
	\begin{align}\label{ineq:der_cross_cubic_8_0}
	\biggabs{\sum_{i,j \in [n]}\mathfrak{S}_{(i,j,k);2}}\vee \biggabs{\sum_{i,j \in [n]}\mathfrak{S}_{(j,i,k);2}}\leq \big(C_2 tK L \log n\big)^{c_2t^2}\cdot n^{-1/2}.
	\end{align}
    Combining (\ref{ineq:der_cross_cubic_6})-(\ref{ineq:der_cross_cubic_8_0}), on an event $\mathscr{E}_{2}\cap \mathscr{E}_{[t]}(L)$ with $\Prob(\mathscr{E}_{2}^c)\leq C_2\cdot c_2^t n^{-D}$, 
	\begin{align}\label{ineq:der_cross_cubic_9}
	\biggabs{\sum_{i,j \in [n]} S_{(i,j,k);2}}&\leq \big(C_2 tK L \log n\big)^{c_2t^2}\cdot n^{-1/2}.
	\end{align}

	\noindent (\textbf{Term $S_{(i,j,k);3}$}). Let
	\begin{align*}
	\mathfrak{S}_{(i,j,k);3}\equiv \sum_{s \in [t-1]}  \big(M^{(t-1:s+1)}M_3^{(s)}\big)_{ki}\big(\partial_{ij} z^{(s)}_{i}\big)^3\cdot A_{ij}^3.
	\end{align*}
	On the event $\mathscr{E}_0\cap \mathscr{E}_{[t]}(L)$, we have
	\begin{align*}
	&\bigabs{S_{(i,j,k);3}- \mathfrak{S}_{(i,j,k);3}-\mathfrak{S}_{(j,i,k);3}}\\
	 &\leq t K^3 n^{-3/2}\cdot  \max_{s \in [t-1]} \sum_{\ell \neq i,j} \bigabs{\big(M^{(t-1:s+1)}M_3^{(s)}\big)_{k\ell}} \bigabs{\partial_{ij} z^{(s)}_{\ell}}^3\leq \big(C_3 tK L \log n\big)^{c_3t^2}\cdot n^{-5/2}.
	\end{align*}
	On the other hand,
	\begin{align*}
	&\sum_{i,j \in [n]} \mathfrak{S}_{(i,j,k);3} =\sum_{j \in [n]}  \sum_{s \in [t-1]}  \sum_{u, u', u'' \in [0:s-1]}  \bigg[\sum_{i \in [n]}   \sum_{(c,d), (c',d'),(c'',d'')}  \\
	&\quad\quad \big(M^{(t-1:s+1)}M_3^{(s)}\big)_{ki}  M_{ic}^{(s-1:u+1)}  M_{ic'}^{(s-1:u'+1)}  M_{ic''}^{(s-1:u''+1)} A_{ij}^3\bigg]\cdot \mathsf{F}_{u+1;d}\mathsf{F}_{u'+1;d'}\mathsf{F}_{u''+1;d''}.
	\end{align*}
	The term in the above bracket can be handled using a trivial bound except for the case $c=c'=c''=i$, in which case may be handled via Proposition \ref{prop:A_cross_cubic} with $q_0=5$ to produce a desired high probability bound of order $\big(C tK L \log n\big)^{ct^2}\cdot n^{-1/2}$. Consequently, on an event $\mathscr{E}_{3}\cap \mathscr{E}_{[t]}(L)$ with $\Prob(\mathscr{E}_{3}^c)\leq C_3\cdot c_3^t n^{-D}$, 
	\begin{align}\label{ineq:der_cross_cubic_10}
	\biggabs{\sum_{i,j \in [n]} S_{(i,j,k);3}}&\leq \big(C_3 tK L \log n\big)^{c_3t^2}\cdot n^{-1/2}.
	\end{align}
	
	\noindent (\textbf{Term $S_{(i,j,k);4}$}). Note that
	\begin{align*}
	S_{(i,j,k);4}&=3 A_{ij}^3 \sum_{s \in [t-1]} \sum_{\ell \in [n]} \tilde{M}^{(t-1:s)}_{k\ell} \bigg(\sum_{(c,d)}\sum_{u \in [0:s-1]} M_{\ell c}^{(s-1:u+1)} \mathsf{F}_{u+1;d}\bigg)\\
	&\quad \times \bigg(2\sum_{(c',d')}\sum_{u' \in [s-1]} M_{\ell c'}^{(s-1:u'+1)}\mathsf{F}_{u'+1;d'}' \partial_{ij} z_{d'}^{(u')}+ \sum_{u' \in [s-1]}\sum_{\ell' \in [n]} \tilde{M}^{(s-1:u')}_{\ell \ell'} (\partial_{ij} z_{\ell'}^{(u')})^2 \bigg)\\
	& = 6 \sum_{\substack{(c,d),(c',d'), (c'',d'')}} \sum_{  \substack{s \in [t-1] } } \sum_{  \substack{ u \in [0:s-1],u' \in [s-1], \\ u'' \in [u'-1] } }   \bigg(\sum_{\ell \in [n]}  \tilde{M}^{(t-1:s)}_{k\ell}M_{\ell c}^{(s-1:u+1)} M_{\ell c'}^{(s-1:u'+1)}\bigg) \\
	&\qquad\qquad \qquad\qquad  \times M_{d',c''}^{(u'-1:u''+1)}A_{ij}^3\cdot \mathsf{F}_{u''+1;d''} \mathsf{F}_{u+1;d}\mathsf{F}_{u'+1;d'}' \\
	&\quad + 3 A_{ij}^3\sum_{(c,d)} \sum_{  \substack{s \in [t-1],\\ u \in [0:s-1],u' \in [s-1] } } \bigg(\sum_{\ell,\ell' \in [n]} \tilde{M}^{(t-1:s)}_{k\ell} M_{\ell c}^{(s-1:u+1)} \tilde{M}^{(s-1:u')}_{\ell \ell'} (\partial_{ij} z_{\ell'}^{(u')})^2\bigg)\cdot \mathsf{F}_{u+1;d}.
	\end{align*}
	We write the two terms on the right hand side of the above display as $S_{(i,j,k);4,I}$ and $S_{(i,j,k);4, II}$. 
	
	\noindent (\textbf{Case 4.I}).  Consider the first term $S_{(i,j,k);4,I}$. 
	
	\noindent (\emph{Case 4.I.1}). For configurations $(c,d),(c',d')$ such that $c\neq c'$ (so $d\neq d'$), which we assume to be $(c,c')=(i,j)$ for notational simplicity, its absolute value can be bounded via Proposition \ref{prop:A_hprod_quad}: on an event $\mathscr{E}_{4,I;1}\cap \mathscr{E}_{[t]}(L)$ with $\Prob(\mathscr{E}_{4,I;1}^c)\leq C_4\cdot c_4^t n^{-D}$, 
	\begin{align}\label{ineq:der_cross_cubic_11}
	&\biggabs{\sum_{  \substack{s \in [t-1]; u,u' \in [0:s-1]} } \bigg(\sum_{\ell \in [n]}  \tilde{M}^{(t-1:s)}_{k\ell}M_{\ell i}^{(s-1:u+1)} M_{\ell j}^{(s-1:u'+1)}\bigg) \partial_{ij} z_{i}^{(u')}A_{ij}^3\cdot  \mathsf{F}_{u+1;j}\mathsf{F}_{u'+1;i}' }\nonumber\\
	&\leq \big(C_4 tK L \log n\big)^{c_4 t^2}\cdot n^{- \abs{\{i,j,k\}}/2-1}.
	\end{align}
	\noindent (\emph{Case 4.I.2}). For configurations $(c,d),(c',d')$ such that $c= c'$ (so $d= d'$), which we assume to be $(c,c')=(i,i)$ for notational simplicity. 
	
	\noindent (\emph{Case 4.I.2-(a)}). For $(c'',d'')=(i,j)$, on an event $\mathscr{E}_{4,I;2}\cap \mathscr{E}_{[t]}(L)$ with $\Prob(\mathscr{E}_{4,I;2}^c)\leq C_4\cdot c_4^t n^{-D}$, 
	\begin{align}\label{ineq:der_cross_cubic_12}
	&\biggabs{\sum_{  \substack{s \in [t-1], u \in [0:s-1],\\u' \in [s-1], u'' \in [u'-1] } } \bigg(\sum_{\ell \in [n]}  \tilde{M}^{(t-1:s)}_{k\ell}M_{\ell i}^{(s-1:u+1)} M_{\ell i}^{(s-1:u'+1)}\bigg) \bigg( M_{j,i}^{(u'-1:u''+1)}A_{ij}^3\cdot \mathsf{F}_{u''+1;j} \mathsf{F}_{u+1;j}\mathsf{F}_{u'+1;j}'\bigg) }\nonumber\\
	&\leq \big(C tK L \log n\big)^{c t^2}\cdot n^{- (\abs{\{i,k\}}-1)/2}\cdot n^{- (\abs{\{i,j\}}-1)/2 }\cdot n^{-3/2}\nonumber\\
	&\leq \big(C_4 tK L \log n\big)^{c_4 t^2}\cdot n^{- \abs{\{i,j,k\}}/2-1}.
	\end{align}
	\noindent (\emph{Case 4.I.2-(b)}). For $(c'',d'')=(j,i)$, summing over $i,j \in [n]$, it follows by Proposition \ref{prop:A_cross_cubic} that on an event $\mathscr{E}_{4,I;3}\cap \mathscr{E}_{[t]}(L)$ with $\Prob(\mathscr{E}_{4,I;3}^c)\leq C_4\cdot c_4^t n^{-D}$, 
	\begin{align}\label{ineq:der_cross_cubic_13}
	&\biggabs{\sum_{j \in [n]} \sum_{  \substack{s \in [t-1], u \in [0:s-1],\\ u' \in [s-1]; u'' \in [u'-1] } } \bigg\{\sum_{i \in [n]}\bigg(\sum_{\ell \in [n]}  \tilde{M}^{(t-1:s)}_{k\ell}M_{\ell i}^{(s-1:u+1)} M_{\ell i}^{(s-1:u'+1)}\bigg) A_{ij}^3 \mathsf{F}_{u''+1;i}\bigg\}\nonumber\\
		&\qquad \times  M_{j,j}^{(u'-1:u''+1)}  \mathsf{F}_{u+1;j}\mathsf{F}_{u'+1;j}'\bigg) }\leq \big(C_4 tK L \log n\big)^{c_4 t^2}\cdot n^{-1/2}. 
	\end{align}
	Combining (\ref{ineq:der_cross_cubic_11})-(\ref{ineq:der_cross_cubic_13}), 
	on an event $\mathscr{E}_{4,I}\cap \mathscr{E}_{[t]}(L)$ with $\Prob(\mathscr{E}_{4,I}^c)\leq C_4\cdot c_4^t n^{-D}$, 
	\begin{align}\label{ineq:der_cross_cubic_14}
	\biggabs{\sum_{i,j \in [n]} S_{(i,j,k);4,I}}&\leq \big(C_4 tK L \log n\big)^{c_4t^2}\cdot n^{-1/2}.
	\end{align}
	\noindent (\textbf{Case 4.II}).  Consider the second term $S_{(i,j,k);4,II}$. Consider the case $(c,d)=(i,j)$. Then on an event $\mathscr{E}_{4,II;1}\cap \mathscr{E}_{[t]}(L)$ with $\Prob(\mathscr{E}_{4,II;1}^c)\leq C_4\cdot c_4^t n^{-D}$, 
	\begin{align}\label{ineq:der_cross_cubic_15}
	&\biggabs{\sum_{i,j \in [n]} A_{ij}^3 \sum_{\ell \in [n],\ell' \neq i,j } \tilde{M}^{(t-1:s)}_{k\ell} M_{\ell i}^{(s-1:u+1)} \tilde{M}^{(s-1:u')}_{\ell \ell'} (\partial_{ij} z_{\ell'}^{(u')})^2}\\
	&\leq \big(C tK L \log n\big)^{c t^2}\cdot n^{-5/2} \sum_{i,j \in [n]} \sum_{\ell' \neq i,j} n^{- (\abs{\{i,k,\ell'\}}-1)/2}\stackrel{(\ast)}{\leq} \big(C_4 tK L \log n\big)^{c_4 t^2}\cdot n^{-1/2}.\nonumber
	\end{align}
	where the last inequality $(\ast)$ follows as 
	\begin{align*}
	&\sum_{i,j \in [n]} \sum_{\ell' \neq i,j} n^{- (\abs{\{i,k,\ell'\}}-1)/2}=\sum_{\ell'} \sum_{i,j\neq \ell'} n^{- (\abs{\{i,k,\ell'\}}-1)/2}\leq n \sum_{\ell'} \sum_{i\neq \ell'} n^{- (\abs{\{i,k,\ell'\}}-1)/2}\\
	& =  n \sum_{i\neq k} n^{- (\abs{\{i,k\}}-1)/2}+ n \sum_{\ell'\neq k} \bigg(\sum_{i\neq \ell',k} n^{- (\abs{\{i,k,\ell'\}}-1)/2}+n^{- (\abs{\{k,\ell'\}}-1)/2}\bigg) \leq 3n^2.
	\end{align*}
	On the other hand, using Proposition \ref{prop:A_hprod_quad}, on an event $\mathscr{E}_{4,II;2}\cap \mathscr{E}_{[t]}(L)$ with $\Prob(\mathscr{E}_{4,II;2}^c)\leq C_4\cdot c_4^t n^{-D}$, 
	\begin{align}\label{ineq:der_cross_cubic_16}
	&\biggabs{\sum_{i,j \in [n]} A_{ij}^3 \sum_{\ell \in [n] } \tilde{M}^{(t-1:s)}_{k\ell} M_{\ell i}^{(s-1:u+1)} \tilde{M}^{(s-1:u')}_{\ell j} (\partial_{ij} z_{j}^{(u')})^2}\nonumber\\
	&\leq \big(C tK L \log n\big)^{c t^2}\cdot n^{-3/2} \sum_{i,j} n^{- (\abs{\{i,j,k\}}-1)/2}\leq \big(C_4 tK L \log n\big)^{c_4 t^2}\cdot n^{-1/2}.
	\end{align}
	Combining (\ref{ineq:der_cross_cubic_15})-(\ref{ineq:der_cross_cubic_16}), by letting
	\begin{align*}
	\mathfrak{S}_{(i,j,k);4}\equiv A_{ij}^3 \sum_{  \substack{s \in [t-1],\\ u \in [0:s-1],u' \in [s-1] } } \bigg(\sum_{\ell \in [n]} \tilde{M}^{(t-1:s)}_{k\ell} M_{\ell i}^{(s-1:u+1)} \tilde{M}^{(s-1:u')}_{\ell i} (\partial_{ij} z_{i}^{(u')})^2\bigg)\cdot \mathsf{F}_{u+1;j},
	\end{align*}
	on the event $\mathscr{E}_{4,II;1}\cap \mathscr{E}_{4,II;2}\cap \mathscr{E}_{[t]}(L)$,
	\begin{align*}
	\biggabs{\sum_{i,j \in [n]} \Big(S_{(i,j,k);4,II}- 3 \mathfrak{S}_{(i,j,k);4}-3\mathfrak{S}_{(j,i,k);4}\Big) }&\leq \big(C_4 tK L \log n\big)^{c_4t^2}\cdot n^{-1/2}.
	\end{align*}
	On the other hand,
	\begin{align*}
	\sum_{i,j \in [n]} \mathfrak{S}_{(i,j,k);4}&= \sum_{j \in [n]} \sum_{  \substack{s \in [t-1],\\ u \in [0:s-1],u' \in [s-1],\\ u_1,u_2 \in [0:u'-1] } } \bigg[\sum_{i \in [n]}\sum_{(c_1,d_1),(c_2,d_2)} \bigg(\sum_{\ell \in [n]} \tilde{M}^{(t-1:s)}_{k\ell} M_{\ell i}^{(s-1:u+1)} \tilde{M}^{(s-1:u')}_{\ell i}\bigg)\\
	&\qquad\qquad \times  M_{ic_1}^{(u'-1:u_1+1)}  M_{ic_2}^{(u'-1:u_2+1)} A_{ij}^3\bigg]\cdot \mathsf{F}_{u+1;j}\mathsf{F}_{u_1+1;d_1}\mathsf{F}_{u_2+1;d_2}.
	\end{align*}
	The term in the above bracket can be handled using a trivial bound except for the case $c_1=c_2=i$, in which case may be handled via Proposition \ref{prop:A_cross_cubic} with $q_0=4$ to produce a desired high probability bound of order $\big(C tK L \log n\big)^{ct^2}\cdot n^{-1/2}$. Consequently, combining the above two displays, on an event $\mathscr{E}_{4,II}\cap \mathscr{E}_{[t]}(L)$ with $\Prob(\mathscr{E}_{4,II}^c)\leq C_4\cdot c_4^t n^{-D}$, 
	\begin{align}\label{ineq:der_cross_cubic_17}
	\biggabs{\sum_{i,j \in [n]} S_{(i,j,k);4,II}}&\leq \big(C_4 tK L \log n\big)^{c_4t^2}\cdot n^{-1/2}.
	\end{align}
	In view of (\ref{ineq:der_cross_cubic_14}) and (\ref{ineq:der_cross_cubic_17}), on an event $\mathscr{E}_{4}\cap \mathscr{E}_{[t]}(L)$ with $\Prob(\mathscr{E}_{4}^c)\leq C_4\cdot c_4^t n^{-D}$, 
	\begin{align}\label{ineq:der_cross_cubic_18}
	\biggabs{\sum_{i,j \in [n]} S_{(i,j,k);4}}&\leq \big(C_4 tK L \log n\big)^{c_4t^2}\cdot n^{-1/2}.
	\end{align}
	The claimed large deviation for $\sum_{i,j\in [n]} A_{ij}^3 \partial_{ij}^3 z^{(t)}_k$ now follows from (\ref{ineq:der_cross_cubic_1}), (\ref{ineq:der_cross_cubic_5}), (\ref{ineq:der_cross_cubic_9}), (\ref{ineq:der_cross_cubic_10}) and (\ref{ineq:der_cross_cubic_18}), and the choice of $L$ in (\ref{ineq:z_der_L}).\qed

\section{Remaining proofs for Section \ref{section:main_results_sym}}

\subsection{Proof of Theorem \ref{thm:universality_avg}}

We shall focus on the case $\psi_k=\psi:\R\to \R$ and the simplified iterate (\ref{def:GFOM_sym_proof}) (and therefore the simplified leave-k-out version of (\ref{def:GFOM_sym_loo})) to avoid unnecessary, cumbersome notational complications. To this end, we first prove a weaker version of Theorem \ref{thm:universality_avg}, assuming strong regularity conditions on both $\{\mathsf{F}_\cdot,\mathsf{G}_\cdot\}$, and the test function $\psi$.

\begin{proposition}\label{prop:universality_avg_smooth}
	Assume the same conditions as in Theorem \ref{thm:universality}. Take any $\psi \in C^3(\R)$ satisfying $\max_{a \in \mathbb{Z}_{\geq 0}^{\abs{S}}, \abs{a}\in [0:3]}\pnorm{ \partial^a\psi(\cdot) }{\infty}\leq \Lambda_{\psi}$ for some $\Lambda_{\psi}\geq 2$. Then there exists some constant $c_1=c_1(\mathfrak{p},q)>0$ such that for any even $q\in \N$,  
	\begin{align*}
	&\E \bigg[\biggabs{\frac{1}{n}\sum_{k \in [n]} \psi\big(z_k^{(t)}(A)\big) - \frac{1}{n}\sum_{k \in [n]}  \psi\big(z_k^{(t)}(B)\big)  }^q\Big|z^{(0)}\bigg]\\
	&\leq \Lambda_\psi^{q}\cdot  \big( K\Lambda \log n\cdot (1+\pnorm{z^{(0)}}{\infty})\big)^{c_1 t^3} \cdot n^{-1/2}.
	\end{align*}
\end{proposition}

The following concentration result is useful for the proof of Proposition \ref{prop:universality_avg_smooth}.

\begin{lemma}\label{lem:conc_avg}
	Fix $t \in \N$. Suppose that for all $s \in [t]$ and $\ell \in [n]$,  $\pnorm{\mathsf{F}_{s,\ell}}{\mathrm{Lip}}\vee \pnorm{\mathsf{G}_{s,\ell}}{\mathrm{Lip}}\leq \Lambda$ for some $\Lambda\geq 2$. Let $V_n$ be a symmetric $n\times n$ matrix whose entries are all bounded by $(K/n)^{1/2}$ for some $K\geq 2$. Then for any $\psi:\R \to \R$ with $\pnorm{\psi}{\infty}\vee \pnorm{\psi'}{\infty}\leq \Lambda_\psi$, and any $p\geq 1$,
	\begin{align*}
	\E^{1/p} \bigg[\biggabs{\frac{1}{n}\sum_{k \in [n]} (\mathrm{id}-\E )\psi\big(z_k^{(t)}(G\circ V_n)\big) }^p\Big|z^{(0)}\bigg]\leq \sqrt{p} \Lambda_{\psi }(K\Lambda)^{c_0 t}(1+\pnorm{z^{(0)}}{\infty})\cdot n^{-1/2}.
	\end{align*}
	Here $G$ denotes a symmetric matrix with i.i.d. $\mathcal{N}(0,1)$ entries on its upper triangle, and $c_0>0$ is a universal constant.
\end{lemma}
\begin{proof}
	First, using the simple recursion $\pnorm{z^{(t)}(A)}{}\leq (\pnorm{A}{\op}+1)\Lambda \pnorm{z^{(t-1)}(A)}{}$, we have $\pnorm{z^{(t)}(A)}{}\leq (\pnorm{A}{\op}+1)^t\Lambda^t \pnorm{z^{(0)}}{}$. Consequently, with $\overline{\psi}(z^{(t)}(A))\equiv n^{-1}\sum_{k \in [n]} \psi\big(z_k^{(t)}(A)\big) $, we have
	\begin{align}\label{ineq:conc_avg_1}
	\abs{\overline{\psi}(z^{(t)}(A))}\leq \Lambda_\psi\cdot  \big(1+n^{-1/2} \pnorm{z^{(t)}(A)}{}\big)\leq L_t\cdot  (\pnorm{A}{\op}+1)^t,
	\end{align}
	where $L_t\equiv \Lambda_{\psi}(2\Lambda)^t (1+\pnorm{z^{(0)}}{\infty})$.
	
	Next, for two symmetric matrices $A_1,A_2 \in M_n(\R)$, we have 
	\begin{align*}
	&\pnorm{z^{(t)}(A_1)-z^{(t)}(A_2)}{}\\
	&\leq \bigpnorm{A_1 \mathsf{F}_t \big(z^{(t-1)} (A_1)-A_2 \mathsf{F}_t \big(z^{(t-1)} (A_2)\big) }{}+\bigpnorm{ \mathsf{G}_t \big(z^{(t-1)} (A_1)- \mathsf{G}_t \big(z^{(t-1)} (A_2)  }{}\\
	&\leq \big(\pnorm{A_1}{\op}+1\big) \cdot \Lambda \pnorm{z^{(t-1)}(A_1)-z^{(t-1)}(A_2)}{}+ \pnorm{A_1-A_2}{\op}\cdot \Lambda \pnorm{z^{(t-1)}(A_2) }{}\\
	&\leq \big(\pnorm{A_1}{\op}+1\big) \cdot \Lambda \pnorm{z^{(t-1)}(A_1)-z^{(t-1)}(A_2)}{}+ (\pnorm{A_2}{\op}+1)^{t-1}\Lambda^t \pnorm{z^{(0)}}{} \cdot  \pnorm{A_1-A_2}{\op}.
	\end{align*}
	Iterating the bound, we obtain 
	\begin{align*}
	\pnorm{z^{(t)}(A_1)-z^{(t)}(A_2)}{}\leq (\pnorm{A_1}{\op}+\pnorm{A_2}{\op}+1)^{t}(2\Lambda)^t \pnorm{z^{(0)}}{} \cdot  \pnorm{A_1-A_2}{\op}.
	\end{align*}
	Consequently, 
	\begin{align}\label{ineq:conc_avg_2}
	\bigabs{\overline{\psi}(z^{(t)}(A_1))-\overline{\psi}(z^{(t)}(A_2))}&\leq \Lambda_{\psi}\cdot n^{-1/2} \pnorm{z^{(t)}(A_1)-z^{(t)}(A_2)}{}\nonumber\\
	&\leq  L_t (\pnorm{A_1}{\op}+\pnorm{A_2}{\op}+1)^{t}\cdot \pnorm{A_1-A_2}{\op}.
	\end{align}
	Now (\ref{ineq:conc_avg_1})-(\ref{ineq:conc_avg_2}) verify the conditions of \cite[Lemma A.2]{bao2025leave}, so an application of that result and noting that $\pnorm{\psi}{\infty}\leq \Lambda_\psi$ yield the concentration estimate: for $t\leq n/(c_0 \log n)$ and $x\geq c_0 L_t K^{c_0 t} e^{-n/c_0}$, 
		\begin{align*}
		&\Prob\Big(\abs{\overline{\psi}(z^{(t)}(G\circ V_n))-\E \overline{\psi}(z^{(t)}(G\circ V_n))}\geq x\big|z^{(0)}\Big)\leq c_0 e^{-n x^2/L_t^2 K^{c_0 t} }.
		\end{align*}
		The range constraint on $x$ may be dropped by suitably enlarging the constant $c_0>0$, and therefore we conclude the $p$-th moment estimate.
\end{proof}

\begin{proof}[Proof of Proposition \ref{prop:universality_avg_smooth}]
	We write $\E \equiv \E[\cdot|z^{(0)}]$ for notational simplicity.  Then 
	\begin{align}\label{ineq:universality_avg_1}
	\mathscr{D}_{\psi}(A,B)&\equiv \E \biggabs{\frac{1}{n}\sum_{k \in [n]} \psi\big(z_k^{(t)}(A)\big) - \frac{1}{n}\sum_{k \in [n]}  \E \psi\big(z_k^{(t)}(B)\big)  }^q \nonumber\\
	&= n^{-q} \sum_{k_1,\ldots,k_q \in [n]} \E \prod_{\ell \in [q]} \Big(\psi(z_{k_\ell}^{(t)}(A))-\E \psi (z_{k_\ell}^{(t)}(B)) \Big)\nonumber\\
	& = n^{-q} \sum_{k_1,\ldots,k_q \in [n]}  \sum_{S\subset [q]} \bigg(\E \prod_{\ell \in S} \psi(z_{k_\ell}^{(t)}(A))\bigg)\cdot \prod_{\ell \in [q]\setminus S} \Big( -\E \psi (z_{k_\ell}^{(t)}(B))  \Big).
	\end{align} 
	Fix $S\subset [q]$. Let $\Psi(z_{S})\equiv \prod_{\ell \in S} \psi(z_\ell)$. Then $\Psi$ satisfies the condition of Theorem \ref{thm:universality} with constant $\Lambda_{\Psi}^{\abs{S}}$, and therefore
	\begin{align}\label{ineq:universality_avg_2}
	\biggabs{\E \prod_{\ell \in S} \psi(z_{k_\ell}^{(t)}(A)) - \E \prod_{\ell \in S} \psi(z_{k_\ell}^{(t)}(B))}\leq \err_S,
	\end{align}
	where $\err_S\equiv \abs{S}^3 \Lambda_\psi^{\abs{S}} \big( K\Lambda \log n\cdot (1+\pnorm{z^{(0)}}{\infty})\big)^{c_1 t^3} \cdot n^{-1/2}$. Combining (\ref{ineq:universality_avg_1}) and (\ref{ineq:universality_avg_2}), and using that $\sum_{S \subset [q]}\abs{S}^3\leq q^3 2^q\leq 6^q$, we have 
	\begin{align}\label{ineq:universality_avg_3}
	&\bigabs{\mathscr{D}_{\psi}(A,B)-\mathscr{D}_{\psi}(B,B)}\leq  \sum\nolimits_{S\subset [q]} \err_S\cdot \Lambda_{\psi}^{q-\abs{S}}\nonumber\\
	&\leq  (6 \Lambda_\psi)^q\cdot  \big( K\Lambda \log n\cdot (1+\pnorm{z^{(0)}}{\infty})\big)^{c_1 t^3} \cdot n^{-1/2}.
	\end{align}
	Now we choose $B\equiv G\circ V_n$, where $V_n\equiv \{\E (A\circ A)/n\}^{1/2}$ and $G$ denotes a symmetric matrix with i.i.d. $\mathcal{N}(0,1)$ entries on its upper triangle. Lemma \ref{lem:conc_avg} yields that
	\begin{align}\label{ineq:universality_avg_4}
	\mathscr{D}_{\psi}\big(G\circ V_n,G\circ V_n\big)\leq q^{q/2} \Lambda_{\psi}^{q} (K\Lambda)^{c_1t}(1+\pnorm{z^{(0)}}{\infty})^q\cdot n^{-q/2}. 
	\end{align}
	Combining (\ref{ineq:universality_avg_3})-(\ref{ineq:universality_avg_4}), we have
	\begin{align*}
	&\bigabs{\mathscr{D}_{\psi}\big(A,G\circ V_n\big)-\mathscr{D}_{\psi}\big(G\circ V_n,G\circ V_n\big)}\\
	&\leq \Lambda_\psi^{q}\cdot  \big(K\Lambda \log n\cdot (1+\pnorm{z^{(0)}}{\infty})\big)^{c_2 t^3} \cdot n^{-1/2}.
	\end{align*}
	The claim now follows by noting that 
	\begin{align*}
	&\E \biggabs{\frac{1}{n}\sum_{k \in [n]} \psi\big(z_k^{(t)}(A)\big) - \frac{1}{n}\sum_{k \in [n]}  \psi\big(z_k^{(t)}(B)\big)  }^q\\
	&\leq 2^{q+1} \max_{M \in \{A,B\}}\bigabs{\mathscr{D}_{\psi}\big(M,G\circ V_n\big)-\mathscr{D}_{\psi}\big(G\circ V_n,G\circ V_n\big)},
	\end{align*}
	completing the proof. 
\end{proof}

To weaken the regularity conditions in Proposition \ref{prop:universality_avg_smooth}, we will need the following lemma that quantifies the smoothing effect for $\{\mathsf{F}_\cdot,\mathsf{G}_\cdot\}$. 

\begin{lemma}\label{lem:smooth_F_G}
	Fix $t \in \N$. Suppose the following hold:
	\begin{enumerate}
		\item $A=A_0/\sqrt{n}$ where, (i) $A_0$ are symmetric $n\times n$ matrices, (ii) the entries of its upper triangle are independent mean $0$ random variables, and (iii) for all $i,j \in [n]$, $\pnorm{A_{0,ij}}{\psi_2}\leq K$ for some $K\geq 2$.
		\item $
		\max_{s \in [t]}\max_{\mathsf{E}_s \in \{\mathsf{F}_s,\mathsf{G}_s\}}\max_{\ell \in [n]}\Big\{\pnorm{\mathsf{E}_{s,\ell}}{\mathrm{Lip}}+\abs{\mathsf{E}_{s,\ell}(0)}\Big\}\leq \Lambda$ for some $\Lambda \geq 2$. 
	\end{enumerate}
	Then for any $\sigma \in (0,1)$, there exist $\{\mathsf{F}_{\sigma;t},\mathsf{G}_{\sigma;t}\}$ such that 
	\begin{align*}
	\max_{s \in [t]}\max_{\mathsf{E} \in \{\mathsf{F},\mathsf{G}\}}\max_{\ell \in [n]}\Big\{\pnorm{\mathsf{E}_{\sigma;s,\ell}}{\mathrm{Lip}}+\max_{q \in [0:3]} \bigpnorm{(1+\abs{\cdot})^{-1}\abs{\mathsf{E}_{\sigma; s,\ell}^{(q)}(\cdot)}}{\infty}\Big\}\leq c_0 \Lambda \sigma^{-3},
	\end{align*}
	and with $\{z_\sigma^{(t)}\}$ being the iterations associated with $A$, non-linearities $\{\mathsf{F}_{\sigma;t},\mathsf{G}_{\sigma;t}\}$ and the same initialization $z^{(0)}$, for any $\Lambda_\psi$-Lipschitz $\psi:\R\to \R$,
	\begin{align*}
	\biggabs{ \frac{1}{n}\sum_{k \in [n]} \psi\big(z_{k}^{(t)}(A)\big)-\frac{1}{n}\sum_{k \in [n]} \psi\big(z_{\sigma,k}^{(t)}(A)\big) }\leq \Lambda_{\psi} (c_0\Lambda)^t (\pnorm{A}{\op}+1)^t \cdot \sigma.
	\end{align*}
	Here $c_0>0$ is a universal constant. 
\end{lemma}
\begin{proof}
	Let $\varphi \in C^\infty(\R)$ be a mollifier supported in $[-1,1]$ such that $\varphi\geq 0$ and $\int \varphi =1$, and $\varphi_\sigma(\cdot)\equiv \sigma^{-1}\varphi(\cdot/\sigma)$. Let $\mathsf{F}_{\sigma;s,\ell}\equiv \mathsf{F}_{s,\ell}*\varphi_\sigma$, and similarly we may define $\mathsf{G}_{\sigma;s,\ell}$. We have the following estimates:
	\begin{itemize}
		\item For $q \in \N$, $\sigma \in (0,1]$ and $\mathsf{E}_{\sigma;s,\ell} \in \{\mathsf{F}_{\sigma;s,\ell},\mathsf{G}_{\sigma;s,\ell}\}$,
		\begin{align*}
		\bigabs{\mathsf{E}_{\sigma;s,\ell}^{(q)}(x)}= \frac{1}{\sigma^q} \biggabs{\int_{-1}^1 \mathsf{E}_{s,\ell}(x-\sigma z) \varphi^{(q)}(z)\,\d{z}}\leq 2\Lambda \pnorm{\varphi^{(q)}}{\infty}  \sigma^{-q}\cdot \big(1+\abs{x}\big).
		\end{align*}
		\item For $q=1$, we may write alternatively $
		\abs{\mathsf{E}_{\sigma;s,\ell}'(x)} = \bigabs{  \int \mathsf{E}_{s,\ell}'(x-z) \varphi_\sigma(z)\,\d{z}  }\leq \Lambda$. 
		\item $
		\abs{\mathsf{E}_{\sigma;s,\ell}(x)-\mathsf{E}_{s,\ell}(x)}=\bigabs{\int_{-1}^{1} \big(\mathsf{E}_{s,\ell}(x-\sigma z)- \mathsf{E}_{s,\ell}(x)\big)\varphi(z)\,\d{z}  }\leq \Lambda \pnorm{\varphi}{\infty} \sigma$. 
	\end{itemize}
	Note that
	\begin{align*}
	&\pnorm{z_\sigma^{(t)}-z^{(t)}}{}\leq \pnorm{A\mathsf{F}_{\sigma;t}(z_\sigma^{(t-1)})- A \mathsf{F}_t(z^{(t-1)})}{}+ \pnorm{\mathsf{G}_{\sigma;t}(z_\sigma^{(t-1)})- \mathsf{G}_t(z^{(t-1)})}{}\\
	&\leq (\pnorm{A}{\op}+1)\max_{\mathsf{E}\in \{\mathsf{F},\mathsf{G}\}} \Big(\pnorm{\mathsf{E}_{\sigma;t}(z_\sigma^{(t-1)})- \mathsf{E}_{\sigma;t}(z^{(t-1)})}{}+ \pnorm{\mathsf{E}_{\sigma;t}(z^{(t-1)})- \mathsf{E}_t(z^{(t-1)})}{}\Big)\\
	&\leq (\pnorm{A}{\op}+1) \Lambda\cdot \big( \pnorm{z_\sigma^{(t-1)}-z^{(t-1)}}{}+ \sqrt{n}\pnorm{\varphi}{\infty} \sigma \big). 
	\end{align*}
	Iterating the bound, we obtain
	\begin{align*}
	\pnorm{z_\sigma^{(t)}-z^{(t)}}{}\leq \sum_{s \in[t]} (\pnorm{A}{\op}+1)^s \Lambda^s\cdot \sqrt{n}\sigma \leq  (2\Lambda)^t (\pnorm{A}{\op}+1)^t \pnorm{\varphi}{\infty}\cdot \sqrt{n}\sigma.  
	\end{align*}
	Now with $\overline{\psi}(z^{(t)}(A))\equiv n^{-1}\sum_{k \in [n]} \psi\big(z_{k}^{(t)}(A)\big) $ and similarly defined $\overline{\psi}(z^{(t)}_\sigma(A))$, 
	\begin{align*}
	\bigabs{ \overline{\psi}(z^{(t)}(A))-\overline{\psi}(z^{(t)}_\sigma(A))}&\leq \Lambda_\psi n^{-1/2} \pnorm{z_\sigma^{(t)}-z^{(t)}}{} \leq \Lambda_{\psi} (2\Lambda)^t (\pnorm{A}{\op}+1)^t \pnorm{\varphi}{\infty}\cdot \sigma,
	\end{align*}
	as desired. 
\end{proof}

Now we may strengthen Proposition \ref{prop:universality_avg_smooth} to the desired Theorem \ref{thm:universality_avg}.

\begin{proof}[Proof of Theorem \ref{thm:universality_avg}]
	Without loss of generality, we assume $\psi(0)=0$. Let $\varphi \in C^\infty(\R)$ be a mollifier supported in $[-1,1]$ such that $\varphi\geq 0$ and $\int \varphi =1$, and $\varphi_\sigma(\cdot)\equiv \sigma^{-1}\varphi(\cdot/\sigma)$. Let $\psi_\sigma^{[M]}\equiv \psi^{[M]}*\varphi_\sigma$, where $\psi^{[M]}\equiv (\psi \wedge M)\vee (-M)$ for some $M>1$ to be chosen later.  Then the following hold:
	\begin{enumerate}
		\item $\abs{\psi_\sigma^{[M]}(x)-\psi(x)}=\bigabs{\int_{-1}^1 \big(\psi^{[M]}(x-\sigma z)-\psi^{[M]}(x)\big)\varphi(z)\,\d{z}}\leq\Lambda_\psi  (c_0M)^{\mathfrak{p}} \cdot \sigma$.
		\item $\abs{\big(\psi_\sigma^{[M]}\big)^{(1)}(x)}= \bigabs{\int_{-1}^1 (\psi^{[M]})^{(1)}(x)\varphi_\sigma(z)\,\d{z}}\leq \Lambda_\psi  (c_0M)^{\mathfrak{p}}$.
		\item For $q =2,3$, $\abs{\big(\psi_\sigma^{[M]}\big)^{(q)}(x)}=\sigma^{-q} \bigabs{\int_{-1}^1 \psi^{[M]}(x-\sigma z)\varphi^{(q)}(z)\,\d{z}}\leq \Lambda_\psi  (c_0M)^{\mathfrak{p}}\sigma^{-q}$.		
	\end{enumerate}
	In particular, (1) above means
	\begin{align*}
	\biggabs{\frac{1}{n}\sum_{k \in [n]} \psi^{[M]}\big(z_{k}^{(t)}(A)\big)-\frac{1}{n}\sum_{k \in [n]} \psi_\sigma^{[M]}\big(z_{k}^{(t)}(A)\big)}\leq \Lambda_\psi  (c_0M)^{\mathfrak{p}} \cdot \sigma. 
	\end{align*}
	Applying Lemma \ref{lem:smooth_F_G} with (2) and using the above display, for some constant $C_0=C_0(\mathfrak{p},q)>0$, 
	\begin{align}\label{ineq:universality_avg_general_1}
	&\E \biggabs{\frac{1}{n}\sum_{k \in [n]} \psi^{[M]}\big(z_{k}^{(t)}(A)\big)-\frac{1}{n}\sum_{k \in [n]} \psi_\sigma^{[M]}\big(z_{\sigma, k}^{(t)}(A)\big)}^q\nonumber\\
	&\leq \Lambda_\psi^q (c_0 M)^{\mathfrak{p}q}  \sigma^q \cdot  \big( 1+ (c\Lambda)^t \E (\pnorm{A}{\op}+1)^t \big)^q\leq M^{C_0}(K\Lambda )^{C_0 t}\cdot \Lambda_\psi^q  \sigma^q.
	\end{align}
	Next, applying Proposition \ref{prop:universality_avg_smooth} with (2)-(3), for some constant $C_1=C_1(\mathfrak{p},q)>0$, 
	\begin{align}\label{ineq:universality_avg_general_2}
	&\E \biggabs{  \frac{1}{n}\sum_{k \in [n]} \psi_\sigma^{[M]}\big(z_{\sigma, k}^{(t)}(A)\big) -\frac{1}{n}\sum_{k \in [n]} \psi_\sigma^{[M]}\big(z_{\sigma, k}^{(t)}(B)\big) }^q\nonumber\\
	&\leq  M^{C_1}\cdot \big(K \Lambda \Lambda_\psi \sigma^{-1} \log n\cdot (1+\pnorm{z^{(0)}}{\infty}) \big)^{C_1 t^3}\cdot n^{-1/2}. 
	\end{align}
	Combining (\ref{ineq:universality_avg_general_1})-(\ref{ineq:universality_avg_general_2}), for some constant $C_2=C_2(\mathfrak{p},q)>0$, 
	\begin{align}\label{ineq:universality_avg_general_3}
	&\E \biggabs{\frac{1}{n}\sum_{k \in [n]} \psi^{[M]}\big(z_{k}^{(t)}(A)\big)-\frac{1}{n}\sum_{k \in [n]} \psi^{[M]}\big(z_{k}^{(t)}(B)\big)}^q\nonumber\\
	& \leq M^{C_2}\cdot \big(K\Lambda\Lambda_\psi \log n\cdot (1+\pnorm{z^{(0)}}{\infty}) \big)^{C_2 t^3} \cdot \big(\sigma^{-C_2 t^3} n^{-1/2}+\sigma^q\big).
	\end{align}
	Fix a large enough $D>0$. Let us choose
	\begin{align*}
	M\equiv  \big(K\Lambda\log n\cdot (1+\pnorm{z^{(0)}}{\infty})\big)^{C_3 t}
	\end{align*}
	for some large enough $C_3=C_3(\mathfrak{p}, q, D)>0$. Let $\mathscr{E}_M\equiv \{\psi(z_k^{(t)})=\psi^{[M]}(z_k^{(t)}), k\in [n]\}$. Proposition \ref{prop:loo_l2_bound} then yields that $\Prob(\mathscr{E}_M^c)\leq C_3 n^{-2D}$. Consequently, 
	\begin{align}\label{ineq:universality_avg_general_4}
	&\E \biggabs{\frac{1}{n}\sum_{k \in [n]} \psi^{[M]}\big(z_{k}^{(t)}(A)\big)-\frac{1}{n}\sum_{k \in [n]} \psi\big(z_{k}^{(t)}(A)\big)}^q\nonumber\\
	& \leq 2^q \big(M^q \Prob(\mathscr{E}_M^c)+ \Lambda_\psi^q \E  (1+\pnorm{z^{(t)}}{\infty})^{\mathfrak{p}q}\bm{1}_{\mathscr{E}_M^c}\big)\nonumber\\
	&\leq \big(K\Lambda\Lambda_\psi\log n\cdot (1+\pnorm{z^{(0)}}{\infty})\big)^{C_4 t}\cdot n^{-D}. 
	\end{align}
	The claim follows by combining the estimates in (\ref{ineq:universality_avg_general_3}) and (\ref{ineq:universality_avg_general_4}), and optimizing $\sigma \in (0,1)$. 
\end{proof}

\subsection{Proof of Theorem \ref{thm:GFOM_se_sym}}\label{section:proof_GFOM_se_sym}

Consider the following symmetric version of the AMP iterate:
\begin{align}\label{def:AMP_sym}
\mathfrak{z}^{(t)} = A \mathfrak{F}_t( \mathfrak{z}^{([0:t-1])} )-\sum_{s \in [t-1]} b_s^{\mathfrak{F}_t}\circ \mathfrak{F}_{s}(\mathfrak{z}^{([0:s-1])}).
\end{align}
Here the correction vectors $\{b_s^{\mathfrak{F}_t}\}_{s \in [t-1]}\subset \R^n$, in conjunction with a centered Gaussian matrix $\mathscr{Z}^{([1:\infty))}\in \R^{n\times [1:\infty)}$ with independent rows (we denote $\mathscr{Z}^{(0)}\equiv \mathfrak{z}^{(0)}$), are defined recursively as follows.
\begin{definition}\label{def:AMP_se_sym}
	For $t=1,2,\ldots$, execute the following steps:
	\begin{enumerate}
		\item For $s \in [1:t-1]$, let $b_s^{\mathfrak{F}_t}\in \R^n$ be defined by 
		\begin{align*}
		b_{s,k}^{\mathfrak{F}_t} \equiv \sum_{\ell \in [n]} \E A_{k\ell}^2 \cdot \E \Big[\partial_{\mathscr{Z}_\ell^{(s)} } \mathfrak{F}_{t,\ell} \big(\mathscr{Z}^{([0:t-1])}_\ell\big)\big| \mathfrak{z}^{(0)}\Big],\quad k \in [n].
		\end{align*}
		\item Let the Gaussian law of $\mathscr{Z}^{(t)}$ be determined via the following correlation specification: for $s \in [1:t]$ and $k \in [n]$,
		\begin{align*}
		\cov\big(\mathscr{Z}^{(t)}_k, \mathscr{Z}^{(s)}_k\big)\equiv \sum_{\ell \in [n]} \E A_{k\ell}^2 \cdot \E \bigg[\prod_{\tau \in \{t,s\}} \mathfrak{F}_{\tau,\ell} \big(\mathscr{Z}^{([0:\tau-1])}_\ell\big)\Big|\mathfrak{z}^{(0)}\bigg].
		\end{align*}
	\end{enumerate}
\end{definition}
The entrywise distributions of $\{\mathfrak{z}^{(t)}\}$ are given by the following variant of \cite[Theorem 2.3]{bao2025leave} in combination with the estimate in Eqn. (2.14) therein. 

\begin{theorem}\label{thm:AMP_entrywise_dist_sym}
	Suppose the following hold.
	\begin{enumerate}
		\item $A\equiv A_0/\sqrt{n}$, where $A_0$ is a symmetric matrix whose upper triangle entries are i.i.d. random $\mathcal{N}(0,1)$ variables such that $\max_{i,j \in [n]} \E A_{ij}^2\leq K/n$ holds for some $K\geq 2$.
		\item $\mathfrak{F}_{s,\ell} \in C^3(\R^{[0:s-1]})$ for all $s \in [1:t], \ell \in [n]$, and there exists some $\Lambda\geq 2$ such that
		\begin{align*}
		\max_{s \in [1:t]}\max_{\ell \in [n]}\Big\{\abs{\mathfrak{F}_{s,\ell}(0)}+\max_{a\in \mathbb{Z}_{\geq 0}^{[0:s-1]}, \abs{a}\leq 3}\pnorm{\partial^a \mathfrak{F}_{s,\ell} }{\infty}\Big\}\leq \Lambda.
		\end{align*}
	\end{enumerate} 
	Fix any $\Lambda$-pseudo-Lipschitz function $\psi:\R^{t} \to \R$ of order $\mathfrak{p}\in \N$. Then there exist some universal constant $c_0>1$, and another constant $c_{\mathfrak{p}}>1$ depending on $\mathfrak{p}$ only, such that 
	\begin{align*}
	\max_{k \in [n]}\bigabs{\E \big[\psi\big( \mathfrak{z}^{([t])}_k\big)|\mathfrak{z}^{(0)}\big]-\E \big[\psi \big(\mathscr{Z}^{([t])}_k\big)|\mathfrak{z}^{(0)}\big] }  \leq \big(K\Lambda\log n\cdot (1+\pnorm{\mathfrak{z}^{(0)}}{\infty})\big)^{c_{\mathfrak{p}} t^3}\cdot n^{-1/c_0^{t}}.
	\end{align*}
\end{theorem}

Strictly speaking, \cite[Theorem 2.3]{bao2025leave} combined with Eqn. (2.14) therein proves a simplified version of (\ref{def:AMP_sym}) where $\mathfrak{F}_t( \mathfrak{z}^{([0:t-1])} )=\mathfrak{F}_t( \mathfrak{z}^{(t-1)} )$ (the variance lower bound in \cite[Eqn. (2.14)]{bao2025leave} can be lifted under the $C^3$ smoothness assumption in the above theorem, as it is not necessary in the second display in the proof of Section 6.4.6 therein). On the other hand, the leave-one-out representation in \cite[Theorem 2.1]{bao2025leave} carries over to the fully general form (\ref{def:AMP_sym}), and therefore an easy extension of \cite[Theorem 2.3]{bao2025leave} leads to the above theorem. 

A weaker, asymptotic and averaged version of the above theorem is previously obtained in \cite[Proposition 2.1]{montanari2021optimization}.

We need the following technical lemma before the proof of Theorem \ref{thm:GFOM_se_sym}.

\begin{lemma}\label{lem:Theta_est}
	Suppose the conditions in Theorem \ref{thm:GFOM_se_sym} hold. Then there exists some universal constant $c_0>0$, such that 
	\begin{align*}
	\max_{\ell \in [n]}\Big\{ \pnorm{\Theta_{t,\ell}(0)}{\infty} + \max_{a \in \mathbb{Z}_{\geq 0}^{[0:t]}: \abs{a}\leq 3}\sup_{z \in \R^{[0:t]}}\pnorm{\partial^a\Theta_{t,\ell}(z)}{\infty}  \Big\}\leq (tK\Lambda)^{c_0 t}.
	\end{align*} 
\end{lemma}
\begin{proof}
	Let $B_t\equiv 1\vee \max_{\ell \in [n]}\pnorm{\Theta_{t,\ell}(0)}{\infty}$. Then by definition of $\Theta_{t,\ell}$, we have an easy estimate $\max_{\ell \in [n]}\pnorm{\big(\Theta_{t,\ell}(0)\big)_{[0:t-1]} }{\infty}\leq B_{t-1}$. Moreover, for any $\ell \in [n]$, using that $\abs{\mathfrak{b}_{s,\ell}^{(t)}}\leq K^2 \Lambda$, we have
	\begin{align*}
	\abs{\big(\Theta_{t,\ell}(0)\big)_{t}} &\equiv   \biggabs{\sum_{s \in [1:t-1]} \mathfrak{b}_{s,\ell}^{(t)}\circ \mathsf{F}_{s,\ell}\big(\Theta_{s-1,\ell}(0)\big)+ \mathsf{G}_{t,\ell}\big(\Theta_{t-1,\ell}(0)\big)}\leq (t K\Lambda)^c\cdot  B_{t-1}.
	\end{align*} 
	Combining the estimates, we arrive at $
	B_t\leq (t K\Lambda)^c\cdot  B_{t-1}$. The claim for $\max_{\ell \in [n]}\pnorm{\Theta_{t,\ell}(0)}{\infty}$ follows by iterating the bound and using $B_0=1$. The derivative bounds are straightforward to derive. 
\end{proof}

We shall now prove Theorem \ref{thm:GFOM_se_sym}.

\begin{proof}[Proof of Theorem \ref{thm:GFOM_se_sym}]
	We only need to prove  the Gaussian case, as the general sub-Gaussian case follows directly from the entrywise universality Theorem \ref{thm:universality}.  
	
	Recall the functions $\{\Theta_t: \R^{n\times [0:t]}\to \R^{n\times [0:t]}\}$ and the Gaussian matrix $\mathfrak{Z}^{([1:\infty))}\in \R^{n\times [1:\infty)}$ in Definition \ref{def:GFOM_se_sym}. We construct the underlying AMP algorithm as follows. With the initialization $\mathfrak{z}^{(0)}\equiv z^{(0)}\in \R^n$, let $\mathfrak{F}_t: \R^{n\times [0:t-1]}\to \R^n$ be defined by $\mathfrak{F}_t( z)\equiv \mathsf{F}_t\big(\Theta_{t-1} (z) \big)$ for $z \in \R^{n\times [0:t-1]}$. Clearly $\mathfrak{F}_t$ is row-separable. Note that by using Lemma \ref{lem:Theta_est}, 
	\begin{align*}
	\max_{\ell \in [n]}\Big\{\abs{\mathfrak{F}_{t,\ell}(0)}+\max_{a\in \mathbb{Z}_{\geq 0}^{[0:t]}, \abs{a}\leq 3}\pnorm{\partial^a \mathfrak{F}_{s,\ell} }{\infty}\Big\}\leq (tK\Lambda)^{c_0 t}.
	\end{align*}
	It is easy to see from Definitions \ref{def:GFOM_se_sym} and \ref{def:AMP_se_sym} that $b_{s,k}^{\mathfrak{F}_t}=\mathfrak{b}_{s,k}^{(t)}$. Now with $w^{([0:t])}\equiv\Theta_t(\mathfrak{z}^{([0:t])})$, the first step in Definition \ref{def:GFOM_se_sym} entails that
	\begin{align*}
	w^{(t)} &\equiv  \mathfrak{z}^{(t)}+ \sum_{s \in [1:t-1]} \mathfrak{b}_s^{(t)}\circ \mathsf{F}_{s}\big(\Theta_{s-1}(\mathfrak{z}^{([0:s-1])})\big)+ \mathsf{G}_t(w^{([0:t-1])}).
	\end{align*}
	Further using the AMP iterate in (\ref{def:AMP_sym}), we have 
	\begin{align*}
	w^{(t)}& = \mathfrak{z}^{(t)}+ \sum_{s \in [t-1]} b_s^{\mathfrak{F}_t}\circ \mathfrak{F}_{s}(\mathfrak{z}^{([0:s-1])})+ \mathsf{G}_t(w^{([0:t-1])})\\
	& = A \mathfrak{F}_t( \mathfrak{z}^{([0:t-1])} )+  \mathsf{G}_t(w^{([0:t-1])}) = A \mathsf{F}_t(w^{([0:t-1])})+\mathsf{G}_t(w^{([0:t-1])}).
	\end{align*}
	Comparing the above display to the GFOM iterate (\ref{def:GFOM_sym}) and noting that $w^{(0)}=\Theta_0(\mathfrak{z}^{(0)})=\mathfrak{z}^{(0)}=z^{(0)}$, we then have $w^{(t)}=z^{(t)}$, completing the proof in view of Theorem \ref{thm:AMP_entrywise_dist_sym}.
\end{proof}

\section{Proofs for Section \ref{section:main_results_asym}}

\subsection{Proofs of Theorems \ref{thm:universality_asym} and \ref{thm:universality_asym_avg}}\label{subsection:proof_GFOM_universality_asym}
The asymmetric version in Theorems \ref{thm:universality_asym} and \ref{thm:universality_asym_avg} can be reduced to the symmetric case in Theorems \ref{thm:universality} and \ref{thm:universality_avg}, in a similar spirit to the arguments in \cite[Section 6]{berthier2020state} designed for the AMP algorithm. We spell out the details below.

\begin{proof}[Proofs of Theorems \ref{thm:universality_asym} and \ref{thm:universality_asym_avg}]
	Let $z^{(-1)}\equiv \binom{u^{(0)}}{0_n}, z^{(0)}\equiv \binom{0_m}{v^{(0)}}\in \R^{m+n}$, and for $t\in \mathbb{N}$, let $
	z^{(2t-1)}\equiv \binom{u^{(t)}}{0_n}$, $z^{(2t)}\equiv \binom{0_m}{v^{(t)}}\in \R^{m+n}$. Now we let $\mathsf{F}_{2t-1},\mathsf{G}_{2t-1}: \R^{(m+n)\times (2t)}\to \R^{m+n}$ and  $\mathsf{F}_{2t}, \mathsf{G}_{2t}: \R^{(m+n)\times (2t+1)}\to \R^{m+n}$ be defined via
	\begin{itemize}
		\item $\mathsf{F}_{2t-1}(z^{[-1:(2t-2)]})\equiv \binom{0_m}{ \mathsf{F}_t^{\langle 1\rangle}(v^{([0:t-1])}) }$, $\mathsf{G}_{2t-1}(z^{[-1:(2t-2)]})\equiv \binom{\mathsf{G}_{t}^{\langle 1\rangle}(u^{([0:t-1])})}{0_n}$;
		\item $\mathsf{F}_{2t}(z^{[-1:(2t-1)]})\equiv \binom{\mathsf{G}_{t}^{\langle 2\rangle}(u^{([0:t])})}{0_n}$, $\mathsf{G}_{2t}(z^{[-1:(2t-1)]})\equiv \binom{0_m}{ \mathsf{F}_t^{\langle 2\rangle}(v^{([0:t-1])}) }$.
	\end{itemize}
	Finally let 
	\begin{align*}
	\overline{A}_0&\equiv 
	\begin{pmatrix}
	0_{m\times m} & A_0\\
	A_0^\top & 0_{n\times n}
	\end{pmatrix}
	,\quad \overline{A}\equiv \frac{A_0}{\sqrt{m+n}}
	=\begin{pmatrix}
	0_{m\times m} & A\\
	A^\top & 0_{n\times n}
	\end{pmatrix}
	\in \R^{(m+n)\times (m+n)}. 
	\end{align*}
	Using these notation, we may rewrite the recursion (\ref{def:GFOM_asym}) as 
	\begin{align*}
	z^{(2t-1)}& = \overline{A} \mathsf{F}_{2t-1}(z^{[-1:(2t-2)]})+ \mathsf{G}_{2t-1}(z^{[-1:(2t-2)]}),\\
	z^{(2t)}& = \overline{A} \mathsf{F}_{2t}(z^{[-1:(2t-1)]})+ \mathsf{G}_{2t}(z^{[-1:(2t-1)]}). 
	\end{align*}
	Now we may apply Theorems \ref{thm:universality} and \ref{thm:universality_avg} to conclude. 
\end{proof}

\subsection{Proof of Theorem \ref{thm:GFOM_se_asym}}\label{subsection:proof_GFOM_se_asym}

We consider the following asymmetric version of the AMP iterate, which is initialized with $(\mathfrak{u}^{(0)},\mathfrak{v}^{(0)})\in \R^m\times \R^n$, and subsequently updated for $t=1,2,\ldots$ according to 
\begin{align}\label{def:AMP_asym}
\begin{cases}
\mathfrak{u}^{(t)} = A \mathfrak{F}_{t}(\mathfrak{v}^{([0:t-1])})-\sum_{s \in [1:t-1]}b_{s}^{\mathfrak{F}_{t}}\circ \mathfrak{G}_{s}(\mathfrak{u}^{([0:s])}) \in \R^m,\\
\mathfrak{v}^{(t)}= A^\top \mathfrak{G}_{t}(\mathfrak{u}^{([0:t])})-\sum_{s \in [1:t]} b_{s}^{\mathfrak{G}_{t}}\circ\mathfrak{F}_{s}(\mathfrak{v}^{([0:s-1])})\in \R^n.
\end{cases}
\end{align}
Here the correction vectors $\{b_{s}^{\mathfrak{F}_t}\}_{s \in [1:t-1]}\subset \R^m$ and $\{b_{s}^{\mathfrak{G}_{t}}\}_{s \in [1:t]}\subset \R^n$, in conjunction with two centered Gaussian matrices $\mathscr{U}^{[1:\infty)} \in \R^{m\times [1:\infty)}$, $\mathscr{V}^{[1:\infty)} \in \R^{n\times [1:\infty)}$ with independent rows (where we denote $\mathscr{U}^{(0)}\equiv \mathfrak{u}^{(0)}$, $\mathscr{V}^{(0)}\equiv \mathfrak{v}^{(0)}$), are defined recursively as follows. 

\begin{definition}\label{def:AMP_se_asym}
	For $t=1,2,\ldots$, with $ \E^{(0)} \equiv \E\big[\cdot | (\mathfrak{u}^{(0)},\mathfrak{v}^{(0)})\big]$, execute the following steps:
	\begin{enumerate}
		\item For $s \in [1:t-1]$, let $b_{s}^{\mathfrak{F}_{t}}\in \R^m$ be defined by
		\begin{align*}
		b_{s,k}^{\mathfrak{F}_{t}}\equiv \sum_{\ell \in [n]} \E A_{k\ell}^2\cdot \E^{(0)} \partial_{ \mathscr{V}_\ell^{(s)} } \mathfrak{F}_{t,\ell}\big(\mathscr{V}_\ell^{[0:t-1]}\big),\quad i \in [m].
		\end{align*}
		\item Let the Gaussian law of $\mathscr{U}^{(t)}$ be determined via the following correlation specification: for $s \in [1:t]$ and $k \in [m]$,
		\begin{align*}
		\cov\big(\mathscr{U}^{(t)}_k, \mathscr{U}^{(s)}_k\big)\equiv \sum_{\ell \in [n]} \E A_{k\ell}^2 \cdot \E^{(0)} \prod_{\tau \in \{t,s\}} \mathfrak{F}_{\tau,\ell} \big(\mathscr{V}^{([0:\tau-1])}_\ell\big).
		\end{align*}
		\item For $s \in [1:t]$, let $b_{s}^{\mathfrak{G}_{t}}\in \R^n$ be defined by
		\begin{align*}
		b_{s,\ell}^{\mathfrak{G}_{t}}\equiv \sum_{k \in [m]} \E A_{k\ell}^2\cdot \E^{(0)} \partial_{ \mathscr{U}_k^{(s)} } \mathfrak{G}_{t,k}\big(\mathscr{U}_k^{[0:t]}\big),\quad \ell \in [n].
		\end{align*}
		\item Let the Gaussian law of $\mathscr{V}^{(t)}$ be determined via the following correlation specification: for $s \in [1:t]$ and $\ell \in [n]$,
		\begin{align*}
		\cov\big(\mathscr{V}^{(t)}_\ell, \mathscr{V}^{(s)}_\ell\big)\equiv \sum_{k \in [m]} \E A_{k\ell}^2 \cdot \E^{(0)}\prod_{\tau \in \{t,s\}} \mathfrak{G}_{\tau,k} \big(\mathscr{U}^{([0:\tau])}_k\big).
		\end{align*}
	\end{enumerate}
\end{definition}
The entrywise distributions of $\{(\mathfrak{u}^{(t)},\mathfrak{v}^{(t)})\}$ are given by the following variant of \cite[Theorem 2.7]{bao2025leave}.

\begin{theorem}\label{thm:AMP_entrywise_dist_asym}
	Suppose the following hold.
	\begin{enumerate}
		\item $A\equiv A_0/\sqrt{n}$, where the entries of $A_0 \in \R^{m\times n}$ are i.i.d. random $\mathcal{N}(0,1)$ variables such that $\max_{i \in [m], j \in [n]} n^{-1}\E A_{ij}^2\leq K$ holds for some $K\geq 2$.
		\item $\mathfrak{F}_{s,\ell} \in C^3(\R^{[0:s-1]})$, $\mathfrak{G}_{s,k} \in C^2(\R^{[0:s]})$ for all $s \in [1:t], k \in [m], \ell \in [n]$, and there exists some $\Lambda\geq 2$ such that
		\begin{align*}
		&\max_{s \in [1:t]}\max_{k \in [m], \ell \in [n]}\Big\{\abs{\mathfrak{F}_{s,\ell}(0)}+\abs{\mathfrak{G}_{s,k}(0)}\\
		&\qquad +\max_{a\in \mathbb{Z}_{\geq 0}^{[0:s-1]}, b \in \mathbb{Z}_{\geq 0}^{[0:s]}, \abs{a}\vee \abs{b} \leq 3} \big(\pnorm{\partial^a \mathfrak{F}_{s,\ell} }{\infty}+\pnorm{\partial^a \mathfrak{G}_{s,k} }{\infty} \big)\Big\}\leq \Lambda.
		\end{align*}
	\end{enumerate} 
	Fix any $\Lambda$-pseudo-Lipschitz function $\psi:\R^{t} \to \R$ of order $\mathfrak{p}\in \N$. Then there exist some universal constant $c_0>1$, and another constant $c_{\mathfrak{p}}>1$ depending on $\mathfrak{p}$ only, such that with $ \E^{(0)} \equiv \E\big[\cdot | (\mathfrak{u}^{(0)},\mathfrak{v}^{(0)})\big]$,
	\begin{align*}
	&\max_{k \in [m]}\bigabs{\E^{(0)} \psi\big( \mathfrak{u}^{([t])}_k\big)-\E^{(0)} \psi \big(\mathscr{U}^{([t])}_k\big) } + \max_{\ell \in [n]}\bigabs{\E^{(0)}\psi\big( \mathfrak{v}^{([t])}_\ell\big)-\E^{(0)} \psi \big(\mathscr{V}^{([t])}_\ell\big) } \\
	&  \leq \big(K\Lambda\log n\cdot (1+\pnorm{\mathfrak{u}^{(0)}}{\infty}+\pnorm{\mathfrak{v}^{(0)}}{\infty})\big)^{c_{\mathfrak{p}} t^3}\cdot n^{-1/c_0^{t}}.
	\end{align*}
\end{theorem}

Similar to the symmetric case, the above theorem is an extension of \cite[Theorem 2.7]{bao2025leave} by using the version of the leave-one-out representation in \cite[Theorem 2.5]{bao2025leave} for the general form in (\ref{def:AMP_asym}).

\begin{proof}[Proof of Theorem \ref{thm:GFOM_se_asym}]
	We only prove the Gaussian case, as universality properties follow from Theorem \ref{thm:universality_asym}. 
	
	We shall now construct the underlying AMP algorithm. Let the initialization be given by $\mathfrak{u}^{(0)}\equiv u^{(0)}$ and $\mathfrak{v}^{(0)}\equiv v^{(0)}$. For $t\geq 1$, let $\mathfrak{F}_t, \mathfrak{G}_t$ be defined as follows: for $\mathfrak{v}^{([0:t-1])}\in \R^{n\times [0:t-1]}$ and $\mathfrak{u}^{([0:t])} \in \R^{m \times [0:t]}$, let $\mathfrak{F}_t(\mathfrak{v}^{([0:t-1])})\equiv \mathsf{F}_t^{\langle 1\rangle}\big(\Xi_{t-1}(\mathfrak{v}^{([0:t-1])})\big)$ and $\mathfrak{G}_t(\mathfrak{u}^{([0:t])})\equiv \mathsf{G}_t^{\langle 2\rangle}\big(\Phi_t(\mathfrak{u}^{([0:t])})\big)$. We may now check from Definitions \ref{def:GFOM_se_asym} and \ref{def:AMP_se_asym} that $b_{s}^{\mathfrak{F}_{t}}= \mathfrak{f}_s^{(t-1)}$ and $b_{s}^{\mathfrak{G}_{t}}=\mathfrak{g}_s^{(t)}$. Now with $x^{([0:t])}\equiv \Phi_{t}(\mathfrak{u}^{([0:t])})$ and $y^{([0:t])}\equiv \Xi_{t}(\mathfrak{v}^{([0:t])})$, (1) and (3) in Definition \ref{def:GFOM_se_asym} reduce to 
	\begin{align*}
	\begin{cases}
	x^{(t)} = \mathfrak{u}^{(t)}+\sum_{s \in [1:t-1]} \mathfrak{f}_{s}^{(t-1) }\circ \mathsf{G}_{s}^{\langle 2\rangle}\big(\Phi_{s}(\mathfrak{u}^{([0:s])}) \big)+\mathsf{G}_{t}^{\langle 1\rangle}(x^{([0:t-1])}),\\
	y^{(t)} = \mathfrak{v}^{(t)}+\sum_{s \in [1:t]} \mathfrak{g}_{s}^{(t)}\circ \mathsf{F}_{s}^{\langle 1\rangle}\big(\Xi_{s-1}(\mathfrak{v}^{([0:s-1])}) \big)+\mathsf{F}_{t}^{\langle 2\rangle}(y^{([0:t-1])}).
	\end{cases}
	\end{align*}
	Using the asymmetric AMP iterate in (\ref{def:AMP_asym}), we have 
	\begin{align*}
	x^{(t)} &= \mathfrak{u}^{(t)}+\sum\nolimits_{s \in [1:t-1]} b_{s}^{\mathfrak{F}_{t}}\circ \mathfrak{G}_s(\mathfrak{u}^{([0:s])})+\mathsf{G}_{t}^{\langle 1\rangle}(x^{([0:t-1])})\\
	&= A \mathfrak{F}_{t}(\mathfrak{v}^{([0:t-1])})+ \mathsf{G}_{t}^{\langle 1\rangle}(x^{([0:t-1])}) \\
	&= A\mathsf{F}_t^{\langle 1\rangle}\big(\Xi_{t-1}(\mathfrak{v}^{([0:t-1])})\big)+ \mathsf{G}_{t}^{\langle 1\rangle}(x^{([0:t-1])}) \\
	&= A \mathsf{F}_t^{\langle 1\rangle}(y^{([0:t-1])})+ \mathsf{G}_{t}^{\langle 1\rangle}(x^{([0:t-1])}).
	\end{align*}
	Similarly, 
	\begin{align*}
	y^{(t)}&=\mathfrak{v}^{(t)}+\sum\nolimits_{s \in [1:t]} b_{s}^{\mathfrak{G}_{t}}\circ \mathfrak{F}_s(\mathfrak{v}^{([0:s-1])})+\mathsf{F}_{t}^{\langle 2\rangle}(y^{([0:t-1])})\\
	&=A^\top \mathfrak{G}_{t}(\mathfrak{u}^{([0:t])})+\mathsf{F}_{t}^{\langle 2\rangle}(y^{([0:t-1])})\\
	& = A^\top \mathsf{G}_t^{\langle 2\rangle}\big(\Phi_t(\mathfrak{u}^{([0:t])})\big)+\mathsf{F}_{t}^{\langle 2\rangle}(y^{([0:t-1])})\\
	&= A^\top \mathsf{G}_t^{\langle 2\rangle}(x^{([0:t])})+ \mathsf{F}_{t}^{\langle 2\rangle}(y^{([0:t-1])}).
	\end{align*}
	Comparing the above two displays with the asymmetric GFOM iterate (\ref{def:GFOM_asym}) and noting that $x^{(0)}=\Phi_0(\mathfrak{u}^{(0)})=\mathfrak{u}^{(0)}=u^{(0)}$, $y^{(0)}=\Xi_0(\mathfrak{v}^{(0)})=\mathfrak{v}^{(0)}=v^{(0)}$, we have $x^{(t)}=u^{(t)},y^{(t)}=v^{(t)}$. The claim follows by checking the regularity conditions in Theorem \ref{thm:AMP_entrywise_dist_asym} via similar estimates as in Lemma \ref{lem:Theta_est} for the symmetric case. 
\end{proof}

\section{Proofs for Section \ref{section:ERM}}\label{section:proof_ERM}

\subsection{Proof of Theorem \ref{thm:avg_universality_ERM}}\label{subsection:proof_thm_avg_universality_ERM}

    For notational simplicity we shall work with $\psi_k\equiv \psi$.
	\noindent (1). 
	Combining  (\ref{ineq:avg_universality_ERM_1})-(\ref{ineq:avg_universality_ERM_2}) and using Lemma \ref{lem:prox_Lip}, 
	\begin{align}\label{ineq:avg_universality_ERM_2_1}
	\pnorm{ \mu^{(t)}-\hat{\mu} }{}&\leq (1+\eta/K)^{-1}\pnorm{\big(I-\eta A^\top D_1^{(t)} A\big)(\mu^{(t-1)}-\hat{\mu}) }{}\nonumber\\
	&\leq \ldots \leq (1+\eta/K)^{-t}\cdot  \prod\nolimits_{s \in [t]} \bigpnorm{I-\eta A^\top D_1^{(s)} A}{\op}\cdot \pnorm{ \hat{\mu} }{}.
	\end{align}
	Here $\big\{D_1^{(s)}\in \R^{m\times m}: s \in [t]\big\}$ are (random) diagonal matrices with $0\leq D_{1,ii}^{(s)}\leq K$ for all $s \in [t]$. As $0\leq A^\top D_1^{(s)} A \leq K\cdot A^\top A$ (in the sense of matrix ordering of p.s.d. matrices), by exponentially high probability boundedness of $\pnorm{A}{\op}$, for $\eta\leq 1/C_1$, on an event $E_1$ with $\Prob^\xi$-probability at least $1-C_1 te^{-n/C_1}$, $\max_{s \in [t]}\pnorm{I-\eta A^\top D_1^{(s)} A}{\op}\leq 1$. Consequently, on the event $E_1$ we have
	\begin{align}\label{ineq:avg_universality_ERM_3}
	\pnorm{ \mu^{(t)}-\hat{\mu} }{}\leq (1+\eta/K)^{-t}\cdot  \pnorm{\hat{\mu}}{}.
	\end{align}
	On the other hand, using (\ref{ineq:avg_universality_ERM_1}) and $\mathsf{L}'(Y-A\hat{\mu})=\mathsf{L}'(Y)-D_2 A\hat{\mu}$ for some (random) diagonal matrix $D_2 \in \R^{m\times m}$ with $0\leq D_{2,ii}\leq K$, 
	\begin{align*}
	\pnorm{\hat{\mu}}{}&\leq (1+\eta/K)^{-1}\pnorm{ (I-\eta A^\top D_2 A) \hat{\mu} }{}+ \eta\cdot \pnorm{A^\top \mathsf{L}'(Y)}{}.
	\end{align*}
	So using the same argument as above, on an event $E_2$ with $\Prob^\xi$-probability at least $1-C_2 e^{-n/C_2}$, for $\eta= 1/C_2$ we have 
	\begin{align}\label{ineq:avg_universality_ERM_4}
	\pnorm{\hat{\mu}}{}\leq  C_2\cdot \big(\pnorm{\mu_0}{}\vee \pnorm{\mathsf{L}'(\xi)}{}\big).
	\end{align}
	Combining (\ref{ineq:avg_universality_ERM_3}) and (\ref{ineq:avg_universality_ERM_4}), on $E_1\cap E_2$, for $\eta\leq 1/C_3$, 
	\begin{align}\label{ineq:avg_universality_ERM_5}
	\pnorm{ \mu^{(t)}-\hat{\mu} }{}\leq C_3 (1+\eta/K)^{-t}\cdot \big(\pnorm{\mu_0}{}\vee \pnorm{\mathsf{L}'(\xi)}{}\big).
	\end{align}
	Now we may apply Theorem \ref{thm:universality_asym_avg} by reformulating the proximal gradient descent algorithm (\ref{ineq:avg_universality_ERM_2}) in the form of (\ref{def:GFOM_asym}) via the following identification: Let $\mu^{(t)}\equiv \prox_{\eta \mathsf{f}_n}(v^{(t)})$, where
	\begin{align*}
	\begin{cases}
	u^{(t)}= A\big(\prox_{\eta \mathsf{f}_n}(v^{(t-1)})-\mu_0\big)\in \R^m,\\
	v^{(t)}= A^\top \mathsf{G}^\xi(u^{(t)})+\prox_{\eta \mathsf{f}_n}(v^{(t-1)})\in \R^n,
	\end{cases}
	\end{align*}
	with $\mathsf{G}^\xi(u)\equiv \eta\mathsf{L}'(\xi-u)$, and the initialization $u^{(0)}=0_m,v^{(0)}=0_n$. Applying Theorem \ref{thm:universality_asym_avg} conditionally on $\xi$, with $E\equiv E_1\cap E_2$, we obtain
	\begin{align*}
	&\E^\xi\biggabs{\frac{1}{n}\sum_{j \in [n]} \psi\big(\hat{\mu}_j(A)\big)-\frac{1}{n}\sum_{j \in [n]} \psi\big(\hat{\mu}_j(B)\big) }^q \\
	&\leq C\cdot\bigg\{ \E^\xi\biggabs{\frac{1}{n}\sum_{j \in [n]} \psi\big(\mu_j^{(t)}(A)\big)-\frac{1}{n}\sum_{j \in [n]} \psi\big(\mu_j^{(t)}(B)\big) }^q\\
	&\qquad  + \max_{M\in \{A,B\}} \Big(\E^\xi \bigabs{n^{-1/2}\pnorm{\mu^{(t)}(M)-\hat{\mu}(M)  }{}}^q\bm{1}_E+ \E^{\xi,1/2} (n^{-1/2} \pnorm{\hat{\mu}(M)}{} )^{2q}\cdot \Prob^{\xi,1/2}(E^c)\Big) \bigg\}\\
	&\leq \Big(\max_{i \in [m]} \abs{\mathsf{L}'(\xi_i)}\cdot\log n\Big)^{C t^3} n^{-1/C t^3}+C\big(1-1/C\big)^{t} \big(1+n^{-1/2}\pnorm{\mathsf{L}'(\xi)}{}\big) + C e^{-n/C}\cdot L_{\xi}^q.
	\end{align*}
	Here in the last inequality, we use the following apriori estimate on $\pnorm{\hat{\mu}}{}$: As $\hat{\mu}$ minimizes the cost function, we have
	\begin{align*}
	K^{-1}\pnorm{\hat{\mu}}{}^2\leq \sum_{j \in [n]} \mathsf{f}(\hat{\mu}_j)\leq \sum_{i \in [m]} \mathsf{L}(Y_i-A_i^\top \hat{\mu})+ \sum_{j \in [n]} \mathsf{f}(\hat{\mu}_j)\leq \sum_{i \in [m]} \mathsf{L}(Y_i)\leq C(1+\pnorm{Y}{}^2).
	\end{align*}
	Finally choosing $t\equiv (\log n)^{1/8}$ and taking expectation with respect to $\xi$, we may conclude the proof for general $\mathsf{L}$'s. 
	
	\noindent (2). For the squared loss $\mathsf{L}(x)=x^2/2$, the estimate (\ref{ineq:avg_universality_ERM_2_1}) can be replaced by
	\begin{align*}
	\pnorm{ \mu^{(t)}-\hat{\mu} }{}&\leq \pnorm{\big(I-\eta A^\top A\big)(\mu^{(t-1)}-\hat{\mu}) }{}.
	\end{align*}
	By \cite[Theorems 2.4-(ii) and 2.9-(i)]{alt2017local}, with probability at least $1-C_4 n^{-100}$, $1/C_4\leq \lambda_{\min}(A^\top A)\leq\lambda_{\max}(A^\top A)\leq C_4$ for some $C_4>1$. So by choosing $\eta$ small, we have $\pnorm{I-\eta A^\top A}{\op}\leq 1-1/C_5$ for some $C_5>1$, which in turn implies an estimate of the type (\ref{ineq:avg_universality_ERM_3}). An estimate of the type (\ref{ineq:avg_universality_ERM_4}) can be obtained similarly. Now we may repeat the arguments in (1) to conclude.
	
	\noindent (3). Let $\mathsf{L}(x)=x^2/2$ and $\mathsf{f}(x)\equiv \lambda \abs{x}$. Using that $\prox_{\eta \lambda \pnorm{\cdot}{1}}(z)=x\Leftrightarrow x-z+\eta \lambda \cdot y =0$ for some $y \in \partial \pnorm{\cdot}{1} (x)$, (\ref{ineq:avg_universality_ERM_2}) can be rewritten as 
	\begin{align*}
	\mu^{(t)}- \Big(\mu^{(t-1)}+\eta\cdot  A^\top (Y-A \mu^{(t-1)})\Big)=-\eta \lambda\cdot y^{(t)},
	\end{align*}
	where $y^{(t)}$ is a sub-gradient of $\pnorm{\cdot}{1}$ evaluated at $\mu^{(t)}$.  Equivalently, with $w^{(t)}\equiv \mu^{(t)}-\mu_0$ and $x\equiv A^\top \xi$,
	\begin{align*}
	-\eta \lambda\cdot y^{(t)} + \eta x = \mu^{(t)}-\mu^{(t-1)}+\eta A^\top Aw^{(t-1)}.
	\end{align*}
	Let $S_+^{(t)}\equiv \{i \in [n]: \mu_i^{(t)}> 0\}$. Taking square on both sides of the above display and summing over $i \in S_+^{(t)}$, with $x_i\equiv (A^\top \xi)_i$,
	\begin{align}\label{ineq:avg_universality_lasso_1}
	\eta^2 \lambda^2\cdot \abs{S_+^{(t)}}- 2\eta^2 \lambda \sum\nolimits_{i \in S_+^{(t)}} x_i  &\leq 2 \pnorm{\mu^{(t)}-\mu^{(t-1)}}{}^2+ 2\eta^2 \pnorm{A}{\op}^2 \pnorm{Aw^{(t-1)}}{}^2.
	\end{align}
	For the first term on the RHS of (\ref{ineq:avg_universality_lasso_1}), using the definition of the gradient descent and non-expansiveness of the proximal operator, on an event with probability at least $1-C e^{-n/C}$, for $\eta\leq 1/C$, $\pnorm{(I-\eta A^\top A)}{\op}\leq 1$, and therefore 
	\begin{align}\label{ineq:avg_universality_lasso_2}
	\pnorm{\mu^{(t)}-\mu^{(t-1)}}{}\leq \pnorm{ (I-\eta A^\top A) (\mu^{(t-1)}-\mu^{(t-2)}) }{}\leq \cdots \leq \pnorm{\mu^{(1)}}{}\leq \eta \pnorm{A^\top Y}{}.
	\end{align}
	For the second term on the RHS of (\ref{ineq:avg_universality_lasso_1}), note that a variant of \cite[Lemma 6.3]{han2023universality} applies to conclude that for small enough $\epsilon_0\in (0,1)$ and $\lambda\geq C_{K,\epsilon_0}$, with probability at least $1-C e^{-n/C}$, we have $\pnorm{\hat{\mu}}{0}\leq \epsilon_0 n$. Moreover, by a slight modification of \cite[Lemma B.2]{han2023universality}, there exists $\delta_0=\delta_0(K) \in (0,1/2)$ such that for $s\leq \delta_0 n$, with probability at least $1-C e^{-n/C}$, the sparse eigenvalue $\phi_-(s)\equiv \inf_{v \in \R^n: \pnorm{v}{0}\leq s} v^\top (A^\top A) v/\pnorm{v}{}^2\geq 1/K$. Using the deterministic estimate part in \cite[Eqn. (6.8)]{han2023universality},  we have $\pnorm{\hat{\mu}}{}^2\leq \phi_-^{-1}(\delta_0/2)\pnorm{A \hat{\mu}}{}^2\lesssim \phi_-^{-1}(\delta_0/2)\big(\pnorm{A \mu_0}{}^2+\pnorm{\xi}{}^2\big)$. So for $\lambda\geq C_K$, both $\pnorm{\hat{\mu}}{0}\leq n/C$ and $\pnorm{\hat{\mu}}{}\leq Cn$ hold with probability at least $1-C e^{-n/C}$. Now with $\hat{w}\equiv \hat{\mu}-\mu_0$, with the prescribed probability,	
	 \begin{align}\label{ineq:avg_universality_lasso_2_1}
	 \pnorm{Aw^{(t-1)}}{}\leq \pnorm{A\hat{w}}{}+ \pnorm{A}{\op} \pnorm{ \mu^{(t-1)}-\hat{\mu} }{}\leq \pnorm{A\hat{w}}{}+ \pnorm{A}{\op} \pnorm{\hat{\mu} }{}\leq Cn.
	 \end{align}	
	Combining the above displays (\ref{ineq:avg_universality_lasso_1})-(\ref{ineq:avg_universality_lasso_2_1}), on an event $E_3$ with $\Prob(E_3^c)\leq C e^{-n/C}$, 
	\begin{align*}
	\lambda^2\abs{S_+^{(t)}}\leq  2\lambda\cdot   \Big|\sum\nolimits_{i \in S_+^{(t)}} x_i \Big|+C n . 
	\end{align*}
	From here, using the same arguments as those below \cite[Eqn. (6.7)]{han2023universality}, for any $\delta \in (0,1/2)$, if $\lambda \geq C_{K,\delta}$, we have $
	\Prob\big(\{\abs{S_+^{(t)}}>\delta n \}\cap E_3  \big)\leq C e^{-\delta\lambda^2 n/C }$. 
	A similar argument applies to $S^{(t)}_-\equiv \{i \in [n]: \mu^{(t)}_i<0\}$, so if $\lambda \geq C_{K,\delta}$, for all $t \in \N$,
	\begin{align}\label{ineq:avg_universality_lasso_3}
	\Prob\big(\pnorm{\mu^{(t)}}{0}>\delta n \big)\leq C e^{-n(\delta\lambda^2 \wedge 1) /C}.
	\end{align}
	The above estimate (\ref{ineq:avg_universality_lasso_3}) in conjunction with \cite[Lemma 6.3]{han2023universality}, implies then with probability at least $1-Ce^{-n/C}$, $\pnorm{\mu^{(t)}}{0}\vee \pnorm{\hat{\mu}}{0}\leq (\delta_0/2) n$ uniformly for all $t \leq n$. Combining this assertion and the sparse eigenvalue estimate above, with probability at least $1-Ce^{-n/C}$,  we have $\pnorm{(I-\eta A^\top A)(\mu^{(t)}-\hat{\mu})}{}\leq (1-1/C)\pnorm{\mu^{(t)}-\hat{\mu}}{}$ for all $t\leq n$. Consequently, with the same probability and for all $t\leq n$,
	\begin{align*}
	&\pnorm{\mu^{(t)}-\hat{\mu}}{}\leq \pnorm{(I-\eta A^\top A)(\mu^{(t-1)}-\hat{\mu})}{}\\
	&\leq (1-1/C)\cdot \pnorm{\mu^{(t-1)}-\hat{\mu}}{}\leq \cdots \leq (1-1/C)^t \cdot \pnorm{\hat{\mu}}{}.
	\end{align*}
	Combining the above display and the proven high probability estimate $\pnorm{\hat{\mu}}{}\leq Cn$, with probability at least $1-C e^{-n/C}$, for all $t\leq n$, 
	\begin{align*}
	n^{-1/2}\pnorm{\mu^{(t)}-\hat{\mu}}{}\leq C(1-1/C)^t .
	\end{align*}
	Now we may repeat the arguments in (1) to conclude. \qed

\subsection{Proof of Theorem \ref{thm:universality_ls}}\label{section:proof_ERM_ls}

Fix $\ell \in [n]$. Let $A_{i;[-\ell]}\equiv (A_{ij}\bm{1}_{j\neq \ell})_{j \in [n]}$. Consider the following leave-one(-predictor)-out version of $\hat{\mu}$:
\begin{align}\label{def:ERM_loo_predictor}
\hat{\mu}_{[-\ell]} = \argmin_{\mu \in \R^n } \bigg\{\sum_{i \in [m]} \Big(A_{i;[-\ell]}^\top (\mu- \mu_0)- \xi_i\Big)^2+  \sum_{j \in [n]} \mathsf{f}(\mu_j) \bigg\},
\end{align}
which satisfies the first-order condition
\begin{align}\label{eqn:ERM_first_order_loo_predictor}
\hat{\mu}_{[-\ell]} = \prox_{\eta \mathsf{f}_n}\Big(\big(I-\eta\hat{\Sigma}_{[-\ell]}\big)\hat{\mu}_{[-\ell]}+\eta\cdot  A_{[-\ell]}^\top Y_{[-\ell]}\Big).
\end{align}
Here $\hat{\Sigma}_{[-\ell]}=A_{[-\ell]}^\top A_{[-\ell]}$. We therefore consider the following leave-one-predictor-out proximal gradient descent algorithm: for $\eta>0$, let for $t=1,2,\ldots,$
\begin{align}\label{eqn:ERM_grad_descent_loo_predictor}
{\mu}_{[-\ell]}^{(t)} \equiv \prox_{\eta \mathsf{f}_n}\Big(\big(I-\eta\hat{\Sigma}_{[-\ell]}\big)\mu^{(t-1)}_{[-\ell]}+\eta\cdot  A^\top_{[-\ell]} Y_{[-\ell]}\Big),
\end{align}
with the initialization $\mu^{(0)}_{[-\ell]}=0$.

We now establish a key estimate that asserts (i) the first-order delocalization of $\hat{\mu},\hat{\mu}_{[-\ell]}$ and $\mu^{(t)},\mu^{(t)}_{[-\ell]}$, and the second-order `almost delocalization' of their differences $\hat{\mu}-\hat{\mu}_{[-\ell]}$ and $\mu^{(t)}-\mu^{(t)}_{[-\ell]}$.

\begin{lemma}\label{lem:ERM_loo_predictor_error}
	Suppose the following hold.
	\begin{enumerate}
		\item For some $K>1$, $m/n \in [1/K,K]$, $\pnorm{\mu_0}{\infty}\leq K$, and $A=A_0/\sqrt{m}$ where $A_0$ is an $m\times n$ random matrix whose entries are independent mean $0$ variables with $\max_{i,j} \pnorm{A_{0,ij}}{\psi_2}\leq K$.
		\item The function $\mathsf{f}\geq \mathsf{f}(0)=0$, $\mathsf{f}\in C^2(\R)$ and $\pnorm{(\mathsf{f}^{(2)})^{-1}}{\infty}\leq \kappa$ for some $ \kappa>1$.
	\end{enumerate}
	Then there exists some $C_0=C_0(K,\kappa)>1$ such that with $L_\xi\equiv 1+m^{-1/2}\pnorm{\xi}{}$,
	\begin{align*}
	\sup_{\eta \in [0,1/C_0]}\Prob^\xi \Big( \big(\pnorm{ \hat{\mu} }{}\vee \pnorm{ \hat{\mu}_{[-\ell]} }{}\big)+\sup_{t \in \N} \big(\pnorm{{\mu}^{(t)} }{}\vee \pnorm{{\mu}^{(t)}_{[-\ell]} }{}\big) \geq C_0 n^{1/2}\cdot L_\xi \Big)\leq C_0 e^{-n/C_0}. 
	\end{align*}
	For any $D>0$, there exists some $C_1=C_1(K,\kappa,D)>1$ such that for any $\ell \in [n]$,
	\begin{align*}
	&\sup_{\eta \in [0,1/C_0]} \bigg\{\Prob^\xi\Big(\max_{k\neq \ell} n^{1/2} \abs{(\hat{\mu}-\hat{\mu}_{[-\ell]})_k}+ \abs{\hat{\mu}_\ell}\geq C_1 L_\xi \log n \Big)\\
	&\quad+ \sup_{t \leq C_1 \log n}\Prob^\xi\Big(\max_{k\neq \ell} n^{1/2} \abs{({\mu}^{(t)}-{\mu}_{[-\ell]}^{(t)})_k}+ \abs{{\mu}^{(t)}_\ell}\geq C_1 L_\xi \log n \Big)\bigg\} \leq  C_1 n^{-D}.
	\end{align*}
	The above estimate holds with $\hat{\mu}-\hat{\mu}_{[-\ell]}$ (resp. ${\mu}^{(t)}-{\mu}_{[-\ell]}^{(t)}$) replaced by $\hat{\Sigma}\hat{\mu}-\hat{\Sigma}_{[-\ell]}\hat{\mu}_{[-\ell]}$ (resp. $\hat{\Sigma}{\mu}^{(t)}-\hat{\Sigma}_{[-\ell]}{\mu}_{[-\ell]}^{(t)}$).
	
	If $n/m\leq 1-\epsilon$ for some $\epsilon \in (0,1)$, then the above estimates hold without the condition $\pnorm{(\mathsf{f}^{(2)})^{-1}}{\infty}\leq \kappa$, with the constants $C_0,C_1>0$ further depending on $\epsilon$ (but not on $\kappa$ any longer). 
\end{lemma}
\begin{proof}
	First using Lemma \ref{lem:prox_Lip} and the smoothness of $\mathsf{f}$, we have $\pnorm{\prox'_{\eta \mathsf{f}}}{\infty}\leq 1/(1+\eta \kappa^{-1})\leq (1-\eta/2\kappa)$ whenever $\eta\leq 1<\kappa$.

	\noindent (1). Using (\ref{eqn:ERM_first_order}), we have
	\begin{align*}
	\pnorm{\hat{\mu}}{}\leq (1-\eta/2\kappa)\pnorm{I-\eta \hat{\Sigma}}{\op} \pnorm{\hat{\mu}}{}+(1-\eta/2\kappa) \eta \pnorm{A^\top Y}{}.
	\end{align*}
	So for $\eta\leq 1/C$, with probability at least $1-Ce^{-n/C}$, 
	\begin{align*}
	\pnorm{\hat{\mu}}{}\leq 2\kappa(1-\eta/2\kappa)  \pnorm{A^\top Y}{}.
	\end{align*}
	The claim now follows by using $\pnorm{A^\top Y}{}\leq \pnorm{\hat{\Sigma}}{\op}\pnorm{\mu_0}{}+ \pnorm{A}{\op}\pnorm{\xi}{}$ and the exponentially high probability boundedness of $\pnorm{\hat{\Sigma}}{\op},\pnorm{A}{\op}$. For the case $n/m\leq 1-\epsilon$ without strong convexity of $\mathsf{f}$, we may use alternatively the exponentially high probability boundedness of $\pnorm{(\hat{\Sigma})^{-1}}{\op}$ (cf. \cite{rudelson2009smallest}) to proceed.  The claim for $\pnorm{\hat{\mu}_{[-\ell]}}{}$ proceeds similarly. 
	
	On the other hand, using (\ref{def:ERM_grad_descent}), for $\eta\leq 1/C$, with probability at least $1-Ce^{-n/C}$, for all $t \in \N$, 
	\begin{align*}
	\pnorm{{\mu}^{(t)}}{}&\leq (1-\eta/2\kappa)\pnorm{{\mu}^{(t-1)} }{}+(1-\eta/2\kappa) \eta \pnorm{A^\top Y}{}\\
	& \leq \cdots \leq \eta \sum_{s\in [t]}(1-\eta/2\kappa)^s \pnorm{A^\top Y}{} \leq 2\kappa \pnorm{A^\top Y}{}. 
	\end{align*}
	Now we may proceed as above to conclude for both $\pnorm{{\mu}^{(t)}}{}$ and its leave-one-out version $\pnorm{{\mu}^{(t)}_{[-\ell]} }{}$.

	\noindent (2). Combining (\ref{eqn:ERM_first_order}) and (\ref{eqn:ERM_first_order_loo_predictor}), we have for some $\beta_{[-\ell]} \in [0,1-\eta/2\kappa]^n$, with $D_{[-\ell]}\equiv \mathrm{diag}(\beta_{[-\ell]})$, 
	\begin{align*}
	\hat{\mu}-\hat{\mu}_{[-\ell]}&=D_{[-\ell]}\Big(\big(I-\eta\hat{\Sigma}\big)(\hat{\mu}-\hat{\mu}_{[-\ell]})-\eta\Delta_{[-\ell]}\Big),
	\end{align*}
	where
	\begin{align*}
	\Delta_{[-\ell]}&\equiv \big(\hat{\Sigma}-\hat{\Sigma}_{[-\ell]}\big)\hat{\mu}_{[-\ell]}- \big(A^\top Y-A^\top_{[-\ell]}Y_{[-\ell]}\big)\\
	& = (A^\top A-A_{[-\ell]}^\top A_{[-\ell]}) \big(\hat{\mu}_{[-\ell]}-\mu_0\big)- (A-A_{[-\ell]})^\top \xi\\
	& = -A_{[-\ell]}^\top A_{\cdot,\ell}\cdot \mu_{0,\ell}+ e_\ell \Big(A_{\cdot,\ell}^\top A_{[-\ell]}\big(\hat{\mu}_{[-\ell]}-\mu_0\big)-\pnorm{A_{\cdot,\ell}}{}^2\cdot \mu_{0,\ell}- A_{\cdot, \ell}^\top \xi \Big).
	\end{align*}
	Here in the last line we used
	\begin{align}\label{ineq:ERM_loo_predictor_error_0}
	A^\top A-A_{[-\ell]}^\top A_{[-\ell]}& = A_{[-\ell]}^\top A_{\cdot,\ell} e_\ell^\top + e_\ell A_{\cdot,\ell}^\top A_{[-\ell]}+ \pnorm{A_{\cdot,\ell}}{}^2 e_\ell e_\ell^\top.
	\end{align}
	Using the independence of $A_{\cdot,\ell}$ and $A_{[-\ell]}\hat{\mu}_{[-\ell]}$, for $x\geq 1$, with probability at least $1- C_1 e^{-(x^2\wedge n)/C_1}$, we have $\pnorm{\Delta_{[-\ell]}}{}\leq C_1 L_\xi\cdot x$. Therefore on the same event, by possibly adjusting $C_1>0$, 
	\begin{align}\label{ineq:ERM_loo_predictor_error_1}
	\abs{\hat{\mu}_\ell}\leq \pnorm{\hat{\mu}-\hat{\mu}_{[-\ell]}}{}\leq C_1 L_\xi\cdot x.
	\end{align}
	The case $n/m\leq {1-\epsilon}$ without strong convexity of $\mathsf{f}$ follows similarly as above. 
	
	Now let us write alternatively
	\begin{align}\label{ineq:ERM_loo_predictor_error_2}
	\hat{\mu}-\hat{\mu}_{[-\ell]}&=D_{[-\ell]}\Big(\big(I-\eta\hat{\Sigma}_{[-\ell]}\big)(\hat{\mu}-\hat{\mu}_{[-\ell]})-\eta\Xi_{[-\ell]}\Big),
	\end{align}
	where
	\begin{align*}
	\Xi_{[-\ell]}&\equiv \big(\hat{\Sigma}-\hat{\Sigma}_{[-\ell]}\big)\hat{\mu}- \big(A^\top Y-A^\top_{[-\ell]}Y_{[-\ell]}\big)\\
	& =  A_{[-\ell]}^\top A_{\cdot,\ell}\cdot (\hat{\mu}_\ell-\mu_{0,\ell})+ e_\ell \Big(A_{\cdot,\ell}^\top A_{[-\ell]}\big(\hat{\mu}-\mu_0\big)+\pnorm{A_{\cdot,\ell}}{}^2\cdot (\hat{\mu}_\ell-\mu_{0,\ell})- A_{\cdot, \ell}^\top \xi \Big).
	\end{align*}
	Iterating (\ref{ineq:ERM_loo_predictor_error_2}), with $\hat{M}_{[-\ell],s}\equiv \big[D_{[-\ell]}\big(I-\eta\hat{\Sigma}_{[-\ell]}\big)\big]^s D_{[-\ell]}$, for any $t \in \N$, 
	\begin{align}\label{ineq:ERM_loo_predictor_error_2_0}
	\hat{\mu}-\hat{\mu}_{[-\ell]}& = \big[D_{[-\ell]}\big(I-\eta\hat{\Sigma}_{[-\ell]}\big)\big]^t(\hat{\mu}-\hat{\mu}_{[-\ell]}) -\eta \sum_{s \in [0:t-1]} \hat{M}_{[-\ell],s}\Xi_{[-\ell]}.
	\end{align}
	Note that for any $k\neq \ell$, using that $\hat{\Sigma}_{[-\ell]}e_\ell=0$, 
	\begin{align*}
	& e_k^\top \hat{M}_{[-\ell],s}\Xi_{[-\ell]} = e_k^\top \hat{M}_{[-\ell],s} A_{[-\ell]}^\top A_{\cdot,\ell}\cdot (\hat{\mu}_\ell-\mu_{0,\ell}).
	\end{align*}
	By subgaussian inequality applied to $\iprod{A_{\cdot,\ell}}{\cdot}$, the exponentially high probability boundedness of $\pnorm{\hat{\Sigma}_{[-\ell]} }{}$ and the estimate (\ref{ineq:ERM_loo_predictor_error_1}), for $\eta\leq 1/C_2$ and $x\geq 1$, with probability at least $1-C_2e^{-(x^2\wedge n)/C_2}$,
	\begin{align}\label{ineq:ERM_loo_predictor_error_3}
	\abs{e_k^\top\hat{M}_{[-\ell],s} \Xi_{[-\ell]}}&\leq \frac{CL_\xi x^2}{\sqrt{n}} \pnorm{ A_{[-\ell] } \hat{M}_{[-\ell],s}^\top e_k}{}\leq \frac{C_2 L_\xi x^2}{\sqrt{n}}(1-\eta/2\kappa)^{s}.
	\end{align}
	Consequently, for $\eta\leq 1/C_3$ and $x\geq 1$, with probability at least $1-C_3 te^{-(x^2\wedge n)/C_3}$, for any $k\neq \ell$ and $t \in \N$,
	\begin{align*}
	\abs{(\hat{\mu}-\hat{\mu}_{[-\ell]})_k}&\leq (1-\eta/2\kappa)^t \pnorm{\hat{\mu}-\hat{\mu}_{[-\ell]}}{ }+\frac{C\eta L_\xi x^2}{\sqrt{n}} \sum_{s \in [0:t-1]} (1-\eta/2\kappa)^s\\
	&\leq  C_3 L_\xi x^2 \cdot \big[(1-\eta/2\kappa)^t + n^{-1/2} \big]. 
	\end{align*}
	Choosing $t=C_4\log n$ for a large enough $C_4>0$ to conclude via a union bound over $k\in [n]\setminus \{\ell\}$. The case $k=\ell$ is already included in (\ref{ineq:ERM_loo_predictor_error_1}).

	On the other hand, by multiplying on the left with $\hat{\Sigma}_{[-\ell]}$ on both sides of the equation (\ref{ineq:ERM_loo_predictor_error_2_0}) and using (almost) the same subsequent arguments, we have for $\eta\leq 1/C_5$ and $x\geq 1$, with probability at least $1-C_5 te^{-(x^2\wedge n)/C_5}$, for any $k\neq \ell$ and $t \in \N$,
	\begin{align}\label{ineq:ERM_loo_predictor_error_4}
	\bigabs{\big(\hat{\Sigma}_{[-\ell]}(\hat{\mu}-\hat{\mu}_{[-\ell]})\big)_k}&\leq  C_5 L_\xi x^2 \cdot \big[(1-\eta/2\kappa)^t + n^{-1/2} \big]. 
	\end{align}
	The bound trivially holds for $k=\ell$ as the left hand side is simply $0$. Finally the claimed entrywise control for $\hat{\Sigma}\hat{\mu}-\hat{\Sigma}_{[-\ell]}\hat{\mu}_{[-\ell]}$ follows from the above display (\ref{ineq:ERM_loo_predictor_error_4}) and the decomposition
	\begin{align*}
	\hat{\Sigma}\hat{\mu}-\hat{\Sigma}_{[-\ell]}\hat{\mu}_{[-\ell]}&= \hat{\Sigma}_{[-\ell]}(\hat{\mu}-\hat{\mu}_{[-\ell]})+ A_{[-\ell]}^\top A_{\cdot,\ell}\cdot \hat{\mu}_{\ell}  +e_\ell\big(A_{\cdot,\ell}^\top A_{[-\ell]}\hat{\mu}+\pnorm{A_{\cdot,\ell}^2}{}^2 \hat{\mu}_\ell\big),
	\end{align*}
	by noting that the scalar term $A_{\cdot,\ell}^\top A_{[-\ell]}\hat{\mu}=A_{\cdot,\ell}^\top A_{[-\ell]}(\hat{\mu}-\hat{\mu}_{[-\ell]})+A_{\cdot,\ell}^\top A_{[-\ell]}\hat{\mu}_{[-\ell]}$ in the above display can be handled via (\ref{ineq:ERM_loo_predictor_error_1}).

	Entrywise control for ${\mu}^{(t)}-{\mu}_{[-\ell]}^{(t)}$ and $\hat{\Sigma}{\mu}^{(t)}-\hat{\Sigma}_{[-\ell]}{\mu}_{[-\ell]}^{(t)}$ follows from (almost) the same proof above with minor modifications, by starting from (\ref{def:ERM_grad_descent}) and (\ref{eqn:ERM_grad_descent_loo_predictor}).
\end{proof}

Using the above lemma, we may now prove delocalization of $\mu^{(t)}-\hat{\mu}$. 

\begin{proposition}\label{prop:ERM_Delta_r_l2}
	Suppose the following hold.
	\begin{enumerate}
		\item For some $K>1$, $m/n \in [1/K,K]$, $\pnorm{\mu_0}{\infty}\leq K$, and $A=A_0/\sqrt{m}$ where $A_0$ is an $m\times n$ random matrix whose entries are independent mean $0$ variables with $\max_{i,j} \pnorm{A_{0,ij}}{\psi_2}\leq K$.
		\item The function $\mathsf{f}\geq \mathsf{f}(0)=0$, $\mathsf{f}\in C^2(\R)$ and $\pnorm{\mathsf{f}^{(3)}}{\infty}\vee \pnorm{(\mathsf{f}^{(2)})^{-1}}{\infty}\leq \kappa$ for some $ \kappa>1$.
	\end{enumerate}
	Then for any $D>0$, there exist some $C_0=C_0(K,\kappa,D)>1$ and $\delta_0 =\delta_0(K,\kappa) \in (0,1/2)$ such that with $L_\xi\equiv 1+m^{-1/2}\pnorm{\xi}{}$,
	\begin{align*}
	\sup_{\eta\leq 1/C_0}\sup_{t\leq C_0\log n}\Prob^\xi\Big(\pnorm{\mu^{(t)}-\hat{\mu}  }{\infty}\geq C_0L_\xi^2\log n \cdot (1-\delta_0\eta)^{t} \Big)\leq C_0n^{-D}.
	\end{align*}
	If $n/m\leq 1-\epsilon$ for some $\epsilon \in (0,1)$, then the above estimate holds without the condition $\pnorm{(\mathsf{f}^{(2)})^{-1}}{\infty}\leq \kappa$, with the constants $C_0,\delta$ in the above displays replaced further by some  $C_1=C_1(K,\epsilon,D)>0,\delta_1 =\delta_1(K,\epsilon)>0$. 
\end{proposition}

\begin{proof}
	Let $
	r^{(t)}\equiv \mu^{(t)}-\hat{\mu}$, and 
	\begin{align*}
	\alpha^{(t)}\equiv \int_0^1 \prox_{\eta \mathsf{f}_n}' \Big[\big(I-\eta\hat{\Sigma}\big)\big( s \mu^{(t)}+ (1-s) \hat{\mu}\big)+\eta\cdot  A^\top Y\Big]\,\d{s}\in \R^n.
	\end{align*}
	Then combining (\ref{eqn:ERM_first_order}) and (\ref{def:ERM_grad_descent}), we obtain
	\begin{align}\label{eqn:ERM_r_recursion}
	r^{(t+1)}&= \mathrm{diag}(\alpha^{(t)})\big(I-\eta\hat{\Sigma}\big) r^{(t)}.
	\end{align}
	Similarly let $
	r^{(t)}_{[-\ell]}\equiv \mu^{(t)}_{[-\ell]}-\hat{\mu}_{[-\ell]}$, and
	\begin{align*}
	\alpha^{(t)}_{[-\ell]}\equiv \int_0^1 \prox_{\eta \mathsf{f}_n}' \Big[\big(I-\eta\hat{\Sigma}_{[-\ell]}\big)\big( s \mu^{(t)}_{[-\ell]}+ (1-s) \hat{\mu}_{[-\ell]}\big)+\eta\cdot  A^\top_{[-\ell]} Y_{[-\ell]}\Big]\,\d{s}\in \R^n.
	\end{align*}
	Then we have
	\begin{align}\label{eqn:ERM_r_recursion_loo_predictor}
	r^{(t+1)}_{[-\ell]}&=\mathrm{diag}(\alpha^{(t)}_{[-\ell]})\big(I-\eta\hat{\Sigma}_{[-\ell]}\big) r^{(t)}_{[-\ell]}.
	\end{align}	
	Now combining (\ref{eqn:ERM_r_recursion}) and (\ref{eqn:ERM_r_recursion_loo_predictor}), with $\Delta r^{(t)}_{[-\ell]}\equiv r^{(t)}-r^{(t)}_{[-\ell]}$, 
	\begin{align*}
	\Delta r^{(t+1)}_{[-\ell]} &=\mathrm{diag}\big(\alpha^{(t)}-\alpha_{[-\ell]}^{(t)}\big) \big(I-\eta \hat{\Sigma}_{[-\ell]}\big) r_{[-\ell]}^{(t)}\\
	&\qquad +\mathrm{diag}(\alpha^{(t)})\Big[(1-\eta \hat{\Sigma})\Delta r^{(t)}_{[-\ell]}-\eta (\hat{\Sigma}-\hat{\Sigma}_{[-\ell]}) r_{[-\ell]}^{(t)}  \Big].
	\end{align*}
	Consequently, by applying (\ref{ineq:ERM_loo_predictor_error_0}) to the last term in the above display, for $\eta\leq 1/C_1$, with probability at least $1-C_1e^{-n/C_1}$,
	\begin{align}\label{ineq:ERM_Delta_r_l2_1}
	\bigpnorm{\Delta r^{(t+1)}_{[-\ell]}}{}&\leq \max_{k\neq \ell} \bigabs{\big(\alpha^{(t)}-\alpha_{[-\ell]}^{(t)}\big)_k  }\cdot  \pnorm{r_{[-\ell]}^{(t)}}{}\nonumber\\
	&\qquad + (1-\eta/2\kappa) \bigpnorm{\Delta r^{(t)}_{[-\ell]}}{}+ \eta  \cdot \bigabs{ A_{\cdot,\ell}^\top A_{[-\ell]} r_{[-\ell]}^{(t)} }.
	\end{align}
	Using $\pnorm{\mathsf{f}^{(3)}}{\infty}\leq \kappa$, so $\pnorm{\prox^{(2)}_{\eta \mathsf{f}}}{\infty}\leq \eta \pnorm{\mathsf{f}^{(3)}}{\infty}\leq \eta \kappa$, we have
	\begin{align*}
	&\max_{k\neq \ell}\abs{\big(\alpha^{(t)}-\alpha_{[-\ell]}^{(t)}\big)_k  }\lesssim \max_{k\neq \ell}\Big[\abs{(\hat{\mu}-\hat{\mu}_{[-\ell]})_k}+\bigabs{\big(\hat{\Sigma}\hat{\mu}-\hat{\Sigma}_{[-\ell]}\hat{\mu}_{[-\ell]}\big)_k}  \Big]\\
	&\qquad + \max_{k\neq \ell}\Big[\abs{({\mu}^{(t)}-{\mu}^{(t)}_{[-\ell]})_k}+\bigabs{\big(\hat{\Sigma}{\mu}^{(t)}-\hat{\Sigma}_{[-\ell]}{\mu}^{(t)}_{[-\ell]}\big)_k}  \Big]+\pnorm{A_{[-\ell]}^\top A_{\cdot,\ell} }{\infty}.
	\end{align*}
	So using the second claim of Lemma \ref{lem:ERM_loo_predictor_error} and a simple large deviation estimate for $\pnorm{A_{[-\ell]}^\top A_{\cdot,\ell} }{\infty}$, for $\eta\leq 1/C_2$ and $t\leq C_2 \log n$, with probability at least $1-C_2 n^{-D}$, 
	\begin{align}\label{ineq:ERM_Delta_r_l2_2}
	\max_{k\neq \ell}\abs{\big(\alpha^{(t)}-\alpha_{[-\ell]}^{(t)}\big)_k  }\leq C_2 L_\xi\cdot n^{-1/2}\log n.
	\end{align}
	On the other hand, using the independence of $A_{\cdot,\ell}$ and  $A_{[-\ell]} r_{[-\ell]}^{(t)} $, by subgaussian inequality followed by exponentially high probability boundedness of $\pnorm{A_{[-\ell]}}{\op}$, with probability at least $1-C n^{-D}$, 
	\begin{align*}
	\bigabs{ A_{\cdot,\ell}^\top A_{[-\ell]} r_{[-\ell]}^{(t)} }&\leq C\sqrt{\frac{\log n}{n}}\bigpnorm{ A_{[-\ell]} r_{[-\ell]}^{(t)} }{ }\leq C' \sqrt{\frac{\log n}{n}} \pnorm{   r_{[-\ell]}^{(t)}}{} \\
	&\leq C' \sqrt{\log n}\cdot (1-\eta/2\kappa)^t\cdot \frac{ \pnorm{\hat{\mu}_{[-\ell]}}{} }{\sqrt{n}}.
	\end{align*}
	Using the first claim of Lemma \ref{lem:ERM_loo_predictor_error}, for $\eta\leq 1/C_3$ and $t\leq C_3\log n$, with probability at least $1-C_3 n^{-D}$, 
	\begin{align}\label{ineq:ERM_Delta_r_l2_3}
	\bigabs{ A_{\cdot,\ell}^\top A_{[-\ell]} r_{[-\ell]}^{(t)} }\leq C_3 L_\xi \sqrt{\log n}\cdot (1-\eta/2\kappa)^t.
	\end{align}
	Combining (\ref{ineq:ERM_Delta_r_l2_1})-(\ref{ineq:ERM_Delta_r_l2_3}), for $\eta\leq 1/C_4$, with probability at least $1-C_4 n^{-D}$, uniformly in $t\leq C_4\log n$, 
	\begin{align*}
	\bigpnorm{\Delta r^{(t+1)}_{[-\ell]}}{}&\leq C_4 L_\xi^2 \log n\cdot  (1-\eta/2\kappa)^t + (1-\eta/2\kappa) \bigpnorm{\Delta r^{(t)}_{[-\ell]}}{}.
	\end{align*}
	Iterating the above inequality to conclude that for $\eta\leq 1/C_5$, with probability at least $1-C_5 n^{-D}$, uniformly in $t\leq C_5\log n$, 
	\begin{align*}
	\abs{r^{(t)}_\ell}\leq \bigpnorm{\Delta r^{(t)}_{[-\ell]}}{}&\leq C_5 L_\xi^2 \log n\cdot  t (1-\eta/2\kappa)^{t},
	\end{align*}
	proving the claim. 
\end{proof}

\begin{proof}[Proof of Theorem \ref{thm:universality_ls}]
	The proximal gradient descent algorithm (\ref{def:ERM_grad_descent}) can be reformulated in the form of (\ref{def:GFOM_asym}) via the following identification: Let $\mu^{(t)}\equiv \prox_{\eta \mathsf{f}_n}(v^{(t)})$, where
	\begin{align*}
	\begin{cases}
	u^{(t)}= A \big(\prox_{\eta \mathsf{f}_n}(v^{(t-1)})-\mu_0\big)\in \R^m,\\
	v^{(t)}= A^\top \big[\eta(\xi-u^{(t)})\big]+\prox_{\eta \mathsf{f}_n}(v^{(t-1)})\in \R^n,
	\end{cases}
	\end{align*}
	with the initialization $u^{(0)}=0_m,v^{(0)}=0_n$. Now using Theorem \ref{thm:universality_asym} with $\Psi_n\equiv \Psi\circ \prox_{\eta \mathsf{f}_n}$, we have for any $t \in \N$, for $\eta\leq 1/C_1$,
	\begin{align*}
	&\max_{j \in [n]}\bigabs{\E^\xi \big[\Psi\big(\mu^{(t)}_j(A)\big)|(u^{(0)},v^{(0)})\big]-  \E^\xi \big[\Psi\big(\mu^{(t)}_j(B)\big)|(u^{(0)},v^{(0)})\big] }\\
	&\leq \big(C_1 (1+\pnorm{\xi}{\infty})\log n\big)^{C_1 t^3}\cdot n^{-1/2}.
	\end{align*}
	On the other hand, using Proposition \ref{prop:ERM_Delta_r_l2}, there exists some $C_2>1,\delta_0\in (0,1/2)$ such that for $\eta\leq 1/C_2$ and $M \in \{A,B\}$, 
	\begin{align*}
	&\max_{j \in [n]} \bigabs{\E^\xi \big[\Psi\big(\mu^{(t)}_j(M)\big)|(u^{(0)},v^{(0)})\big]-\E^\xi\big[\Psi\big(\hat{\mu}_j(M)\big)\big]}\\
	&\leq C_2(1+m^{-1/2}\pnorm{\xi}{})^2  \log n\cdot (1-\delta_0 \eta)^t+ C_2 n^{-100}. 
	\end{align*}
	Combining the above two displays and taking expectation over $\xi$ by using $\pnorm{\xi_1}{\psi_2}\leq K$, there exists some constant $C_3>0$ such that for all $\eta\leq 1/C_3$ and $t \leq \log n$, 
	\begin{align*}
	&\max_{j \in [n]} \bigabs{\E\big[\Psi\big(\hat{\mu}_j(A)\big)\big]-  \E\big[\Psi\big(\hat{\mu}_j(B)\big)\big]  }\\
	&\leq C_3\log n\cdot (1-\delta_0 \eta)^t+   \big(C_3 \log n\big)^{C_3 t^3}\cdot n^{-1/2}.
	\end{align*}
	Finally we choose $\eta=1/C_3$ and $t=(\log n)^{1/4}$ to conclude. 
\end{proof}

\subsection{Proof of Theorem \ref{thm:universality_logistic}}\label{section:proof_ERM_logistic}

Let $\varphi \in C^\infty(\R)$ be such that $\varphi \in [0,1]$, $\varphi(x)|_{x\geq 1}=1$ and $\varphi(x)|_{x\leq -1}=0$. For any $\sigma>0$, let $\varphi_\sigma(x)\equiv \varphi(x/\sigma)$, and $\mathsf{L}_\sigma(x,y,\xi)\equiv \rho\big(-(2\varphi_\sigma(y+\xi)-1)x\big)$. We then define $\hat{\mu}_\sigma, \mu^{(t+1)}_\sigma$ similarly as in (\ref{def:ERM_logit_equiv}) and (\ref{def:ERM_grad_descent_logit}), with $\mathsf{L}$ replaced by $\mathsf{L}_\sigma$.  By convention, we shall write $\sigma_0(x)=\bm{1}_{x\geq 0}$ and $\mathsf{L}_0\equiv \mathsf{L}$. 

The following lemma establishes the $\ell_2$ boundedness of $\hat{\mu}_\sigma$ and the geometric decay of $\pnorm{\mu^{(t)}_\sigma-\hat{\mu}_\sigma }{}$, for any smoothing level $\sigma\geq 0$.

\begin{lemma}\label{lem:logistic_error}
	Suppose the following hold.
	\begin{enumerate}
		\item For some $K>1$, $m/n \in [1/K,K]$, $\pnorm{\mu_0}{}/\sqrt{n}\leq K$, and $A=A_0/\sqrt{m}$ where $A_0$ is an $m\times n$ random matrix whose entries are independent mean $0$ variables with $\max_{i,j} \pnorm{A_{0,ij}}{\psi_2}\leq K$.
		\item The regularizer $\mathsf{f}\geq \mathsf{f}(0)=0$ is $\kappa$-strongly convex for some $\kappa>0$. 
	\end{enumerate}
	Then for any $D>0$, there exist some $C_0=C_0(K,\kappa,D)>1$ and $\delta_0 =\delta_0(K,\kappa) \in (0,1/2)$ such that for all $\sigma\geq 0$ and $t \leq e^{n/C_0}$,
	\begin{align*}
	&\sup_{\eta\leq 1/C_0} \Prob\bigg(\frac{\pnorm{\mu^{(t)}_\sigma-\hat{\mu}_\sigma }{}}{\sqrt{n} (1-\delta_0\eta)^t}\vee \frac{\pnorm{\hat{\mu}_\sigma}{} }{\sqrt{n}}\geq C_0 \sqrt{\log n} \bigg)\leq C_0 n^{-D}.
	\end{align*}
\end{lemma}
\begin{proof}
	Note that $\rho'(x)=1/(1+e^{-x})$, $\rho''(x)=e^x/(1+e^x)^2$, and
	\begin{align*}
	\partial_1\mathsf{L}_\sigma(x,y;\xi)&=-(2\varphi_\sigma(y+\xi)-1)\rho'\big(-(2\varphi_\sigma(y+\xi)-1)x\big)\in [-1,1],\\
	\partial_1^2 \mathsf{L}_\sigma (x,y;\xi)&=(2\varphi_\sigma(y+\xi)-1)^2\rho''\big(-(2\varphi_\sigma(y+\xi)-1)x\big) \in [0,1].
	\end{align*}
	Using (\ref{eqn:ERM_first_order_logit})-(\ref{def:ERM_grad_descent_logit}) with $\mathsf{L}$ replaced by $\mathsf{L}_\sigma$, and Lemma \ref{lem:prox_Lip}, for some $\alpha_\sigma^{(t)}\in [0,1/(1+\eta\kappa)]^n$ and $\beta_\sigma^{(t)}\in [0,1]^m$,
	\begin{align*}
	\mu^{(t+1)}_\sigma-\hat{\mu}_\sigma&=\mathrm{diag}(\alpha_\sigma^{(t)})\bigg(\mu^{(t)}_\sigma-\hat{\mu}_\sigma- \eta \sum_{i \in [m]} \beta^{(t)}_{\sigma;i} A_i A_i^\top \big(\mu^{(t)}_\sigma-\hat{\mu}_\sigma\big) \bigg)\\
	& = \mathrm{diag}(\alpha^{(t)}_\sigma)\bigg(I-\eta \sum_{i \in [m]} \beta^{(t)}_{\sigma;i} A_i A_i^\top \bigg)\big(\mu^{(t)}_\sigma-\hat{\mu}_\sigma\big).
	\end{align*}
	Consequently, for $\eta\leq 1/C_1$, with probability at least $1-C_1 e^{-n/C_1}$, 
	\begin{align*}
	\pnorm{\mu^{(t+1)}_\sigma-\hat{\mu}_\sigma}{} \leq \frac{1}{1+\eta\kappa} \pnorm{\mu^{(t)}_\sigma-\hat{\mu}_\sigma}{}.
	\end{align*}
	Iterating the bound, we have for the same range of $\eta$, with probability at least $1-C_1 te^{-n/C_1}$, 
	\begin{align}\label{ineq:logistic_error_1}
	\frac{\pnorm{\mu^{(t)}_\sigma-\hat{\mu}_\sigma}{}}{\sqrt{n}} \leq \frac{1}{(1+\eta\kappa)^t} \cdot  \frac{\pnorm{\hat{\mu}_\sigma}{}}{\sqrt{n}}.
	\end{align}
	Below we shall provide a bound on $\pnorm{\hat{\mu}_\sigma}{}/\sqrt{n}$. Using  (\ref{eqn:ERM_first_order_logit}) with $\mathsf{L}$ replaced by $\mathsf{L}_\sigma$, Lemma \ref{lem:prox_Lip} and $\prox_{\eta \mathsf{f}_n}(0)=0$, for some $\alpha_\sigma \in [0,1/(1+\eta\kappa)]^n$ and $\beta_\sigma \in [0,1]^m$,
	\begin{align*}
	\hat{\mu}_\sigma=\mathrm{diag}(\alpha_\sigma)\bigg[\bigg(I-\eta \sum_{i \in [m]} \beta_{\sigma;i} A_i A_i^\top \bigg) \hat{\mu}_\sigma-\eta \sum_{i \in [m]} A_i \partial_1 \mathsf{L}_\sigma(0,A_i^\top \mu_0;\xi_i)\bigg].
	\end{align*}
	So for $\eta= 1/C_2$, with probability at least $1-C_2 e^{-n/C_2}$, 
	\begin{align}\label{ineq:logistic_error_2}
	\pnorm{\hat{\mu}_\sigma}{}&\leq \frac{1}{1+\kappa/C_2}\pnorm{\hat{\mu}_\sigma}{}+\frac{1}{2C_2} \biggpnorm{\sum_{i \in [m]} A_i \big(2\varphi_\sigma(A_i^\top \mu_0+\xi_i)-1\big) }{}.
	\end{align}
	Note that with $\mathsf{F}_\sigma(t)\equiv \E_\xi \varphi_\sigma (t+\xi)$, 
	\begin{align}\label{ineq:logistic_error_3}
	&\biggpnorm{\sum_{i \in [m]} A_i \big(2 \varphi_\sigma(A_i^\top \mu_0+\xi_i)-1\big) }{}^2\lesssim \sum_{j \in [n]} \bigg(\sum_{i \in [m]} A_{ij}\big(\varphi_\sigma(A_i^\top \mu_0+\xi_i)-\mathsf{F}_\sigma(A_i^\top \mu_0)\big) \bigg)^2\nonumber\\
	&\qquad + \sum_{j \in [n]}\bigg(\sum_{i \in [m]} A_{ij}\big(2\mathsf{F}_\sigma(A_i^\top \mu_0)-1 \big)\bigg)^2 \equiv S_1+S_2. 
	\end{align}
	
	We first handle $S_1$. For fixed $j \in [n]$, $\big\{D_{ij}\equiv A_{ij} \big(\varphi_\sigma(A_i^\top \mu_0+\xi_i)-\mathsf{F}_\sigma(A_i^\top \mu_0)\big): i \in [m]\big\}$ are centered, independent variables with $\max_{i \in [m]} \pnorm{n^{1/2} D_{ij}}{\psi_2}\lesssim K$. By subgaussian concentration and a union bound, we conclude that with probability at least $1-C n^{-D}$, $S_1\leq C n \log n$. 
	
	Next we handle $S_2$. To this end, note that for any $t \in \R$,
	\begin{align*}
	\abs{\mathsf{F}_\sigma'(t)}& = \frac{1}{\sigma}\biggabs{\E \varphi'\bigg(\frac{t+\xi}{\sigma}\bigg)}\leq \frac{\pnorm{\varphi'}{\infty}}{\sigma}\Prob\big(\xi \in [-t\pm \sigma]\big) \leq \pnorm{\varphi'}{\infty}. 
	\end{align*}
	This means that for fixed $j \in [n]$, for some (random) $\tau_j \in [\pm \pnorm{\varphi'}{\infty}]$,
	\begin{align*}
	\mathsf{F}_\sigma(A_i^\top \mu_0)&= \mathsf{F}_\sigma(A_{i;[-j]}^\top \mu_0+ A_{ij}\mu_{0,j})= \mathsf{F}_\sigma\big(A_{i;[-j]}^\top \mu_0\big)+ \tau_j A_{ij}\mu_{0,j},
	\end{align*}
	and therefore 
	\begin{align*}
	S_2&\lesssim \sum_{j \in [n]}\bigg(\sum_{i \in [m]} A_{ij}\bigg)^2+ \sum_{j \in [n]}\bigg(\sum_{i \in [m]} A_{ij}\mathsf{F}_\sigma\big(A_{i;[-j]}^\top \mu_0\big)\bigg)^2+ \sum_{j \in [n]}\bigg(\sum_{i \in [m]} A_{ij}^2\bigg)^2 \mu_{0,j}^2.
	\end{align*}
	For the first two terms above, we may use subgaussian inequality and a union bound to conclude that with probability at least $1-Cn^{-D}$, these terms are bounded by $C n\log n$. For the third term, a standard large deviation bound yields that $\sum_{i \in [m]} A_{ij}^2\leq C$ holds for all $j \in [n]$ with probability at least $1-Cn^{-D-1}$, so a union bound concludes that the third term is bounded by $Cn$ with probability at least $1-Cn^{-D}$. So with probability at least $1-Cn^{-D}$, $S_2\leq C n\log n$. 
	
	Combining the above arguments with (\ref{ineq:logistic_error_3}), we have with probability at least $1-C_3 n^{-D}$,
	\begin{align*}
	\biggpnorm{\sum_{i \in [m]} A_i \big(2 \varphi_\sigma(A_i^\top \mu_0+\xi_i)-1\big) }{}\leq C_3 \sqrt{n\log n}.  
	\end{align*}
	Combined with (\ref{ineq:logistic_error_2}), with probability at least $1-C_4 n^{-D}$,
	\begin{align}\label{ineq:logistic_error_4}
	\frac{\pnorm{\hat{\mu}_\sigma}{}}{\sqrt{n}}\leq C_4 \sqrt{\log n}.
	\end{align}
	It is easy to verify the same estimate for $\sigma=0$. The claim now follows from (\ref{ineq:logistic_error_1}). 
\end{proof}

Next we prove that $\hat{\mu}_\sigma$ is close to $\hat{\mu}$ in $\ell_2$ for small smoothing parameters $\sigma>0$.

\begin{lemma}\label{lem:logistic_smooth_error}
	Suppose the following hold.
	\begin{enumerate}
		\item For some $K>1$, $m/n \in [1/K,K]$, $\pnorm{\mu_0}{}/\sqrt{n}\leq K$, and $A=A_0/\sqrt{m}$ where $A_0$ is an $m\times n$ random matrix whose entries are independent mean $0$ variables with $\max_{i,j} \pnorm{A_{0,ij}}{\psi_2}\leq K$.
		\item The regularizer $\mathsf{f}\geq \mathsf{f}(0)=0$ is $\kappa$-strongly convex for some $\kappa>0$. 
	\end{enumerate}
	Then for any $D>0$, there exists some $C_0=C_0(K,\kappa,D)>1$ such that if $\sigma\geq \log n/n$, with probability at least $1-C_0 n^{-D}$,
	\begin{align*}
	\frac{ \pnorm{\hat{\mu}-\hat{\mu}_\sigma}{} }{\sqrt{n}}\leq C_0 (\sigma \log n)^{1/4}. 
	\end{align*}
\end{lemma}
\begin{proof}
	Let for $\sigma\geq 0$
	\begin{align*}
	\mathcal{L}_\sigma(\mu)\equiv \frac{1}{n}\sum_{i \in [m]} \mathsf{L}_\sigma(A_i^\top \mu, A_i^\top \mu_0;\xi_i)+ \frac{1}{n} \sum_{j \in [n]} \mathsf{f}(\mu_j). 
	\end{align*}
	Using the optimality of $\hat{\mu}_\sigma$, we have
	\begin{align}\label{ineq:logistic_smooth_error_1}
	\mathcal{L}_\sigma(\hat{\mu}_\sigma)&\leq \mathcal{L}_\sigma(\hat{\mu})\leq \mathcal{L}_0(\hat{\mu})+\err_{1;\sigma},
	\end{align}
	where
	\begin{align*}
	\err_{1;\sigma}&\equiv \frac{1}{n}\sum_{i \in [m]} \abs{(\mathsf{L}_\sigma-\mathsf{L}_0)(A_i^\top \hat{\mu}, A_i^\top \mu_0;\xi_i)}\\
	&\leq \frac{2\pnorm{\rho'}{\infty}}{n}\sum_{i \in [m]} \abs{A_i^\top \hat{\mu}}\cdot\abs{(\varphi_\sigma-\varphi_0)(A_i^\top \mu_0+\xi_i)}\\
	&\leq 2\pnorm{A}{\op}\cdot \frac{ \pnorm{\hat{\mu}}{} }{\sqrt{n}}\cdot \bigg(\frac{1}{n}\sum_{i \in [m]}\bm{1}\big(\abs{A_i^\top \mu_0+\xi_i}\leq \sigma\big)\bigg)^{1/2}.
	\end{align*}
	Now using (i) subgaussian inequality for $\pnorm{A}{\op}$ and Bernstein's inequality for $\sum_{i \in [m]}\bm{1}(\abs{A_i^\top \mu_0+\xi_i})$, and (ii) the bound for $\pnorm{\hat{\mu}}{}$ in Lemma \ref{lem:logistic_error}, with $\phi(\sigma)\equiv \Prob\big(A_1^\top \mu_0+\xi_1\in [\pm \sigma] \big)$, we have with probability at least $1-C_1 n^{-D}$, if $\sigma\geq \log n/n$,
	\begin{align*}
	\err_{1;\sigma}&\lesssim \sqrt{\log n}\cdot \Big(\phi(\sigma)+n^{-1/2}\sqrt{\log n}\cdot \phi^{1/2}(\sigma)+n^{-1}\log n\Big)^{1/2}\leq C_1 \sqrt{\sigma \log n}.
	\end{align*}
	Combined with (\ref{ineq:logistic_smooth_error_1}), with the same probability estimate, we have 
	\begin{align}\label{ineq:logistic_smooth_error_2}
	\mathcal{L}_\sigma(\hat{\mu}_\sigma)&\leq \mathcal{L}_\sigma(\hat{\mu})\leq \mathcal{L}_0(\hat{\mu})+C_1\sqrt{\sigma \log n}.
	\end{align}
	On the other hand, using the optimality of $\hat{\mu}$, if $\sigma\geq \log n/n$,
	\begin{align*}
	\mathcal{L}_0(\hat{\mu})\leq \mathcal{L}_0(\hat{\mu}_\sigma)\leq \mathcal{L}_\sigma(\hat{\mu}_\sigma)+\err_{2;\sigma},
	\end{align*}
	where 
	\begin{align*}
	\err_{2;\sigma}&\equiv \frac{1}{n}\sum_{i \in [m]} \abs{(\mathsf{L}_\sigma-\mathsf{L}_0)(A_i^\top \hat{\mu}_\sigma, A_i^\top \mu_0;\xi_i)}\\
	&\leq 2\pnorm{A}{\op}\cdot \frac{ \pnorm{\hat{\mu}_\sigma}{} }{\sqrt{n}}\cdot \bigg(\frac{1}{n}\sum_{i \in [m]}\bm{1}\big(\abs{A_i^\top \mu_0+\xi_i}\leq \sigma\big)\bigg)^{1/2}.
	\end{align*}
	Using the same argument as above, if $\sigma\geq \log n/n$, we have with probability at least $1-C_2 n^{-D}$, 
	\begin{align}\label{ineq:logistic_smooth_error_3}
	\mathcal{L}_0(\hat{\mu})\leq \mathcal{L}_0(\hat{\mu}_\sigma)\leq \mathcal{L}_\sigma(\hat{\mu}_\sigma)+C_2 \sqrt{\sigma \log n}.
	\end{align}
	Combining (\ref{ineq:logistic_smooth_error_2}) and (\ref{ineq:logistic_smooth_error_3}), if $\sigma\geq \log n/n$, with probability at least $1-C_3 n^{-D}$, 
	\begin{align*}
	\abs{\mathcal{L}_0(\hat{\mu})- \mathcal{L}_0(\hat{\mu}_\sigma)  }\leq C_3\sqrt{\sigma \log n}.
	\end{align*}
	Finally using the $\kappa/n$-strong convexity of $\mathcal{L}_0$ to conclude. 
\end{proof}

Fix $k \in [m]$ and $\sigma\geq 0$. Due to technical reasons, we consider the following leave-one-sample-out version of $\hat{\mu}_\sigma$: 
\begin{align}\label{def:ERM_logit_equiv_loo}
\hat{\mu}_{\sigma;(-k)} \equiv \argmin_{\mu \in \R^n } \bigg\{\sum_{i \neq k} \mathsf{L}_\sigma(A_i^\top \mu, A_i^\top \mu_0;\xi_i)+  \sum_{j \in [n]} \mathsf{f}(\mu_j) \bigg\}.
\end{align}
Similar to (\ref{eqn:ERM_first_order_logit}), $\hat{\mu}_{\sigma;(-k)}$ satisfies the following first-order condition: for any $\eta>0$, 
\begin{align}\label{eqn:ERM_first_order_logit_loo}
\hat{\mu}_{\sigma;(-k)} = \prox_{\eta \mathsf{f}_n}\bigg(\hat{\mu}_{\sigma;(-k)}-\eta \sum_{i \neq k} A_i \partial_1 \mathsf{L}_\sigma(A_i^\top \hat{\mu}_{\sigma;(-k)}, A_i^\top \mu_0;\xi_i) \bigg).
\end{align}
Similar to (\ref{def:ERM_grad_descent_logit}), consider the following leave-one-sample-out proximal gradient descent algorithm: for $\eta>0$, let for $t=1,2,\ldots,$
\begin{align}\label{def:ERM_grad_descent_logit_loo}
\mu^{(t)}_{\sigma;(-k)} \equiv \prox_{\eta \mathsf{f}_n}\bigg(\mu^{(t-1)}_{\sigma;(-k)}-\eta \sum_{i \neq k} A_i \partial_1 \mathsf{L}_\sigma(A_i^\top \mu^{(t-1)}_{\sigma;(-k)}, A_i^\top \mu_0;\xi_i) \bigg),
\end{align}
with the initialization $\mu^{(0)}_{\sigma;(-k)}=0$.

\begin{lemma}\label{lem:logistic_loo}
Suppose the following hold.
\begin{enumerate}
	\item For some $K>1$, $m/n \in [1/K,K]$, $\pnorm{\mu_0}{\infty}\leq K$, and $A=A_0/\sqrt{m}$ where $A_0$ is an $m\times n$ random matrix whose entries are independent mean $0$ variables with $\max_{i,j} \pnorm{A_{0,ij}}{\psi_2}\leq K$.
	\item The regularizer $\mathsf{f}\geq \mathsf{f}(0)=0$ is $\kappa$-strongly convex for some $\kappa>0$.
\end{enumerate}
Then there exists some $C_0=C_0(K,\kappa)>1$  such that for all $\sigma\geq 0$,
\begin{align*}
& \sup_{\xi \in \R^m} \bigg\{\sup_{\eta\leq 1/C_0}\sup_{t \leq e^{n/C_0}}\Prob^\xi\Big(\pnorm{\mu_\sigma^{(t)}-\mu^{(t)}_{\sigma;(-k)}}{}\geq C_0 \Big)+ \Prob^\xi\Big(\pnorm{\hat{\mu}_\sigma-\hat{\mu}_{\sigma;(-k)}}{}\geq C_0 \Big)  \bigg\}\leq C_0 e^{-n/C_0}.
\end{align*}
\end{lemma}
\begin{proof}
Let $\Delta \mu^{(t)}_{\sigma;(-k)}\equiv \mu_\sigma^{(t)}-\mu^{(t)}_{\sigma;(-k)}$. Using (\ref{def:ERM_grad_descent_logit}) with $\mathsf{L}$ replaced by $\mathsf{L}_\sigma$, (\ref{def:ERM_grad_descent_logit_loo}) and Lemma \ref{lem:prox_Lip}, for some $\alpha_{\sigma;(-k)}^{(t)}\in [0,1/(1+\eta\kappa)]^n$ and $\beta_{\sigma;(-k)}^{(t)}\in [0,1]^m$,
\begin{align*}
\Delta \mu^{(t)}_{\sigma;(-k)}&=\mathrm{diag}\big(\alpha_{\sigma;(-k)}^{(t)}\big)\bigg[\bigg(I-\eta \sum_{i \in [m]}\big(\beta_{\sigma;(-k)}^{(t)}\big)_i A_iA_i^\top \bigg) \Delta \mu^{(t-1)}_{\sigma;(-k)}\\
&\qquad -\eta \cdot A_k \partial_1 \mathsf{L}_\sigma \big(A_k^\top \mu^{(t-1)}_{\sigma;(-k)}, A_k^\top \mu_0;\xi_k\big) \bigg].
\end{align*}
Consequently, using that $\pnorm{\partial_1 \mathsf{L}_\sigma}{\infty}\leq 1$, for $\eta\leq 1/C_1$, with probability at least $1-C_1 e^{-n/C_1}$,
\begin{align*}
\pnorm{\Delta \mu^{(t)}_{\sigma;(-k)}}{}\leq \frac{1}{1+\eta\kappa} \pnorm{\Delta \mu^{(t-1)}_{\sigma;(-k)}}{}+ \eta \pnorm{A_k}{}.
\end{align*}
Iterating the bound and using the subgaussian estimate for $\pnorm{A_k}{}$ to conclude the estimate for $\pnorm{\mu_\sigma^{(t)}-\mu^{(t)}_{\sigma;(-k)}}{}$. A completely similar argument (replacing both terms $\pnorm{\Delta \mu^{(t)}_{\sigma;(-k)}}{}$ and $\pnorm{\Delta \mu^{(t-1)}_{\sigma;(-k)}}{}$ above by $\pnorm{\hat{\mu}_\sigma-\hat{\mu}_{\sigma;(-k)}}{}$) yields the desired estimate for $\pnorm{\hat{\mu}_\sigma-\hat{\mu}_{\sigma;(-k)}}{}$.
\end{proof}

\begin{proof}[Proof of Theorem \ref{thm:universality_logistic}]
	For notational simplicity, we only work with $\psi_k\equiv \psi$.
	Fix $\sigma>0$. By Lemmas \ref{lem:logistic_error} and \ref{lem:logistic_loo}, on an event $E_{0;\sigma}$ with $\Prob(E_{0;\sigma}^c)\leq C_0 n^{-D}$, for $t\leq e^{n/C_0}$,
	\begin{align*}
	\max_{k\in [m]}\abs{A_k^\top \mu^{(t)}_\sigma}\leq \max_{k\in [m]}\abs{A_k^\top \mu^{(t)}_{\sigma;(-k)} }+ \max_{k\in [m]}\pnorm{A_k}{}\cdot \pnorm{ \mu^{(t)}_\sigma -\mu^{(t)}_{\sigma;(-k)}  }{}\leq C_0\log n\equiv L_n.
	\end{align*}	
	Let $\mathsf{G}^\xi_\sigma:\R^{m\times 2}\to \R^m $ be 
	\begin{align*}
	\mathsf{G}^\xi_\sigma(u)\equiv \big(\mathsf{G}^\xi_{\sigma,k}(u_{k\cdot})\big)_{k \in [m]}\equiv  \Big(\partial_1 \mathsf{L}_\sigma\big( (u_{k1}\wedge L_n)\vee (-L_n),u_{k2};\xi_k\big)\Big)_{k \in [m]}.
	\end{align*}
	Then some calculations show that
	\begin{align*}
	\sup_{\{\xi_i\}}\max_{k \in [m]}\Big(\abs{\mathsf{G}^\xi_{\sigma,k}(0)}+\pnorm{\mathsf{G}^\xi_{\sigma,k}}{\mathrm{Lip}}\Big)\lesssim 1\vee \frac{L_n}{\sigma}.
	\end{align*}
	The ($\mu^{(t)}_\sigma$ version of the) proximal gradient descent algorithm (\ref{def:ERM_grad_descent_logit}) can be reformulated in the form of (\ref{def:GFOM_asym}) via the following identification: Consider initialization $u^{(-1)}\equiv 0_m$ and $v^{(-1)}\equiv\mu_0$ at $t=-1$. For $t=0$, let $u^{(0)}\equiv A v^{(-1)}=A\mu_0$ and $v^{(0)}\equiv 0$. For $t\geq 1$, let
	\begin{align}\label{ineq:universality_logistic_0}
	\begin{cases}
	u^{(t)} = A\prox_{\eta \mathsf{f}_n}(v^{(t-1)}) \in \R^{m},\\
	v^{(t)} = A^\top \big[-\eta \mathsf{G}^\xi_\sigma (u^{(t)},u^{(0)})\big]+\prox_{\eta \mathsf{f}_n}(v^{(t-1)}) \in \R^{n},
	\end{cases}
	\end{align}
	Then on the event $E_{0;\sigma}$, for $t\leq e^{n/C_0}$, we have $\mu^{(t)}_\sigma\equiv \prox_{\eta \mathsf{f}_n}(v^{(t)})$.

	We assume without loss of generality that $\psi(0)=0$. Now using Theorem \ref{thm:universality_asym_avg} with $\psi_n\equiv \psi\circ \prox_{\eta \mathsf{f}_n}$, by first conditionally on $\xi$ and then taking expectation, we have for $\sigma\in (0,1)$ and $t \leq e^{n/C_1}$, 
	\begin{align*}
	&\E \biggabs{\frac{1}{n}\sum_{j \in [n]} \psi\big(\mu_{\sigma,j}^{(t)}(A)\big)-\frac{1}{n}\sum_{j \in [n]} \psi\big(\mu_{\sigma,j}^{(t)}(B)\big) }\bm{1}_{E_{0;\sigma}}\leq \big(C_1 \sigma^{-1}\log n\big)^{C_1 t^3}\cdot n^{-1/C_1 t^3}.
	\end{align*}	
	Using the simple apriori estimate $\abs{n^{-1}\sum_{j \in [n]} \psi\big(\mu_{\sigma,j}^{(t)}\big)}\lesssim  \pnorm{\mu_{\sigma}^{(t)}}{}^2/n+\big(\pnorm{\mu_{\sigma}^{(t)}}{}^2/n\big)^{1/2}$ and Lemma \ref{lem:logistic_error}, for $\sigma\in (0,1)$ and $t \leq e^{n/C_1}$, 
	\begin{align}\label{ineq:universality_logistic_1}
	\E \biggabs{\frac{1}{n}\sum_{j \in [n]} \psi\big(\mu_{\sigma,j}^{(t)}(A)\big)-\frac{1}{n}\sum_{j \in [n]} \psi\big(\mu_{\sigma,j}^{(t)}(B)\big) }\leq \big(C_1 \sigma^{-1}\log n\big)^{C_1 t^3}\cdot n^{-1/C_1 t^3}. 
	\end{align}
	On the other hand, for $M \in \{A,B\}$, using Lemma \ref{lem:logistic_error}, for $\eta\leq 1/C_2$ and $t\leq e^{n/C_2}$,
	\begin{align*}
	&\E \biggabs{\frac{1}{n}\sum_{j \in [n]} \psi\big(\mu_{\sigma,j}^{(t)}(M)\big)-\frac{1}{n}\sum_{j \in [n]} \psi\big(\hat{\mu}_{\sigma,j}(M)\big) }\\
	&\lesssim \frac{1}{n}\cdot  \E \pnorm{\mu^{(t)}_\sigma(M)-\hat{\mu}_\sigma(M) }{}\cdot \big(\sqrt{n} +\E \pnorm{\mu^{(t)}_\sigma(M)}{}+ \E \pnorm{\hat{\mu}_\sigma(M)}{}\big)\\
	&\lesssim \Big( (1-\delta_0 \eta)^t \sqrt{\log n}+ b_{t,n}(M)\cdot n^{-D}\Big)\cdot \big(\sqrt{\log n}+b_{t,n}(M)\cdot n^{-D}\big),
	\end{align*}
	where $b_{n,t}^2(M)\equiv \E \pnorm{\mu^{(t)}_\sigma(M)}{}^2/n+ \E \pnorm{\hat{\mu}_\sigma(M)}{}^2/n$. Now let us get a (crude) bound on $b_{n,t}^2(M)$. First, using the optimality of $\hat{\mu}_\sigma(M)$, we have
	\begin{align*}
	\frac{ \pnorm{\hat{\mu}_\sigma}{}^2}{n}\lesssim \frac{1}{n}\sum_{i \in [m]} \mathsf{L}_\sigma(0,A_i^\top \mu_0;\xi_i)\, \implies\, \frac{\E \pnorm{\hat{\mu}_\sigma}{}^2}{n}\leq C_3. 
	\end{align*}
	Next, using the definition (\ref{def:ERM_grad_descent_logit}), we have for $\eta\leq 1$,
	\begin{align*}
	\pnorm{\mu^{(t)}}{}\leq \pnorm{\mu^{(t-1)}}{}+\sum_{i \in [m]} \pnorm{A_i}{}\leq \cdots \leq t \sum_{i \in [m]} \pnorm{A_i}{}\,\implies\, \E \pnorm{\mu^{(t)}}{}^2\leq C_3 t^2 n^2. 
	\end{align*}
	Combining the above three displays and adjusting constants, we have for $M\in \{A,B\}$, $\eta\leq 1/C_4$ and $t\leq n$, 
	\begin{align}\label{ineq:universality_logistic_2}
	\E \biggabs{\frac{1}{n}\sum_{j \in [n]} \psi\big(\mu_{\sigma,j}^{(t)}(M)\big)-\frac{1}{n}\sum_{j \in [n]} \psi\big(\hat{\mu}_{\sigma,j}(M)\big) }\leq C_4 \cdot  \big[\log n (1-\delta_0 \eta)^t+n^{-D}\big]. 
	\end{align}
	Finally, using Lemma \ref{lem:logistic_smooth_error} and $\sup_{\sigma>0}\E \pnorm{\hat{\mu}_\sigma}{}^2\vee \E \pnorm{\hat{\mu}}{}^2\leq C_5 n$, for $M \in \{A,B\}$, 
	\begin{align}\label{ineq:universality_logistic_3}
	&\E \biggabs{\frac{1}{n}\sum_{j \in [n]} \psi\big(\hat{\mu}_{j}(M)\big)-\frac{1}{n}\sum_{j \in [n]} \psi\big(\hat{\mu}_{\sigma,j}(M)\big) }\nonumber\\
	&\lesssim \frac{1}{n}\cdot  \E \pnorm{\hat{\mu}(M)-\hat{\mu}_\sigma(M) }{}\cdot \big(\sqrt{n} +\E \pnorm{\hat{\mu}(M)}{}+ \E \pnorm{\hat{\mu}_\sigma(M)}{}\big)\nonumber\\
	&\leq C_5\cdot \big[(\sigma \log n)^{1/4}+n^{-D}\big].
	\end{align}
	Now combining (\ref{ineq:universality_logistic_1})-(\ref{ineq:universality_logistic_3}), for $\sigma \in (0,1)$, $\eta=1/C_6$ and $t\leq n$,
	\begin{align*}
	&\E\biggabs{\frac{1}{n}\sum_{j \in [n]} \psi\big(\hat{\mu}_j(A)\big)-\frac{1}{n}\sum_{j \in [n]} \psi\big(\hat{\mu}_j(B)\big) }\\
	&\lesssim (\sigma \log n)^{1/4}+ \log n\cdot (1-\delta_0/C_6)^t + \big(C_6 \sigma^{-1}\log n\big)^{C_6 t^3}\cdot n^{-1/C_6 t^3}. 
	\end{align*}
	Fix $\epsilon \in (0,1/7)$ to be chosen later. For the choice $t\equiv (\log n)^\epsilon$, $\log(1/\sigma)=(\log n)^{1-7\epsilon}$, the above bound reduces to $C \exp\big(-\delta' (\log n)^{\min\{\epsilon,1-7\epsilon\}}\big)$. Now we may choose $\epsilon=1/8$ to conclude. 
\end{proof}

\section{Proof of Theorem \ref{thm:grad_descent}}\label{section:proof_grad_descent}

For notational simplicity, we work with the full sample case $S_\cdot\equiv [m]$.

\noindent (\textbf{Step 1}). The gradient descent algorithm (\ref{def:ERM_stoc_grad_descent_general_loss}) can be reformulated in the form of the GFOM iterate (\ref{def:GFOM_asym}) via the following identification: Let $\mu^{(t)}-\mu_0\equiv v^{(t)}$, and 
\begin{align*}
\begin{cases}
u^{(t)}= A v^{(t-1)}\in \R^m,\\
v^{(t)}= A^\top \big[\eta \mathsf{L}'(\xi-u^{(t)})\big]+(1-\eta \lambda) v^{(t-1)}-\eta\lambda \mu_0\in \R^n,
\end{cases}
\end{align*}
with the initialization $u^{(0)}=0_m,v^{(0)}=-\mu_0$. In other words, we may take
\begin{align*}
&\mathsf{F}_t^{\langle 1\rangle}(v^{([0:t-1])})\equiv v^{(t-1)},\,\mathsf{F}_t^{\langle 2\rangle}(v^{([0:t-1])})\equiv (1-\eta \lambda)v^{(t-1)}-\eta\lambda \mu_0,\\
& \mathsf{G}_t^{\langle 1\rangle}(u^{([0:t-1])})\equiv 0,\, \mathsf{G}_t^{\langle 2\rangle}(u^{([0:t])})\equiv \eta \mathsf{L}'(\xi-u^{(t)})
\end{align*}
in the GFOM iterate (\ref{def:GFOM_asym}). Consequently, the transformation maps $\Phi_{t}: \R^{m\times [0:t]}\to \R^{m\times [0:t]},\Xi_{t}:\R^{n\times [0:t]}\to \R^{n\times [0:t]}$ are specified as follows. Let $\Phi_0\equiv \mathrm{id}(\R^m)$, $\Xi_0\equiv \mathrm{id}(\R^n)$, and $\mathfrak{U}^{(0)}=u^{(0)}=0_m$, $\mathfrak{V}^{(0)}\equiv v^{(0)}=-\mu_0$. For $t=1,2,\ldots$, execute the following steps:
\begin{enumerate}
	\item[(G1)] Let $\Phi_{t}:\R^{m\times [0:t]}\to \R^{m\times [0:t]}$ be defined as follows: for $w \in [0:t-1]$, $\big[\Phi_{t}(\mathfrak{u}^{([0:t])})\big]_{\cdot,w}\equiv \big[\Phi_{w}(\mathfrak{u}^{([0:w])})\big]_{\cdot,w}$, and for $w=t$,
	\begin{align*}
	\big[\Phi_{t}(\mathfrak{u}^{([0:t])})\big]_{\cdot,t} \equiv \mathfrak{u}^{(t)}+\eta\sum_{s \in [1:t-1]} \mathfrak{f}_{s}^{(t-1) }\circ  \mathsf{L}'\big(\xi-\big[\Phi_{s}(\mathfrak{u}^{([0:s])})\big]_{\cdot,s}\big),
	\end{align*}
	where the coefficient vectors $\{\mathfrak{f}_{s}^{(t-1) } \}_{s \in [1:t-1]}\subset \R^m$ are determined by
	\begin{align*}
	\mathfrak{f}_{s,k}^{(t-1) } \equiv \sum_{\ell \in [n]} \E A_{k\ell}^2\cdot \E^\xi \partial_{\mathfrak{V}_\ell^{(s)}}\bigiprod{e_\ell}{\big[ \Xi_{t-1} (\mathfrak{V}^{([0:t-1])})\big]_{\cdot,t-1}}, \quad  k \in [m].
	\end{align*}
	\item[(G2)] Let the Gaussian law of $\mathfrak{U}^{(t)}$ be determined via the following correlation specification: for $s \in [1:t]$ and $k \in [m]$,
	\begin{align*}
	\mathrm{Cov}\big(\mathfrak{U}_k^{(t)}, \mathfrak{U}_k^{(s)} \big)\equiv \sum_{\ell \in [n]} \E A_{k\ell}^2\cdot \E^\xi \prod_{\tau \in \{t,s\}} \bigiprod{e_\ell}{\big[ \Xi_{\tau-1} (\mathfrak{V}^{([0:\tau-1])})\big]_{\cdot,\tau-1}}.
	\end{align*}
	\item[(G3)] Let $\Xi_{t}:\R^{n\times [0:t]}\to \R^{n\times [0:t]}$ be defined as follows: for $w \in [0:t-1]$, $\big[\Xi_{t}(\mathfrak{v}^{([0:t])})\big]_{\cdot,w}\equiv \big[\Xi_{w}(\mathfrak{v}^{([0:w])})\big]_{\cdot,w}$, and for $w=t$,
	\begin{align*}
	\big[\Xi_{t}(\mathfrak{v}^{([0:t])})\big]_{\cdot,t} &\equiv \mathfrak{v}^{(t)}+\sum_{s \in [1:t]} \mathfrak{g}_{s}^{(t)}\circ \big[\Xi_{s-1}(\mathfrak{v}^{([0:s-1])}) \big]_{\cdot,s-1}\\
	&\qquad \qquad +(1-\eta\lambda)\cdot\big[\Xi_{t-1}(\mathfrak{v}^{([0:t-1])}) \big]_{\cdot,t-1}-\eta\lambda \mu_0,
	\end{align*}
	where the coefficient vectors $\{\mathfrak{g}_{s}^{(t)}\}_{s \in [1:t]}\subset \R^n$ are determined via
	\begin{align*}
	\mathfrak{g}_{s,\ell}^{(t)}\equiv \eta \sum_{k \in [m]} \E A_{k\ell}^2\cdot  \E^\xi \partial_{\mathfrak{U}_k^{(s)}} \bigiprod{e_k}{ \mathsf{L}'\big(\xi-\big[\Phi_{t} (\mathfrak{U}^{([0:t])})\big]_{\cdot,t}\big)},\quad \ell \in [n].
	\end{align*}
	\item[(G4)] Let the Gaussian law of $\mathfrak{V}^{(t)}$ be determined via the following correlation specification: for $s \in [1:t]$ and $\ell \in [n]$,
	\begin{align*}
	\mathrm{Cov}(\mathfrak{V}_\ell^{(t)},\mathfrak{V}_\ell^{(s)})\equiv \eta^2 \sum_{k \in [m]} \E A_{k\ell}^2\cdot  \E^\xi \prod_{\tau \in \{t,s\}} \bigiprod{e_k}{ \mathsf{L}'\big(\xi-\big[\Phi_{\tau} (\mathfrak{U}^{([0:\tau])})\big]_{\cdot,\tau}\big)}.
	\end{align*}
\end{enumerate}
Note that the above evolution only depends on $\Phi_t$ only via its last column, we may then identify $\Phi_t$ as its restriction to its last column.

\noindent (\textbf{Step 2}). In this step, we provide correspondence of (G1)-(G4) above to (1)-(4) in Definition \ref{def:grad_descent_se} as in the statement of the theorem. A major simplification in this case is to note the row-wise linearity of $\big\{\big[\Xi_{t}(\mathfrak{v}^{([0:t])})\big]_{\cdot,t}\big\}$. In the sequel, we always use the indices $k \in [m]$ and $\ell \in [n]$. Let $\mathsf{M}_\ell^{\mathfrak{V}}\in \R^{[0:\infty)\times [0:\infty)}$ be a matrix such that
\begin{align}\label{ineq:grad_descent_0}
\bigiprod{e_\ell}{\big[\Xi_{t}(\mathfrak{v}^{([0:t])})\big]_{\cdot,t}}\equiv \sum_{s \in [0:t]} (\mathsf{M}_\ell^{\mathfrak{V}})_{s,t} \mathfrak{v}^{(s)}_\ell = \bigiprod{e_t}{(\mathsf{M}_\ell^{\mathfrak{V}})^\top  \mathfrak{v}_\ell^{[0:\infty)}  }.
\end{align}
Clearly, $
\mathfrak{f}_{s,k}^{(t)} = \sum_{\ell \in [n]} \E A_{k\ell}^2\cdot (\mathsf{M}_{\ell}^{\mathfrak{V}})_{s,t}$ holds for $s \in [1:t]$. 
Recall $\Sigma^{\mathfrak{U}}_k\in \R^{[1:\infty)\times [1:\infty)}, \Sigma^{\mathfrak{V}}_\ell \in \R^{[0:\infty)\times [0:\infty)}$, which are understood as the covariance matrices associated with the Gaussian vectors $\mathfrak{U}^{([1:\infty))}_k, \mathfrak{V}^{([0:\infty))}_\ell\in \R^\infty$ (except for $\mathfrak{V}^{(0)}$). 	

Now we shall convert the above (G1)-(G4) into the stated recursion. First, by using (\ref{ineq:grad_descent_0}), (G2) can be rewritten as follows: for $s \in [1:t]$,
\begin{align*}
(\Sigma^{\mathfrak{U}}_k)_{t,s}&\equiv \mathrm{Cov}\big(\mathfrak{U}^{(t)}_k, \mathfrak{U}^{(s)}_k \big) = \sum_{\ell \in [n]} \E A_{k\ell}^2\cdot \bigiprod{(\mathsf{M}_{\ell}^{\mathfrak{V}})_{\cdot,t-1} }{\Sigma_\ell^{\mathfrak{V}} (\mathsf{M}_{\ell}^{\mathfrak{V}})_{\cdot,s-1}}.
\end{align*}
This corresponds to (1).

Next, (G1) corresponds to the identity in (2). Moreover, with 
\begin{align*}
W^{(t)}&\equiv W^{(t)}(\mathfrak{U}^{([0:t])})=\mathsf{L}'' \big(\xi-\big[\Phi_{t}(\mathfrak{U}^{([0:t])})\big]_{\cdot,t}\big) \in \R^m,\\
D_{s}^{(t)}&\equiv \Big(\partial_{\mathfrak{U}_k^{(s)}}\bigiprod{e_k}{\big[\Phi_{t}(\mathfrak{U}^{([0:t])})\big]_{\cdot,t}}\Big)_{k \in [m]} \in \R^m,
\end{align*}
(G1) implies that
\begin{align}\label{ineq:grad_descent_1}
D_{s}^{(t)}& = \delta_{s,t} 1_m - \eta \sum_{r \in [s:t-1]} \mathfrak{f}_{r}^{(t-1) } \circ W^{(r)} \circ D_{s}^{(r)}.
\end{align}
Applying (\ref{ineq:grad_descent_1}) to the term $D_{s}^{(r)}$ on the far right hand side of the above display, the left hand side of (\ref{ineq:grad_descent_1}) is equal to 
\begin{align*}
& \delta_{s,t} 1_m - \eta \sum_{r_1 \in [s:t-1]} \mathfrak{f}_{r_1}^{(t-1) } \circ W^{(r_1)}\circ \bigg( \delta_{s,r_1} 1_m - \eta \sum_{r_2 \in [s:r_1-1]} \mathfrak{f}_{r_2}^{(r_1-1) } \circ W^{(r_2)}\circ D_{s}^{(r_2)}  \bigg)\\
& = \delta_{s,t}1_m - \eta \cdot \mathfrak{f}_{s}^{(t-1) }\circ W^{(s)}+\eta^2 \sum_{s\leq r_2<r_1\leq t-1} \mathfrak{f}_{r_1}^{(t-1) }\circ \mathfrak{f}_{r_2}^{(r_1-1) }\circ W^{(r_1)} \circ W^{(r_2)} \circ D_{s}^{(r_2)}.
\end{align*}
From here, we may again apply (\ref{ineq:grad_descent_1}) to the term $D_{s}^{(r_2)}$ on the far right hand side of the above display, which shows that $D_{s}^{(t)}$ is equal to 
\begin{align*}
& \delta_{s,t}1_m - \eta\cdot\mathfrak{f}_{s}^{(t-1) }\circ W^{(s)}+ \eta^2\sum_{s<r_1\leq t-1} \mathfrak{f}_{r_1}^{(t-1) }\circ \mathfrak{f}_{s}^{(r_1-1) }\circ W^{(r_1)}\circ W^{(s)}\\
&\quad - \eta^3 \sum_{s\leq r_3< r_2<r_1\leq t-1} \mathfrak{f}_{r_1}^{(t-1) }\circ \mathfrak{f}_{r_2}^{(r_1-1) }\circ \mathfrak{f}_{r_3}^{(r_2-1) }\circ W^{(r_1)} \circ W^{(r_2)}  \circ W^{(r_3)} \circ D_{s}^{(r_3)}.
\end{align*}
Iterating this procedure, we finally arrive at the formula
\begin{align}\label{ineq:grad_descent_2}
D_{s}^{(t)}&= \delta_{s,t}1_m-\eta\cdot\mathfrak{f}_{s}^{(t-1) }\circ W^{(s)}\nonumber\\
&+\sum_{\tau \in [1:t-s-1]}(-\eta)^{\tau+1}  \sum_{s+1\leq r_\tau<r_{\tau-1}<\cdots<r_1\leq t-1} \bigodot_{\iota \in [1:\tau+1]}  \mathfrak{f}_{r_\iota}^{(r_{\iota-1}-1) } \circ W^{(r_\iota)}.
\end{align}
Using the above formula (\ref{ineq:grad_descent_2}), 
\begin{align*}
\mathfrak{g}_{s,\ell}^{(t)}&\equiv -\eta \sum_{k \in [m]} \E A_{k\ell}^2\cdot  \E^\xi   W_k^{(t)} \bigg\{ \delta_{s,t}-\eta\cdot\mathfrak{f}_{s,k}^{(t-1) } W_k^{(s)}\\
&\qquad\qquad +\sum_{\tau \in [1:t-s-1]}(-\eta)^{\tau+1}  \sum_{r_{[1:\tau]}\subset [s+1: t-1]} \prod_{\iota \in [1:\tau+1]}  \mathfrak{f}_{r_\iota,k}^{(r_{\iota-1}-1) } W_k^{(r_\iota)}\bigg\}.
\end{align*}
This gives the second identity in (3).

For (G3), using (\ref{ineq:grad_descent_0}) it can be rewritten as follows:
\begin{align*}
&\sum_{r \in [0:t]} (\mathsf{M}_\ell^{\mathfrak{V}})_{r,t} \mathfrak{v}^{(r)}_\ell= \bigiprod{e_\ell}{\big[\Xi_{t}(\mathfrak{v}^{([0:t])})\big]_{\cdot,t}} \\
& = \mathfrak{v}_\ell^{(t)}+\sum_{s \in [1:t]} \mathfrak{g}_{s,\ell}^{(t)}  \sum_{r \in [0:s-1]} (\mathsf{M}_\ell^{\mathfrak{V}})_{r,s-1} \mathfrak{v}^{(r)}_\ell  +(1-\eta \lambda) \sum_{r \in [0:t-1]} (\mathsf{M}_\ell^{\mathfrak{V}})_{r,t-1} \mathfrak{v}^{(r)}_\ell +\eta \lambda \mathfrak{v}^{(0)}_\ell\\
&=\eta \lambda \mathfrak{v}^{(0)}_\ell+\sum_{r \in [0:t-1]} \bigg(\sum_{s \in [r+1:t]} \mathfrak{g}_{s,\ell}^{(t)}  (\mathsf{M}_{\ell}^{\mathfrak{V}})_{r,s-1}  +(1-\eta \lambda)\cdot (\mathsf{M}_\ell^{\mathfrak{V}})_{r,t-1}\bigg)\cdot \mathfrak{v}^{(r)}_\ell+ \mathfrak{v}^{(t)}_\ell.
\end{align*}
This corresponding to the formula in (3).

Finally, (G4) corresponds to the formula in (4).

\noindent (\textbf{Step 3}).  By the identification between the gradient descent of the GFOM, we have 
\begin{align*}
\mu^{(t)}_\ell-\mu_{0,\ell}&= v^{(t)}_\ell=\bigiprod{e_\ell}{\big[\Xi_{t}(\mathfrak{v}^{([0:t])})\big]_{\cdot,t}}= \sum_{s \in [0:t]} (\mathsf{M}_{\ell}^{\mathfrak{V}})_{s,t} \mathfrak{v}^{(s)}_\ell.
\end{align*}
Using the state evolution for $\{\mathfrak{v}^{(t)}\}_{t\geq 1}$ in Theorem \ref{thm:GFOM_se_asym}, we have
\begin{align*}
\mu^{(t)}_\ell\approx \mathcal{N}\Big(\big(1-(\mathsf{M}_{\ell}^{\mathfrak{V}})_{0,t}\big)\cdot \mu_{0,\ell}, \bigiprod{(\mathsf{M}^{\mathfrak{V}}_\ell)_{{[1:\infty),t}}}{(\Sigma_\ell^{\mathfrak{V}})_{[1:\infty)^2} (\mathsf{M}^{\mathfrak{V}}_\ell)_{{[1:\infty),t}} }\Big),
\end{align*}
where $\approx$ contains relevant error terms in the formulation of Theorem \ref{thm:GFOM_se_asym}. \qed

\appendix

\section{Auxiliary results}

\begin{lemma}\label{lem:count_total_num}
	For any $V \in \N$, $n \in \N$ and $S \subset [n]$,
	\begin{align*}
	\sum_{\ell_{[V]} \in [n]} n^{- \abs{ \{S,\ell_{[V]}\} }}\leq (2\abs{S}V)^{V+1} n^{-\abs{S}}.
	\end{align*}
\end{lemma}
\begin{proof}
	Let $s=\abs{S}\leq n$. Then $
	\sum_{\ell_{[V]}\in [n]} n^{- \abs{ \{S, \ell_{[V]}\}  } } = \sum_{\ell_{[V]}\in [n]}  n^{ - \abs{ \{[s], \ell_{[V]}\}  }}$. 
	Using 
	\begin{align*}
	\sum_{\ell_{[v]}\in [n]} n^{- \abs{\{\ell_{[v]}\} }  }= \sum_{q= 1}^{v\wedge n} n^{-q} \bigg(\sum_{\ell_{[v]}\in [n],  \abs{\{\ell_{[v]}\} } =q  } 1\bigg)\leq \sum_{q=1}^v q^v\leq v^{v+1},
	\end{align*}
	we may then bound
	\begin{align*}
	\sum_{\ell_{[V]}\in [n]}  n^{ - \abs{ \{[s], \ell_{[V]}\}  } }  &=   \sum_{\mathcal{Q}\subset [V]} \sum_{\ell_{\mathcal{Q}}\in [s],\ell_{[V]\setminus \mathcal{Q}} \in [n]\setminus [s] } n^{- \abs{ \{\ell_{[V]\setminus \mathcal{Q}}\}  }-s}\\
	&\leq n^{-s}\sum_{\mathcal{Q}\subset [V]} s^{\abs{\mathcal{Q}}}\sum_{\ell_{[V]\setminus \mathcal{Q}} \in [n]\setminus [s] } (n-s)^{- \abs{ \{\ell_{[V]\setminus \mathcal{Q}}\}  }} \leq n^{-s} \cdot (2sV)^{V+1}.
	\end{align*}
	Here in the last inequality we used the estimates $\sum_{\ell_{[V]\setminus \mathcal{Q}} \in [n]\setminus [s] } (n-s)^{- \abs{ \{\ell_{[V]\setminus \mathcal{Q}}\}  }}\leq V^{V+1}$ and $\sum_{\mathcal{Q}\subset [V]} s^{\abs{\mathcal{Q}}}\leq s^V\cdot \big(\sum_{\mathcal{Q}\subset [V]} 1\big)= (2s)^V$.
\end{proof}

\begin{lemma}\label{lem:set_merge}
	For two sets $T,S$ with $\abs{T\cap S}\geq 1$, and another two elements $t,s$, 
	\begin{align*}
	\big(\abs{\{T,t\}}-1\big)+\big(\abs{\{S,s\}}-1\big)\geq \abs{\{T,S,t,s\} }-1.
	\end{align*}
\end{lemma}
\begin{proof}
	We discuss all possible scenarios below:
	\begin{itemize}
		\item Suppose $t \in T $. Then the LHS is equal to
		\begin{align*}
		\begin{cases}
		\abs{T}+\abs{S}-2\stackrel{(\ast)}{\geq} \abs{\{T,S\}}-1, & s \in S;\\
		\abs{T}+\abs{\{S,s\}}-2=\abs{T}+\abs{S}-1\geq \abs{\{T,S\}}-1, & s \in T\setminus S;\\
		\abs{T}+\abs{\{S,s\}}-2=\abs{T}+\abs{S}-1\stackrel{(\ast)}{\geq} \abs{\{T,S\}}= \abs{\{T,S,s\}}-1, & s \in (T\cup S)^c.
		\end{cases}
		\end{align*}
		\item Suppose $t \in S \setminus T$. Then the LHS is equal to
		\begin{align*}
		\begin{cases}
		\abs{\{T,t\}}+\abs{S}-2=\abs{T}+\abs{S}-1\geq \abs{\{T,S\}}-1, & s \in S;\\
		\abs{\{T,t\}}+\abs{\{S,s\}}-2=\abs{T}+\abs{S}\geq \abs{\{T,S\}}-1, & s \in T\setminus S;\\
		\abs{\{T,t\}}+\abs{\{S,s\}}-2=\abs{T}+\abs{S} \geq \abs{\{T,S\}}= \abs{\{T,S,s\}}-1, & s \in (T\cup S)^c.
		\end{cases}
		\end{align*}
		\item Suppose $t \in (T\cup S)^c$. Then the LHS is equal to
		\begin{align*}
		\begin{cases}
		\abs{\{T,t\}}+\abs{S}-2=\abs{T}+\abs{S}-1\stackrel{(\ast)}{\geq} \abs{\{T,S\}}=\abs{\{T,S,t\}}-1, & s \in S;\\
		\abs{\{T,t\}}+\abs{\{S,s\}}-2=\abs{T}+\abs{S}\geq \abs{\{T,S\}}=\abs{\{T,S,t\}}-1, & s \in T\setminus S;\\
		\abs{\{T,t\}}+\abs{\{S,s\}}-2=\abs{T}+\abs{S} \stackrel{(\ast)}{\geq} \abs{\{T,S\}}+1= \abs{\{T,S,t,s\}}-1, & s \in (T\cup S)^c.
		\end{cases}
		\end{align*}
	\end{itemize}
	Here in $(\ast)$ we used the condition $\abs{T\cap S}\geq 1$.
\end{proof}

\begin{lemma}\label{lem:prox_Lip}
	If $f:\R\to \R$ is $\alpha$-strongly convex, then $\prox_f:\R\to \R$ is $1/(1+\alpha)$-Lipschitz.
\end{lemma}
\begin{proof}
	The result is standard in convex analysis. We provide a proof for completeness. Let us fix $x_1,x_2\in \R$ and write $w_\ell\equiv \prox_f(x_\ell)$ for $\ell=1,2$. By first-order optimality, for any $x\in \R$, $x-\prox_f(x)\in \partial f(\prox_f(x))$. So there exists $g_\ell \in \partial f(w_\ell)$ such that $x_\ell-w_\ell=g_\ell$. In particular,
	\begin{align*}
	x_1-x_2 = w_1-w_2+(g_1-g_2). 
	\end{align*}
	On the other hand, using the subgradient inequality $
	f(w_\ell)-f(w_{3-\ell})\geq g_{3-\ell}(w_\ell-w_{3-\ell})+\frac{\alpha}{2} \big(w_\ell-w_{3-\ell}\big)^2$ that holds for both $\ell=1,2$, we have $
	(g_1-g_2)(w_1-w_2)\geq \alpha (w_1-w_2)^2$. 
	Combined with the above display, we have
	\begin{align*}
	(x_1-x_2) (w_1-w_2)\geq (1+\alpha) (w_1-w_2)^2. 
	\end{align*}
	The claim follows.
\end{proof}

\section*{Acknowledgments}
The author would like to thank Sheng Xu, Xiaocong Xu, Linjun Zhang, and three referees for numerous helpful comments and suggestions that significantly improved the quality of the paper.

\bibliographystyle{alpha}
\bibliography{mybib}

\end{document}